\documentclass[11pt]{article}
\RequirePackage{graphicx,wrapfig}
\RequirePackage{tikz,pgf}
\usepackage{amsmath}
\usepackage{amssymb}
\usepackage[citecolor=red]{hyperref}
\usepackage{latexsym}
\usepackage{amsmath, amsfonts,amssymb, amsthm, euscript,makeidx,color,mathrsfs}

\newtheorem{definition}{Definition}
\newtheorem{theorem}{Theorem}

\newtheorem{remark}{Remark}
\newtheorem{example}{Example}

\newtheorem{assumption}{Assumption}

\def\qed{ \ \vrule width.2cm height.2cm depth0cm\smallskip}

\newcommand{\la}{\langle}
\newcommand{\ra}{\rangle}

\newcommand{\dd}{\mathcal{\dagger}}
\newcommand{\ts}{\mathsf{T}}

\newcommand{\ba}{\begin{array}}
\newcommand{\ea}{\end{array}}
\newcommand{\be}{\begin{equation}}
\newcommand{\ee}{\end{equation}}
\newcommand{\bea}{\begin{eqnarray}}
\newcommand{\eea}{\end{eqnarray}}
\newcommand{\beaa}{\begin{eqnarray*}}
\newcommand{\eeaa}{\end{eqnarray*}}

\newcommand{\WH}{W_a}

\def\Vol{\pi}
\def\div{\mathbf{div}}

\def\D{\Delta}

\def\cC{{\cal C}}

\def\cF{{\cal F}}

\def\cH{{\cal H}}
\def\cI{{\cal I}}
\def\cJ{{\cal J}}
\def\cK{{\cal K}}

\def\cS{{\cal S}}

\def\hE{\mathbb{E}}

\def\hM{\mathbb{R}}
\def\hN{\mathbb{N}}

\def\hR{\mathbb{R}}

\def\pa{\partial}
\def\cd{\cdot}

\def\qed{ \hfill \vrule width.25cm height.25cm depth0cm\smallskip}

\newcommand{\basa}{\begin{assumption}}
\newcommand{\easa}{\end{assumption}}

\newcommand{\bas}{\begin{assum}}
\newcommand{\eas}{\end{assum}}

\def\1{{\bf 1}}

\def\:{\!:\!}

at 9pt

\begin{document}

\newtheorem{thm}{Theorem}[section]
\newtheorem{lem}[thm]{Lemma}
\newtheorem{nota}[thm]{Notation}
\newtheorem{cor}[thm]{Corollary}
\newtheorem{prop}[thm]{Proposition}
\newtheorem{rem}[thm]{Remark}
\newtheorem{eg}[thm]{Example}
\newtheorem{defn}[thm]{Definition}
\newtheorem{assum}[thm]{Assumption}

\renewcommand {\theequation}{\arabic{section}.\arabic{equation}}
\def\thesection{\arabic{section}}

\title{Entropy dissipation for degenerate stochastic differential equations via sub-Riemannian density manifold} 
\author{ Qi Feng\thanks{\noindent Department of
Mathematics, University of Michigan, Ann Arbor, 48109; email: qif@umich.edu}
~ and ~ Wuchen Li\thanks{ \noindent Department of
Mathematics, University of South Carolina;
email: wuchen@mailbox.sc.edu.} 
}
\date{}
\maketitle
\abstract{We study the dynamical behaviors of degenerate stochastic differential equations (SDEs). We select an auxiliary Fisher information functional as the Lyapunov functional. Using generalized Fisher information, we conduct the Lyapunov exponential convergence analysis of degenerate SDEs. We derive the convergence rate condition by generalized Gamma calculus. Examples of the generalized Bochner's formula are provided in the Heisenberg group, displacement group, and Martinet sub-Riemannian structure. We show that the generalized Bochner's formula follows a generalized second-order calculus of Kullback-Leibler divergence in density space embedded with a sub-Riemannian type optimal transport metric.}



\noindent\textbf{Thanks}: {Wuchen Li is supported by AFOSR MURI FA9550-18-1-0502, AFOSR YIP award 2023, and NSF RTG: 2038080.}

\tableofcontents

\section{Introduction} 
Consider the following Stratonovich stochastic differential equation:
\begin{equation}\label{SDE framework1}
dX_t=b(X_t)dt+ \sqrt{2} a(X_t)\circ dB_t,    
\end{equation}
where $(B_t^1,B_t^2,\cdots, B_t^n)$ is { a $n$-dimensional Brownian motion in $\hR^n$, $a\in\hR^{n+m}\rightarrow\mathbb{R}^{(n+m)\times n}$ is a matrix-valued function, and $b:\hR^{n+m}\rightarrow  \mathbb{R}^{n+m}$ is a drift vector field.}
{The convergence analysis of SDE \eqref{SDE framework1} to its invariant distribution lies in the intersection of differential geometry, analysis, Lie group (subgroup in quantum mechanics), and probability. The convergence analysis also has broad applications in designing fast algorithms in artificial intelligence (AI), and Bayesian sampling/optimization problems. One key question arises: how fast does the probability density function of SDE \eqref{SDE framework1} converge to its invariant distribution?}

The Gamma calculus, also named Bakry-{\'E}mery iterative calculus \cite{bakryemery1985}, provides analytical approaches to derive the convergence rate for SDE \eqref{SDE framework1}. This lower bound is known as the Ricci curvature lower bound.  However, classical studies are limited to the non-degenerate diffusion coefficient matrix $a$. The classical Gamma calculus is no longer valid when $a$ is a degenerate matrix function; see the generalization of Bakry-{\'E}mery calculus in \cite{BaudoinGarofalo09}.

This paper presents a Lyapunov convergence analysis for the degenerate diffusion process. We select a class of $z$-Fisher information as the Lyapunov functional, where $z$ is a matrix function different from matrix $a$. We derive a generalized Gamma calculus by the dissipation of Lyapunov functional along the diffusion process. We then derive the generalized Bochner's formula and obtain the exponential convergence condition. Several concrete examples are presented: gradient-drift diffusions on the Heisenberg group, the displacement group, and the Martinet sub-Riemannian structure. Our approach extends the classical optimal transport geometry, in particular, the second-order calculus of relative entropy in density manifold studied in \cite{AC, LiG, otto2001, OV}. 

{The  generalized Gamma calculus is first introduced by Baudoin-Garofalo \cite{BaudoinGarofalo09} for sub-Riemannian manifolds. Related results are studied later in \cite{Baudoin2017, BBG, baudoin2019gamma, BGK, baudoinwang2012, Feng, Grong_2015_1, Grong_2015_2}. The commutative property of iteration of $\Gamma_1$ and $\Gamma_1^z$ (Hypothesis $1.2$ in \cite{BaudoinGarofalo09}) is crucial in the previous works.} Our algebraic condition \ref{assumption:main result} { does} not have this requirement. We can remove this commutative condition in {the} weak sense. Thus our results go beyond the step two-bracket generating condition. We present algebraic conditions for the existence of the generalized Bochner's formula. 
 
On the other hand, optimal transport on the sub-Riemannian manifold has been studied by \cite{agrachev2009optimal, figalli2010mass, Juilleet,KL}. 
An optimal transport metric on a sub-Riemannian manifold is proposed in \cite{Juilleet,KL}. In this case, the density manifold still forms an infinite-dimensional Riemannian manifold. The Monge-Amp{\`e}r{e} equation in sub-Riemannian settings is studied in \cite{figalli2010mass}. Our approach is different. We introduce the sub-Riemannian density manifold (SDM) and study its second-order geometric calculations of relative entropies in SDM. Using those, we propose new Gamma z calculus for degenerate stochastic differential equations and establish the generalized curvature dimension type bound. Besides, \cite{Lott_Villani, Strum} use the analytical property of optimal transport to formulate Ricci curvature lower bound in general metric space. Different from \cite{Lafferty, Lott_Villani, Strum}, we focus on the geometric calculations in density manifold introduced by the $z$ direction. Following the second-order geometric calculations in the density manifold, we formulate new Gamma calculus and the corresponding Ricci curvature tensor for the sub-Riemannian manifold. Besides, our derivation also relates to the entropy methods \cite{jungel2016entropy, MV}. Using entropy methods, \cite{ARNOLD20186843, AE} derive the convergence rate for degenerate drift-diffusion processes with constant diffusion coefficients $a$. Compared to previous works, we apply the entropy method with Gamma calculus and geometric calculations in density manifold. It derives a generalized Gamma calculus from the dissipation of auxiliary Fisher information. Several concrete examples of convergence conditions are derived in  Lie group-induced drift-diffusion processes.

We organize the paper below. We introduce the main result in Section \ref{section 2}. It is an explicit convergence rate condition for {the density of} degenerate SDEs in $L^1$ distance. In Section \ref{sec: example}, we provide three examples of the proposed convergence analysis, including gradient-drift diffusions on the Heisenberg group, the displacement group, and the Martinet sub-Riemannian structure. In Section \ref{section 4}, we present the Lyapunov analysis in the sub-Riemannian density manifold. The generalized Gamma calculus and { the proof of the generalized Bochner's formula} is presented in Section \ref{sec: generalized gamma z}. Some further discussions { for other functional inequalities} are presented in Section \ref{sec:func inq}.

\section{Main results}\label{section 2}
In this section, we present this paper's setting and main results. 
\subsection{Setting}
Consider a Stratonovich SDE 
\bea\label{SDE framework with drift}
dX_t=b(X_t)dt+ \sqrt{2} a(X_t)\circ dB_t,
\eea
 where $(B_t^1,B_t^2,\cdots, B_t^n)$ is a $n$-dimensional Brownian motion in $\hR^n$, $a: \mathbb R^{n+m}\rightarrow  \hR^{(n+m)\times n}$ is a matrix valued function, and $b: \hR^{n+m}\rightarrow \hR^{n+m}$ is a vector field. { We refer the definition of Stratonovich SDE to \cite{Karatzas}[Section 3.13]. According to \cite{baudoinflow}[Appendix A.7],  the SDE \eqref{SDE framework with drift} can also be written as the following It\^o SDE: }
\begin{equation}
 dX_{l,t}=\tilde b_l(X_t)dt+\sum_{i=1}^n  \sqrt{2} a_{li}(X_t) dB_t^i,\quad \text{for}\quad l =1,\cdots, n+m,
 \end{equation}
 where 
\begin{equation}
\tilde b_l=b_l+(\sum_{i=1}^n\nabla_{a_i}a_i)_l,\quad \text{for}\quad l=1,\cdots, n+m. 
 \end{equation}
{We denote $\{a_1,\cdots, a_n\}$ as the column vectors of matrix $a$}, and  $\sum_{i=1}^n\nabla_{a_i}a_i\in\mathbb{R}^{n+m}$ represents,  
\begin{equation}
(\sum_{i=1}^n\nabla_{a_i}a_i)_l=\sum_{i=1}^n\sum_{k=1}^{n+m} a_{ki}\frac{\partial a_{li}}{\partial x_k},\quad \text{for} \quad l=1,\cdots, n+m.    
\end{equation}
{ We denote $a^{\ts}$ as the transpose of matrix $a$, and denote $\{a^{\ts}_1,\cdots, a^{\ts}_n\}$ as the row vectors of matrix $a^{\ts}$. In particular, we have $a_{\hat i i}=a^{\ts}_{i\hat i}$, for $i=1,\cdots, n$ and $\hat i = 1,\cdots, n+m$. With some abuse of notation, we also denote $a^{\ts}_i$ as the vector fields corresponding to the row vectors $a^{\ts}_i$, for $i=1,\cdots, n$.
 We assume that $\{a^{\ts}_1(x),a^{\ts}_2(x),\cdots,a^{\ts}_n(x)\}$ satisfies the strong H\"ormander condition (or bracket generating condition):
 \begin{equation*} 
    \mathrm{Span}\Big\{a^{\ts}_1(x),\cdots, a^{\ts}_n(x),~ [a^{\ts}_{i_1},\cdots,[a^{\ts}_{i_{k-1}},a^{\ts}_{i_k}]\cdots](x),~1\leq i_1,\cdots,i_k\leq n, k\geq 2 \Big\}=\hR^{n+m},
\end{equation*}
where $[\cdot , \cdot ]$ represents the Lie bracket between two vector fields. The strong H\"ormander condition means that the Lie algebra generated by the vector fields $\{a^{\ts}_1(x),\cdots,a^{\ts}_n(x)\}$ is of full rank at every point $x\in\hR^{n+m}$ (see e.g.: \cite{Stroock}[Section 7.4]). This condition ensures the existence of a smooth probability density function of SDE \eqref{SDE framework with drift}, see the original proofs in \cite{Hormander, Bismut}. For simplicity of presentation, we assume the probability density function is strictly positive. Indeed, the positivity of the density follows from the H\"ormander condition \cite{benarous}, for the more technical conditions to show the positivity by using Malliavin calculus, we refer to \cite{barlownualart} and  \cite{BNOT}[Theorem 1.4 with $H$=1/2]}. Denote $X_t\sim \rho(t,x)$, where $\rho=\rho(t,x)$ is the probability density function of SDE \eqref{SDE framework with drift}. The density function $\rho$ satisfies the Fokker-Planck equation of SDE \eqref{SDE framework with drift}   
{
\begin{equation}\label{FPE}
\partial_t\rho(t,x)=-\nabla_x\cdot(\rho(t,x)\tilde b(x))+\sum_{i=1}^{n+m}\sum_{j=1}^{n+m}\frac{\partial^2}{\partial x_i\partial x_j} \Big((a(x)a(x)^{\ts})_{ij}\rho(t,x)\Big),
\end{equation}}
with a smooth initial condition 
\begin{equation*}
\rho_0(x)=\rho(0,x), \quad \int_{\mathbb{R}^{n+m}} \rho_0(x) dx=1,\quad \rho_0(x)> 0.    
\end{equation*}
In this paper, we assume that SDE \eqref{SDE framework with drift} has a unique invariant symmetric measure $\mu$, where $d\mu=\pi(x)dx$ with $\pi\in C^{\infty}(\mathbb{R}^{n+m})$. 
Here $\pi$ solves the equilibrium of Fokker-Planck equation \eqref{FPE}: 
{\begin{equation*}
    -\nabla_x\cdot(\pi(x)\tilde b(x))+\sum_{i=1}^{n+m}\sum_{j=1}^{n+m}\frac{\partial^2}{\partial x_i\partial x_j}\Big((a(x)a(x)^{\ts})_{ij} \pi(x)\Big)=0. 
\end{equation*}}
We study a particular class of the vector field $b$ for a given invariant distribution $\pi$.

\noindent{\textbf{Assumption (Gradient flow formulation):}} Suppose that $b$, $a$, and $\pi$ satisfy the relation: 
\begin{equation}\label{assu} 
b=a\otimes \nabla a+aa^{\ts}\nabla \log \pi,  
\end{equation}
where $a\otimes \nabla a\in\mathbb{R}^{n+m}$ represents, for $\hat k = 1,\cdots, n+m$,
\bea \label{Notation A}
(a\otimes \nabla a)_{\hat k} &=&\sum_{k=1}^n\sum_{k'=1}^{n+m}a_{\hat kk}\frac{\partial}{\partial x_{k'}}a_{k'k}. 
\eea
{In the It\^o formulation, $\tilde b$, $a$, and $\pi$ satisfy 
\begin{equation*}
 \tilde b_l=(aa^{\ts}\nabla \log \pi)_l+(a\otimes \nabla a)_l+(\sum_{i=1}^n\nabla_{a_i}a_i)_l,   
\end{equation*}
for $l=1,\cdots, n+m$.} In this case, we can reformulate equation \eqref{FPE} as 
\begin{equation}\label{FPE_New}
    \pa_t \rho(t,x)=\nabla\cdot\Big(\rho(t,x)a(x)a(x)^{\ts}\nabla\log\frac{\rho(t,x)}{\pi(x)}\Big). 
\end{equation} 
We leave the derivation of the formula \eqref{FPE_New} in the appendix. If $\rho(t,x)=\pi(x)$, then $\log\frac{\rho(t,x)}{\pi(x)}=0$, and $\pi$ is an invariant density function for SDE \eqref{SDE framework with drift}.
In section \ref{section 4}, we demonstrate that the Fokker-Planck equation \eqref{FPE}, or its equivalent formulation \eqref{FPE_New}, forms a ``horizontal'' gradient flow in sub-Riemannian density manifold. We design a Lyapunov functional to study the convergence behavior of this ``horizontal'' gradient flow \eqref{FPE_New}. 
{
\begin{remark}
Formula \eqref{FPE_New} can be written as  
\begin{equation*}
\partial_t(\log\rho(t,x)-\log \pi(x))\rho(t,x)=\nabla\cdot\Big(\rho(t,x)a(x)a(x)^{\ts}\nabla\log\frac{\rho(t,x)}{\pi(x)}\Big).
\end{equation*}
It has a weak formulation that 
\begin{equation*}
  \int_{\mathbb{R}^{n+m}} (\partial_t\log\frac{\rho(t,x)}{\pi(x)}, \phi(x))\rho(t,x) dx=-\int_{\mathbb{R}^{n+m}} \Big(\nabla\phi(x), a(x)a(x)^{\ts}\nabla\log\frac{\rho(t,x)}{\pi(x)}\Big)\rho(t,x)dx, 
\end{equation*}
where $\phi\in C^\infty(\mathbb{R}^{n+m})$ is a smooth test function. 
\end{remark}}
\begin{remark}[Non-gradient flow drift]
In fact, the proposed method is not limited to the gradient flow assumption of the drift vector field $b$ in \eqref{assu}. See details in \cite{FengLi2021}.   
\end{remark}

\subsection{Main result}
We now briefly sketch the main results. 
Denote a sub-elliptic operator $L\colon C^{\infty}(\mathbb{R}^{n+m})\rightarrow C^{\infty}(\mathbb{R}^{n+m})$ as follows: 
\begin{equation*}
 L f=\nabla\cdot(aa^{\ts}\nabla f)-\la a\otimes \nabla a, \nabla f\ra_{\mathbb{R}^{n+m}}+\la b,\nabla f\ra_{\mathbb{R}^{n+m}},
\end{equation*}
where $f\in C^{\infty}(\mathbb{R}^{n+m})$.  
\begin{defn}[Generalized Gamma z calculus]\label{defn:tilde gamma 2 znew}
Consider a smooth matrix function $z: \hR^{n+m}\rightarrow \mathbb{R}^{(n+m)\times m}$. Denote Gamma one bilinear forms $\Gamma_{1},\Gamma_{1}^z\colon C^{\infty}(\mathbb{R}^{n+m})\times C^{\infty}(\mathbb{R}^{n+m})\rightarrow C^{\infty}(\mathbb{R}^{n+m}) $ as 
\begin{equation*}
\Gamma_{1}(f,g)=\la a^{\ts}\nabla f, a^{\ts}\nabla g\ra_{\hR^n},\quad \Gamma_{1}^z(f,g)=\la z^{\ts}\nabla f, z^{\ts}\nabla g\ra_{\hR^m}.
\end{equation*}
Define Gamma two bilinear forms $\Gamma_{2},\Gamma_{2}^{z,\Vol}\colon C^{\infty}(\mathbb{R}^{n+m})\times C^{\infty}(\mathbb{R}^{n+m})\rightarrow C^{\infty}(\mathbb{R}^{n+m}) $ as
\begin{equation*}
\Gamma_{2}(f,g)=\frac{1}{2}\Big[ L\Gamma_{1}(f,g)-\Gamma_{1}( Lf, g)-\Gamma_{1}(f,  Lg)\Big],    
\end{equation*}
and 
\bea
\label{old gamma two z}
{\Gamma}_{2}^{z,\Vol}(f,g)&=&\quad\frac{1}{2}\Big[ L\Gamma_1^{z}(f,g)-\Gamma_{1}^z( Lf,g)-\Gamma_{1}^z(f,  Lg)\label{Gz1}\Big]\\
&& \label{new term}+\div^{\Vol}_z\Big(\Gamma_{1,\nabla(aa^{\ts})}(f,g )\Big)-\div^{\Vol }_a\Big(\Gamma_{1,\nabla(zz^{\ts})}(f,g )\Big)\label{Gz2}.
  \eea 
Here $\div^{\Vol }_a$, $\div^{\Vol }_z$ are divergence operators defined by: 
\begin{equation*}
\div^{\Vol}_a(F)=\frac{1}{\Vol }\nabla\cdot(\Vol~ aa^{\ts} F), \quad\div^{\Vol }_z(F)=\frac{1}{\Vol }\nabla\cdot(\Vol ~zz^{\ts}F),
\end{equation*}
for any smooth vector field $F\in \mathbb{R}^{n+m}$, and $\Gamma_{\nabla (aa^{\ts})}$, $\Gamma_{\nabla (zz^{\ts})}$ are vector Gamma one bilinear forms defined as
\beaa
\Gamma_{1,\nabla(aa^{\ts)}}(f,g)&=&\la \nabla f,\nabla(aa^{\ts})\nabla g\ra=(\la \nabla f,\frac{\partial}{\partial x_{\hat k}}(aa^{\ts})\nabla g\ra)_{\hat k=1}^{n+m},\\
\Gamma_{1,\nabla(zz^{\ts)}}(f,g)&=&\la \nabla f,\nabla(aa^{\ts})\nabla g\ra=(\la \nabla f,\frac{\partial}{\partial x_{\hat k}}(zz^{\ts})\nabla g\ra)_{\hat k=1}^{n+m},
\eeaa
with 
	\beaa
\div^{\Vol }_z\Big(\Gamma_{\nabla(aa^{\ts})}f,g \Big)&=&\frac{\nabla\cdot(zz^{\ts}\Vol \la \nabla f,\nabla(aa^{\ts})\nabla g\ra ) }{\Vol },\\
\div^{\Vol }_a\Big(\Gamma_{\nabla(zz^{\ts})}f,g \Big)&=&\frac{\nabla\cdot(aa^{\ts}\Vol  \la \nabla f,\nabla(zz^{\ts})\nabla g\ra) }{\Vol }.
	\eeaa
\end{defn}

We next demonstrate that the summation of $\Gamma_2$ and $\Gamma_{2}^{z,\Vol}$ can induce the following decomposition and bilinear forms. They are naturally extensions of classical Bakry-{\'E}mery calculus in Riemannian manifold, i.e., non-degenerate matrix function $a$. 
\begin{nota} \label{notation}
For matrix function $a: \hR^{n+m}\rightarrow \hR^{(n+m)\times n}$, we define matrix $Q$ as
\bea\label{matrix Q}
Q=\begin{pmatrix}
	a^{\ts}_{11}a^{\ts}_{11}&\cdots &a^{\ts}_{1(n+m)}a^{\ts}_{1(n+m)}\\
	\cdots & a^{\ts}_{i\hat i}a^{\ts}_{k\hat k}&\cdots \\
	a^{\ts}_{n 1}a^{\ts}_{n 1} &\cdots &a^{\ts}_{n(n+m)}a^{\ts}_{n(n+m)} 
	\end{pmatrix}\in \mathbb{R}^{n^2\times(n+m)^2},
\eea
with $Q_{ik\hat i\hat k}=a^{\ts}_{i\hat i}a^{\ts}_{k\hat k}$. More precisely, for each row (resp. column) of $Q$, the row (resp. column)  indices of $ Q_{ik\hat i\hat k}$ follows $\sum_{i=1}^n\sum_{k=1}^n$  (resp. $\sum_{\hat i=1}^{n+1}\sum_{\hat k=1}^{n+m}$).  For matrix function $z: \hR^{n+m}\rightarrow \hR^{(n+m)\times m}$, we define matrix $P$ as 
\bea\label{matrix P}
P=\begin{pmatrix}
	z^{\ts}_{11}a^{\ts}_{11}&\cdots &z^{\ts}_{1(n+m)}a^{\ts}_{1(n+m)}\\
	\cdots & z^{\ts}_{i\hat i}a^{\ts}_{k\hat k}&\cdots \\
	z^{\ts}_{m\hat 1}a^{\ts}_{n\hat 1} &\cdots &z^{\ts}_{m(n+m)}a^{\ts}_{n(n+m)} 
	\end{pmatrix}\in \mathbb{R}^{(nm)\times(n+m)^2},
\eea
with { $P_{ik\hat i\hat k}=z^{\ts}_{k\hat k}a^{\ts}_{i\hat i}$}. For smooth function $f\in  C^{\infty}(\hR^{n+m})$,
for any $\hat i,\hat k,\hat j=1,\cdots,n+m$ and $i,k=1,\cdots,n$ $($or $1,\cdots, m)$,
we define vector $C\in\hR^{(n+m)^2\times 1}$ with components,
 \bea\label{vector C}
C_{\hat i\hat k}= \left[\sum_{i,k=1}^n\sum_{i'=1}^{n+m}\left(\la a^{\ts}_{i\hat i}a^{\ts}_{i i'} ( \frac{\partial a^{\ts}_{k\hat k}}{\partial x_{i'}}) ,(a^{\ts}\nabla)_kf\ra_{\hR^n} - \la a^{\ts}_{k i'}a^{\ts}_{i\hat k}\frac{\partial a^{\ts}_{i \hat i}}{\partial x_{i'}}   ,(a^{\ts}\nabla)_kf\ra_{\hR^n}\right)\right],
\eea
where we denote $(a^{\ts}\nabla)_kf=\sum_{k'=1}^{n+m}a^{\ts}_{kk'}\frac{\pa f}{\pa x_{k'}}$. We define vector $D\in\hR^{n^2\times 1}$ with components,
\bea\label{vector D}
D_{ik}=\sum_{\hat i,\hat k=1}^{n+m}a^{\ts}_{i\hat i}\frac{\partial a^{\ts}_{k\hat k}}{\partial x_{\hat i}}\frac{\partial f}{\partial x_{ \hat k}}, \quad \text{and} \quad D^{\ts}D=\sum_{i,k}D_{ik}D_{ik}.
\eea
 We define vector $F\in\hR^{(n+m)^2\times 1}$ with components,
 \bea\label{vector F}
F_{\hat i\hat k}= \left[\sum_{i=1}^n\sum_{k=1}^m\sum_{i'=1}^{n+m}\left(\la a^{\ts}_{i\hat i}a^{\ts}_{i i'} ( \frac{\partial z^{\ts}_{k\hat k}}{\partial x_{i'}}) ,(z^{\ts}\nabla)_kf\ra_{\hR^m} - \la z^{\ts}_{k i'}a^{\ts}_{i\hat k}\frac{\partial a^{\ts}_{i \hat i}}{\partial x_{i'}}   ,(z^{\ts}\nabla)_kf\ra_{\hR^m}\right)\right].
\eea
 We define vector $E\in\hR^{(n\times m)\times 1}$ with components,
\bea\label{vector E}
E_{ik}=\sum_{\hat i,\hat k=1}^{n+m}a^{\ts}_{i\hat i}\frac{\partial z^{\ts}_{k\hat k}}{\partial x_{\hat i}}\frac{\partial f}{\partial x_{ \hat k}}, \quad \text{and} \quad E^{\ts}E=\sum_{i,k}E_{ik}E_{ik}. 
\eea
We define vector $G\in\hR^{(n+m)^2\times 1}$ with components,
\bea\label{vector G}
G_{\hat i\hat j}&=& \sum_{i=1}^n\sum_{j=1}^m \sum_{j',\hat j,i',\hat i=1}^{n+m}\left[\left( z^{\ts}_{j\hat j} z^{\ts}_{j j'} \frac{\pa}{\pa x_{ j'}}a^{\ts}_{i\hat i}a^{\ts}_{ii'}\frac{\pa f}{\pa x_{i'}}+z^{\ts}_{j \hat j} z^{\ts}_{j j'} \frac{\pa}{\pa x_{ j'}}a^{\ts}_{i i'}\frac{\pa f}{\pa x_{ i'}}a^{\ts}_{i\hat i}\right)\right. \\
&&\left.\quad\quad\quad\quad\quad\quad-\left(a^{\ts}_{i\hat i}a^{\ts}_{i i'} \frac{\pa}{\pa x_{ i'}}z^{\ts}_{j\hat j}z^{\ts}_{jj'}\frac{\pa f}{\pa x_{j'}} +a^{\ts}_{i\hat i}a^{\ts}_{ii'} \frac{\pa}{\pa x_{ i'}}z^{\ts}_{j j'}\frac{\pa f}{\pa x_{ j'}}z^{\ts}_{j\hat j}\right)\right].\nonumber
\eea
We define $X$ as the vectorization of the Hessian matrix of function $f$: 
\bea\label{vector X}
X^{\ts}=\begin{pmatrix}
	\frac{\pa^2 f}{\pa x_{1} \pa x_{1}}&
	\cdots&
		\frac{\pa^2 f}{\pa x_{\hat i} \pa x_{\hat k}}&
		\cdots &
		\frac{\pa^2 f}{\pa x_{n+m} \pa x_{n+m}}
	\end{pmatrix}\in \mathbb{R}^{1\times (n+m)^2}.
\eea
\end{nota}{}

\begin{assum}\label{assumption:main result}
	Assume that there exists vectors $\Lambda_1, \Lambda_2\in \mathbb{R}^{(n+m)^2\times 1}$, such that
	\beaa
	(Q^{\ts}Q\Lambda_1+P^{\ts}P \Lambda_2)^{\ts}X&=&(F+C+G+Q^{\ts}D+P^{\ts}E)^{\ts}X.
	\eeaa
\end{assum}

\begin{definition}[Hessian matrix]
	\label{def: Ricci curvature}
	For smooth function $f\in C^{\infty}(\hR^{n+m})$, define a matrix function $\mathfrak R:\hR^{n+m}\rightarrow\hR^{(n+m)\times(n+m)}$ as 
	\bea
	\mathfrak R(\nabla f,\nabla f)=-\Lambda_1^{\ts}Q^{\ts}Q\Lambda_1-\Lambda_2^{\ts}P^{\ts}P\Lambda_2+D^{\ts}D+E^{\ts}E\nonumber 
	+(\mathfrak R_{ab}+\mathfrak R_{zb}+\mathfrak R_{\pi} ) (\nabla f,\nabla f),
	\eea 
where we define the following bilinear forms: 
\bea\label{new tensor R a with b}
\mathfrak{R}_{ab}(\nabla f,\nabla f)&=&\mathfrak{R}_{a}(\nabla f,\nabla f)-\sum_{i=1}^n \sum_{\hat i,\hat k=1}^{n+m}\la (a^{\ts}_{i\hat i}\frac{\pa b_{\hat k}}{\pa x_{\hat i}}\frac{\pa f}{\pa x_{\hat k}}-b_{\hat k}\frac{\pa a^{\ts}_{i\hat i}}{\pa x_{\hat k}} \frac{\pa f}{\pa x_{\hat i}} ),(a^{\ts}\nabla f)_i\ra_{\hR^n},\nonumber \\
\mathfrak{R}_a(\nabla f,\nabla f) \label{new tensor R a}
&=&\sum_{i,k=1}^n\sum_{i',\hat i,\hat k=1}^{n+m} \la a^{\ts}_{ii'} (\frac{\partial a^{\ts}_{i \hat i}}{\partial x_{i'}} \frac{\partial a^{\ts}_{k\hat k}}{\partial x_{\hat i}}\frac{\partial f}{\partial x_{\hat k}}) ,(a^{\ts}\nabla)_kf\ra_{\hR^n}  \nonumber\\
&&+\sum_{i,k=1}^n\sum_{i',\hat i,\hat k=1}^{n+m} \la a^{\ts}_{ii'}a^{\ts}_{i \hat i} (\frac{\partial }{\partial x_{i'}} \frac{\partial a^{\ts}_{k\hat k}}{\partial x_{\hat i}})(\frac{\partial f}{\partial x_{\hat k}}) ,(a^{\ts}\nabla)_kf\ra_{\hR^n} \nonumber\\
&&-\sum_{i,k=1}^n\sum_{i',\hat i,\hat k=1}^{n+m} \la a^{\ts}_{k\hat k}\frac{\partial a^{\ts}_{ii'}}{\partial x_{\hat k}} \frac{\partial a^{\ts}_{i \hat i}}{\partial x_{i'}} \frac{\partial f}{\partial x_{\hat i}})  ,(a^{\ts}\nabla)_kf\ra_{\hR^n}  \nonumber\\
&&-\sum_{i,k=1}^n\sum_{i',\hat i,\hat k=1}^{n+m} \la a^{\ts}_{k\hat k} a^{\ts}_{ii'} (\frac{\partial }{\partial x_{\hat k}} \frac{\partial a^{\ts}_{i \hat i}}{\partial x_{i'}}) \frac{\partial f}{\partial x_{\hat i}}  ,(a^{\ts}\nabla)_kf\ra_{\hR^n},\nonumber
\eea 
\bea 
\mathfrak{R}_{zb}(\nabla f,\nabla f)\label{new tensor R z}&=&\mathfrak{R}_{z}(\nabla f,\nabla f) -\sum_{i=1}^m \sum_{\hat i,\hat k=1}^{n+m}\la (z^{\ts}_{i\hat i}\frac{\pa b_{\hat k}}{\pa x_{\hat i}}\frac{\pa f}{\pa x_{\hat k}}-b_{\hat k}\frac{\pa z^{\ts}_{i\hat i}}{\pa x_{\hat k}} \frac{\pa f}{\pa x_{\hat i}} ),(z^{\ts}\nabla f)_i\ra_{\hR^m},\nonumber \\
\mathfrak{R}_z(\nabla f,\nabla f)\label{new tensor R z}&=&\sum_{i=1}^n\sum_{k=}^m\sum_{i',\hat i,\hat k=1}^{n+m} \la a^{\ts}_{ii'} (\frac{\partial a^{\ts}_{i \hat i}}{\partial x_{i'}} \frac{\partial z^{\ts}_{k\hat k}}{\partial x_{\hat i}}\frac{\partial f}{\partial x_{\hat k}}) ,(z^{\ts}\nabla)_kf\ra_{\hR^m}\nonumber \\
	&&+\sum_{i=1}^n\sum_{k=}^m\sum_{i',\hat i,\hat k=1}^{n+m} \la a^{\ts}_{ii'}a^{\ts}_{i \hat i} (\frac{\partial }{\partial x_{i'}} \frac{\partial z^{\ts}_{k\hat k}}{\partial x_{\hat i}})(\frac{\partial f}{\partial x_{\hat k}}) ,(z^{\ts}\nabla)_kf\ra_{\hR^m} \nonumber \\
	&&-\sum_{i=1}^n\sum_{k=1}^m\sum_{i',\hat i,\hat k=1}^{n+m} \la z^{\ts}_{k\hat k}\frac{\partial a^{\ts}_{ii'}}{\partial x_{\hat k}} \frac{\partial a^{\ts}_{i \hat i}}{\partial x_{i'}} \frac{\partial f}{\partial x_{\hat i}})  ,(z^{\ts}\nabla)_kf\ra_{\hR^m} \nonumber \\
&&-\sum_{i=1}^n\sum_{k=1}^m\sum_{i',\hat i,\hat k=1}^{n+m} \la z^{\ts}_{k\hat k} a^{\ts}_{ii'} (\frac{\partial }{\partial x_{\hat k}} \frac{\partial a^{\ts}_{i \hat i}}{\partial x_{i'}}) \frac{\partial f}{\partial x_{\hat i}}  ,(z^{\ts}\nabla)_kf\ra_{\hR^m},\nonumber
\eea 
and 
\bea 
\mathfrak{R}_{\pi}(\nabla f,\nabla f)\label{new tensor R Psi}&=&2\sum_{k=1}^m \sum_{i=1}^n\sum_{k',\hat k,\hat i,i'=1}^{n+m}\left[\frac{\pa }{\pa x_{k'}} z^{\ts}_{kk'} z^{\ts}_{k\hat k}\frac{\pa}{\pa x_{\hat k}}a^{\ts}_{i\hat i}\frac{\pa f}{\pa x_{\hat i}}a^{\ts}_{ii'}\frac{\pa f}{\pa x_{i'}}\right]\nonumber\\
 &&+2\sum_{k=1}^m \sum_{i=1}^n\sum_{k',\hat k,\hat i,i'=1}^{n+m}\left[z^{\ts}_{kk'}\frac{\pa }{\pa x_{k'}} z^{\ts}_{k\hat k} \frac{\pa}{\pa x_{\hat k}}a^{\ts}_{i\hat i}\frac{\pa f}{\pa x_{\hat i}}a^{\ts}_{ii'}\frac{\pa f}{\pa x_{i'}} \right.\nonumber\\
 &&\quad\quad\quad\quad\quad\quad\quad\quad\quad+z^{\ts}_{kk'} z^{\ts}_{k\hat k} \frac{\pa^2}{\pa x_{k'}\pa x_{\hat k}}a^{\ts}_{i\hat i}\frac{\pa f}{\pa x_{\hat i}}a^{\ts}_{ii'}\frac{\pa f}{\pa x_{i'}}\nonumber\\
&&\left.\quad\quad\quad\quad\quad\quad\quad\quad\quad+z^{\ts}_{kk'} z^{\ts}_{k\hat k} \frac{\pa}{\pa x_{\hat k}}a^{\ts}_{i\hat i}\frac{\pa f}{\pa x_{\hat i}}\frac{\pa }{\pa x_{k'}}a^{\ts}_{ii'}\frac{\pa f}{\pa x_{i'}} \right]\nonumber
 \eea 
 \bea 
&&+2\sum_{k=1}^m \sum_{i=1}^n\sum_{\hat k,\hat i,i'=1}^{n+m}(z^{\ts}\nabla\log\pi)_k \left[ z^{\ts}_{k\hat k}\frac{\pa}{\pa x_{\hat k}}a^{\ts}_{i\hat i}\frac{\pa f}{\pa x_{\hat i}}a^{\ts}_{ii'}\frac{\pa f}{\pa x_{i'}} \right] \nonumber\\
&&-2\sum_{j=1}^m\sum_{l=1}^n\sum_{ l',\hat l,\hat j,j'=1}^{n+m}\left[\frac{\pa }{\pa x_{l'}} a^{\ts}_{ll'} a^{\ts}_{l\hat l} \frac{\pa}{\pa x_{\hat l}}z^{\ts}_{j\hat j}\frac{\pa f}{\pa x_{\hat j}}z^{\ts}_{jj'}\frac{\pa f}{\pa x_{j'}} \right] \nonumber\\ 
 &&- 2\sum_{j=1}^m\sum_{l=1}^n\sum_{ l',\hat l,\hat j,j'=1}^{n+m}\left[ a^{\ts}_{ll'}\frac{\pa }{\pa x_{l'}}a^{\ts}_{l\hat l} \frac{\pa}{\pa x_{\hat l}}z^{\ts}_{j\hat j}\frac{\pa f}{\pa x_{\hat j}}z^{\ts}_{jj'}\frac{\pa f}{\pa x_{j'}}  \right.\nonumber\\
 &&\quad\quad\quad\quad\quad\quad\quad\quad\quad+a^{\ts}_{ll'}a^{\ts}_{l\hat l} \frac{\pa^2}{\pa x_{l'}\pa x_{\hat l}}z^{\ts}_{j\hat j}\frac{\pa f}{\pa x_{\hat j}}z^{\ts}_{jj'}\frac{\pa f}{\pa x_{j'}}\nonumber \\
&&\left.\quad\quad\quad\quad\quad\quad\quad\quad\quad+a^{\ts}_{ll'}a^{\ts}_{l\hat l} \frac{\pa}{\pa x_{\hat l}}z^{\ts}_{j\hat j}\frac{\pa f}{\pa x_{\hat j}}\frac{\partial}{\pa x_{l'}}z^{\ts}_{jj'}\frac{\pa f}{\pa x_{j'}}  \right] \nonumber \\
&&-2\sum_{j=1}^m\sum_{l=1}^n\sum_{\hat l,\hat j,j'=1}^{n+m}(a^{\ts}\nabla\log\pi)_l \left[ a^{\ts}_{l\hat l}\frac{\pa}{\pa x_{\hat l}}z^{\ts}_{j\hat j}\frac{\pa f}{\pa x_{\hat j}}z^{\ts}_{jj'}\frac{\pa f}{\pa x_{j'}} \right].\nonumber
\eea
Here we also denote $\mathfrak{R}=\mathfrak{R}(x)\in \mathbb{R}^{(n+m)\times(n+m)}$, such that $(\nabla f)^\ts\mathfrak{R}(x)\nabla f=\mathfrak{R}(\nabla f, \nabla f)$.
\end{definition}
{ The main theorem is presented below and its proof is postponed to Theorem \ref{thm: generalized gamma 2 z with drift} in Section \ref{sec: generalized gamma z}.}
\begin{theorem}[Generalized $z$ Bochner's formula]\label{thm1}
If Assumption \ref{assumption:main result} is satisfied, then the following decomposition holds
	\beaa
\Gamma_{2}(f,f)+\Gamma_{2}^{z,\Vol}(f,f)&=&\|\mathfrak{Hess}_{a,z}f\|^2
+\mathfrak{R}(\nabla f,\nabla f),
\eeaa 
where we define 
\beaa
\|\mathfrak{Hess}_{a,z}f\|^2&=&[X+\Lambda_1]^{\ts}Q^{\ts}Q[X+\Lambda_1]+[X+\Lambda_2]^{\ts}P^{\ts}P[X+\Lambda_2],\\
\mathfrak R(\nabla f,\nabla f)&=&-\Lambda_1^{\ts}Q^{\ts}Q\Lambda_1-\Lambda_2^{\ts}P^{\ts}P\Lambda_2+D^{\ts}D+E^{\ts}E \\
&&+\mathfrak{R}_{ab}(\nabla f,\nabla f) +\mathfrak{R}_{zb}(\nabla f,\nabla f)+\mathfrak{R}^{\Vol }(\nabla f,\nabla f).
\eeaa
\end{theorem}

We are now ready to prove the convergence property of degenerate drift diffusion process \eqref{SDE framework1} and related functional inequalities. 
Denote the Kullback–Leibler divergence as
\begin{equation*}
    \mathrm{D}_{\mathrm{KL}}(\rho\|\pi):=\int_{\hM^{n+m}}\rho(x)\log\frac{\rho(x)}{\Vol(x)} dx. 
\end{equation*}
Denote the $a,z$--relative Fisher information functional as  
\beaa
\mathrm{I}_{a,z}(\rho):=\int_{\mathbb{R}^{n+m}} (\nabla \log\frac{\rho}{\Vol}, aa^{\ts}\nabla\log\frac{\rho}{\Vol})\rho dx+\int_{\mathbb{R}^{n+m}} (\nabla \log\frac{\rho}{\Vol}, zz^{\ts}\nabla\log\frac{\rho}{\Vol})\rho dx.
\eeaa

\begin{theorem}[Exponential convergence in $L^1$ distance]\label{prop1.4}
Suppose there exists a constant $\kappa>0$, such that
\begin{equation*}
\mathfrak R\succeq \kappa(aa^{\ts}+zz^{\ts}).
\end{equation*}
 Let $\rho_0$ be a smooth initial distribution and $\rho=\rho(t,x)$ be the probability density function of \eqref{SDE framework1}. Then $\rho$ converges to the invariant measure $\Vol$ in the sense of
\begin{equation*}
\mathrm{I}_{a,z}(\rho) \leq e^{-2\kappa t} \mathrm{I}_{a,z}(\rho_0). 
\end{equation*}
In addition, 
    \begin{equation*}
 \int_{{\hM}^{n+m}} |\rho(t,x)-\pi(x)| dx\leq \sqrt{2  \mathrm{D}_{\mathrm{KL}}(\rho_0\|\pi) }e^{-\kappa t}.
\end{equation*}
\end{theorem}
{ The proof of Theorem \ref{prop1.4} is postponed to Proposition \eqref{prop:z log-soblov}.}

\begin{rem}[Functional inequalities]\label{prop1.5}
Suppose $\mathfrak R\succeq \kappa(aa^{\ts}+zz^{\ts})$ with $\kappa>0$,   
then the z-log-Sobolev inequalities holds:
\begin{equation*}
\int_{\hM^{n+m}}\rho\log\frac{\rho}{\Vol} dx\leq \frac{1}{2\kappa}
\mathrm{I}_{a,z}(\rho),
\end{equation*}
for any smooth density function $\rho$. 
\end{rem}

\begin{rem}
In literature \cite{BaudoinGarofalo09}, the $\Gamma_{2,z}$ operator is defined by \eqref{old gamma two z}, i.e. $\Gamma_{2}^z(f,f)=\frac{1}{2}L\Gamma_1^z(f,f)-\Gamma_1^z(Lf,f)$.
In fact, this definition is under the assumption of $\Gamma_{1}(\Gamma_{1}^z(f,f),f)=\Gamma_{1}^z(\Gamma_{1}(f,f),f)$. This assumption holds true only for the special choice of $a$ and $z$. In generalized Gamma z calculus,
we introduce a new term \eqref{new term}, which removes the assumption $\Gamma_{1}(\Gamma_{1}^z(f,f),f)=\Gamma_{1}^z(\Gamma_{1}(f,f),f)$. In fact, in the paper, we show that \eqref{new term} is exactly the new bilinear form behind assumption in \cite{baudoin2016wasserstein, Baudoin2017} by considering the weak form. See details in Remark \ref{iterative Gamma conditon}.
\end{rem}

\begin{rem}
Following \cite{FengLi2021}[Assumption 1], we know that for any $i\in \{1,\cdots,n\}$ and $k\in\{1,\cdots,m\}$, if
\begin{align}\label{assump}
	z^{\ts}_k\nabla a^{\ts}_i\in \text{Span}\{a^{\ts}_1,\cdots,a^{\ts}_n\},
\end{align} 
there exists vectors $\widehat\Lambda_1$ and $\widehat\Lambda_2$, such that the Hessian operator associated with the generator of the SDE and the metric $(aa^{\ts})^{\dd}$ could be represented as
\[
\|\mathfrak{Hess}f\|^2=[\mathsf Q\mathsf X+\widehat \Lambda_1]^{\ts} [\mathsf Q\mathsf X+\widehat\Lambda_1]+[\mathsf P\mathsf X+\widehat\Lambda_2]^{\ts} [\mathsf P\mathsf X+\widehat\Lambda_2].
\]
Furthermore, we have the following relation
\begin{align*}
& [\mathsf Q\mathsf X+\widehat\Lambda_1]^{\ts} [\mathsf Q\mathsf X+\widehat\Lambda_1]+[\mathsf P\mathsf X+\widehat\Lambda_2]^{\ts}[\mathsf P\mathsf X+\widehat\Lambda_2]-\widehat\Lambda_1^{\ts}\widehat\Lambda_1-\widehat\Lambda_2^{\ts}\widehat\Lambda_2\\
=& [\mathsf X+ \Lambda_1]^{\ts}\mathsf Q^{\ts}\mathsf Q [\mathsf X+\Lambda_1]+[\mathsf X+\Lambda_2]^{\ts}\mathsf P^{\ts}\mathsf P [\mathsf X+\Lambda_2]-\Lambda_1^{\ts}\mathsf Q^{\ts}\mathsf Q\Lambda_1-\Lambda_2^{\ts}\mathsf P^{\ts}\mathsf P\Lambda_2,
\end{align*}
if there exist $\Lambda_1$ and $\Lambda_2$ as in Assumption \ref{assumption:main result}, such that 
\begin{align}
	\label{condition}
	\widehat \Lambda_1^{\ts} = \Lambda_1^{\ts}\mathsf Q^{\ts}\quad \text{and}\quad  \widehat\Lambda_2^{\ts} =\Lambda_2^{\ts}\mathsf P^{\ts}.
\end{align}
 The Assumption \ref{assumption:main result} is true if condition \ref{assump} and \ref{condition} hold. See detailed connections in \cite{FengLi2021}[Remark 11].  
\end{rem}

\section{Examples}\label{sec: example}
In this section, we consider the following degenerate drift-diffusion process
\begin{equation}\label{SDE framework12}
dX_t=-a(X_t)a(X_t)^{\ts}\nabla V(X_t)dt+\sqrt{2}a(X_t)\circ dB_t,
\end{equation}
where $a: \hR^{n+m}\rightarrow \mathbb{R}^{(n+m)\times n}$ is a matrix valued function, for  $n, m\in \mathbb Z_+$, and $V\in C^{\infty}(\mathbb{R}^{n+m})$ is a smooth potential function. We denote the invariant measure of SDE \eqref{SDE framework12} as $\pi$. We further assume that 
\beaa 
-aa^{\ts}\nabla V=a\otimes \nabla a+aa^{\ts}\nabla \log \pi.
\eeaa 
The above assumption holds for later on three examples. 

\begin{rem}\label{rem: canonital drift}
For $V=0$, the invariant measure $\pi$ in the above assumption exists if $\{a_1,\cdots,a_n\}$ forms  left-invariant structures on unimodular Lie groups. In this case, the sub-Laplacian is a sum of squares of horizontal vector fields and invariant measure is also symmetric. The Stratonovich SDE \eqref{SDE framework12} defines the horizontal Brownian motion on sub-Riemannian structure $(\hM^{n+m},\tau, (aa^{\ts})^{\dd}|_{\tau})$ and the $\Vol$ is the volume form associated with horizontal Laplacian. In general, if the Lie group structure is not unimodular, the drift $b\neq 0$. \cite{agrachev2012, barilari2013formula, baudoin2019integration, eldredge2018, elworthy1982, GL2016, GL2017, Malliavin}. See the related studies on log-Sobolev inequality in \cite{baudoin2015log, inglis2009logarithmic}. 
\end{rem}
\begin{rem}
It is also worth mentioning that many sub-Riemannian manifolds are non-compact. 
Hence there may not exist a positive constant $\kappa$ for both classical $\Gamma_1$ and $\Gamma^z_1$ directions in the non-compact domain. The non-compactness of the domain brings additional difficulties. To prove the associated inequalities in this case, we need to extend the result derived in \cite{baudoin2012log, wang1997logarithmic}. This is a direction for future work.  
\end{rem}

\begin{rem}
{ It is known that the Heisenberg group is an example of Lie groups in quantum mechanics \cite{woit2017quantum}. In future work, we shall investigate the general convergence analysis of SDEs in Lie groups and their connections with quantum SDEs.}
\end{rem}

\subsection{Heisenberg group}
In this subsection, we apply our general theory to the well-known example in sub-Riemannian geometry, which is the Heisenberg group. A related LSI for the horizontal Wiener measure has been studied in  \cite{baudoin2012log}. Recall briefly that the Heisenberg group $\mathbb H^1$ admits left invariant vector fields: $X=\frac{\partial}{\partial x}-\frac{1}{2}y\frac{\partial}{\partial z},\quad Y=\frac{\partial}{\partial y}+\frac{1}{2}x\frac{\partial}{\partial z},\quad Z=\frac{\partial}{\partial z}$. Here $\{X,Y,Z\}$ forms an orthonormal basis for the tangent bundle of $\mathbb H^1$. In this case, $\Vol=e^{-V}$. In particular, $X$ and $Y$ generate the horizontal distribution $\tau$. To fit into our general  theory from the previous section, we take matrices $a$ and $z$ as below
\bea\label{heisenberg a and z}
a^{\ts}=\begin{pmatrix}{}1&0&-y/2\\
0&1&x/2
\end{pmatrix},\quad z^{\ts}=(0,0,1).
\eea
In particular, we have 
\beaa
a^{\ts}\nabla f=\Big((a^{\ts}\nabla)_1f,(a^{\ts}\nabla)_2f\Big)^{\ts},\quad (a^{\ts}\nabla)_1f=(\frac{\pa f}{\pa x}-\frac{y}{2}\frac{\pa f}{\pa z}),\quad (a^{\ts}\nabla)_2f=(\frac{\pa f}{\pa y}+\frac{x}{2}\frac{\pa f}{\pa z}).
\eeaa 

We have the following proposition for Heisenberg group following Theorem \ref{thm1}.
\begin{prop}\label{prop heisenberg} For any smooth function $f\in C^{\infty}(\mathbb{H}^1)$, one has
\beaa
\Gamma_2(f,f)+\Gamma_2^{z,\pi}(f,f)&=&\|\mathfrak{Hess}_{a,z}f\|^2+\mathfrak{R}(\nabla f,\nabla f),
\eeaa
where 
\beaa
\Lambda_1^{\ts}&=&(0,0,0,0,0,0,0,0,0); \\
\Lambda_2^{\ts}&=&(0,0,0,0,0,0,(a^{\ts}\nabla)_2f,-(a^{\ts}\nabla)_1f,0);\\
&&\mathfrak{R}_{ab}(\nabla f,\nabla f)-\Lambda_1^{\ts}Q^{\ts}Q\Lambda_1-\Lambda_2^{\ts}P^{\ts}P\Lambda_2+D^{\ts}D+E^{\ts}E\\
&=&-\Gamma_1(f,f)+\frac{1}{2}\Gamma_1^z(f,f)-(a^{\ts}\nabla)_1V\pa_z f(a^{\ts}\nabla)_2f+(a^{\ts}\nabla)_2V\pa_z f(a^{\ts}\nabla)_1f\\
&&+\Big[\frac{\pa^2 V}{\pa x\pa x}+\frac{y^2}{4}\frac{\pa^2 V}{\pa z\pa z}-y\frac{\pa^2 V}{\pa x\pa z}\Big] |(a^{\ts}\nabla)_1 f|^2\\
&&+\Big[\frac{\pa^2 V}{\pa y\pa y}+\frac{x^2}{4}\frac{\pa^2 V}{\pa z\pa z}+x\frac{\pa^2 V}{\pa y\pa z}\Big] |(a^{\ts}\nabla)_2 f|^2\\
&&+2\Big[\frac{\pa^2 V}{\pa x\pa y}+\frac{x}{2}\frac{\pa^2 V}{\pa x\pa z}-\frac{y}{2}\frac{\pa^2 V}{\pa y\pa z}-\frac{xy}{4}\frac{\pa^2 V}{\pa z\pa z}\Big] (a^{\ts}\nabla)_1 f(a^{\ts}\nabla)_2 f;\\
\mathfrak{R}_{zb}(\nabla f, \nabla f)&=& \Big(\frac{\pa^2 V}{\pa x\pa z }-\frac{y}{2}\frac{\pa^2 V}{\pa z\pa z}\Big)(a^{\ts}\nabla)_1 f (z^{\ts}\nabla)_1 f+\Big(\frac{\pa^2 V}{\pa y\pa z}+\frac{x}{2}\frac{\pa^2 V}{\pa z\pa z}\Big)(z^{\ts}\nabla)_1 f(a^{\ts}\nabla)_2 f;\\
\mathfrak{R}_{\pi}(\nabla f, \nabla f)&=&0.
\eeaa 
\end{prop}
{ 
The proof of Proposition of \ref{prop heisenberg} follows from the proof of Theorem \ref{thm1} (i.e. Theorem \ref{thm: generalized gamma 2 z with drift}), Lemma \ref{heisenberg vector matrix}, Lemma \ref{heisenberg hess}, and Lemma \ref{heisenberg tensor}. The following convergence result follows directly from Theorem \ref{prop1.4}.}
\begin{prop} If there exists $\kappa>0$ as shown in Theorem \ref{prop1.4}, the exponential dissipation result in $L^1$ distance holds:
\begin{equation*}
    \int |\rho(t,x)-\pi(x)| dx=O(e^{-\kappa t}). 
\end{equation*}
\end{prop}
We next formulate the curvature tensor into a matrix format. Denote 
\bea \label{vector U 3d}
\mathsf{U}=\Big(
(a^{\ts}\nabla)_1f , (a^{\ts}\nabla)_2f , (z^{\ts}\nabla)_1f
\Big)_{3\times 1},
\eea  and denote $\mathbf{I}_{3\times 3}$ as the identity matrix. With a little abuse of notation,
there exists a symmetric matrix $\mathfrak{R}$ such that we can represent the tensor as below.
\bea \label{matrix form for R}
\mathfrak{R}(\nabla f,\nabla f)=(\mathsf U)^{\ts}\cd\mathfrak{R}\cd\mathsf U,
\eea
which implies that 
\beaa
\mathfrak{R}\succeq \kappa (aa^{\ts}+zz^{\ts})
\Rightarrow \mathfrak{R}(\nabla f,\nabla f)\succeq  \kappa (\Gamma_1(f,f)+\Gamma_1^z(f,f)).
\eeaa
In other words, we need to estimate the smallest eigenvalue of matrix $\mathfrak R$. We next present the formulation of matrix $\mathfrak R$ for the Heisenberg group as follows. \\ 
\begin{cor}\label{matrix R heisenberg}
 The matrix $\mathfrak{R}$ associated  with Heisenberg group has the following form
\beaa
\mathfrak{R}_{11}&=& \Big[\frac{\pa^2 V}{\pa x\pa x}+\frac{y^2}{4}\frac{\pa^2 V}{\pa z\pa z}-y\frac{\pa^2 V}{\pa x\pa z}\Big]-1;\\
\mathfrak{R}_{22}&=&\Big[\frac{\pa^2 V}{\pa y\pa y}+\frac{x^2}{4}\frac{\pa^2 V}{\pa z\pa z}+x\frac{\pa^2 V}{\pa y\pa z}\Big]-1;\quad \mathfrak{R}_{33}=\frac{1}{2};  \\
\mathfrak{R}_{12}&=&\mathfrak{R}_{21}=\Big[\frac{\pa^2 V}{\pa x\pa y}+\frac{x}{2}\frac{\pa^2 V}{\pa x\pa z}-\frac{y}{2}\frac{\pa^2 V}{\pa y\pa z}-\frac{xy}{4}\frac{\pa^2 V}{\pa z\pa z}\Big];\\
\mathfrak{R}_{13}&=&\mathfrak{R}_{31}=\frac{1}{2}(a^{\ts}\nabla)_2V+\frac{1}{2}\Big(\frac{\pa^2 V}{\pa x\pa z }-\frac{y}{2}\frac{\pa^2 V}{\pa z\pa z}\Big);\\
\mathfrak{R}_{23}&=&\mathfrak{R}_{32}=-\frac{1}{2}(a^{\ts}\nabla)_1V+\frac{1}{2}\Big(\frac{\pa^2 V}{\pa y\pa z}+\frac{x}{2}\frac{\pa^2 V}{\pa z\pa z}\Big).
\eeaa
\end{cor}
\begin{proof} 
The explicit form of matrix $\mathfrak R$ follows from the definition in Theorem \ref{thm1} and the notation in  \eqref{vector U 3d} and \eqref{matrix form for R}. We have
\beaa
\mathfrak R(\nabla f,\nabla f)&=&-\Lambda_1^{\ts}Q^{\ts}Q\Lambda_1-\Lambda_2^{\ts}P^{\ts}P\Lambda_2+D^{\ts}D+E^{\ts}E \\
&&+\mathfrak{R}_{ab}(\nabla f,\nabla f) +\mathfrak{R}_{zb}(\nabla f,\nabla f)+\mathfrak{R}^{\Vol }(\nabla f,\nabla f)\\
&=&(\mathsf U)^{\ts}\cd\mathfrak{R}\cd\mathsf U.
\eeaa 
Plugging the explicit representation from Proposition \ref{prop heisenberg} into the above formula, and applying matrix symmetrization for the off-diagonal terms, we get the desired matrix $\mathfrak R$.
\end{proof}
%
Next, we present the three key lemmas.
\begin{lem}\label{heisenberg vector matrix}
For Heisenberg group, we have 
\beaa
Q&=&\left(
\begin{array}{ccccccccc}
 1 & 0 & -\frac{y}{2} & 0 & 0 & 0 & -\frac{y}{2} & 0 & \frac{y^2}{4} \\
 0 & 1 & \frac{x}{2} & 0 & 0 & 0 & 0 & -\frac{y}{2} & -\frac{x y}{4} \\
 0 & 0 & 0 & 1 & 0 & -\frac{y}{2} & \frac{x}{2} & 0 & -\frac{x y}{4} \\
 0 & 0 & 0 & 0 & 1 & \frac{x}{2} & 0 & \frac{x}{2} & \frac{x^2}{4} \\
\end{array}
\right);\\
P&=&\left(
\begin{array}{ccccccccc}
 0 & 0 & 0 & 0 & 0 & 0 & 1 & 0 & -\frac{y}{2} \\
 0 & 0 & 0 & 0 & 0 & 0 & 0 & 1 & \frac{\text{x}}{2} \\
\end{array}
\right);\\
D^{\ts}&=&(0,\frac{1}{2}\pa_zf,-\frac{1}{2}\pa_zf,0);\quad E^{\ts}=(0,0);\quad F^{\ts}=G^{\ts}=(0,0,0,0,0,0,0,0,0);\\
C^{\ts}&=&(0,0,\frac{x}{4}\pa_zf+\frac{1}{2}\pa_y f,0,0,\frac{y}{4}\pa_z f-\frac{1}{2}\pa_x f,\frac{x}{4}\pa_zf+\frac{1}{2}\pa_yf,\frac{y}{4}\pa_zf-\frac{1}{2}\pa_xf,-\frac{y}{2}\pa_yf-\frac{x}{2}\pa_xf).
\eeaa
\end{lem}{}
\begin{proof}
The proof of this lemma follows from routine computations. Plugging the matrices $a$ and $z$ from \eqref{heisenberg a and z} into Notation \ref{notation}, we get the desired vectors and matrices. We skip the detailed computation here. \qed 
\end{proof}
\begin{lem}\label{heisenberg hess} On $\mathbb H^1$, vectors $F$ and $G$ are zero vectors, we have
\beaa
&&[QX+D]^{\ts}[QX+D]+[PX+E]^{\ts}[PX+E]+2C^{\ts}X\\
&=&\|\mathfrak{Hess}_{a,z}f\|^2-\Lambda_1^{\ts}Q^{\ts}Q\Lambda_1-\Lambda_2^{\ts}P^{\ts}P\Lambda_2+D^{\ts}D+E^{\ts}E.
\eeaa
In particular, we have 
\beaa
\|\mathfrak{Hess}_{a,z}f\|^2&=&[X+\Lambda_1]^{\ts}Q^{\ts}Q[X+\Lambda_1]+[X+\Lambda_2]^{\ts}P^{\ts}P[X+\Lambda_2];\\
\Lambda_1^{\ts}&=&(0,0,0,0,0,0,0,0,0); \\
\Lambda_2^{\ts}&=&(0,0,0,0,0,0,(a^{\ts}\nabla)_2f,-(a^{\ts}\nabla)_1f,0);\\
-\Lambda_1^{\ts}Q^{\ts}Q\Lambda_1-\Lambda_2^{\ts}P^{\ts}P\Lambda_2+D^{\ts}D+E^{\ts}E)&=&-\Gamma_1(f,f)+\frac{1}{2}\Gamma_1^z(f,f).
\eeaa
\end{lem}{}

\begin{lem}\label{heisenberg tensor}
By routine computations, we obtain
\beaa
\mathfrak{R}_{ab}(\nabla f,\nabla f)&=&-(a^{\ts}\nabla)_1V\pa_z f(a^{\ts}\nabla)_2f+(a^{\ts}\nabla)_2V\pa_z f(a^{\ts}\nabla)_1f\\
&&+\Big[\frac{\pa^2 V}{\pa x\pa x}+\frac{y^2}{4}\frac{\pa^2 V}{\pa z\pa z}-y\frac{\pa^2 V}{\pa x\pa z}\Big] |(a^{\ts}\nabla)_1 f|^2\\
&&+\Big[\frac{\pa^2 V}{\pa y\pa y}+\frac{x^2}{4}\frac{\pa^2 V}{\pa z\pa z}+x\frac{\pa^2 V}{\pa y\pa z}\Big] |(a^{\ts}\nabla)_2 f|^2\\
&&+2\Big[\frac{\pa^2 V}{\pa x\pa y}+\frac{x}{2}\frac{\pa^2 V}{\pa x\pa z}-\frac{y}{2}\frac{\pa^2 V}{\pa y\pa z}-\frac{xy}{4}\frac{\pa^2 V}{\pa z\pa z}\Big] (a^{\ts}\nabla)_1 f(a^{\ts}\nabla)_2 f;\\
\mathfrak{R}_{zb}(\nabla f,\nabla f)&=& \Big(\frac{\pa^2 V}{\pa x\pa z }-\frac{y}{2}\frac{\pa^2 V}{\pa z\pa z}\Big)(a^{\ts}\nabla)_1 f(z^{\ts}\nabla)_1 f+\Big(\frac{\pa^2 V}{\pa y\pa z}+\frac{x}{2}\frac{\pa^2 V}{\pa z\pa z}\Big)(z^{\ts}\nabla)_1 f(a^{\ts}\nabla)_2 f;\\
\mathfrak{R}_{\pi}(\nabla f,\nabla f)&=&0.
\eeaa
\end{lem}{}

\begin{proof}
[Proof of Lemma \ref{heisenberg hess}]
We first have 
\beaa
2C^{\ts}X&=&\sum_{\hat i,\hat k=1}^32C^{\ts}_{\hat i\hat k}X_{\hat i\hat k}\\
&=&2\Big[\frac{\pa^2 f}{\pa x\pa z}\frac{(a^{\ts}\nabla)_2f}{2}-\frac{\pa^2 f}{\pa y\pa z}\frac{(a^{\ts}\nabla)_1f}{2}+\frac{\pa^2 f}{\pa z\pa x}\frac{(a^{\ts}\nabla)_2f}{2}\Big]\\
&&-2\Big[\frac{\pa^2 f}{\pa z\pa y}\frac{(a^{\ts}\nabla)_1f}{2}+\frac{\pa^2 f}{\pa z\pa z}(\frac{y}{2}\pa_yf+\frac{x}{2}\pa_xf)\Big]\\
&=&2\frac{\pa^2 f}{\pa x\pa z}(a^{\ts}\nabla)_2f-2\frac{\pa^2 f}{\pa y\pa z}(a^{\ts}\nabla)_1f-2\frac{\pa^2 f}{\pa z\pa z}(\frac{y}{2}\pa_yf+\frac{x}{2}\pa_xf)\\
&=&2(a^{\ts}\nabla)_2f\left[\frac{\pa^2 f}{\pa x\pa z}-\frac{y}{2}\frac{\pa^2 f}{\pa z\pa z}\right]-2(a^{\ts}\nabla )_1f\left[\frac{\pa^2 f}{\pa y\pa z}+\frac{x}{2}\frac{\pa^2 f}{\pa z\pa z}\right].
\eeaa
By direct computations, we have
\beaa
&&[QX+D]^{\ts}[QX+D]+[PX+E]^{\ts}[PX+E] +2C^{\ts}X\\
&=& \left[\frac{\pa^2 f}{\pa x\pa x}-y\frac{\pa^2 f}{\pa x\pa z}+\frac{y^2}{4}\frac{\pa^2 f}{\pa z\pa z} \right]^2 +\left[\frac{\pa^2 f}{\pa x\pa y}+\frac{x}{2}\frac{\pa^2 f}{\pa x\pa z}-\frac{y}{2}\frac{\pa^2 f}{\pa y\pa z}-\frac{xy}{4}\frac{\pa^2f}{\pa z\pa z}+\frac{1}{2}\pa_z f \right]^2\\
&&+\left[\frac{\pa^2 f}{\pa x\pa y}-\frac{y}{2}\frac{\pa^2 f}{\pa y\pa z}+\frac{x}{2}\frac{\pa^2 f}{\pa x\pa z}-\frac{xy}{4}\frac{\pa^2f}{\pa z\pa z}-\frac{1}{2}\pa_z f\right]^2
+\left[\frac{\pa^2 f}{\pa y\pa y}+x\frac{\pa^2 f}{\pa y\pa z}+\frac{x^2}{4}\frac{\pa^2f}{\pa z\pa z} \right]^2\\
&&+ \left[\frac{\pa^2 f}{\pa x\pa z}-\frac{y}{2}\frac{\pa^2 f}{\pa z\pa z}\right]^2+\left[\frac{\pa^2 f}{\pa y\pa z}+\frac{x}{2}\frac{\pa^2 f}{\pa z\pa z}\right]^2\\
&&+2(a^{\ts}\nabla )_2f\left[\frac{\pa^2 f}{\pa x\pa z}-\frac{y}{2}\frac{\pa^2 f}{\pa z\pa z}\right]-2(a^{\ts}\nabla )_1f\left[\frac{\pa^2 f}{\pa y\pa z}+\frac{x}{2}\frac{\pa^2 f}{\pa z\pa z}\right].
\eeaa
Completing squares for the cross terms involving the type of $``\nabla f\nabla^2 f"$ and following the reformulation as below
\beaa
&&\left[\frac{\pa^2 f}{\pa x\pa y}+\frac{x}{2}\frac{\pa^2 f}{\pa x\pa z}-\frac{y}{2}\frac{\pa^2 f}{\pa y\pa z}-\frac{xy}{4}\frac{\pa^2f}{\pa z\pa z}+\frac{1}{2}\pa_z f \right]^2\\
&&+\left[\frac{\pa^2 f}{\pa x\pa y}-\frac{y}{2}\frac{\pa^2 f}{\pa y\pa z}+\frac{x}{2}\frac{\pa^2 f}{\pa x\pa z}-\frac{xy}{4}\frac{\pa^2f}{\pa z\pa z}-\frac{1}{2}\pa_z f\right]^2\\
&=&2\left[\frac{\pa^2 f}{\pa x\pa y}-\frac{y}{2}\frac{\pa^2 f}{\pa y\pa z}+\frac{x}{2}\frac{\pa^2 f}{\pa x\pa z}-\frac{xy}{4}\frac{\pa^2f}{\pa z\pa z} \right]^2+\frac{1}{2}|\pa_z f|^2,
\eeaa
we have 
\beaa
&&[QX+D]^{\ts}[QX+D]+[PX+E]^{\ts}[PX+E] +2C^{\ts}X\\
&=& \left[\frac{\pa^2 f}{\pa x\pa x}-y\frac{\pa^2 f}{\pa x\pa z}+\frac{y^2}{4}\frac{\pa^2 f}{\pa z\pa z} \right]^2+2 \left[\frac{\pa^2 f}{\pa x\pa y}-\frac{y}{2}\frac{\pa^2 f}{\pa y\pa z}+\frac{x}{2}\frac{\pa^2 f}{\pa x\pa z}-\frac{xy}{4}\frac{\pa^2f}{\pa z\pa z}\right]^2\\
&&+\left[\frac{\pa^2 f}{\pa y\pa y}+x\frac{\pa^2 f}{\pa y\pa z}+\frac{x^2}{4}\frac{\pa^2f}{\pa z\pa z}\right]^2+ \left[\frac{\pa^2 f}{\pa x\pa z}-\frac{y}{2}\frac{\pa^2 f}{\pa z\pa z}+(a^{\ts}\nabla)_2f\right]^2\\
&&+\left[\frac{\pa^2 f}{\pa y\pa z}+\frac{x}{2}\frac{\pa^2 f}{\pa z\pa z}-(a^{\ts}\nabla)_1f\right]^2-|(a^{\ts}\nabla )_2f|^2-|(a^{\ts}\nabla )_1f|^2+\frac{1}{2} |(z^{\ts}\nabla )_1f|^2.
\eeaa

The sum of square terms give $\|\mathfrak{Hess}_{a,z}\|^2_{\mathrm F}$, hence $\Lambda_1$ and $\Lambda_2$. The remainders generate $-\Lambda_1^{\ts}Q^{\ts}Q\Lambda_1-\Lambda_2^{\ts}P^{\ts}P\Lambda_2+D^{\ts}D+E^{\ts}E$, which equals $-\Gamma_1(f,f)+\frac{1}{2}\Gamma_1^z(f,f)$.
\qed 
\end{proof}

We are now left to compute the tensors. 

\begin{proof}[Proof of Lemma \ref{heisenberg tensor}]
By direct computation, we have 
\beaa
\mathfrak{R}_a(\nabla f,\nabla f)
&=&\sum_{i,k=1}^2\sum_{i',\hat i,\hat k=1}^{3} \la a^{\ts}_{ii'} (\frac{\partial a^{\ts}_{i \hat i}}{\partial x_{i'}} \frac{\partial a^{\ts}_{k\hat k}}{\partial x_{\hat i}}\frac{\partial f}{\partial x_{\hat k}}) ,(a^{\ts}\nabla)_kf\ra_{\hR^2}  \nonumber\\
&&+\sum_{i,k=2}^2\sum_{i',\hat i,\hat k=1}^{3} \la a^{\ts}_{ii'}a^{\ts}_{i \hat i} (\frac{\partial }{\partial x_{i'}} \frac{\partial a^{\ts}_{k\hat k}}{\partial x_{\hat i}})(\frac{\partial f}{\partial x_{\hat k}}) ,(a^{\ts}\nabla)_kf\ra_{\hR^2} \nonumber\\ 
&&-\sum_{i,k=1}^2\sum_{i',\hat i,\hat k=1}^{3} \la a^{\ts}_{k\hat k}\frac{\partial a^{\ts}_{ii'}}{\partial x_{\hat k}} \frac{\partial a^{\ts}_{i \hat i}}{\partial x_{i'}} \frac{\partial f}{\partial x_{\hat i}})  ,(a^{\ts}\nabla)_kf\ra_{\hR^2}  \nonumber\\
&&-\sum_{i,k=1}^2\sum_{i',\hat i,\hat k=1}^{3} \la a^{\ts}_{k\hat k} a^{\ts}_{ii'} (\frac{\partial }{\partial x_{\hat k}} \frac{\partial a^{\ts}_{i \hat i}}{\partial x_{i'}}) \frac{\partial f}{\partial x_{\hat i}}  ,(a^{\ts}\nabla)_kf\ra_{\hR^2},  \nonumber\\
&=& \cI_1+\cI_2+\cI_3+\cI_4.
\eeaa
For the four terms above, we have 
\beaa
\cI_1&=&\sum_{i=1}^2\sum_{i',\hat i,\hat k=1}^{3}  a^{\ts}_{ii'} (\frac{\partial a^{\ts}_{i \hat i}}{\partial x_{i'}} \frac{\partial a^{\ts}_{1\hat k}}{\partial x_{\hat i}}\frac{\partial f}{\partial x_{\hat k}})(a^{\ts}\nabla)_1f+\sum_{i=1}^2\sum_{i',\hat i,\hat k=1}^{3}  a^{\ts}_{ii'} (\frac{\partial a^{\ts}_{i \hat i}}{\partial x_{i'}} \frac{\partial a^{\ts}_{2\hat k}}{\partial x_{\hat i}}\frac{\partial f}{\partial x_{\hat k}}) (a^{\ts}\nabla)_2f=0\\
\cI_2&=&\sum_{i=2}^n\sum_{i',\hat i,\hat k=1}^{3} a^{\ts}_{ii'}a^{\ts}_{i \hat i} (\frac{\partial }{\partial x_{i'}} \frac{\partial a^{\ts}_{1\hat k}}{\partial x_{\hat i}})(\frac{\partial f}{\partial x_{\hat k}}) (a^{\ts}\nabla)_1f+\sum_{i=2}^n\sum_{i',\hat i,\hat k=1}^{3} a^{\ts}_{ii'}a^{\ts}_{i \hat i} (\frac{\partial }{\partial x_{i'}} \frac{\partial a^{\ts}_{2\hat k}}{\partial x_{\hat i}})(\frac{\partial f}{\partial x_{\hat k}})(a^{\ts}\nabla)_2f=0\\
\cI_3&=&-\sum_{i=1}^2\sum_{i',\hat i,\hat k=1}^{3}  a^{\ts}_{1\hat k}\frac{\partial a^{\ts}_{ii'}}{\partial x_{\hat k}} \frac{\partial a^{\ts}_{i \hat i}}{\partial x_{i'}} \frac{\partial f}{\partial x_{\hat i}})(a^{\ts}\nabla)_1f-\sum_{i=1}^2\sum_{i',\hat i,\hat k=1}^{3}  a^{\ts}_{2\hat k}\frac{\partial a^{\ts}_{ii'}}{\partial x_{\hat k}} \frac{\partial a^{\ts}_{i \hat i}}{\partial x_{i'}} \frac{\partial f}{\partial x_{\hat i}})(a^{\ts}\nabla)_2f=0\\
\cI_4&=&-\sum_{i=1}^2\sum_{i',\hat i,\hat k=1}^{3}  a^{\ts}_{1\hat k} a^{\ts}_{ii'} (\frac{\partial }{\partial x_{\hat k}} \frac{\partial a^{\ts}_{i \hat i}}{\partial x_{i'}}) \frac{\partial f}{\partial x_{\hat i}}(a^{\ts}\nabla)_1f-\sum_{i=1}^2\sum_{i',\hat i,\hat k=1}^{3}  a^{\ts}_{2\hat k} a^{\ts}_{ii'} (\frac{\partial }{\partial x_{\hat k}} \frac{\partial a^{\ts}_{i \hat i}}{\partial x_{i'}}) \frac{\partial f}{\partial x_{\hat i}}(a^{\ts}\nabla)_2f=0.
\eeaa
Similar computation applies to the tensor terms $\mathfrak{R}_{\pi}$ and $\mathfrak{R}_{zb}$. Since $z$ is a constant matrix, we get
\beaa
\mathfrak{R}_{zb}(\nabla f,\nabla f)&=&-\sum_{\hat i,\hat k=1}^{3}\la (z^{\ts}_{1\hat i}\frac{\pa b_{\hat k}}{\pa x_{\hat i}}\frac{\pa f}{\pa x_{\hat k}}-b_{\hat k}\frac{\pa z^{\ts}_{1\hat i}}{\pa x_{\hat k}} \frac{\pa f}{\pa x_{\hat i}} ),(z^{\ts}\nabla f)_1\ra_{\hR},\\
\mathfrak{R}_{\pi}&=&0.
\eeaa
We now compute the tensor terms involving the drift $b$. For the drift term in tensor $\mathfrak{R}_{ab}$, taking $b=-aa^{\ts}\nabla V$, which means  $b=-(a_{\hat k k}a^{\ts}_{kk'}\frac{\pa V}{\pa x_{k'}})_{\hat k=1,2,3}$ in local coordinates,
\beaa
\mathfrak{R}_b^a
&=&\sum_{i,k=1}^2\sum_{\hat i,\hat k,k'=1}^3\left[a^{\ts}_{i\hat i}\frac{\pa a^{\ts}_{k\hat k} }{\pa x_{\hat i}}a^{\ts}_{kk'}\frac{\pa V}{\pa x_{k'}} \frac{\pa f}{\pa x_{\hat k}}(a^{\ts}\nabla)_i f \right]\\ 
&&+\sum_{i,k=1}^2\sum_{\hat i,\hat k,k'=1}^3\left[a^{\ts}_{i\hat i}\frac{\pa a^{\ts}_{kk'}}{\pa x_{\hat i}}a^{\ts}_{k\hat k}\frac{\pa V}{\pa x_{k'}} \frac{\pa f}{\pa x_{\hat k}}(a^{\ts}\nabla)_i f\right]\\
&&+\sum_{i,k=1}^2\sum_{\hat i,\hat k,k'=1}^3\left[ a^{\ts}_{i\hat i}a^{\ts}_{k\hat k}a^{\ts}_{kk'}\frac{\pa^2 V}{\pa x_{\hat i}\pa x_{k'}}\frac{\pa f}{\pa x_{\hat k}}(a^{\ts}\nabla)_i f\right]
\eeaa 
\beaa 
&&-\sum_{i,k=1}^2\sum_{\hat i,\hat k,k'=1}^3\left[a^{\ts}_{k\hat k}a^{\ts}_{kk'}\frac{\pa a^{\ts}_{i\hat i}}{\pa x_{\hat k}}\frac{\pa V}{\pa x_{k'}} \frac{\pa f}{\pa x_{\hat i}}(a^{\ts}\nabla)_if\right] \\
&=& \cJ_1+\cJ_2+\cJ_3+\cJ_4.
\eeaa
We now derive the explicit formulas for the above four terms. 
\beaa
\cJ_1
&=&\sum_{\hat i,\hat k,k'=1}^3\left[ a^{\ts}_{1\hat i}\frac{\pa a^{\ts}_{1\hat k} }{\pa x_{\hat i}}a^{\ts}_{1k'}\frac{\pa V}{\pa x_{k'}} \frac{\pa f}{\pa x_{\hat k}}(a^{\ts}\nabla)_1 f+a^{\ts}_{2\hat i}\frac{\pa a^{\ts}_{1\hat k} }{\pa x_{\hat i}}a^{\ts}_{1k'}\frac{\pa V}{\pa x_{k'}} \frac{\pa f}{\pa x_{\hat k}}(a^{\ts}\nabla)_2 f\right]\\
&&+\sum_{\hat i,\hat k,k'=1}^3\left[ a^{\ts}_{1\hat i}\frac{\pa a^{\ts}_{2\hat k} }{\pa x_{\hat i}}a^{\ts}_{2k'}\frac{\pa V}{\pa x_{k'}} \frac{\pa f}{\pa x_{\hat k}}(a^{\ts}\nabla)_1 f+a^{\ts}_{2\hat i}\frac{\pa a^{\ts}_{2\hat k} }{\pa x_{\hat i}}a^{\ts}_{2k'}\frac{\pa V}{\pa x_{k'}} \frac{\pa f}{\pa x_{\hat k}}(a^{\ts}\nabla)_2 f\right]\\
&=&-\frac{1}{2}(a^{\ts}\nabla)_1V\pa_zf(a^{\ts}\nabla)_2f +\frac{1}{2}(a^{\ts}\nabla)_2V\pa_z f(a^{\ts }\nabla)_1f;\\ 
\cJ_2&=&\sum_{\hat i,\hat k,k'=1}^3\left[a^{\ts}_{1\hat i}\frac{\pa a^{\ts}_{1k'}}{\pa x_{\hat i}}a^{\ts}_{1\hat k}\frac{\pa V}{\pa x_{k'}} \frac{\pa f}{\pa x_{\hat k}}(a^{\ts}\nabla)_1 f+a^{\ts}_{2\hat i}\frac{\pa a^{\ts}_{1k'}}{\pa x_{\hat i}}a^{\ts}_{1\hat k}\frac{\pa V}{\pa x_{k'}} \frac{\pa f}{\pa x_{\hat k}}(a^{\ts}\nabla)_2 f\right]\\
&&\sum_{\hat i,\hat k,k'=1}^3\left[a^{\ts}_{1\hat i}\frac{\pa a^{\ts}_{2k'}}{\pa x_{\hat i}}a^{\ts}_{2\hat k}\frac{\pa V}{\pa x_{k'}} \frac{\pa f}{\pa x_{\hat k}}(a^{\ts}\nabla)_1 f+a^{\ts}_{2\hat i}\frac{\pa a^{\ts}_{2k'}}{\pa x_{\hat i}}a^{\ts}_{2\hat k}\frac{\pa V}{\pa x_{k'}} \frac{\pa f}{\pa x_{\hat k}}(a^{\ts}\nabla)_2 f\right]\\
&=&-\frac{1}{2}\frac{\pa V}{\pa z}(a^{\ts}\nabla)_1f(a^{\ts}\nabla)_2f+\frac{1}{2}\frac{\pa V}{\pa z}(a^{\ts}\nabla)_1f(a^{\ts}\nabla)_2f=0;\\ 
\cJ_3&=&\sum_{\hat i,\hat k,k'=1}^3\left[ a^{\ts}_{1\hat i}a^{\ts}_{1\hat k}a^{\ts}_{1k'}\frac{\pa^2 V}{\pa x_{\hat i}\pa x_{k'}}\frac{\pa f}{\pa x_{\hat k}}(a^{\ts}\nabla)_1 f+a^{\ts}_{2\hat i}a^{\ts}_{1\hat k}a^{\ts}_{1k'}\frac{\pa^2 V}{\pa x_{\hat i}\pa x_{k'}}\frac{\pa f}{\pa x_{\hat k}}(a^{\ts}\nabla)_2 f\right]\\
&&\sum_{\hat i,\hat k,k'=1}^3\left[ a^{\ts}_{1\hat i}a^{\ts}_{2\hat k}a^{\ts}_{2k'}\frac{\pa^2 V}{\pa x_{\hat i}\pa x_{k'}}\frac{\pa f}{\pa x_{\hat k}}(a^{\ts}\nabla)_1 f+a^{\ts}_{2\hat i}a^{\ts}_{2\hat k}a^{\ts}_{2k'}\frac{\pa^2 V}{\pa x_{\hat i}\pa x_{k'}}\frac{\pa f}{\pa x_{\hat k}}(a^{\ts}\nabla)_2 f\right]\\
&=&\sum_{\hat i,k'=1}^3\left[ a^{\ts}_{1\hat i}a^{\ts}_{1k'}\frac{\pa^2 V}{\pa x_{\hat i}\pa x_{k'}} |(a^{\ts}\nabla)_1 f|^2+a^{\ts}_{2\hat i}a^{\ts}_{1k'}\frac{\pa^2 V}{\pa x_{\hat i}\pa x_{k'}}(a^{\ts}\nabla)_1 f(a^{\ts}\nabla)_2 f\right]\\
&&\sum_{\hat i,k'=1}^3\left[ a^{\ts}_{1\hat i}a^{\ts}_{2k'}\frac{\pa^2 V}{\pa x_{\hat i}\pa x_{k'}}(a^{\ts}\nabla)_2 f(a^{\ts}\nabla)_1 f+a^{\ts}_{2\hat i}a^{\ts}_{2k'}\frac{\pa^2 V}{\pa x_{\hat i}\pa x_{k'}}|(a^{\ts}\nabla)_2 f|^2\right]
\eeaa 
\beaa 
&=&\Big[\frac{\pa^2 V}{\pa x\pa x}+\frac{y^2}{4}\frac{\pa^2 V}{\pa z\pa z}-y\frac{\pa^2 V}{\pa x\pa z}\Big] |(a^{\ts}\nabla)_1 f|^2+\Big[\frac{\pa^2 V}{\pa y\pa y}+\frac{x^2}{4}\frac{\pa^2 V}{\pa z\pa z}+x\frac{\pa^2 V}{\pa y\pa z}\Big] |(a^{\ts}\nabla)_2 f|^2\\
&&+2\Big[\frac{\pa^2 V}{\pa x\pa y}+\frac{x}{2}\frac{\pa^2 V}{\pa x\pa z}-\frac{y}{2}\frac{\pa^2 V}{\pa y\pa z}-\frac{xy}{4}\frac{\pa^2 V}{\pa z\pa z}\Big] (a^{\ts}\nabla)_1 f(a^{\ts}\nabla)_2 f;\\
\cJ_4&=&-\sum_{\hat i,\hat k,k'=1}^3\left[a^{\ts}_{1\hat k}a^{\ts}_{1k'}\frac{\pa a^{\ts}_{1\hat i}}{\pa x_{\hat k}}\frac{\pa V}{\pa x_{k'}} \frac{\pa f}{\pa x_{\hat i}}(a^{\ts}\nabla)_1f+a^{\ts}_{1\hat k}a^{\ts}_{1k'}\frac{\pa a^{\ts}_{2\hat i}}{\pa x_{\hat k}}\frac{\pa V}{\pa x_{k'}} \frac{\pa f}{\pa x_{\hat i}}(a^{\ts}\nabla)_2f\right]\\
&&-\sum_{\hat i,\hat k,k'=1}^3\left[a^{\ts}_{2\hat k}a^{\ts}_{2k'}\frac{\pa a^{\ts}_{1\hat i}}{\pa x_{\hat k}}\frac{\pa V}{\pa x_{k'}} \frac{\pa f}{\pa x_{\hat i}}(a^{\ts}\nabla)_1f+a^{\ts}_{2\hat k}a^{\ts}_{2k'}\frac{\pa a^{\ts}_{2\hat i}}{\pa x_{\hat k}}\frac{\pa V}{\pa x_{k'}} \frac{\pa f}{\pa x_{\hat i}}(a^{\ts}\nabla)_2f\right]\\
&=&-\frac{1}{2}(a^{\ts}\nabla)_1V\pa_z f(a^{\ts}\nabla)_2f+\frac{1}{2}(a^{\ts}\nabla)_2V\pa_z f(a^{\ts}\nabla)_1f.
\eeaa
Summing up the above formulas, we get $\mathfrak{R}_{ab}$. We now compute the drift tensor term of $\mathfrak{R}_{zb}$.
By taking $b=-aa^{\ts}\nabla V$, we have
\beaa
\mathfrak{R}_{b}^z(\nabla f,\nabla f) 
&=&-\sum_{\hat i,\hat k=1}^{3}\left[ z^{\ts}_{1\hat i}\frac{\pa b_{\hat k}}{\pa x_{\hat i}}\frac{\pa f}{\pa x_{\hat k}}(z^{\ts}\nabla f)_1-b_{\hat k}\frac{\pa z^{\ts}_{i\hat i}}{\pa x_{\hat k}} \frac{\pa f}{\pa x_{\hat i}} (z^{\ts}\nabla f)_1\right]\\ 
&=&\sum_{k=1}^2\sum_{\hat i,\hat k,k'=1}^3\left[z^{\ts}_{1\hat i}\frac{\pa a^{\ts}_{k\hat k} }{\pa x_{\hat i}}a^{\ts}_{kk'}\frac{\pa V}{\pa x_{k'}} \frac{\pa f}{\pa x_{\hat k}}(z^{\ts}\nabla)_1 f \right]\\
&&+\sum_{k=1}^2\sum_{\hat i,\hat k,k'=1}^3\left[z^{\ts}_{1\hat i}\frac{\pa a^{\ts}_{kk'}}{\pa x_{\hat i}}a^{\ts}_{k\hat k}\frac{\pa V}{\pa x_{k'}} \frac{\pa f}{\pa x_{\hat k}}(z^{\ts}\nabla)_1 f\right]\\
&&+\sum_{k=1}^2\sum_{\hat i,\hat k,k'=1}^3\left[ z^{\ts}_{1\hat i}a^{\ts}_{k\hat k}a^{\ts}_{kk'}\frac{\pa^2 V}{\pa x_{\hat i}\pa x_{k'}}\frac{\pa f}{\pa x_{\hat k}}(z^{\ts}\nabla)_1 f\right]\\
&&-\sum_{k=1}^2\sum_{\hat i,\hat k,k'=1}^3\left[a^{\ts}_{k\hat k}a^{\ts}_{kk'}\frac{\pa z^{\ts}_{1\hat i}}{\pa x_{\hat k}}\frac{\pa V}{\pa x_{k'}} \frac{\pa f}{\pa x_{\hat i}}(z^{\ts}\nabla)_1f\right] \\
&=& \cJ_1^z+\cJ_2^z+\cJ_3^z+\cJ_4^z.
\eeaa
We further compute as blow by taking advantage of the constant matrix $z$,
\beaa
\cJ_1^z&=&\sum_{k=1}^2\sum_{\hat i,\hat k,k'=1}^3\left[z^{\ts}_{1\hat i}\frac{\pa a^{\ts}_{k\hat k} }{\pa x_{\hat i}}a^{\ts}_{kk'}\frac{\pa V}{\pa x_{k'}} \frac{\pa f}{\pa x_{\hat k}}(z^{\ts}\nabla)_1 f \right]=0;\\
\cJ_2^z&=&\sum_{k=1}^2\sum_{\hat i,\hat k,k'=1}^3\left[z^{\ts}_{1\hat i}\frac{\pa a^{\ts}_{kk'}}{\pa x_{\hat i}}a^{\ts}_{k\hat k}\frac{\pa V}{\pa x_{k'}} \frac{\pa f}{\pa x_{\hat k}}(z^{\ts}\nabla)_1 f\right]=0;\\
\cJ_4^z&=&-\sum_{k=1}^2\sum_{\hat i,\hat k,k'=1}^3\left[a^{\ts}_{k\hat k}a^{\ts}_{kk'}\frac{\pa z^{\ts}_{1\hat i}}{\pa x_{\hat k}}\frac{\pa V}{\pa x_{k'}} \frac{\pa f}{\pa x_{\hat i}}(z^{\ts}\nabla)_1f\right]=0\\
\cJ_3^z&=&\sum_{k=1}^2\sum_{\hat i,\hat k,k'=1}^3\left[ z^{\ts}_{1 \hat i}a^{\ts}_{k\hat k}a^{\ts}_{kk'}\frac{\pa^2 V}{\pa x_{\hat i}\pa x_{k'}}\frac{\pa f}{\pa x_{\hat k}}(z^{\ts}\nabla)_1 f\right]\\
&=& \Big(\frac{\pa^2 V}{\pa x\pa z }-\frac{y}{2}\frac{\pa^2 V}{\pa z\pa z}\Big)(a^{\ts}\nabla)_1 f (z^{\ts}\nabla)_1 f+\Big(\frac{\pa^2 V}{\pa y\pa z}+\frac{x}{2}\frac{\pa^2 V}{\pa z\pa z}\Big)(z^{\ts}\nabla)_1 f(a^{\ts}\nabla)_2 f. 
\eeaa
The proof is thus completed. \qed 
\end{proof}{}

\subsection{Displacement group}
In this subsection, we derive the generalized curvature dimension bound for displacement group, which is one example of three dimensional solvable Lie groups. We adapt the general setting from  \cite{baudoin2015subelliptic} below. Denote $\mathfrak{g}$ as the three dimensional solvable Lie algebra and denote $H\subset \mathfrak{g}$ as the horizontal subspace satisfying H\"ormander's condition, then for a given inner product $\la\cd ,\cd\ra$ on $H$, there exists a canonical basis $\{X,Y,Z\}$ for $(\mathfrak{g},H,\la\cd,\cd\ra)$, such that $\{X,Y\}$ forms an orthonormal basis for $H$ and satisfies the following Lie bracket generating condition for parameters $\alpha$ and $\beta\ge 0$:
\[
[X,Y]=Z,\quad [X,Z]=\alpha Y+\beta Z,\quad [Y,Z]=0.
\]
When the parameter $\alpha=0$ and $\beta\neq 0$, the Lie algebra $\mathfrak{g}$ has a faithful representation. In particular, it is shown in \cite{baudoin2015subelliptic} that the elements of $\mathfrak{g}$, in local coordinates $(\theta,x,y)$, corresponds to the following left-invariant differential operators:
\beaa
X=\frac{\partial}{\partial{\theta}},\quad Y=e^{\beta\theta}\frac{\partial}{\pa x}+\frac{\pa}{\pa y},\quad R=-\beta\frac{\pa}{\pa y},
\eeaa
with the following relation
\beaa
[X,Y]=\beta Y+R,\quad [X,R]=0,\quad [Y,R]=0.
\eeaa
In terms of local coordinates $(\theta,x,y)$, we have 
\beaa
X=\begin{pmatrix}
	1\\
	0\\
	0
\end{pmatrix},\quad Y=\begin{pmatrix}
	0\\
	e^{\beta\theta}\\
	1
\end{pmatrix},\quad R=\begin{pmatrix}
	0\\
	0\\
	-\beta 
\end{pmatrix}.
\eeaa
The corresponding Lie group of this special Lie algebra $\mathfrak{g}$ is called displacement group, denoted as $\mathsf G$. We choose $\{X,Y\}$ as the horizontal orthonormal basis for subalgebra $H$.
To fit into the general framework from the previous section, we take 
\bea\label{matrix a and z SE2}
a=(X,Y)=\begin{pmatrix}1&0\\
0&e^{\beta \theta}\\
0&1
\end{pmatrix}, \quad a^{\ts}= \begin{pmatrix}{}
1&0&0\\
0&e^{\beta \theta}&1
\end{pmatrix},\quad z^{\ts}=\begin{pmatrix}{}0&0&-g(\theta,x,y)
\end{pmatrix},
\eea
with $g(\theta,x,y)\neq 0$. Our focus here is to derive the curvature tensor in terms of $\pi=\frac{1}{Z}e^{-V}$.  We then use $(aa^{\ts})^{\dd}_{|H}$ as the horizontal metric on $H$. Thus the sub-Riemannian structure is given by $(\mathsf G,H,(aa^{\ts})^{\dd}_{|H})$. 
By direct computations, it is easy to show that, for general smooth function $f$, $\Gamma_1(f,\Gamma_1^z(f,f))\neq\Gamma_1^z(f,\Gamma_1(f,f))$. Hence classical Gamma $z$ calculus proposed in \cite{BaudoinGarofalo09} can not be extended for this case to derive zLSI. Thus we need to compute vector $G$ and the tensor term $\mathfrak{R}_{\pi}$. Following Theorem \ref{thm1}, we have the following z-Bochner's formula for $\mathsf G$.   
\begin{prop}\label{prop se2} For any smooth function $f\in C^{\infty}(\mathsf G)$, one has
\beaa
\Gamma_2(f,f)+\Gamma_2^{z,\pi}(f,f)&=&\|\mathfrak{Hess}_{a,z}f\|^2+\mathfrak{R}(\nabla f,\nabla f),
\eeaa
where 
\beaa
\Lambda_1^{\ts}&=& (0,\beta\pa_xf,\frac{\beta\pa_yf}{2},\beta\pa_xf,0,0,\frac{\beta\pa_yf}{2},0,-\beta\pa_{\theta}f); \\
\Lambda_2^{\ts}&=& (0,0,0,0,0,0,\lambda_6,0,\lambda_9);\\
\lambda_6&=&\frac{\pa_{\theta}g\pa_yf}{g}-\frac{\beta(a^{\ts}\nabla)_2f}{g^2}-\frac{\pa_{\theta}g\pa_yf}{g};\\ \lambda_9&=&\frac{(a^{\ts}\nabla)_2g\pa_yf}{g}+\frac{\beta \pa_{\theta}f}{g^2}-\frac{(a^{\ts}\nabla)_2g\pa_yf}{g};
\eeaa
And
\beaa
&&\mathfrak{R}_{ab}(\nabla f,\nabla f)-\Lambda_1^{\ts}Q^{\ts}Q\Lambda_1-\Lambda_2^{\ts}P^{\ts}P\Lambda_2+D^{\ts}D+E^{\ts}E)\\
&=&\Gamma_1(\log g,\log g)\Gamma_1^z(f,f)-\beta^2(1+\frac{1}{g^2})\Gamma_1(f,f)+\frac{\beta^2}{2g^2}\Gamma_1^z(f,f)\\
&&+\beta^2e^{\beta \theta}\frac{\pa f}{ \pa x}(a^{\ts}\nabla )_2f +\beta e^{\beta\theta} (a^{\ts}\nabla)_2V\frac{\pa f}{\pa x}(a^{\ts}\nabla )_1 f  +\beta e^{\beta \theta}\frac{\pa V}{\pa x}(a^{\ts}\nabla)_2f(a^{\ts}\nabla)_1f \\
&&+\frac{\pa^2 V}{\pa \theta\pa \theta} |(a^{\ts}\nabla)_1 f|^2+2(e^{\beta\theta}\frac{\pa^2 V}{\pa \theta\pa x}+\frac{\pa^2 V}{\pa \theta \pa y} )(a^{\ts}\nabla)_1 f(a^{\ts}\nabla)_2 f\\
&&+\sum_{\hat i,k'=1}^3a^{\ts}_{2\hat i}a^{\ts}_{2k'}\frac{\pa^2 V}{\pa x_{\hat i}\pa x_{k'}})|(a^{\ts}\nabla)_2 f|^2-\beta e^{\beta\theta}(a^{\ts}\nabla)_1V \frac{\pa f}{\pa x}(a^{\ts}\nabla)_2f;\\
\mathfrak{R}_{zb}(\nabla f,\nabla f)&=& \sum_{i=1}^2\sum_{i',\hat i=1}^{3}  a^{\ts}_{ii'}a^{\ts}_{i \hat i} \frac{\pa^2 z^{\ts}_{13}}{\partial x_{i'}\partial x_{\hat i}}\pa_yf (z^{\ts}\nabla)_1f -\sum_{k=1}^2(a^{\ts}\nabla)_kz^{\ts}_{13}(a^{\ts}\nabla)_kV\pa_yf(z^{\ts}\nabla)_1f
\\ && -g\frac{\pa^2 V}{\pa \theta\pa y}(a^{\ts}\nabla)_1 f(z^{\ts}\nabla)_1 f-g(e^{\beta\theta}\frac{\pa^2 V}{\pa x\pa y}+\frac{\pa^2 V}{\pa y\pa y} )(a^{\ts}\nabla)_2 f(z^{\ts}\nabla)_1 f;\\
\mathfrak{R}_{\pi}(\nabla f,\nabla f)&=& -2\sum_{l=1}^2\sum_{l',\hat l=1}^3a^{\ts}_{ll'}a^{\ts}_{l\hat l}\frac{\pa^2 z^{\ts}_{13}}{\pa x_{l'}\pa x_{\hat l}}\pa_yf (z^{\ts}\nabla)_1f\\
&&-2\Gamma_1(\log \pi,\log g)|(z^{\ts}\nabla)_1f|^2-2\Gamma_1(\log g,\log g)|(z^{\ts}\nabla)_1f|^2.
\eeaa
In particular, we have  
\beaa
\sum_{\hat i,k'=1}^3a^{\ts}_{2\hat i}a^{\ts}_{2k'}\frac{\pa^2 V}{\pa x_{\hat i}\pa x_{k'}}|(a^{\ts}\nabla)_2 f|^2&=&\Big[e^{2\beta\theta}\frac{\pa^2 V}{\pa x\pa x}+2e^{\beta \theta}\frac{\pa^2 V}{\pa x\pa y}+\frac{\pa^2 V}{\pa y\pa y}\Big] |(a^{\ts}\nabla)_2 f|^2;\\
\sum_{i=1}^2\sum_{i',\hat i=1}^{3}  a^{\ts}_{ii'}a^{\ts}_{i \hat i} \frac{\pa^2 z^{\ts}_{13}}{\partial x_{i'}\partial x_{\hat i}}\pa_yf (z^{\ts}\nabla)_1f&=&\Big[\frac{\pa ^2 g}{\pa \theta\pa \theta}+e^{2\beta\theta}\frac{\pa ^2 g}{\pa x\pa x}+\frac{\pa ^2 g}{\pa y\pa y}+2e^{\beta\theta}\frac{\pa ^2 g}{\pa x\pa y}\Big]\frac{|(z^{\ts}\nabla)_1f|^2}{g}.
\eeaa
\end{prop}{}
{ The proof of Proposition \ref{prop se2} follows from
the proof of Theorem \ref{thm1} (i.e. Theorem \ref{thm: generalized gamma 2 z with drift}), Lemma \ref{lemma: vector and mat for se2}, Lemma \ref{SE2 hess} and Lemma \ref{SE2 tensor} below. The following convergence result follows directly from Theorem \ref{prop1.4}.} 
\begin{prop} If there exists $\kappa>0$ as shown in Theorem \ref{prop1.4}, the exponential dissipation result in $L^1$ distance holds:
\begin{equation*}
    \int |\rho(t,x)-\pi(x)| dx=O(e^{-\kappa t}). 
\end{equation*}
\end{prop}
Similarly, we formulate the curvature tensor into a matrix format of $\mathfrak R$. Using the fact $e^{\beta \theta}\frac{\partial f}{\partial x}=(a^{\ts}\nabla)_2f+\frac{1}{g}(z^{\ts}\nabla f)_1f$, we have the following representation.
\begin{cor}
The matrix $\mathfrak{R}$ associated with $\mathsf G$ has the following representation
\beaa
\mathfrak{R}_{11}&=& \frac{\pa^2 V}{\pa \theta\pa \theta}-\beta^2(1+\frac{1}{g^2});\\
\mathfrak{R}_{22}&=& \Big[e^{2\beta\theta}\frac{\pa^2 V}{\pa x\pa x}+2e^{\beta \theta}\frac{\pa^2 V}{\pa x\pa y}+\frac{\pa^2 V}{\pa y\pa y}\Big]-\frac{\beta^2}{g^2}-\beta(a^{\ts}\nabla)_1V;\\
\mathfrak{R}_{33}&=&\frac{\beta^2}{2g^2}-\Gamma_1(\log g,\log g)-2\Gamma_1(\log \pi,\log g)-\Gamma_1(\log g,V)-\frac{1}{g}\Big[\frac{\pa ^2 g}{\pa \theta\pa \theta}+e^{2\beta\theta}\frac{\pa ^2 g}{\pa x\pa x}+\frac{\pa ^2 g}{\pa y\pa y}+2e^{\beta\theta}\frac{\pa ^2 g}{\pa x\pa y}\Big];  \\
\mathfrak{R}_{12}&=&\mathfrak{R}_{21}=\frac{1}{2}\Big(\beta e^{\beta \theta}\frac{\pa V}{\pa x}+2(e^{\beta\theta}\frac{\pa^2 V}{\pa \theta\pa x}+\frac{\pa^2 V}{\pa \theta \pa y} )+\beta (a^{\ts}\nabla)_2V \Big);\\
\mathfrak{R}_{13}&=&\mathfrak{R}_{31}=\frac{1}{2}\Big(\frac{\beta}{g}(a^{\ts}\nabla)_2V-g\frac{\pa^2 V}{\pa \theta\pa y}\Big);\\
\mathfrak{R}_{23}&=&\mathfrak{R}_{32}=-\frac{1}{2}\Big(\frac{\beta}{g}(a^{\ts}\nabla)_1-\frac{\beta^2}{g} \Big)-\frac{1}{2}g(e^{\beta\theta}\frac{\pa^2 V}{\pa x\pa y}+\frac{\pa^2 V}{\pa y\pa y} ).
\eeaa
\end{cor}{}
\begin{proof}
The derivation for the explicit form of matrix $\mathfrak R$ follows from a similar equivalent representation as shown in the proof of Corollary \ref{matrix R heisenberg} and the explicit bilinear terms derived in Proposition \ref{prop se2}. \qed 
\end{proof}

\begin{rem}
By taking $g(\theta,x,y)=\beta$ as a constant, Proposition \ref{prop se2} reduces to a simple version, in particular, the tensors reduce to be 
\beaa
&&\mathfrak{R}_{ab}(\nabla f,\nabla f)-\Lambda_1^{\ts}Q^{\ts}Q\Lambda_1-\Lambda_2^{\ts}P^{\ts}P\Lambda_2+D^{\ts}D+E^{\ts}E)\\
&=&-(1+\beta^2)\Gamma_1(f,f)+\frac{1}{2}\Gamma_1^z(f,f)\\
&&+\beta^2e^{\beta \theta}\frac{\pa f}{ \pa x}(a^{\ts}\nabla )_2f +\beta e^{\beta\theta} (a^{\ts}\nabla)_2V\frac{\pa f}{\pa x}(a^{\ts}\nabla )_1 f  +\beta e^{\beta \theta}\frac{\pa V}{\pa x}(a^{\ts}\nabla)_2f(a^{\ts}\nabla)_1f \\
&&+\frac{\pa^2 V}{\pa \theta\pa \theta} |(a^{\ts}\nabla)_1 f|^2+2(e^{\beta\theta}\frac{\pa^2 V}{\pa \theta\pa x}+\frac{\pa^2 V}{\pa \theta \pa y} )(a^{\ts}\nabla)_1 f(a^{\ts}\nabla)_2 f\\
&&+\sum_{\hat i,k'=1}^3a^{\ts}_{2\hat i}a^{\ts}_{2k'}\frac{\pa^2 V}{\pa x_{\hat i}\pa x_{k'}})|(a^{\ts}\nabla)_2 f|^2-\beta e^{\beta\theta}(a^{\ts}\nabla)_1V \frac{\pa f}{\pa x}(a^{\ts}\nabla)_2f;\\
\mathfrak{R}_{zb}(\nabla f,\nabla f)&=& -\beta\frac{\pa^2 V}{\pa \theta\pa y}(a^{\ts}\nabla)_1 f(z^{\ts}\nabla)_1 f-\beta(e^{\beta\theta}\frac{\pa^2 V}{\pa x\pa y}+\frac{\pa^2 V}{\pa y\pa y} )(a^{\ts}\nabla)_2 f(z^{\ts}\nabla)_1 f;\\
\mathfrak{R}_{\pi}(\nabla f,\nabla f)&=&0.
\eeaa
\end{rem}{}


{ Next, we present the following three key lemmas. }
\begin{lem}\label{lemma: vector and mat for se2}
For displacement group $\mathsf G$, we have 
\beaa
Q&=&\left(
\begin{array}{ccccccccc}
 1 & 0 & 0 & 0 & 0 & 0 & 0 & 0 & 0 \\
 0 & e^{\beta \theta} & 1 & 0 & 0 & 0 & 0 & 0 & 0 \\
 0 & 0 & 0 & e^{\beta \theta} & 0 & 0 & 1 & 0 & 0 \\
 0 & 0 & 0 & 0 & e^{2 \beta \theta} & e^{\beta \theta} & 0 & e^{\beta \theta} & 1 \\
\end{array}
\right);\\
P&=&\left(
\begin{array}{ccccccccc}
 0 & 0 & 0 & 0 & 0 & 0 & -g(\theta,x,y) & 0 & 0 \\
 0 & 0 & 0 & 0 & 0 & 0 & 0 & -g(\theta,x,y) e^{\beta \theta} & -g(\theta,x,y) \\
\end{array}
\right);\\
D^{\ts}&=&(0,\beta e^{\beta\theta}\pa_xf ,0,0),\quad E^{\ts}=(- \pa_yf \pa_{\theta}g , -\pa_yf \pa_yg-e^{\beta \theta} \pa_yf \pa_xg);\\
C^{\ts}&=&(0,\beta e^{\beta \theta}\pa_y f+\beta e^{2\beta\theta} \pa_xf,0,0,-\beta e^{2\beta \theta}\pa_{\theta} f,-\beta e^{\beta \theta}\pa_{\theta} f,0,0,0).
\eeaa
$F=\left(
\begin{array}{c}
 0 \\
 0 \\
 g\pa_{\theta}g \pa_yf \\
 0 \\
 0 \\
 e^{\beta \theta} g \pa_yf \pa_yg+e^{2 \beta \theta} g \pa_yf \pa_x g \\
 0 \\
 0 \\
 g \pa_yf \pa_yg+e^{\beta \theta} g \pa_y f \pa_x g \\
\end{array}
\right),\quad G=\left(
\begin{array}{c}
 0 \\
 0 \\
 -2 g \pa_yf \pa_{\theta}g \\
 0 \\
 0 \\
 -2 e^{\beta \theta} g \pa_yf \pa_yg-2 e^{2 \beta \theta} g \pa_yf \pa_xg\\
 0 \\
 0 \\
 -2 g \pa_yf \pa_yg-2 e^{\beta \theta} g \pa_yf \pa_xg \\
\end{array}
\right)$.
\end{lem}{}
\begin{proof}
The proof follows by plugging matrices $a$ and $z$ from \eqref{matrix a and z SE2} into Notation \ref{notation}.\qed 
\end{proof}

\begin{lem}\label{SE2 hess} On displacement group $\mathsf G$, we have
\beaa
&&[QX+D]^{\ts}[QX+D]+[PX+E]^{\ts}[PX+E]+2[C^{\ts}+F^{\ts}+G^{\ts}]X\\
&=&\|\mathfrak{Hess}_{a,z}f\|^2-\Lambda_1^{\ts}Q^{\ts}Q\Lambda_1-\Lambda_2^{\ts}P^{\ts}P\Lambda_2+D^{\ts}D+E^{\ts}E.
\eeaa
In particular, we have 
\beaa
\|\mathfrak{Hess}_{a,z}f\|^2&=&[X+\Lambda_1]^{\ts}Q^{\ts}Q[X+\Lambda_1]+[X+\Lambda_2]^{\ts}P^{\ts}P[X+\Lambda_2];\\
\Lambda_1^{\ts}&=& (0,\beta\pa_xf,\frac{\beta\pa_yf}{2},\beta\pa_xf,0,0,\frac{\beta\pa_yf}{2},0,-\beta\pa_{\theta}f); \\
\Lambda_2^{\ts}&=& (0,0,0,0,0,0,\lambda_6,0,\lambda_9);\\
\lambda_6&=&\frac{\pa_{\theta}g\pa_yf}{g}-\frac{\beta(a^{\ts}\nabla)_2f}{g^2}-\frac{\pa_{\theta}g\pa_yf}{g};
\eeaa 
\beaa \lambda_9&=&\frac{(a^{\ts}\nabla)_2g\pa_yf}{g}+\frac{\beta \pa_{\theta}f}{g^2}-\frac{(a^{\ts}\nabla)_2g\pa_yf}{g};\\
&&-\Lambda_1^{\ts}Q^{\ts}Q\Lambda_1-\Lambda_2^{\ts}P^{\ts}P\Lambda_2+D^{\ts}D+E^{\ts}E\\
&=&\Gamma_1(\log g,\log g)\Gamma_1^z(f,f)-\beta^2(1+\frac{1}{g^2})\Gamma_1(f,f)+\frac{\beta^2}{2g^2}\Gamma_1^z(f,f).
\eeaa
\end{lem}{}
\begin{lem}\label{SE2 tensor}
By routine computations, we obtain
\beaa
\mathfrak{R}_{ab}(\nabla f,\nabla f)&=&\beta^2e^{\beta \theta}\frac{\pa f}{ \pa x}(a^{\ts}\nabla )_2f +\beta e^{\beta\theta} (a^{\ts}\nabla)_2V\frac{\pa f}{\pa x}(a^{\ts}\nabla )_1 f  +\beta e^{\beta \theta}\frac{\pa V}{\pa x}(a^{\ts}\nabla)_2f(a^{\ts}\nabla)_1f \\
&&+\frac{\pa^2 V}{\pa \theta\pa \theta} |(a^{\ts}\nabla)_1 f|^2+2(e^{\beta\theta}\frac{\pa^2 V}{\pa \theta\pa x}+\frac{\pa^2 V}{\pa \theta \pa y} )(a^{\ts}\nabla)_1 f(a^{\ts}\nabla)_2 f\\
&&+\sum_{\hat i,k'=1}^3a^{\ts}_{2\hat i}a^{\ts}_{2k'}\frac{\pa^2 V}{\pa x_{\hat i}\pa x_{k'}}|(a^{\ts}\nabla)_2 f|^2-\beta e^{\beta\theta}(a^{\ts}\nabla)_1V \frac{\pa f}{\pa x}(a^{\ts}\nabla)_2f;\\
\mathfrak{R}_{zb}(\nabla f,\nabla f)&=& \sum_{i=1}^2\sum_{i',\hat i=1}^{3}  a^{\ts}_{ii'}a^{\ts}_{i \hat i} \frac{\pa^2 z^{\ts}_{1\hat k}}{\partial x_{i'}\partial x_{\hat i}}\pa_yf (z^{\ts}\nabla)_1f -\sum_{k=1}^2(a^{\ts}\nabla)_kz^{\ts}_{13}(a^{\ts}\nabla)_kV\pa_yf(z^{\ts}\nabla)_1f
\\ && -g\frac{\pa^2 V}{\pa \theta\pa y}(a^{\ts}\nabla)_1 f(z^{\ts}\nabla)_1 f-g(e^{\beta\theta}\frac{\pa^2 V}{\pa x\pa y}+\frac{\pa^2 V}{\pa y\pa y} )(a^{\ts}\nabla)_2 f(z^{\ts}\nabla)_1 f;\\
\mathfrak{R}_{\pi}(\nabla f,\nabla f)&=& -2\sum_{l=1}^2\sum_{l',\hat l=1}^3a^{\ts}_{ll'}a^{\ts}_{l\hat l}\frac{\pa^2 z^{\ts}_{13}}{\pa x_{l'}\pa x_{\hat l}}\pa_yf (z^{\ts}\nabla)_1f-2\sum_{l=1}^2\sum_{l',\hat l=1}^3a^{\ts}_{ll'}a^{\ts}_{l\hat l}\frac{\pa z^{\ts}_{13}}{\pa x_{\hat l}}\frac{\pa z^{\ts}_{13}}{\pa x_{l'}} |\pa_yf|^2\\
&&-2 \sum_{l=1}^2\sum_{\hat l=1}^3(a^{\ts}\nabla)_l\log \pi a^{\ts}_{l\hat l}\frac{\pa z^{\ts}_{13}}{\pa x_{\hat l}}\pa_yf (z^{\ts}\nabla)_1f.
\eeaa
\end{lem}{}
\begin{proof}[Proof of Lemma \ref{SE2 hess}]
According to Lemma \ref{lemma: vector and mat for se2} and observe the fact that $G=-2F$ and $(a^{\ts}\nabla)_2f=e^{\beta\theta}\pa_xf+\pa_yf$, we first have 
\beaa
2C^{\ts}X&=&2[\beta e^{\beta \theta}\pa_y f+\beta e^{2\beta\theta} \pa_xf]\frac{\pa^2 f}{\pa \theta \pa x}+2[-\beta e^{2\beta \theta}\pa_{\theta} f]\frac{\pa^2 f}{\pa x \pa x}+2[-\beta e^{\beta \theta}\pa_{\theta} f]\frac{\pa^2 f}{\pa x\pa y};\\
2[F^{\ts}+G^{\ts}]X&=&-2\Big(g\pa_{\theta}g\pa_yf\frac{\pa^2 f}{\pa \theta \pa y}+e^{\beta\theta}g (a^{\ts}\nabla)_2g\pa_yf\frac{\pa^2 f}{\pa x\pa y}+g(a^{\ts}\nabla)_2g\pa_yf\frac{\pa^2 f}{\pa y\pa y} \Big).
\eeaa
By direct computations, we have 
\beaa
&&[QX+D]^{\ts}[QX+D]+ [PX+E]^{\ts}[PX+E]+2C^{\ts}X+2F^{\ts}X+2G^{\ts}X\\
&=&
\left[ \frac{\pa^2 f}{\pa \theta\pa \theta}\right]^2+\left[ e^{2\beta\theta} \frac{\pa^2 f}{\pa x\pa x}+2e^{\beta\theta}\frac{\pa^2 f}{\pa x\pa y}+\frac{\pa^2 f}{\pa y\pa y}\right]^2+\left[e^{\beta \theta}\frac{\pa^2 f}{\pa \theta \pa x}+\frac{\pa^2 f}{\pa \theta\pa y}+\beta e^{\beta\theta}\frac{\pa f}{\pa x} \right]^2\\
&&+\left[e^{\beta \theta}\frac{\pa^2 f}{\pa \theta \pa x}+\frac{\pa^2 f}{\pa \theta\pa y} \right]^2+\left[-g\frac{\pa^2 f}{\pa \theta\pa y} -\pa_yf\pa_{\theta}g\right]^2+ \left[-g e ^{\beta\theta}\frac{\pa^2 f}{\pa x\pa y} -g\frac{\pa^2f}{\pa y\pa y}-(a^{\ts}\nabla)_2g \pa_yf\right]^2\\
&&+2[\beta e^{\beta \theta}\pa_y f+\beta e^{2\beta\theta} \pa_xf]\frac{\pa^2 f}{\pa \theta \pa x}+2[-\beta e^{2\beta \theta}\pa_{\theta} f]\frac{\pa^2 f}{\pa x \pa x}+2[-\beta e^{\beta \theta}\pa_{\theta} f]\frac{\pa^2 f}{\pa x\pa y}\\
&&-2\Big(g\pa_{\theta}g\pa_yf\frac{\pa^2 f}{\pa \theta \pa y}+e^{\beta\theta}g (a^{\ts}\nabla)_2g\pa_yf\frac{\pa^2 f}{\pa x\pa y}+g(a^{\ts}\nabla)_2g\pa_yf\frac{\pa^2 f}{\pa y\pa y} \Big)
\eeaa 
\beaa 
&=&\left[ \frac{\pa^2 f}{\pa \theta\pa \theta}\right]^2+\left[ e^{2\beta\theta} \frac{\pa^2 f}{\pa x\pa x}+2e^{\beta\theta}\frac{\pa^2 f}{\pa x\pa y}+\frac{\pa^2 f}{\pa y\pa y}\right]^2+\left[e^{\beta \theta}\frac{\pa^2 f}{\pa \theta \pa x}+\frac{\pa^2 f}{\pa \theta\pa y}+\beta e^{\beta\theta}\frac{\pa f}{\pa x} \right]^2\\
&&+\left[e^{\beta \theta}\frac{\pa^2 f}{\pa \theta \pa x}+\frac{\pa^2 f}{\pa \theta\pa y} \right]^2+\left[-g\frac{\pa^2 f}{\pa \theta\pa y} -\pa_yf\pa_{\theta}g\right]^2+ \left[-g e ^{\beta\theta}\frac{\pa^2 f}{\pa x\pa y} -g\frac{\pa^2f}{\pa y\pa y}-(a^{\ts}\nabla)_2g \pa_yf\right]^2\\
&&+2\beta (a^{\ts}\nabla)_2f\Big[ e^{\beta \theta} \frac{\pa^2 f}{\pa \theta \pa x}+\frac{\pa^2 f}{\pa \theta \pa y}\Big]-2\beta (a^{\ts}\nabla)_2f\frac{\pa^2 f}{\pa \theta \pa y}-2g\pa_{\theta}g\pa_yf\frac{\pa^2 f}{\pa \theta \pa y}\\
&&-2\beta \pa_{\theta}f\Big[2e^{\beta\theta}\frac{\pa^2 f}{\pa x\pa y}+e^{2\beta\theta}\frac{\pa^2 f}{\pa x\pa x}+\frac{\pa^2 f}{\pa y\pa y} \Big]+2\beta \pa_{\theta}f\Big[e^{\beta\theta}\frac{\pa^2 f}{\pa x\pa y}+\frac{\pa^2 f}{\pa y\pa y} \Big]\\
&&-2g(a^{\ts}\nabla)_2g\pa_yf\Big[e^{\beta\theta}\frac{\pa^2f}{\pa x\pa y}+\frac{\pa^2 f}{\pa y\pa y}\Big].
\eeaa
Completing square for the above terms, we have 
\beaa 
&&[QX+D]^{\ts}[QX+D]+[PX+E]^{\ts}[PX+E]+2C^{\ts}X+2F^{\ts}X+2G^{\ts}X\\
&=&\left[ \frac{\pa^2 f}{\pa \theta\pa \theta}\right]^2+\left[ e^{2\beta\theta} \frac{\pa^2 f}{\pa x\pa x}+2e^{\beta\theta}\frac{\pa^2 f}{\pa x\pa y}+\frac{\pa^2 f}{\pa y\pa y}-\beta \pa_{\theta}f\right]^2-\beta^2|\pa_{\theta}f|^2\\
&&+\left[e^{\beta \theta}\frac{\pa^2 f}{\pa \theta \pa x}+\frac{\pa^2 f}{\pa \theta\pa y}+\beta e^{\beta\theta}\frac{\pa f}{\pa x} \right]^2
+\left[e^{\beta \theta}\frac{\pa^2 f}{\pa \theta \pa x}+\frac{\pa^2 f}{\pa \theta\pa y}+\beta (a^{\ts}\nabla)_2f \right]^2-\beta^2|(a^{\ts}\nabla)_2f|^2\\
&&+\Big[g\frac{\pa^2 f}{\pa \theta\pa y}+\pa_{\theta}g\pa_yf-\frac{\beta(a^{\ts}\nabla)_2f}{g}-\pa_{\theta}g\pa_yf \Big]^2-\Big[\frac{\beta(a^{\ts}\nabla)_2f}{g}+\pa_{\theta}g\pa_yf \Big]^2\\
&&+\Big[ge^{\beta\theta}\frac{\pa^2 f}{\pa x\pa y}+g\frac{\pa^2 f}{\pa y\pa y} +(a^{\ts}\nabla)_2g\pa_yf+\frac{\beta \pa_{\theta}f}{g}-(a^{\ts}\nabla)_2g\pa_yf \Big]^2-\Big[\frac{\beta \pa_{\theta}f}{g}-(a^{\ts}\nabla)_2g\pa_yf \Big]^2\\
&&+2\Big[\frac{\beta(a^{\ts}\nabla)_2f}{g}+\pa_{\theta}g\pa_yf \Big]\pa_{\theta}g\pa_yf-2\pa_yf(a^{\ts}\nabla)_2g\times \Big[ \frac{\beta \pa_{\theta}f}{g}-(a^{\ts}\nabla)_2g\pa_yf \Big]
\eeaa
The first order terms generate $-\Lambda_1^{\ts}Q^{\ts}Q\Lambda_1-\Lambda_2^{\ts}P^{\ts}P\Lambda_2+D^{\ts}D+E^{\ts}E$ and the sum of square terms generate vectors $\Lambda_1$ and $\Lambda_2$. We further formulate the above two terms as below 
\beaa
&&\left[e^{\beta \theta}\frac{\pa^2 f}{\pa \theta \pa x}+\frac{\pa^2 f}{\pa \theta\pa y}+\beta e^{\beta\theta}\frac{\pa f}{\pa x} \right]^2
+\left[e^{\beta \theta}\frac{\pa^2 f}{\pa \theta \pa x}+\frac{\pa^2 f}{\pa \theta\pa y}+\beta (a^{\ts}\nabla)_2f \right]^2\\
&=&2\left[e^{\beta \theta}\frac{\pa^2 f}{\pa \theta \pa x}+\frac{\pa^2 f}{\pa \theta\pa y}+\beta e^{\beta\theta}\frac{\pa f}{\pa x}+\frac{\beta}{2}\pa_yf \right]^2+\frac{\beta^2}{2}|\pa_yf|^2.
\eeaa
Adding $\frac{\beta^2}{2}|\pa_yf|^2$ into the term $-\Lambda_1^{\ts}Q^{\ts}Q\Lambda_1-\Lambda_2^{\ts}P^{\ts}P\Lambda_2+D^{\ts}D+E^{\ts}E$ again, we further expand as below, 
\beaa
&&-\Lambda_1^{\ts}Q^{\ts}Q\Lambda_1-\Lambda_2^{\ts}P^{\ts}P\Lambda_2+D^{\ts}D+E^{\ts}E\\
&=&-\beta^2[|\pa_{\theta}f|^2+|(a^{\ts}\nabla)_2f|^2]-\Big[\frac{\beta \pa_{\theta}f}{g}-(a^{\ts}\nabla)_2g\pa_yf \Big]^2-\Big[\frac{\beta(a^{\ts}\nabla)_2f}{g}+\pa_{\theta}g\pa_yf \Big]^2\\
&&+2\Big[\frac{\beta(a^{\ts}\nabla)_2f}{g}+\pa_{\theta}g\pa_yf \Big]\pa_{\theta}g\pa_yf-2\pa_yf(a^{\ts}\nabla)_2g\times \Big[ \frac{\beta \pa_{\theta}f}{g}-(a^{\ts}\nabla)_2g\pa_yf \Big]+\frac{\beta^2}{2}|\pa_yf|^2\\ 
&=& -\beta^2\Gamma_1(f,f)-\frac{\beta^2}{g^2}|(a^{\ts}\nabla)_1f|^2
-|(a^{\ts}\nabla)_2(\log g)|^2 |(z^{\ts}\nabla)_1f|^2-2\frac{\beta}{g}(a^{\ts}\nabla)_2\log g(a^{\ts}\nabla)_1f(z^{\ts}\nabla)_1f\\
&&-\frac{\beta^2}{g^2}|(a^{\ts}\nabla)_2f|^2-|(a^{\ts}\nabla)_1\log g|^2|(z^{\ts}\nabla)_1f|^2+2\frac{\beta}{g} (a^{\ts}\nabla)_1\log g (a^{\ts}\nabla)_2f(z^{\ts}\nabla)_1f
\eeaa 
\beaa 
&&-2\frac{\beta}{g}(a^{\ts}\nabla)_1\log g(a^{\ts}\nabla)_2f (z^{\ts}\nabla)_1f
+2|(a^{\ts}\nabla)_1\log g|^2|(z^{\ts}\nabla)_1f|^2\\
&&+2\frac{\beta}{g}(a^{\ts}\nabla)_2\log g(a^{\ts}\nabla)_1f(z^{\ts}\nabla)_1f+2|(a^{\ts}\nabla)_2\log g|^2|(z^{\ts}\nabla)_1f|^2+\frac{\beta^2}{2g^2}\Gamma_1^z(f,f).
\eeaa
By grouping the bilinear terms of $\nabla f$, we get 
\beaa
&&-\Lambda_1^{\ts}Q^{\ts}Q\Lambda_1-\Lambda_2^{\ts}P^{\ts}P\Lambda_2+D^{\ts}D+E^{\ts}E\\
&=&\Gamma_1(\log g,\log g)\Gamma_1^z(f,f)-\beta^2(1+\frac{1}{g^2})\Gamma_1(f,f)+\frac{\beta^2}{2g^2}\Gamma_1^z(f,f).
\eeaa
\qed
\end{proof}{}

We are now left to compute the three tensor terms.

\begin{proof}[Proof of Lemma \ref{SE2 tensor}]
For displacement group $\mathbf{G}$, we have $n=2$ and $m=1$. Recall from Theorem \ref{thm1}, we denote 
$\mathfrak{R}_{ab}(\nabla f,\nabla f)=\mathfrak{R}_a(\nabla f,\nabla f)+\mathfrak{R}_b(\nabla f,\nabla f)$ where $\mathfrak{R}_b(\nabla f,\nabla f)$ represents the tensor term involving drift $b$. We thus have
\beaa
\mathfrak{R}_a(\nabla f,\nabla f)
&=&\sum_{i,k=1}^2\sum_{i',\hat i,\hat k=1}^{3} \la a^{\ts}_{ii'} (\frac{\partial a^{\ts}_{i \hat i}}{\partial x_{i'}} \frac{\partial a^{\ts}_{k\hat k}}{\partial x_{\hat i}}\frac{\partial f}{\partial x_{\hat k}}) ,(a^{\ts}\nabla)_kf\ra_{\hR^2}  \nonumber\\
&&+\sum_{i,k=2}^n\sum_{i',\hat i,\hat k=1}^{3} \la a^{\ts}_{ii'}a^{\ts}_{i \hat i} (\frac{\partial }{\partial x_{i'}} \frac{\partial a^{\ts}_{k\hat k}}{\partial x_{\hat i}})(\frac{\partial f}{\partial x_{\hat k}}) ,(a^{\ts}\nabla)_kf\ra_{\hR^2} \nonumber\\
&&-\sum_{i,k=1}^2\sum_{i',\hat i,\hat k=1}^{3} \la a^{\ts}_{k\hat k}\frac{\partial a^{\ts}_{ii'}}{\partial x_{\hat k}} \frac{\partial a^{\ts}_{i \hat i}}{\partial x_{i'}} \frac{\partial f}{\partial x_{\hat i}})  ,(a^{\ts}\nabla)_kf\ra_{\hR^2}  \nonumber\\
&&-\sum_{i,k=1}^2\sum_{i',\hat i,\hat k=1}^{3} \la a^{\ts}_{k\hat k} a^{\ts}_{ii'} (\frac{\partial }{\partial x_{\hat k}} \frac{\partial a^{\ts}_{i \hat i}}{\partial x_{i'}}) \frac{\partial f}{\partial x_{\hat i}}  ,(a^{\ts}\nabla)_kf\ra_{\hR^2},  \nonumber\\
&=& \cI_1+\cI_2+\cI_3+\cI_4.
\eeaa
By direct computations, we have
\beaa
\cI_1&=&\sum_{i=1}^2\sum_{i',\hat i,\hat k=1}^{3}  \Big[ a^{\ts}_{ii'} (\frac{\partial a^{\ts}_{i \hat i}}{\partial x_{i'}} \frac{\partial a^{\ts}_{1\hat k}}{\partial x_{\hat i}}\frac{\partial f}{\partial x_{\hat k}})(a^{\ts}\nabla)_1f+  a^{\ts}_{ii'} (\frac{\partial a^{\ts}_{i \hat i}}{\partial x_{i'}} \frac{\partial a^{\ts}_{2\hat k}}{\partial x_{\hat i}}\frac{\partial f}{\partial x_{\hat k}}) (a^{\ts}\nabla)_2f\Big]=0 ; \\
\cI_2&=&\sum_{i=2}^n\sum_{i',\hat i,\hat k=1}^{3} \Big[ a^{\ts}_{ii'}a^{\ts}_{i \hat i} (\frac{\partial }{\partial x_{i'}} \frac{\partial a^{\ts}_{1\hat k}}{\partial x_{\hat i}})(\frac{\partial f}{\partial x_{\hat k}}) (a^{\ts}\nabla)_1f+ a^{\ts}_{ii'}a^{\ts}_{i \hat i} (\frac{\partial }{\partial x_{i'}} \frac{\partial a^{\ts}_{2\hat k}}{\partial x_{\hat i}})(\frac{\partial f}{\partial x_{\hat k}})(a^{\ts}\nabla)_2f\Big]\\
&=&a^{\ts}_{11}a^{\ts}_{11}\frac{\pa^2}{\pa \theta \pa \theta}a^{\ts}_{22}\frac{\pa f}{ \pa x}(a^{\ts}\nabla )_2f=\beta^2e^{\beta \theta}\frac{\pa f}{ \pa x}(a^{\ts}\nabla )_2f;  \\
\cI_3&=&-\sum_{i=1}^2\sum_{i',\hat i,\hat k=1}^{3} \Big[ a^{\ts}_{1\hat k}\frac{\partial a^{\ts}_{ii'}}{\partial x_{\hat k}} \frac{\partial a^{\ts}_{i \hat i}}{\partial x_{i'}} \frac{\partial f}{\partial x_{\hat i}})(a^{\ts}\nabla)_1f+  a^{\ts}_{2\hat k}\frac{\partial a^{\ts}_{ii'}}{\partial x_{\hat k}} \frac{\partial a^{\ts}_{i \hat i}}{\partial x_{i'}} \frac{\partial f}{\partial x_{\hat i}})(a^{\ts}\nabla)_2f\Big]=0;\\
\cI_4&=&-\sum_{i=1}^2\sum_{i',\hat i,\hat k=1}^{3} \Big[ a^{\ts}_{1\hat k} a^{\ts}_{ii'} (\frac{\partial }{\partial x_{\hat k}} \frac{\partial a^{\ts}_{i \hat i}}{\partial x_{i'}}) \frac{\partial f}{\partial x_{\hat i}}(a^{\ts}\nabla)_1f+ a^{\ts}_{2\hat k} a^{\ts}_{ii'} (\frac{\partial }{\partial x_{\hat k}} \frac{\partial a^{\ts}_{i \hat i}}{\partial x_{i'}}) \frac{\partial f}{\partial x_{\hat i}}(a^{\ts}\nabla)_2f\Big]=0.
\eeaa
For the drift term in tensor $\mathfrak{R}_{ab}$, taking $b=-aa^{\ts}\nabla V$, we get
\beaa
\mathfrak{R}_b^a
&=&\sum_{i,k=1}^2\sum_{\hat i,\hat k,k'=1}^3\left[a^{\ts}_{i\hat i}\frac{\pa a^{\ts}_{k\hat k} }{\pa x_{\hat i}}a^{\ts}_{kk'}\frac{\pa V}{\pa x_{k'}} \frac{\pa f}{\pa x_{\hat k}}(a^{\ts}\nabla)_i f \right]\\
&&+\sum_{i,k=1}^2\sum_{\hat i,\hat k,k'=1}^3\left[a^{\ts}_{i\hat i}\frac{\pa a^{\ts}_{kk'}}{\pa x_{\hat i}}a^{\ts}_{k\hat k}\frac{\pa V}{\pa x_{k'}} \frac{\pa f}{\pa x_{\hat k}}(a^{\ts}\nabla)_i f\right]\\
&&+\sum_{i,k=1}^2\sum_{\hat i,\hat k,k'=1}^3\left[ a^{\ts}_{i\hat i}a^{\ts}_{k\hat k}a^{\ts}_{kk'}\frac{\pa^2 V}{\pa x_{\hat i}\pa x_{k'}}\frac{\pa f}{\pa x_{\hat k}}(a^{\ts}\nabla)_i f\right]\\
&&-\sum_{i,k=1}^2\sum_{\hat i,\hat k,k'=1}^3\left[a^{\ts}_{k\hat k}a^{\ts}_{kk'}\frac{\pa a^{\ts}_{i\hat i}}{\pa x_{\hat k}}\frac{\pa V}{\pa x_{k'}} \frac{\pa f}{\pa x_{\hat i}}(a^{\ts}\nabla)_if\right] \\
&=& \cJ_1+\cJ_2+\cJ_3+\cJ_4.
\eeaa
Plugging into the matrix $a^{\ts}$, we get
\beaa
\cJ_1
&=&\sum_{\hat i,\hat k,k'=1}^3\left[ a^{\ts}_{1\hat i}\frac{\pa a^{\ts}_{1\hat k} }{\pa x_{\hat i}}a^{\ts}_{1k'}\frac{\pa V}{\pa x_{k'}} \frac{\pa f}{\pa x_{\hat k}}(a^{\ts}\nabla)_1 f+a^{\ts}_{2\hat i}\frac{\pa a^{\ts}_{1\hat k} }{\pa x_{\hat i}}a^{\ts}_{1k'}\frac{\pa V}{\pa x_{k'}} \frac{\pa f}{\pa x_{\hat k}}(a^{\ts}\nabla)_2 f\right]\\
&&+\sum_{\hat i,\hat k,k'=1}^3\left[ a^{\ts}_{1\hat i}\frac{\pa a^{\ts}_{2\hat k} }{\pa x_{\hat i}}a^{\ts}_{2k'}\frac{\pa V}{\pa x_{k'}} \frac{\pa f}{\pa x_{\hat k}}(a^{\ts}\nabla)_1 f+a^{\ts}_{2\hat i}\frac{\pa a^{\ts}_{2\hat k} }{\pa x_{\hat i}}a^{\ts}_{2k'}\frac{\pa V}{\pa x_{k'}} \frac{\pa f}{\pa x_{\hat k}}(a^{\ts}\nabla)_2 f\right]\\
&=&\beta e^{\beta\theta} (a^{\ts}\nabla)_2V\frac{\pa f}{\pa x}(a^{\ts}\nabla )_1 f;
\eeaa
\beaa
\cJ_2
&=&\sum_{\hat i,\hat k,k'=1}^3\left[a^{\ts}_{1\hat i}\frac{\pa a^{\ts}_{1k'}}{\pa x_{\hat i}}a^{\ts}_{1\hat k}\frac{\pa V}{\pa x_{k'}} \frac{\pa f}{\pa x_{\hat k}}(a^{\ts}\nabla)_1 f+a^{\ts}_{2\hat i}\frac{\pa a^{\ts}_{1k'}}{\pa x_{\hat i}}a^{\ts}_{1\hat k}\frac{\pa V}{\pa x_{k'}} \frac{\pa f}{\pa x_{\hat k}}(a^{\ts}\nabla)_2 f\right]\\
&&+\sum_{\hat i,\hat k,k'=1}^3\left[a^{\ts}_{1\hat i}\frac{\pa a^{\ts}_{2k'}}{\pa x_{\hat i}}a^{\ts}_{2\hat k}\frac{\pa V}{\pa x_{k'}} \frac{\pa f}{\pa x_{\hat k}}(a^{\ts}\nabla)_1 f+a^{\ts}_{2\hat i}\frac{\pa a^{\ts}_{2k'}}{\pa x_{\hat i}}a^{\ts}_{2\hat k}\frac{\pa V}{\pa x_{k'}} \frac{\pa f}{\pa x_{\hat k}}(a^{\ts}\nabla)_2 f\right]\\
&=&\beta e^{\beta \theta}\frac{\pa V}{\pa x}(a^{\ts}\nabla)_2f(a^{\ts}\nabla)_1f;
\eeaa
\beaa
\cJ_3
&=&\sum_{\hat i,k'=1}^3\left[ a^{\ts}_{1\hat i}a^{\ts}_{1k'}\frac{\pa^2 V}{\pa x_{\hat i}\pa x_{k'}} |(a^{\ts}\nabla)_1 f|^2+a^{\ts}_{2\hat i}a^{\ts}_{1k'}\frac{\pa^2 V}{\pa x_{\hat i}\pa x_{k'}}(a^{\ts}\nabla)_1 f(a^{\ts}\nabla)_2 f\right]\\
&&+\sum_{\hat i,k'=1}^3\left[ a^{\ts}_{1\hat i}a^{\ts}_{2k'}\frac{\pa^2 V}{\pa x_{\hat i}\pa x_{k'}}(a^{\ts}\nabla)_2 f(a^{\ts}\nabla)_1 f+a^{\ts}_{2\hat i}a^{\ts}_{2k'}\frac{\pa^2 V}{\pa x_{\hat i}\pa x_{k'}}|(a^{\ts}\nabla)_2 f|^2\right]\\
&=&\frac{\pa^2 V}{\pa \theta\pa \theta} |(a^{\ts}\nabla)_1 f|^2+2(e^{\beta\theta}\frac{\pa^2 V}{\pa \theta\pa x}+\frac{\pa^2 V}{\pa \theta \pa y} )(a^{\ts}\nabla)_1 f(a^{\ts}\nabla)_2 f\\
&&+\sum_{\hat i,k'=1}^3a^{\ts}_{2\hat i}a^{\ts}_{2k'}\frac{pa^2 V}{\pa x_{\hat i}\pa x_{k'}})|(a^{\ts}\nabla)_2 f|^2;
\eeaa

\beaa
\cJ_4
&=&-\sum_{\hat i,\hat k,k'=1}^3\left[a^{\ts}_{1\hat k}a^{\ts}_{1k'}\frac{\pa a^{\ts}_{1\hat i}}{\pa x_{\hat k}}\frac{\pa V}{\pa x_{k'}} \frac{\pa f}{\pa x_{\hat i}}(a^{\ts}\nabla)_1f+a^{\ts}_{1\hat k}a^{\ts}_{1k'}\frac{\pa a^{\ts}_{2\hat i}}{\pa x_{\hat k}}\frac{\pa V}{\pa x_{k'}} \frac{\pa f}{\pa x_{\hat i}}(a^{\ts}\nabla)_2f\right]\\
&&-\sum_{\hat i,\hat k,k'=1}^3\left[a^{\ts}_{2\hat k}a^{\ts}_{2k'}\frac{\pa a^{\ts}_{1\hat i}}{\pa x_{\hat k}}\frac{\pa V}{\pa x_{k'}} \frac{\pa f}{\pa x_{\hat i}}(a^{\ts}\nabla)_1f+a^{\ts}_{2\hat k}a^{\ts}_{2k'}\frac{\pa a^{\ts}_{2\hat i}}{\pa x_{\hat k}}\frac{\pa V}{\pa x_{k'}} \frac{\pa f}{\pa x_{\hat i}}(a^{\ts}\nabla)_2f\right]\\
&=&-\beta e^{\beta\theta}(a^{\ts}\nabla)_1V \frac{\pa f}{\pa x}(a^{\ts}\nabla)_2f.
\eeaa
 Combining the above computations, we get the tensor $\mathfrak{R}_{ab}$.  Now we turn to the second tensor $\mathfrak{R}_{zb}$, which has the following form,
 \beaa
\mathfrak{R}_{zb}(\nabla f,\nabla f)&=&\sum_{i=1}^2\sum_{i',\hat i,\hat k=1}^{3} \la a^{\ts}_{ii'} (\frac{\partial a^{\ts}_{i \hat i}}{\partial x_{i'}} \frac{\partial z^{\ts}_{1\hat k}}{\partial x_{\hat i}}\frac{\partial f}{\partial x_{\hat k}}) ,(z^{\ts}\nabla)_1f\ra_{\hR}\nonumber \\
	&&+\sum_{i=1}^2\sum_{i',\hat i,\hat k=1}^{3} \la a^{\ts}_{ii'}a^{\ts}_{i \hat i} (\frac{\partial }{\partial x_{i'}} \frac{\partial z^{\ts}_{1\hat k}}{\partial x_{\hat i}})(\frac{\partial f}{\partial x_{\hat k}}) ,(z^{\ts}\nabla)_1f\ra_{\hR} \nonumber \\
	&&-\sum_{i=1}^2\sum_{i',\hat i,\hat k=1}^{3} \la z^{\ts}_{1\hat k}\frac{\partial a^{\ts}_{ii'}}{\partial x_{\hat k}} \frac{\partial a^{\ts}_{i \hat i}}{\partial x_{i'}} \frac{\partial f}{\partial x_{\hat i}})  ,(z^{\ts}\nabla)_1f\ra_{\hR} \nonumber \\
&&-\sum_{i=1}^2\sum_{i',\hat i,\hat k=1}^{3} \la z^{\ts}_{1\hat k} a^{\ts}_{ii'} (\frac{\partial }{\partial x_{\hat k}} \frac{\partial a^{\ts}_{i \hat i}}{\partial x_{i'}}) \frac{\partial f}{\partial x_{\hat i}}  ,(z^{\ts}\nabla)_1f\ra_{\hR}\nonumber \\
&&-\sum_{i=1}^2 \sum_{\hat i,\hat k=1}^{3}\la (z^{\ts}_{1\hat i}\frac{\pa b_{\hat k}}{\pa x_{\hat i}}\frac{\pa f}{\pa x_{\hat k}}-b_{\hat k}\frac{\pa z^{\ts}_{1\hat i}}{\pa x_{\hat k}} \frac{\pa f}{\pa x_{\hat i}} ),(z^{\ts}\nabla f)_1\ra_{\hR},\nonumber\\
&=&\cI_1^z+\cI_2^z+\cI_3^z+\cI_4^z+\mathfrak R^{z}_b(\nabla f,\nabla f).
\eeaa 
where we denote further that 
\beaa
\mathfrak{R}_{b}^z(\nabla f,\nabla f)=-\sum_{\hat i,\hat k=1}^{3} (z^{\ts}_{1\hat i}\frac{\pa b_{\hat k}}{\pa x_{\hat i}}\frac{\pa f}{\pa x_{\hat k}}-b_{\hat k}\frac{\pa z^{\ts}_{i\hat i}}{\pa x_{\hat k}} \frac{\pa f}{\pa x_{\hat i}} )(z^{\ts}\nabla f)_1.
\eeaa
By taking $b=-aa^{\ts}\nabla V$, we further obtain that
\beaa
\mathfrak{R}_{b}^z(\nabla f,\nabla f)
&=& -\sum_{\hat i,\hat k=1}^{3}\left[ z^{\ts}_{1\hat i}\frac{\pa b_{\hat k}}{\pa x_{\hat i}}\frac{\pa f}{\pa x_{\hat k}}(z^{\ts}\nabla f)_1-b_{\hat k}\frac{\pa z^{\ts}_{i\hat i}}{\pa x_{\hat k}} \frac{\pa f}{\pa x_{\hat i}} (z^{\ts}\nabla f)_1\right]\\
&=&\sum_{k=1}^2\sum_{\hat i,\hat k,k'=1}^3\left[z^{\ts}_{1\hat i}\frac{\pa a^{\ts}_{k\hat k} }{\pa x_{\hat i}}a^{\ts}_{kk'}\frac{\pa V}{\pa x_{k'}} \frac{\pa f}{\pa x_{\hat k}}(z^{\ts}\nabla)_1 f \right]\\
&&+\sum_{k=1}^2\sum_{\hat i,\hat k,k'=1}^3\left[z^{\ts}_{1\hat i}\frac{\pa a^{\ts}_{kk'}}{\pa x_{\hat i}}a^{\ts}_{k\hat k}\frac{\pa V}{\pa x_{k'}} \frac{\pa f}{\pa x_{\hat k}}(z^{\ts}\nabla)_1 f\right]\\ 
&&+\sum_{k=1}^2\sum_{\hat i,\hat k,k'=1}^3\left[ z^{\ts}_{1\hat i}a^{\ts}_{k\hat k}a^{\ts}_{kk'}\frac{\pa^2 V}{\pa x_{\hat i}\pa x_{k'}}\frac{\pa f}{\pa x_{\hat k}}(z^{\ts}\nabla)_1 f\right]
\eeaa 
\beaa 
&&-\sum_{k=1}^2\sum_{\hat i,\hat k,k'=1}^3\left[a^{\ts}_{k\hat k}a^{\ts}_{kk'}\frac{\pa z^{\ts}_{1\hat i}}{\pa x_{\hat k}}\frac{\pa V}{\pa x_{k'}} \frac{\pa f}{\pa x_{\hat i}}(z^{\ts}\nabla)_1f\right] \\
&=& \cJ_1^z+\cJ_2^z+\cJ_3^z+\cJ_4^z.
\eeaa
By direct computations, it is not hard to observe that
\beaa 
\cI_1^z&=& \sum_{i=1}^2\sum_{i',\hat i,\hat k=1}^{3} \la a^{\ts}_{ii'} (\frac{\partial a^{\ts}_{i \hat i}}{\partial x_{i'}} \frac{\partial z^{\ts}_{1\hat k}}{\partial x_{\hat i}}\frac{\partial f}{\partial x_{\hat k}}) ,(z^{\ts}\nabla)_1f\ra_{\hR}=0 \\
	\cI_2^z&=&\sum_{i=1}^2\sum_{i',\hat i,\hat k=1}^{3} \la a^{\ts}_{ii'}a^{\ts}_{i \hat i} (\frac{\partial }{\partial x_{i'}} \frac{\partial z^{\ts}_{1\hat k}}{\partial x_{\hat i}})(\frac{\partial f}{\partial x_{\hat k}}),(z^{\ts}\nabla)_1f\ra_{\hR}\\
	&&= \sum_{i=1}^2\sum_{i',\hat i=1}^{3}  a^{\ts}_{ii'}a^{\ts}_{i \hat i} \frac{\pa^2 z^{\ts}_{1\hat k}}{\partial x_{i'}\partial x_{\hat i}}\pa_yf (z^{\ts}\nabla)_1f \\
	\cI_3^z&=&-\sum_{i=1}^2\sum_{i',\hat i,\hat k=1}^{3} \la z^{\ts}_{1\hat k}\frac{\partial a^{\ts}_{ii'}}{\partial x_{\hat k}} \frac{\partial a^{\ts}_{i \hat i}}{\partial x_{i'}} \frac{\partial f}{\partial x_{\hat i}})  ,(z^{\ts}\nabla)_1f\ra_{\hR}=0 \nonumber \\
\cI_4^z&=&-\sum_{i=1}^2\sum_{i',\hat i,\hat k=1}^{3} \la z^{\ts}_{1\hat k} a^{\ts}_{ii'} (\frac{\partial }{\partial x_{\hat k}} \frac{\partial a^{\ts}_{i \hat i}}{\partial x_{i'}}) \frac{\partial f}{\partial x_{\hat i}}  ,(z^{\ts}\nabla)_1f\ra_{\hR}=0,\nonumber 
\eeaa 
and 
\beaa 
\cJ_1^z&=&\sum_{k=1}^2\sum_{\hat i,\hat k,k'=1}^3\left[z^{\ts}_{1\hat i}\frac{\pa a^{\ts}_{k\hat k} }{\pa x_{\hat i}}a^{\ts}_{kk'}\frac{\pa V}{\pa x_{k'}} \frac{\pa f}{\pa x_{\hat k}}(z^{\ts}\nabla)_1 f \right]=0;\\
\cJ_2^z&=&\sum_{k=1}^2\sum_{\hat i,\hat k,k'=1}^3\left[z^{\ts}_{1\hat i}\frac{\pa a^{\ts}_{kk'}}{\pa x_{\hat i}}a^{\ts}_{k\hat k}\frac{\pa V}{\pa x_{k'}} \frac{\pa f}{\pa x_{\hat k}}(z^{\ts}\nabla)_1 f\right]=0;\\
\cJ_4^z&=&-\sum_{k=1}^2\sum_{\hat i,\hat k,k'=1}^3\left[a^{\ts}_{k\hat k}a^{\ts}_{kk'}\frac{\pa z^{\ts}_{1\hat i}}{\pa x_{\hat k}}\frac{\pa V}{\pa x_{k'}} \frac{\pa f}{\pa x_{\hat i}}(z^{\ts}\nabla)_1f\right]\\
&=& -\sum_{k=1}^2(a^{\ts}\nabla)_kz^{\ts}_{13}(a^{\ts}\nabla)_kV\pa_yf(z^{\ts}\nabla)_1f.
\eeaa 
\beaa 
\cJ_3^z&=& \sum_{k=1}^2\sum_{\hat i,\hat k,k'=1}^3\left[ z^{\ts}_{1 \hat i}a^{\ts}_{k\hat k}a^{\ts}_{kk'}\frac{\pa^2 V}{\pa x_{\hat i}\pa x_{k'}}\frac{\pa f}{\pa x_{\hat k}}(z^{\ts}\nabla)_1 f\right]\\
&=&\sum_{\hat i,\hat k,k'=1}^3\left[ z^{\ts}_{1\hat i}a^{\ts}_{1\hat k}a^{\ts}_{1k'}\frac{\pa^2 V}{\pa x_{\hat i}\pa x_{k'}}\frac{\pa f}{\pa x_{\hat k}}(z^{\ts}\nabla)_1 f+ z^{\ts}_{1\hat i}a^{\ts}_{2\hat k}a^{\ts}_{2k'}\frac{\pa^2 V}{\pa x_{\hat i}\pa x_{k'}}\frac{\pa f}{\pa x_{\hat k}}(z^{\ts}\nabla)_1 f\right]\\
&=&\sum_{\hat i,k'=1}^3\left[ z^{\ts}_{1\hat i}a^{\ts}_{1k'}\frac{\pa^2 V}{\pa x_{\hat i}\pa x_{k'}} (a^{\ts}\nabla)_1 f(z^{\ts}\nabla)_1 f+ z^{\ts}_{1\hat i}a^{\ts}_{2k'}\frac{\pa^2 V}{\pa x_{\hat i}\pa x_{k'}}(a^{\ts}\nabla)_2 f(z^{\ts}\nabla)_1 f\right]\\
&=&-g\frac{\pa^2 V}{\pa \theta\pa y}(a^{\ts}\nabla)_1 f(z^{\ts}\nabla)_1 f-g(e^{\beta\theta}\frac{\pa^2 V}{\pa x\pa y}+\frac{\pa^2 V}{\pa y\pa y} )(a^{\ts}\nabla)_2 f(z^{\ts}\nabla)_1 f;
\eeaa 
Now we are left to compute the term $\mathfrak{R}_{\pi}$. Recall that, 
\beaa
\mathfrak{R}_{\pi}(\nabla f,\nabla f)&=&2\sum_{k=1}^1 \sum_{i=1}^2\sum_{k',\hat k,\hat i,i'=1}^{3}\left[\frac{\pa }{\pa x_{k'}} z^{\ts}_{kk'} z^{\ts}_{k\hat k}\frac{\pa}{\pa x_{\hat k}}a^{\ts}_{i\hat i}\frac{\pa f}{\pa x_{\hat i}}a^{\ts}_{ii'}\frac{\pa f}{\pa x_{i'}}\right]\\
 &&+2\sum_{k=1}^1 \sum_{i=1}^2\sum_{k',\hat k,\hat i,i'=1}^{3}\left[z^{\ts}_{kk'}\frac{\pa }{\pa x_{k'}} z^{\ts}_{k\hat k} \frac{\pa}{\pa x_{\hat k}}a^{\ts}_{i\hat i}\frac{\pa f}{\pa x_{\hat i}}a^{\ts}_{ii'}\frac{\pa f}{\pa x_{i'}} \right.\nonumber\\
 &&\quad\quad\quad\quad\quad\quad\quad\quad+z^{\ts}_{kk'} z^{\ts}_{k\hat k} \frac{\pa^2}{\pa x_{k'}\pa x_{\hat k}}a^{\ts}_{i\hat i}\frac{\pa f}{\pa x_{\hat i}}a^{\ts}_{ii'}\frac{\pa f}{\pa x_{i'}}\nonumber \\
&&\left.\quad\quad\quad\quad\quad\quad\quad\quad+z^{\ts}_{kk'} z^{\ts}_{k\hat k} \frac{\pa}{\pa x_{\hat k}}a^{\ts}_{i\hat i}\frac{\pa f}{\pa x_{\hat i}}\frac{\pa }{\pa x_{k'}}a^{\ts}_{ii'}\frac{\pa f}{\pa x_{i'}} \right].\nonumber \\
&&+2\sum_{k=1}^1 \sum_{i=1}^2\sum_{\hat k,\hat i,i'=1}^{3}(z^{\ts}\nabla\log\pi)_k \left[ z^{\ts}_{k\hat k}\frac{\pa}{\pa x_{\hat k}}a^{\ts}_{i\hat i}\frac{\pa f}{\pa x_{\hat i}}a^{\ts}_{ii'}\frac{\pa f}{\pa x_{i'}} \right] \nonumber\\
&&-2\sum_{j=1}^1\sum_{l=1}^2\sum_{ l',\hat l,\hat j,j'=1}^{3}\left[\frac{\pa }{\pa x_{l'}} a^{\ts}_{ll'} a^{\ts}_{l\hat l} \frac{\pa}{\pa x_{\hat l}}z^{\ts}_{j\hat j}\frac{\pa f}{\pa x_{\hat j}}z^{\ts}_{jj'}\frac{\pa f}{\pa x_{j'}} \right] \nonumber\\
 &&- 2\sum_{j=1}^1\sum_{l=1}^2\sum_{ l',\hat l,\hat j,j'=1}^{3}\left[ a^{\ts}_{ll'}\frac{\pa }{\pa x_{l'}}a^{\ts}_{l\hat l} \frac{\pa}{\pa x_{\hat l}}z^{\ts}_{j\hat j}\frac{\pa f}{\pa x_{\hat j}}z^{\ts}_{jj'}\frac{\pa f}{\pa x_{j'}}  \right.\nonumber\\
 &&\quad\quad\quad\quad\quad\quad\quad\quad+a^{\ts}_{ll'}a^{\ts}_{l\hat l} \frac{\pa^2}{\pa x_{l'}\pa x_{\hat l}}z^{\ts}_{j\hat j}\frac{\pa f}{\pa x_{\hat j}}z^{\ts}_{jj'}\frac{\pa f}{\pa x_{j'}}\nonumber \\
&&\left.\quad\quad\quad\quad\quad\quad\quad\quad+a^{\ts}_{ll'}a^{\ts}_{l\hat l} \frac{\pa}{\pa x_{\hat l}}z^{\ts}_{j\hat j}\frac{\pa f}{\pa x_{\hat j}}\frac{\partial}{\pa x_{l'}}z^{\ts}_{jj'}\frac{\pa f}{\pa x_{j'}}  \right] \nonumber \\
&&-2\sum_{j=1}^1\sum_{l=1}^2\sum_{\hat l,\hat j,j'=1}^{3}(a^{\ts}\nabla\log\pi)_l \left[ a^{\ts}_{l\hat l}\frac{\pa}{\pa x_{\hat l}}z^{\ts}_{j\hat j}\frac{\pa f}{\pa x_{\hat j}}z^{\ts}_{jj'}\frac{\pa f}{\pa x_{j'}} \right]\nonumber\\
&=&\sum_{i=1}^{10}\cK_i.
\eeaa
By direct computation, we get 
\beaa
\cK_1&=&0, \quad \cK_2=0,\quad \cK_3=0,\quad \cK_4=0,\quad \cK_5=0,\quad \cK_6=0,\quad \cK_7=0;\\
\cK_8&=& -2\sum_{l=1}^2\sum_{l',\hat l=1}^3a^{\ts}_{ll'}a^{\ts}_{l\hat l}\frac{\pa^2 z^{\ts}_{13}}{\pa x_{l'}\pa x_{\hat l}}\pa_yf (z^{\ts}\nabla)_1f;
\eeaa 
\beaa 
\cK_9&=&-2\sum_{l=1}^2\sum_{l',\hat l=1}^3a^{\ts}_{ll'}a^{\ts}_{l\hat l}\frac{\pa z^{\ts}_{13}}{\pa x_{\hat l}}\frac{\pa z^{\ts}_{13}}{\pa x_{l'}} |\pa_yf|^2=-2\Gamma_1(\log g,\log g)|(z^{\ts}\nabla)_1f|^2;\\
\cK_{10}&=&-2 \sum_{l=1}^2\sum_{\hat l=1}^3(a^{\ts}\nabla)_l\log \pi a^{\ts}_{l\hat l}\frac{\pa z^{\ts}_{13}}{\pa x_{\hat l}}\pa_yf (z^{\ts}\nabla)_1f=-2\Gamma_1(\log \pi,\log g)|(z^{\ts}\nabla)_1f|^2.
\eeaa 
\end{proof}{}

\subsection{Martinet flat sub-Riemannian structure}
In this part, we apply our result to Martinet flat sub-Riemannian structure, which satisfies the bracket generating condition and has non-equiregular sub-Riemannian structure (see \cite{barilari2013formula}). The sub-Riemannian structure is defined on $\mathbb R^3$ through the kernel of one-form $\eta :=dz-\frac{1}{2}y^2dx.$ A global orthonormal basis for the horizontal distribution $\mathcal{H}$ adapt the following differential operator representation, in local coordinates $(x,y,z)$,  
\[
X=\frac{\pa}{\pa x}+\frac{y^2}{2}\frac{\pa}{\pa z},\quad Y=\frac{\pa}{\pa y}.
\]
The commutative relation gives
\beaa
[X,Y]=-yZ,\quad [Y,[X,Y]]=Z,\quad \text{where}\quad Z=\frac{\pa}{\pa z}.
\eeaa
To apply in our framework, we take
\bea \label{matrix a and z martinet}
a=\begin{pmatrix}
	1&0\\
	0&1\\
	\frac{y^2}{2}&0
\end{pmatrix},\quad  a^{\ts}=\begin{pmatrix}
	1&0&\frac{y^2}{2}\\
	0&1&0
\end{pmatrix},\quad z^{\ts}=(0,0,1),\quad aa^{\ts}=\begin{pmatrix}
	1&0&\frac{y^2}{2}\\
	0&1&0\\
	\frac{y^2}{2}&0&\frac{y^4}{4}
\end{pmatrix}.
\eea
Thus the sub-Riemanian structure has the form $(\mathbb M,\mathcal H, (aa^{\ts})^{\dd}_{|\mathcal{H}})$.

\begin{prop}
In this setting, 
\begin{equation*}
\Vol=e^{-\frac{y^2}{2}-V},   
\end{equation*}
then 
\[
-aa^{\ts}\nabla \log \Vol=a\otimes\nabla a+aa^{\ts}\nabla V.
\]
\end{prop}
\begin{proof}
The poof follows from the observation that,
\[
a\otimes\nabla a= \begin{pmatrix}
0&
y&
0
\end{pmatrix}^{\ts}, \quad aa^{\ts}\nabla \log e^{-\frac{y^2}{2}}= \begin{pmatrix}
0&
y&
0
\end{pmatrix}^{\ts}.
\]\qed 
\end{proof}
Similar to the previous displacement group case, we have the following identity.
\begin{prop}\label{prop Martinet}  For any smooth function $f\in C^{\infty}(\mathbb{M})$, one has
\beaa
\Gamma_2(f,f)+\Gamma_2^{z,\pi}(f,f)&=&\|\mathfrak{Hess}_{a,z}f\|^2+\mathfrak{R}(\nabla f,\nabla f),
\eeaa
where 
\beaa
\Lambda_1^{\ts}&=&(0,y\pa_zf/2,0,y\pa_zf/2,0,0,0,0,0); \\
\Lambda_2^{\ts}&=&(0,0,0,0,0,0,-y\pa_yf,\frac{y^3}{2}\pa_zf+y\pa_xf,0);\\
&&\mathfrak{R}_{ab}(\nabla f,\nabla f)-\Lambda_1^{\ts}Q^{\ts}Q\Lambda_1-\Lambda_2^{\ts}P^{\ts}P\Lambda_2+D^{\ts}D+E^{\ts}E\\
&=&\frac{y^2}{2}\Gamma_1^z(f,f)-y^2\Gamma_1(f,f)\\
&&+\frac{\pa f}{\pa z}(a^{\ts}\nabla )_1f+y(a^{\ts}\nabla)_1V \frac{\pa f}{\pa z}(a^{\ts}\nabla)_2f +y \frac{\pa V}{\pa z}(a^{\ts}\nabla)_1f(a^{\ts}\nabla)_2f \\
&&+\sum_{\hat i,k'=1}^3a^{\ts}_{1\hat i}a^{\ts}_{1k'}\frac{\pa^2 V}{\pa x_{\hat i}\pa x_{k'}}|(a^{\ts}\nabla)_1 f|^2+2(\frac{\pa^2 V}{\pa x\pa y}+\frac{y^2}{2}\frac{\pa^2 V}{\pa y \pa z} )(a^{\ts}\nabla)_1 f(a^{\ts}\nabla)_2 f\\
&&+\frac{\pa^2 V}{\pa y\pa y}|(a^{\ts}\nabla)_2 f|^2-y\frac{\pa V}{\pa y}\frac{\pa f}{\pa z}(a^{\ts}\nabla)_1f;\\ 
\mathfrak{R}_{zb}(\nabla f,\nabla f)&=& (\frac{\pa^2 V}{\pa x\pa z}+\frac{y^2}{2}\frac{\pa^2 V}{\pa z\pa z})(a^{\ts}\nabla)_1 f(z^{\ts}\nabla)_1 f+\frac{\pa^2 V}{\pa y\pa z}(a^{\ts}\nabla)_2 f(z^{\ts}\nabla)_1 f;\\
\mathfrak{R}_{\pi}(\nabla f,\nabla f)&=&0.
\eeaa
In particular, we have 
\beaa
\sum_{\hat i,k'=1}^3a^{\ts}_{1\hat i}a^{\ts}_{1k'}\frac{\pa^2 V}{\pa x_{\hat i}\pa x_{k'}}|(a^{\ts}\nabla)_1 f|^2=\Big(\frac{\pa^2 V}{\pa x\pa x}+y^2\frac{\pa^2 V}{\pa x\pa z}+\frac{y^4}{4}\frac{\pa^2 V}{\pa z\pa z} \Big)|(a^{\ts}\nabla)_1 f|^2.
\eeaa 
\end{prop}{}

{The proof of Proposition \ref{prop Martinet} follows from the proof of Theorem \ref{thm1} (i.e. Theorem \ref{thm: generalized gamma 2 z with drift}), Lemma \ref{martinet}, Lemma \ref{martinet hess}, and Lemma \ref{tensor martinet} below. The following convergence results is a direct consequence of Theorem \ref{prop1.4}.
}

\begin{prop} If there exists $\kappa>0$ as shown in Theorem \ref{prop1.4}, the exponential dissipation result in $L^1$ distance holds:
\begin{equation*}
    \int |\rho(t,x)-\pi(x)| dx=O(e^{-\kappa t}). 
\end{equation*}
\end{prop}
Similarly, we summarize the sub-Riemannian Ricci tensor in terms of $\mathfrak R$ as follows. 
\begin{cor} The matrix $\mathfrak{R}$ associated  with Martinet sub-Riemannian structure has the following form
\beaa
\mathfrak{R}_{11}&=&\Big(\frac{\pa^2 V}{\pa x\pa x}+y^2\frac{\pa^2 V}{\pa x\pa z}+\frac{y^4}{4}\frac{\pa^2 V}{\pa z\pa z} \Big)-y^2;\\
\mathfrak{R}_{22}&=&\frac{\pa^2 V}{\pa y\pa y}-y^2;\quad \mathfrak{R}_{33}=\frac{y^2}{2};  \\
\mathfrak{R}_{12}&=&\mathfrak{R}_{21}=\frac{y}{2}\frac{\pa V}{\pa z}+(\frac{\pa^2 V}{\pa x\pa y}+\frac{y^2}{2}\frac{\pa^2 V}{\pa y\pa z});\\
\mathfrak{R}_{13}&=&\mathfrak{R}_{31}=\frac{1}{2}-\frac{y}{2}\frac{\pa V}{\pa y}+\frac{1}{2}(\frac{\pa^2 V}{\pa x\pa z}+\frac{y^2}{2}\frac{\pa^2 V}{\pa z\pa z}) ;\quad \mathfrak{R}_{23}=\mathfrak{R}_{32}=\frac{1}{2}y(a^{\ts}\nabla)_1V+\frac{1}{2}\frac{\pa^2 V}{\pa y\pa z}.
\eeaa
\end{cor}{}
\begin{proof}
The proof follows from the similar equivalent matrix formulation as shown in the proof of Corollary \ref{matrix R heisenberg} and the explicit bilinear forms in Proposition \ref{prop Martinet}. \qed 

\end{proof}
%

Next, we prove the following three key lemmas.
\begin{lem}\label{martinet} For Martinet sub-Riemannian structure $(\mathbb M,\mathcal H, (aa^{\ts})^{\dd}_{|\mathcal{H}})$, we have
\beaa
Q&=&\left(
\begin{array}{ccccccccc}
 1 & 0 & \frac{y^2}{2} & 0 & 0 & 0 & \frac{y^2}{2} & 0 & \frac{y^4}{4} \\
 0 & 1 & 0 & 0 & 0 & 0 & 0 & \frac{y^2}{2} & 0 \\
 0 & 0 & 0 & 1 & 0 & \frac{y^2}{2} & 0 & 0 & 0 \\
 0 & 0 & 0 & 0 & 1 & 0 & 0 & 0 & 0 \\
\end{array}
\right);\\ 
P&=&\begin{pmatrix}
	0&0&0&0&0&0&1&0&y^2/2\\
	0&0&0&0&0&0&0&1&0
\end{pmatrix};\\
C^{\ts}&=&(0,0,0,0,0,\frac{y^3}{2}\pa_zf+y\pa_xf,-y\pa_yf,0,-\frac{y^3}{2}\pa_yf);\\
D^{\ts}&=&(0,0,y\pa_zf,0),\quad E^{\ts}=(0,0);\\
F^{\ts}&=&G^{\ts}=(0,0,0,0,0,0,0,0,0).
\eeaa
\end{lem}{}
\begin{proof} Plugging matrices $a$ and $z$ from \eqref{matrix a and z martinet} into Notation \ref{notation}, we complete the proof. \qed 
\end{proof}

\begin{lem}\label{martinet hess} For Martinet sub-Riemannian structure, $F$ and $G$ are zero vectors, we have
\beaa
&&[QX+D]^{\ts}[QX+D]+[PX+E]^{\ts}[PX+E]+2C^{\ts}X\\
&=&\|\mathfrak{Hess}_{a,z}f\|^2-\Lambda_1^{\ts}Q^{\ts}Q\Lambda_1-\Lambda_2^{\ts}P^{\ts}P\Lambda_2+D^{\ts}D+E^{\ts}E.
\eeaa
In particular, we have
\beaa
\|\mathfrak{Hess}_{a,z}f\|^2&=&[X+\Lambda_1]^{\ts}Q^{\ts}Q[X+\Lambda_1]+[X+\Lambda_2]^{\ts}P^{\ts}P[X+\Lambda_2];\\
\Lambda_1^{\ts}&=&(0,y\pa_zf/2,0,y\pa_zf/2,0,0,0,0,0); \\
\Lambda_2^{\ts}&=&(0,0,0,0,0,0,-y\pa_yf,\frac{y^3}{2}\pa_zf+y\pa_xf,0);\\
-\Lambda_1^{\ts}Q^{\ts}Q\Lambda_1-\Lambda_2^{\ts}P^{\ts}P\Lambda_2+D^{\ts}D+E^{\ts}E
&=&\frac{y^2}{2}\Gamma_1^z(f,f)-y^2\Gamma_1(f,f).
\eeaa
\end{lem}{}
\begin{lem}\label{tensor martinet}
By routine computations, we obtain
\beaa
\mathfrak{R}_{ab}(\nabla f,\nabla f)&=&\frac{\pa f}{\pa z}(a^{\ts}\nabla )_1f+y(a^{\ts}\nabla)_1V \frac{\pa f}{\pa z}(a^{\ts}\nabla)_2f +y \frac{\pa V}{\pa z}(a^{\ts}\nabla)_1f(a^{\ts}\nabla)_2f \\
&&+\sum_{\hat i,k'=1}^3a^{\ts}_{1\hat i}a^{\ts}_{1k'}\frac{\pa^2 V}{\pa x_{\hat i}\pa x_{k'}}|(a^{\ts}\nabla)_1 f|^2+2(\frac{\pa^2 V}{\pa x\pa y}+\frac{y^2}{2}\frac{\pa^2 V}{\pa y \pa z} )(a^{\ts}\nabla)_1 f(a^{\ts}\nabla)_2 f\\
&&+\frac{\pa^2 V}{\pa y\pa y}|(a^{\ts}\nabla)_2 f|^2-y\frac{\pa V}{\pa y}\frac{\pa f}{\pa z}(a^{\ts}\nabla)_1f;\\
\mathfrak{R}_{zb}(\nabla f,\nabla f)&=& (\frac{\pa^2 V}{\pa x\pa z}+\frac{y^2}{2}\frac{\pa^2 V}{\pa z\pa z})(a^{\ts}\nabla)_1 f (z^{\ts}\nabla)_1 f+\frac{\pa^2 V}{\pa y\pa z}(a^{\ts}\nabla)_2 f(z^{\ts}\nabla)_1 f;\\
\mathfrak{R}_{\pi}(\nabla f,\nabla f)&=&0.
\eeaa
\end{lem}{}
\begin{proof}[Proof of Lemma \ref{martinet hess}] Since $F$ and $G$ are zero vectors, we have
\beaa
2C^{\ts}X=2\Big[\frac{\pa^2 f}{\pa y\pa z}(\frac{y^3}{2}\pa_zf+y\pa_xf)-\frac{\pa^2 f}{\pa x \pa z}(y\pa_yf)-\frac{\pa^2 f}{\pa z\pa z}(\frac{y^3}{2}\pa_yf)\Big].
\eeaa
By routine computation, we observe that 
\beaa
&&[QX+D]^{\ts}[QX+D]+[PX+E]^{\ts}[PX+E]+2C^{\ts}X\\
&=&\left[\frac{\pa^2 f}{\pa x\pa x}+\frac{y^2}{2}\frac{\pa^2 f}{\pa x\pa z} +\frac{y^2}{2}\frac{\pa^2 f}{\pa z \pa x}+\frac{y^4}{4}\frac{\pa^2 f}{\pa z \pa z}\right]^2 +\left[\frac{\pa^2 f}{\pa y \pa x}+\frac{y^2}{2}\frac{\pa^2 f}{\pa z\pa y}+y\pa_zf\right]^2\\
&&+\left[\frac{\pa^2 f}{\pa y \pa x}+\frac{y^2}{2}\frac{\pa^2 f}{\pa z\pa y}\right]^2+\left[\frac{\pa^2 f}{\pa y\pa y} \right]^2\\
&&+\left[\frac{\pa^2 f}{\pa z\pa x}+\frac{y^2}{2}\frac{\pa^2 f}{\pa z\pa z} \right]^2+\left[ \frac{\pa^2 f}{\pa z\pa y}\right]^2\\
&&+2\frac{\pa^2 f}{\pa y\pa z}(\frac{y^3}{2}\pa_zf+y\pa_xf)-2\frac{\pa^2 f}{\pa x \pa z}(y\pa_yf)-2\frac{\pa^2 f}{\pa z\pa z}(\frac{y^3}{2}\pa_yf)\\ 
&=&\left[\frac{\pa^2 f}{\pa x\pa x}+\frac{y^2}{2}\frac{\pa^2 f}{\pa x\pa z} +\frac{y^2}{2}\frac{\pa^2 f}{\pa z \pa x}+\frac{y^4}{4}\frac{\pa^2 f}{\pa z \pa z}\right]^2 +\left[\frac{\pa^2 f}{\pa y \pa x}+\frac{y^2}{2}\frac{\pa^2 f}{\pa z\pa y}+y\pa_zf\right]^2\\
&&+\left[\frac{\pa^2 f}{\pa y \pa x}+\frac{y^2}{2}\frac{\pa^2 f}{\pa z\pa y}\right]^2+\left[\frac{\pa^2 f}{\pa y\pa y} \right]^2\\
&&+\left[\frac{\pa^2 f}{\pa z\pa x}+\frac{y^2}{2}\frac{\pa^2 f}{\pa z\pa z}-y\pa_y f \right]^2+\left[ \frac{\pa^2 f}{\pa z\pa y}+(\frac{y^3}{2}\pa_zf+y\pa_xf)\right]^2\\
&&-y^2|\pa_y f|^2-(\frac{y^3}{2}\pa_zf+y\pa_xf)^2\\
&=&|\mathfrak{Hess}_{a,z}f|^2+\frac{y^2}{2}\Gamma_1^z(f,f)-y^2\Gamma_1(f,f),
\eeaa 
where we use the fact
\beaa 
\left[\frac{\pa^2 f}{\pa y \pa x}+\frac{y^2}{2}\frac{\pa^2 f}{\pa z\pa y}+y\pa_zf\right]^2+\left[\frac{\pa^2 f}{\pa y \pa x}+\frac{y^2}{2}\frac{\pa^2 f}{\pa z\pa y}\right]^2=2\left[\frac{\pa^2 f}{\pa y \pa x}+\frac{y^2}{2}\frac{\pa^2 f}{\pa z\pa y}+\frac{1}{2}y\pa_zf\right]^2+\frac{y^2}{2}|\pa_zf|^2.
\eeaa 
The proof is thus completed.
\qed 
\end{proof}
We are now left to compute the three tensor terms.

\begin{proof}[Proof of Lemma \ref{tensor martinet}]
Similar to the proof of Lemma \ref{SE2 tensor}, we have 
\beaa
\mathfrak{R}_a(\nabla f,\nabla f)
&=&\sum_{i,k=1}^2\sum_{i',\hat i,\hat k=1}^{3} \la a^{\ts}_{ii'} (\frac{\partial a^{\ts}_{i \hat i}}{\partial x_{i'}} \frac{\partial a^{\ts}_{k\hat k}}{\partial x_{\hat i}}\frac{\partial f}{\partial x_{\hat k}}) ,(a^{\ts}\nabla)_kf\ra_{\hR^2}  \nonumber \\ 
&&+\sum_{i,k=2}^n\sum_{i',\hat i,\hat k=1}^{3} \la a^{\ts}_{ii'}a^{\ts}_{i \hat i} (\frac{\partial }{\partial x_{i'}} \frac{\partial a^{\ts}_{k\hat k}}{\partial x_{\hat i}})(\frac{\partial f}{\partial x_{\hat k}}) ,(a^{\ts}\nabla)_kf\ra_{\hR^2} \nonumber\\
&&-\sum_{i,k=1}^2\sum_{i',\hat i,\hat k=1}^{3} \la a^{\ts}_{k\hat k}\frac{\partial a^{\ts}_{ii'}}{\partial x_{\hat k}} \frac{\partial a^{\ts}_{i \hat i}}{\partial x_{i'}} \frac{\partial f}{\partial x_{\hat i}})  ,(a^{\ts}\nabla)_kf\ra_{\hR^2}  \nonumber\\
&&-\sum_{i,k=1}^2\sum_{i',\hat i,\hat k=1}^{3} \la a^{\ts}_{k\hat k} a^{\ts}_{ii'} (\frac{\partial }{\partial x_{\hat k}} \frac{\partial a^{\ts}_{i \hat i}}{\partial x_{i'}}) \frac{\partial f}{\partial x_{\hat i}}  ,(a^{\ts}\nabla)_kf\ra_{\hR^2},  \nonumber\\
&=& \cI_1+\cI_2+\cI_3+\cI_4.
\eeaa
By direct computations, we have
\beaa
\cI_1&=&\sum_{i=1}^2\sum_{i',\hat i,\hat k=1}^{3} \Big[ a^{\ts}_{ii'} (\frac{\partial a^{\ts}_{i \hat i}}{\partial x_{i'}} \frac{\partial a^{\ts}_{1\hat k}}{\partial x_{\hat i}}\frac{\partial f}{\partial x_{\hat k}})(a^{\ts}\nabla)_1f+  a^{\ts}_{ii'} (\frac{\partial a^{\ts}_{i \hat i}}{\partial x_{i'}} \frac{\partial a^{\ts}_{2\hat k}}{\partial x_{\hat i}}\frac{\partial f}{\partial x_{\hat k}}) (a^{\ts}\nabla)_2f\Big]=0;\\
\cI_2&=&\sum_{i=2}^n\sum_{i',\hat i,\hat k=1}^{3}\Big[  a^{\ts}_{ii'}a^{\ts}_{i \hat i} (\frac{\partial }{\partial x_{i'}} \frac{\partial a^{\ts}_{1\hat k}}{\partial x_{\hat i}})(\frac{\partial f}{\partial x_{\hat k}}) (a^{\ts}\nabla)_1f+ a^{\ts}_{ii'}a^{\ts}_{i \hat i} (\frac{\partial }{\partial x_{i'}} \frac{\partial a^{\ts}_{2\hat k}}{\partial x_{\hat i}})(\frac{\partial f}{\partial x_{\hat k}})(a^{\ts}\nabla)_2f\Big]\\
&=&a^{\ts}_{22}a^{\ts}_{22}\frac{\pa^2}{\pa y \pa y}a^{\ts}_{13}\frac{\pa f}{\pa z}(a^{\ts}\nabla )_1f=\frac{\pa f}{\pa z}(a^{\ts}\nabla )_1f ;  \\
\cI_3&=&-\sum_{i=1}^2\sum_{i',\hat i,\hat k=1}^{3} \Big[  a^{\ts}_{1\hat k}\frac{\partial a^{\ts}_{ii'}}{\partial x_{\hat k}} \frac{\partial a^{\ts}_{i \hat i}}{\partial x_{i'}} \frac{\partial f}{\partial x_{\hat i}})(a^{\ts}\nabla)_1f+  a^{\ts}_{2\hat k}\frac{\partial a^{\ts}_{ii'}}{\partial x_{\hat k}} \frac{\partial a^{\ts}_{i \hat i}}{\partial x_{i'}} \frac{\partial f}{\partial x_{\hat i}})(a^{\ts}\nabla)_2f\Big]=0;\\
\cI_4&=&-\sum_{i=1}^2\sum_{i',\hat i,\hat k=1}^{3} \Big[  a^{\ts}_{1\hat k} a^{\ts}_{ii'} (\frac{\partial }{\partial x_{\hat k}} \frac{\partial a^{\ts}_{i \hat i}}{\partial x_{i'}}) \frac{\partial f}{\partial x_{\hat i}}(a^{\ts}\nabla)_1f+  a^{\ts}_{2\hat k} a^{\ts}_{ii'} (\frac{\partial }{\partial x_{\hat k}} \frac{\partial a^{\ts}_{i \hat i}}{\partial x_{i'}}) \frac{\partial f}{\partial x_{\hat i}}(a^{\ts}\nabla)_2f\Big]=0.
\eeaa
For the drift term, we take $b=-aa^{\ts}\nabla V$,
\beaa
\mathfrak{R}_b^a
&=&\sum_{i,k=1}^2\sum_{\hat i,\hat k,k'=1}^3\left[a^{\ts}_{i\hat i}\frac{\pa a^{\ts}_{k\hat k} }{\pa x_{\hat i}}a^{\ts}_{kk'}\frac{\pa V}{\pa x_{k'}} \frac{\pa f}{\pa x_{\hat k}}(a^{\ts}\nabla)_i f +a^{\ts}_{i\hat i}\frac{\pa a^{\ts}_{kk'}}{\pa x_{\hat i}}a^{\ts}_{k\hat k}\frac{\pa V}{\pa x_{k'}} \frac{\pa f}{\pa x_{\hat k}}(a^{\ts}\nabla)_i f\right]\\
&&+\sum_{i,k=1}^2\sum_{\hat i,\hat k,k'=1}^3\left[ a^{\ts}_{i\hat i}a^{\ts}_{k\hat k}a^{\ts}_{kk'}\frac{\pa^2 V}{\pa x_{\hat i}\pa x_{k'}}\frac{\pa f}{\pa x_{\hat k}}(a^{\ts}\nabla)_i f\right]\\
&&-\sum_{i,k=1}^2\sum_{\hat i,\hat k,k'=1}^3\left[a^{\ts}_{k\hat k}a^{\ts}_{kk'}\frac{\pa a^{\ts}_{i\hat i}}{\pa x_{\hat k}}\frac{\pa V}{\pa x_{k'}} \frac{\pa f}{\pa x_{\hat i}}(a^{\ts}\nabla)_if\right] \\
&=& \cJ_1+\cJ_2+\cJ_3+\cJ_4.
\eeaa
Plugging into the matrices of $a^{\ts}$, we get
\beaa
\cJ_1
&=&\sum_{\hat i,\hat k,k'=1}^3\left[ a^{\ts}_{1\hat i}\frac{\pa a^{\ts}_{1\hat k} }{\pa x_{\hat i}}a^{\ts}_{1k'}\frac{\pa V}{\pa x_{k'}} \frac{\pa f}{\pa x_{\hat k}}(a^{\ts}\nabla)_1 f+a^{\ts}_{2\hat i}\frac{\pa a^{\ts}_{1\hat k} }{\pa x_{\hat i}}a^{\ts}_{1k'}\frac{\pa V}{\pa x_{k'}} \frac{\pa f}{\pa x_{\hat k}}(a^{\ts}\nabla)_2 f\right]\\
&&+\sum_{\hat i,\hat k,k'=1}^3\left[ a^{\ts}_{1\hat i}\frac{\pa a^{\ts}_{2\hat k} }{\pa x_{\hat i}}a^{\ts}_{2k'}\frac{\pa V}{\pa x_{k'}} \frac{\pa f}{\pa x_{\hat k}}(a^{\ts}\nabla)_1 f+a^{\ts}_{2\hat i}\frac{\pa a^{\ts}_{2\hat k} }{\pa x_{\hat i}}a^{\ts}_{2k'}\frac{\pa V}{\pa x_{k'}} \frac{\pa f}{\pa x_{\hat k}}(a^{\ts}\nabla)_2 f\right]\\
&=& a^{\ts}_{22}\frac{\pa a^{\ts}_{13}}{\pa y}(a^{\ts}\nabla)_1V \frac{\pa f}{\pa z}(a^{\ts}\nabla)_2f=y(a^{\ts}\nabla)_1V \frac{\pa f}{\pa z}(a^{\ts}\nabla)_2f;
\eeaa
\beaa
\cJ_2
&=&\sum_{\hat i,\hat k,k'=1}^3\left[a^{\ts}_{1\hat i}\frac{\pa a^{\ts}_{1k'}}{\pa x_{\hat i}}a^{\ts}_{1\hat k}\frac{\pa V}{\pa x_{k'}} \frac{\pa f}{\pa x_{\hat k}}(a^{\ts}\nabla)_1 f+a^{\ts}_{2\hat i}\frac{\pa a^{\ts}_{1k'}}{\pa x_{\hat i}}a^{\ts}_{1\hat k}\frac{\pa V}{\pa x_{k'}} \frac{\pa f}{\pa x_{\hat k}}(a^{\ts}\nabla)_2 f\right]\\
&&+\sum_{\hat i,\hat k,k'=1}^3\left[a^{\ts}_{1\hat i}\frac{\pa a^{\ts}_{2k'}}{\pa x_{\hat i}}a^{\ts}_{2\hat k}\frac{\pa V}{\pa x_{k'}} \frac{\pa f}{\pa x_{\hat k}}(a^{\ts}\nabla)_1 f+a^{\ts}_{2\hat i}\frac{\pa a^{\ts}_{2k'}}{\pa x_{\hat i}}a^{\ts}_{2\hat k}\frac{\pa V}{\pa x_{k'}} \frac{\pa f}{\pa x_{\hat k}}(a^{\ts}\nabla)_2 f\right]\\
&=&y \frac{\pa V}{\pa z}(a^{\ts}\nabla)_1f(a^{\ts}\nabla)_2f;
\eeaa
\beaa
\cJ_3
&=&\sum_{\hat i,k'=1}^3\left[ a^{\ts}_{1\hat i}a^{\ts}_{1k'}\frac{\pa^2 V}{\pa x_{\hat i}\pa x_{k'}} |(a^{\ts}\nabla)_1 f|^2+a^{\ts}_{2\hat i}a^{\ts}_{1k'}\frac{\pa^2 V}{\pa x_{\hat i}\pa x_{k'}}(a^{\ts}\nabla)_1 f(a^{\ts}\nabla)_2 f\right]\\
&&+\sum_{\hat i,k'=1}^3\left[ a^{\ts}_{1\hat i}a^{\ts}_{2k'}\frac{\pa^2 V}{\pa x_{\hat i}\pa x_{k'}}(a^{\ts}\nabla)_2 f(a^{\ts}\nabla)_1 f+a^{\ts}_{2\hat i}a^{\ts}_{2k'}\frac{\pa^2 V}{\pa x_{\hat i}\pa x_{k'}}|(a^{\ts}\nabla)_2 f|^2\right]\\
&=&\sum_{\hat i,k'=1}^3a^{\ts}_{1\hat i}a^{\ts}_{1k'}\frac{\pa^2 V}{\pa x_{\hat i}\pa x_{k'}}|(a^{\ts}\nabla)_1 f|^2+2(\frac{\pa^2 V}{\pa x\pa y}+\frac{y^2}{2}\frac{\pa^2 V}{\pa y \pa z} )(a^{\ts}\nabla)_1 f(a^{\ts}\nabla)_2 f\\
&&+\frac{\pa^2 V}{\pa y\pa y}|(a^{\ts}\nabla)_2 f|^2;
\eeaa

\beaa
\cJ_4
&=&-\sum_{\hat i,\hat k,k'=1}^3\left[a^{\ts}_{1\hat k}a^{\ts}_{1k'}\frac{\pa a^{\ts}_{1\hat i}}{\pa x_{\hat k}}\frac{\pa V}{\pa x_{k'}} \frac{\pa f}{\pa x_{\hat i}}(a^{\ts}\nabla)_1f+a^{\ts}_{1\hat k}a^{\ts}_{1k'}\frac{\pa a^{\ts}_{2\hat i}}{\pa x_{\hat k}}\frac{\pa V}{\pa x_{k'}} \frac{\pa f}{\pa x_{\hat i}}(a^{\ts}\nabla)_2f\right]
\eeaa 
\beaa 
&&-\sum_{\hat i,\hat k,k'=1}^3\left[a^{\ts}_{2\hat k}a^{\ts}_{2k'}\frac{\pa a^{\ts}_{1\hat i}}{\pa x_{\hat k}}\frac{\pa V}{\pa x_{k'}} \frac{\pa f}{\pa x_{\hat i}}(a^{\ts}\nabla)_1f+a^{\ts}_{2\hat k}a^{\ts}_{2k'}\frac{\pa a^{\ts}_{2\hat i}}{\pa x_{\hat k}}\frac{\pa V}{\pa x_{k'}} \frac{\pa f}{\pa x_{\hat i}}(a^{\ts}\nabla)_2f\right]\\ 
&=&-y\frac{\pa V}{\pa y}\frac{\pa f}{\pa z}(a^{\ts}\nabla)_1f.
\eeaa
 Combing the above computations, we get the tensor $\mathfrak{R}_{ab}$. Now we turn to the second tensor $\mathfrak{R}_{zb}$.
Since $z^{\ts}=(0,0,1)$, it is obvious to see that only the drift term of the tensor $\mathfrak{R}_{zb}$ remains, where we denote
\beaa
\mathfrak{R}_{b}^z(\nabla f,\nabla f)=-\sum_{\hat i,\hat k=1}^{3} (z^{\ts}_{1\hat i}\frac{\pa b_{\hat k}}{\pa x_{\hat i}}\frac{\pa f}{\pa x_{\hat k}}-b_{\hat k}\frac{\pa z^{\ts}_{1\hat i}}{\pa x_{\hat k}} \frac{\pa f}{\pa x_{\hat i}} )(z^{\ts}\nabla )_1f.
\eeaa
By taking $b=-aa^{\ts}\nabla V$, we further obtain that
\beaa
\mathfrak{R}_{b}^z(\nabla f,\nabla f)
&=&  -\sum_{\hat i,\hat k=1}^{3}\left[ z^{\ts}_{1\hat i}\frac{\pa b_{\hat k}}{\pa x_{\hat i}}\frac{\pa f}{\pa x_{\hat k}}(z^{\ts}\nabla f)_1-b_{\hat k}\frac{\pa z^{\ts}_{1\hat i}}{\pa x_{\hat k}} \frac{\pa f}{\pa x_{\hat i}} (z^{\ts}\nabla f)_1\right]\\
&=&\sum_{k=1}^2\sum_{\hat i,\hat k,k'=1}^3\left[z^{\ts}_{1\hat i}\frac{\pa a^{\ts}_{k\hat k} }{\pa x_{\hat i}}a^{\ts}_{kk'}\frac{\pa V}{\pa x_{k'}} \frac{\pa f}{\pa x_{\hat k}}(z^{\ts}\nabla)_1 f\right]\\
&&+\sum_{k=1}^2\sum_{\hat i,\hat k,k'=1}^3\left[z^{\ts}_{1\hat i}\frac{\pa a^{\ts}_{kk'}}{\pa x_{\hat i}}a^{\ts}_{k\hat k}\frac{\pa V}{\pa x_{k'}} \frac{\pa f}{\pa x_{\hat k}}(z^{\ts}\nabla)_1 f\right]\\ 
&&+\sum_{k=1}^2\sum_{\hat i,\hat k,k'=1}^3\left[ z^{\ts}_{1\hat i}a^{\ts}_{k\hat k}a^{\ts}_{kk'}\frac{\pa^2 V}{\pa x_{\hat i}\pa x_{k'}}\frac{\pa f}{\pa x_{\hat k}}(z^{\ts}\nabla)_1 f\right]\\
&&-\sum_{k=1}^2\sum_{\hat i,\hat k,k'=1}^3\left[a^{\ts}_{k\hat k}a^{\ts}_{kk'}\frac{\pa z^{\ts}_{1\hat i}}{\pa x_{\hat k}}\frac{\pa V}{\pa x_{k'}} \frac{\pa f}{\pa x_{\hat i}}(z^{\ts}\nabla)_1f\right] \\
&=& \cJ_1^z+\cJ_2^z+\cJ_3^z+\cJ_4^z.
\eeaa
By direct computations, it is not hard to observe that
\beaa
\cJ_1^z&=&\sum_{k=1}^2\sum_{\hat i,\hat k,k'=1}^3\left[z^{\ts}_{1\hat i}\frac{\pa a^{\ts}_{k\hat k} }{\pa x_{\hat i}}a^{\ts}_{kk'}\frac{\pa V}{\pa x_{k'}} \frac{\pa f}{\pa x_{\hat k}}(z^{\ts}\nabla)_1 f \right]=0;\\
\cJ_2^z&=&\sum_{k=1}^2\sum_{\hat i,\hat k,k'=1}^3\left[z^{\ts}_{1\hat i}\frac{\pa a^{\ts}_{kk'}}{\pa x_{\hat i}}a^{\ts}_{k\hat k}\frac{\pa V}{\pa x_{k'}} \frac{\pa f}{\pa x_{\hat k}}(z^{\ts}\nabla)_1 f\right]=0;\\
\cJ_4^z&=&-\sum_{k=1}^2\sum_{\hat i,\hat k,k'=1}^3\left[a^{\ts}_{k\hat k}a^{\ts}_{kk'}\frac{\pa z^{\ts}_{1\hat i}}{\pa x_{\hat k}}\frac{\pa V}{\pa x_{k'}} \frac{\pa f}{\pa x_{\hat i}}(z^{\ts}\nabla)_1f\right]=0.
\eeaa
The only non-zero term has the following form,
\beaa
\cJ_3^z&=&\sum_{k=1}^2\sum_{\hat i,\hat k,k'=1}^3\left[ z^{\ts}_{1 \hat i}a^{\ts}_{k\hat k}a^{\ts}_{kk'}\frac{\pa^2 V}{\pa x_{\hat i}\pa x_{k'}}\frac{\pa f}{\pa x_{\hat k}}(z^{\ts}\nabla)_1 f\right]\\
&=&\sum_{\hat i,\hat k,k'=1}^3\left[ z^{\ts}_{1\hat i}a^{\ts}_{1\hat k}a^{\ts}_{1k'}\frac{\pa^2 V}{\pa x_{\hat i}\pa x_{k'}}\frac{\pa f}{\pa x_{\hat k}}(z^{\ts}\nabla)_1 f+ z^{\ts}_{1\hat i}a^{\ts}_{2\hat k}a^{\ts}_{2k'}\frac{\pa^2 V}{\pa x_{\hat i}\pa x_{k'}}\frac{\pa f}{\pa x_{\hat k}}(z^{\ts}\nabla)_1 f\right]
\eeaa 
\beaa 
&=&\sum_{\hat i,k'=1}^3\left[ z^{\ts}_{1\hat i}a^{\ts}_{1k'}\frac{\pa^2 V}{\pa x_{\hat i}\pa x_{k'}} (a^{\ts}\nabla)_1 f(z^{\ts}\nabla)_1 f+ z^{\ts}_{1\hat i}a^{\ts}_{2k'}\frac{\pa^2 V}{\pa x_{\hat i}\pa x_{k'}}(a^{\ts}\nabla)_2 f(z^{\ts}\nabla)_1 f\right]\\
&=&(\frac{\pa^2 V}{\pa x\pa z}+\frac{y^2}{2}\frac{\pa^2 V}{\pa z\pa z})(a^{\ts}\nabla)_1 f(z^{\ts}\nabla)_1 f+\frac{\pa^2 V}{\pa y\pa z}(a^{\ts}\nabla)_2 f(z^{\ts}\nabla)_1 f.
\eeaa
Since matrix $z^{\ts}$ is a constant matrix and matrix $a^{\ts}$ contains only variable $y$, it is easy to observe that 
\[
\mathfrak{R}_{\pi}(\nabla f,\nabla f)=0.
\]
\qed 
\end{proof}{}

\section{Lyapunov analysis in Sub-Riemannian Density manifold}\label{section 4}
In this section, we illustrate the motivation in this paper. This is to design a matrix condition, whose smallest eigenvalue characterizes the convergence rate of degenerate SDE. 

The outline of this section is given below. Consider a density space over the sub-Riemannian manifold. The finite dimensional sub-Rimennian structure introduces the density space the {\em infinite dimensional sub-Riemannian structure}. We name it the {\em sub-Riemannian density manifold} (SDM). We provide the geometric calculations in SDM. We study the Fokker-Planck equation as the sub-Riemannian gradient flow in SDM. We derive the equivalence relation between second order calculus of relative entropy in SDM and generalized Gamma z calculus.  
\subsection{Sub-Riemannian Density manifold} 
Given a finite dimensional sub-Riemannian manifold $(\hM^{n+m}, \tau, g_\tau)$ with $g_\tau=(aa^{\ts})^{\dd}$, consider the probability density space 
\begin{equation*}
\mathcal{P}(\hM^{n+m})=\Big\{\rho(x)\in C^{\infty}(\hM^{n+m})\colon \int \rho(x)dx=1,\quad \rho(x)\geq 0\Big\}.
\end{equation*} 
Consider the tangent space at $\rho\in \mathcal{P}(\hM^{n+m})$:
\[
T_{\rho}\mathcal{P}(\hM^{n+m})=\{\sigma(x)\in C^{\infty}(\hM^{n+m})\colon \int \sigma(x)dx=0 \}.
\]
We introduce the sub-Riemannian structure in probability density space $\mathcal{P}(\hM^{n+m})$. 
 \begin{defn}[sub-Riemannian Wasserstein metric tensor]\label{eqn:riemannian whole density space}
 The $L^2$ sub-Riemannian-Wasserstein metric $g^{W_a}_\rho\colon {T_\rho}\mathcal{P}(\hM^{n+m})\times {T_\rho}\mathcal{P}(\hM^{n+m})\rightarrow \mathbb{R}$ is defined by
 \begin{equation*}
 g^{W_a}_\rho(\sigma_1, \sigma_2)=\int \Big(\sigma_1(x), (-\Delta^a_\rho)^{\dd}\sigma_2(x)\Big) dx.
  \end{equation*}
Here $\sigma_1, \sigma_2\in T_\rho\mathcal{P}(\hM^{n+m})$, $(\cdot, \cdot)$ is the metric on $\hM^{n+m}$ and
 $(\Delta^a_\rho)^{\dd}\colon {T_\rho}\mathcal{P}(\hM^{n+m})\rightarrow{T_\rho}\mathcal{P}(\hM^{n+m})$ is the pseudo-inverse of the sub-elliptic operator $$\Delta^a_\rho = \nabla\cdot(\rho aa^{\ts}\nabla).$$
 \end{defn}
 For some special choices of $a$ as studied in \cite{KL} or $aa^{\ts}$ forming a positive definite matrix, then $\Delta^a_{\rho}$ is an elliptic operator. In this case, $(\mathcal{P}(\hM^{n+m}), g^{W_a})$ still forms a Riemannian density manifold. In general, given a sub-Riemannian manifold $(\hM^{n+m}, (aa^{\ts})^{\dd})$, $\Delta_{\rho}^a$ is only a sub-elliptic operator. Thus $(\mathcal{P}(\hM^{n+m}), g^{W_a})$ forms an infinite-dimensional sub-Riemannian manifold. 

We next present the sub-Riemannian calculus in $(\mathcal{P}(\hM^{n+m}), g^{W_a})$, including both geodesics and Hessian operator in tangent bundle. 
Consider an identification map: $$V\colon C^{\infty}(\hM^{n+m})\rightarrow T_\rho\mathcal{P}(\hM^{n+m}),\quad V_\Phi=-\Delta_{\rho}^a\Phi= -\nabla\cdot(\rho aa^{\ts} \nabla \Phi).$$
Here $\Phi \in T_\pi\mathcal{P}(\hM^{n+m})=C^{\infty}(\hM^{n+m})/\sim$. This $T_\pi\mathcal{P}(\hM^{n+m})$ is the cotangent space in SDM and $\sim$ represents a constant shift relation. 
Thus
\begin{equation*}
g^{W_a}_\rho(V_{\Phi_1}, V_{\Phi_2})=\int  \Gamma_{1}(\Phi_1,\Phi_2)\rho(x)dx.
\end{equation*}
In other words, 
\begin{equation}\label{dual}
\begin{split}
g^{W_a}_\rho(V_{\Phi_1}, V_{\Phi_2})=& \int V_{\Phi_1} (-\Delta^a_\rho)^{\dd}V_{\Phi_2} dx\\
=&\int  \Phi_1 (-\Delta_{\rho}^a)(-\Delta_{\rho}^a)^{\dd}(-\Delta_{\rho}^a) \Phi_2 dx\\
=&\int (\Phi_1 ,-\Delta_\rho^a\Phi_2) dx\\
=&\int \Phi_1 (-\nabla\cdot (\rho aa^{\ts}\nabla\Phi_2) dx\\
=&\int  (a^{\ts} \nabla\Phi_1, a^{\ts}\nabla\Phi_2)\rho dx,\end{split}
\end{equation}
where the second equality holds by $(-\Delta_{\rho}^a)(-\Delta_{\rho}^a)^{\dd}(-\Delta_{\rho}^a)=-\Delta_\rho^a$ and the last equality holds by the integration by parts. 

We next derive several basic geometric calculations in SDM. 
\begin{prop}[Geodesics in SDM]
The sub-Riemannian geodesics in cotangent bundle forms 
\begin{equation}\label{geo}
\begin{cases}
\partial_t\rho_t+\nabla\cdot(\rho_t aa^{\ts} \nabla \Phi_t)=0,\\
\partial_t\Phi_t+ \frac{1}{2}(a^{\ts}\nabla\Phi_t, a^{\ts}\nabla\Phi_t)=0.
\end{cases}
\end{equation}
\end{prop}
\begin{proof}
We consider the Lagrangian formulation of geodesics in density. Here the minimization of the geometric action functional forms
\begin{equation*}
\mathcal{L}(\rho_t, \partial_t\rho_t)=\int_0^1\int  \frac{1}{2}(\partial_t\rho_t, (-\Delta^a_{\rho_t})^{\dd}\partial_t\rho_t)dxdt, 
\end{equation*}
where $\rho_t=\rho(t,x)$ is a density path with fixed boundary points $\rho_0$, $\rho_1$. Then the Euler-Lagrange equation in density space forms
\begin{equation*}\label{EL}
\frac{\partial}{\partial t}\delta_{\partial_t\rho_t}\mathcal{L}(\rho_t, \partial_t\rho_t)=\delta_{\rho_t} \mathcal{L}(\rho_t, \partial_t\rho_t),
\end{equation*}
where $\delta_{\partial_t\rho_t}$ is the $L^2$ first variation w.r.t. $\partial_t\rho_t$ and $\delta_{\rho_t}$ is the $L^2$ first variation w.r.t. $\rho_t$. Here
\begin{equation}\label{EL}
\begin{split}
\partial_t \Big((-\Delta^a_{\rho_t})^{\dd}\partial_t\rho_t\Big)=&\delta_\rho \int \frac{1}{2}\Big(\partial_t\rho, (-\Delta^a_{\rho_t})^{\dd}\partial_t\rho_t\Big)dx\\
=&-{\frac{1}{2} \Big(a^\ts\nabla(-\Delta_{\rho}^a)^\dd\partial_t\rho_t,a^\ts\nabla(-\Delta_{\rho}^a)^\dd\partial_t\rho_t\Big)},
\end{split}
\end{equation}
where the last equality uses the following fact 
\begin{equation*}
\partial_t\Big(\Delta_{\rho_t}^a\Big)^{\dd}=-\Big(\Delta^a_{\rho_t}\Big)^{\dd}\cdot\Delta^a_{\partial_t\rho_t}\cdot\Big(\Delta^a_{\rho_t}\Big)^{\dd},
\end{equation*}
Denote $\partial_t\rho_t=-\Delta^a_{\rho_t}\Phi_t$, then the Euler-Lagrange equation \eqref{EL} forms the sub-Riemannian geodesics flow \eqref{geo}. {In other words, 
\begin{equation*}
\partial_t\Phi+\frac{1}{2}(a^{\ts}\nabla\Phi, a^{\ts}\nabla\Phi)=0.    
\end{equation*}}
\end{proof}

\begin{prop}[Gradient and Hessian operators in SDM]\label{hessian operators}
Given a functional $\mathcal{F}\colon \mathcal{P}(\hM^{n+m})\rightarrow \mathbb{R}$, the gradient operator of $\mathcal{F}$ in $(\mathcal{P}, g^{W_a})$ satisfies 
\begin{equation*}
grad_{W_a}\mathcal{F}(\rho)=-\nabla\cdot(\rho aa^{\ts}\nabla\delta\mathcal{F}(\rho)).
\end{equation*}
And the Hessian operator of $\mathcal{F}$ in $(\mathcal{P}, g^{W_a})$ satisfies
\begin{equation*}
\begin{split}
&Hess_{W_a} \cF(V_{\Phi},V_{\Phi})\\
=&\int \int (a(y)^{\ts}\nabla_y)(a(x)^{\ts}\nabla_x)\delta^2\mathcal{F}(\rho)(x,y) \Big(a(x)^{\ts}\nabla_x\Phi(x), a(y)^{\ts}\nabla_y\Phi(y)\Big) \rho(x)\rho(y) dx dy \\
&+\int \textrm{Hess}_a\delta\mathcal{F}(\rho)(\Phi, \Phi)\rho dx,
\end{split}
\end{equation*}
where 
\begin{equation*}
\textrm{Hess}_{a}\delta\mathcal{F}(\rho)(\Phi,\Phi)=\frac{1}{2}\Big\{2\Gamma_{1}(\Gamma_{1}(\delta\mathcal{F}, \Phi), \Phi) - \Gamma_{1}(\Gamma_{1}(\Phi, \Phi), \delta\mathcal{F})\Big\}.
\end{equation*}
\end{prop}
\begin{proof}
We first derive the sub-Riemannian gradient operator. {We recall the identification map by $-\Delta_\rho^a\Phi=-\nabla\cdot(\rho aa^{\ts}\nabla\Phi)$. Hence the gradient operator in SDM satisfies} 
\begin{equation*}
\begin{split}
\textrm{grad}_{W_a}\mathcal{F}(\rho)=&\Big((-\Delta^a_{\rho})^{\dd}\Big)^{\dd}\frac{\delta}{\delta\rho(x)}\mathcal{F}(\rho)\\
=&-\Delta^a_{\rho}\frac{\delta}{\delta\rho(x)}\mathcal{F}(\rho)\\
=&-\nabla\cdot(\rho aa^{\ts}\nabla \frac{\delta}{\delta\rho(x)}\mathcal{F}(\rho)).
\end{split}
\end{equation*}
And the Hessian operator in SDM satisfies 
\begin{equation*}
\textrm{Hess}_{W_a}\mathcal{F}(\rho)(V_\Phi, V_\Phi)=\frac{d^2}{dt^2}\mathcal{F}(\rho_t)|_{t=0},
\end{equation*}
where $(\rho_t, \Phi_t)$ satisfies the geodesics equation \eqref{geo} with $\rho_0=\rho$, $\Phi_0=\Phi$. Notice the fact that 
\begin{equation*}
\begin{split}
\frac{d}{dt}\mathcal{F}(\rho_t)|_{t=0}=&\int \partial_t\rho_t\delta\mathcal{F}(\rho)dx|_{t=0}\\
=&\int  (-\nabla\cdot(\rho aa^{\ts}\nabla\Phi))\delta\mathcal{F}(\rho)dx\\
=&\int  (a^{\ts}\nabla\delta\mathcal{F}(\rho), a^{\ts}\nabla\Phi)\rho dx.
\end{split}
\end{equation*}
In addition, 
\begin{equation}\label{Hess}
\begin{split}
&\frac{d^2}{dt^2}\mathcal{F}(\rho_t)|_{t=0}=\frac{d}{dt}\int  (a^{\ts}\nabla\delta\mathcal{F}(\rho_t), a^{\ts}\nabla\Phi_t)\rho_t dx|_{t=0}\\
=&\int  \int    \delta^2\mathcal{F}(\rho)(x,y) \partial_t\rho_t(x)\partial_t\rho_t(y)dxdy\\
&+ \int  (a^{\ts}\nabla\delta\mathcal{F}(\rho_t), a^{\ts}\nabla \partial_t\Phi_t)\rho_t dx\\
&+\int  (a^{\ts}\nabla\delta\mathcal{F}(\rho_t), a^{\ts}\nabla \partial_t\Phi_t)\partial_t\rho_t dx|_{t=0}\\
=&\int \int  \delta^2\mathcal{F}(\rho)(x,y) \nabla\cdot(\rho aa^{\ts}\nabla\Phi)(x) \nabla\cdot(\rho aa^{\ts}\nabla\Phi)(y)dxdy\\
&-\frac{1}{2}\int  (a^{\ts}\nabla\delta\mathcal{F}(\rho), a^{\ts}\nabla \Gamma_{1}(\Phi, \Phi))\rho dx\\
&+\int  \Gamma_{1}(\Phi, \delta\mathcal{F}(\rho)) \Big(-\nabla\cdot(\rho aa^{\ts}\nabla\Phi)\Big)dx\\
=&\int \int (a(y)^{\ts}\nabla_y)(a(x)^{\ts}\nabla_x)\delta^2\mathcal{F}(\rho)(x,y) \Big(a(x)^{\ts}\nabla_x\Phi(x), a(y)^{\ts}\nabla_y\Phi(y)\Big) \rho(x)\rho(y) dx dy \\
+&\frac{1}{2}\int \Big\{2\Gamma_{1}(\Gamma_{1}(\delta\mathcal{F}, \Phi), \Phi)- \Gamma_{1}(\Gamma_{1}(\Phi, \Phi), \delta\mathcal{F})\Big\}\rho dx,
\end{split}
\end{equation}
where the last equality holds by the integration by parts formula.  
\end{proof}

We next show the equivalence relation between Hessian of relative entropy in SDM and classical Gamma two operator. We first demonstrate the relation among $L^*$, $\Delta_a$ and gradient operator of entropy. In particular, we show that the Fokker-Planck equation is a sub-Riemannian gradient flow in SDM. Denote the KL divergence as 
\begin{equation}\label{D}
\mathrm{D}(\rho)=\mathrm{D}_{\mathrm{KL}}(\rho\|\pi)=\int \rho(x)\log\frac{\rho(x)}{\Vol(x)}dx. 
\end{equation}

\begin{prop}[Gradient flow]
The negative gradient operator in $(\mathcal{P}, g^{W_a})$ forms 
\begin{equation*}
-\textrm{grad}_{\WH}\mathrm{D}(\rho)=L^*\rho=\nabla \cdot (\rho aa^{\ts}\nabla\log\frac{\rho}{\pi}).
\end{equation*}
In addition, the sub-Riemannian gradient flow of $\mathrm{D}(\rho)$ in $(\mathcal{P}, g^{W_a})$ forms the Fokker-Planck equation 
\begin{equation}\label{newFPE}
\partial_t\rho=\nabla\cdot(\rho aa^{\ts}\nabla\log\frac{\rho}{\pi}).
\end{equation}
\end{prop}
\begin{proof}
We first derive the sub-Riemannian gradient operator of entropy and relative entropy. 
Notice 
\begin{equation*}
\delta_{\rho(x)}\mathrm{D}(\rho)=\log \rho(x)+1- \log \Vol(x).
\end{equation*}
Thus 
\begin{equation*}
\textrm{grad}_{\WH}\mathrm{D}(\rho)=-\nabla\cdot(\rho aa^{\ts}\nabla\log\rho)+\nabla\cdot(\rho aa^{\ts}\nabla\log\Vol),
\end{equation*}
where $\rho\nabla\log\rho=\rho\frac{\nabla\rho}{\rho}=\nabla\rho$. Following the gradient flow formulation
$$\frac{\partial\rho_t}{\partial t}=-\textrm{grad}_{W_a}\mathrm{D}(\rho_t)=L^*\rho_t,$$
we finish the derivation of \eqref{newFPE}. \qed 
\end{proof}

We next demonstrate that the Hessian of relative entropy (KL divergence) is equivalent to the classical Bakry-{\'E}mery calculus. 
\begin{prop}[Hessian of Entropy and Bakry-{\'E}mery calculus]\label{prop1}
Given $\Phi_1$, $\Phi_2\in C^{\infty}(\hM^{n+m})$, then
\begin{equation*}
\begin{split}
Hess_{W_a} \mathrm{D}(\rho)(V_{\Phi},V_{\Phi})=&\int \Gamma_2(\Phi, \Phi)\rho(x)dx.
\end{split}
\end{equation*}
\end{prop}
\begin{proof}
We first derive the Hessian of $\mathrm{D}(\rho)$ in SDM. Notice the fact $\delta^2\mathrm{D}(\rho)(x,y)=\frac{1}{\rho}\delta_{x=y}$. For simplicity, we denote $\delta^2\mathrm{D}(\rho)=\frac{1}{\rho(x)}$. 
By using \eqref{Hess}, we have
\begin{equation}\label{Hessa}
\begin{split}
\textrm{Hess}_{\WH}\mathrm{D}(\rho)(V_{\Phi}, V_{\Phi})=&\int  \delta^2\mathrm{D}(\rho)(x)\Big(\nabla\cdot(\rho aa^{\ts}\nabla\Phi)\Big)^2 dx \hspace{2cm} (a)\\
&-\frac{1}{2}\int  (a^{\ts}\nabla\delta\mathrm{D}(\rho), a^{\ts}\nabla \Gamma_{1}(\Phi, \Phi))\rho dx\hspace{1.4cm}(b)\\
&+\int  \Gamma_{1}(\Phi, \delta\mathrm{D}(\rho)) \Big(-\nabla\cdot(\rho aa^{\ts}\nabla\Phi)\Big)dx. \hspace{0.8cm}(c)
\end{split}
\end{equation}
We next rewrite \eqref{Hessa} into iterative Gamma calculus. We first show 
\begin{equation*}
\begin{split}
(a)+(c)=&\int  \Big(\delta^2\mathrm{D}(\rho)\nabla\cdot(\rho aa^{\ts}\nabla\Phi)- \Gamma_{1}(\Phi, \delta\mathrm{D}(\rho))\Big)\nabla\cdot(\rho aa^{\ts}\nabla\Phi)\rho dx\\
=&\int  \Big(\frac{1}{\rho}\nabla\cdot(\rho aa^{\ts}\nabla\Phi)- (aa^{\ts}\nabla\log\frac{\rho}{\Vol }, \nabla\Phi)\Big)\nabla\cdot(\rho aa^{\ts}\nabla\Phi)\rho dx\\
=&\int  \Big(\frac{1}{\rho}(\nabla\rho, aa^{\ts}\nabla\Phi)+ \nabla\cdot(aa^{\ts}\nabla\Phi)- (aa^{\ts}\nabla\log\frac{\rho}{\Vol }, \nabla\Phi)\Big) \nabla\cdot(\rho aa^{\ts}\nabla\Phi) \rho dx\\
=&\int  \Big((\nabla\log\rho, aa^{\ts}\nabla\Phi)+ \nabla\cdot(aa^{\ts}\nabla\Phi)\\
&\quad-(\nabla\log\rho, aa^{\ts}\nabla\Phi)+(aa^\ts\nabla\log\pi, \nabla\Phi)\Big) \nabla\cdot(\rho aa^{\ts}\nabla\Phi)\rho dx\\
=&\int  \Big((\nabla\cdot(aa^{\ts}\nabla\Phi)+(aa^\ts\nabla\log\pi, \nabla\Phi)\Big)\nabla\cdot(\rho aa^{\ts}\nabla\Phi)\rho dx\\
=&\int  L\Phi \nabla\cdot(\rho aa^{\ts}\nabla\Phi) dx\\
=&-\int  \Gamma_{1}(L\Phi, \Phi)\rho dx,
\end{split}
\end{equation*}
where the fourth equality uses the fact that $\frac{\nabla\rho}{\rho}=\nabla\log\rho$, while the last equality follows the integration by parts.

We second show 
\begin{equation*}
\begin{split}
(b)=&-\frac{1}{2}\int  (a^{\ts}\nabla\delta\mathrm{D}(\rho), a^{\ts}\nabla \Gamma_{1}(\Phi, \Phi))\rho dx\\
=&\frac{1}{2}\int  \Gamma_{1}(\Phi, \Phi)) \nabla\cdot(\rho aa^{\ts}\nabla\delta\mathrm{D}(\rho)) dx\\
=&\frac{1}{2}\int  \Gamma_{1}(\Phi, \Phi)) L^*\rho dx\\
=&\frac{1}{2}\int  L\Gamma_{1}(\Phi, \Phi)) \rho dx,
\end{split}
\end{equation*}
where the second equality applies the fact that $L^*\rho=\nabla\cdot(\rho aa^{\ts}\nabla\delta\mathrm{D})$, while the last inequality uses the dual relation between Kolomogrov operators $L$ and $L^*$ in $L^2(\rho)$, i.e.  
\begin{equation*}
\int  f(x) L^*\rho(t,x)dx= \int Lf(x)\rho(t,x) dx, \qquad \textrm{for any $f\in C_c^{\infty}(\hM^{n+m})$}.
\end{equation*}
Combining the equality of $(a), (b), (c)$, we prove the result. 
\end{proof}

{\begin{rem}
We remark that the above formulations in term of $aa^{\ts}$ hold for both Riemannian and sub-Riemannian density manifolds. Here the major difference is that whether matrix function $a$ is full rank or degenerate. In this sense, all formulas derived in this subsection recover the classical Bakry-Emery calculus. However, the classical Hessian operator of entropy is not enough to study the convergence behavior of degenerate diffusion processes. Shortly, we use a modified Lyapunov functional and derive a tensor for the gradient flow in SDM. It provides the convergence rate of the degenerate diffusion process. 
\end{rem}}

\subsection{Gamma z calculus via second order calculus of relative entropy in SDM}\label{section3.3}
In this subsection, we introduce the motivation of our new Gamma z calculus from SDM viewpoint. Consider the SDM gradient flow \eqref{newFPE}
$$\partial_t\rho_t=\Delta_{\rho_t}^a\delta \mathrm{D}(\rho_t).$$ 
When $a$ is degenerate matrix, the classical relative Fisher information $\mathrm{I}_a$ may not be the Lyapunov functional. In other words, along the gradient flow, it is possible that  
$\frac{d}{dt}\mathrm{I}_a(\rho_t)\geq 0$. 

To handle this issue, a new Lyapunov function is considered. It is to add a new direction $z$ into the relative Fisher information functional. Denote $\Delta_\rho^z=\nabla\cdot(\rho zz^{\ts}\nabla)$ and $\mathrm{I}_{z}(\rho)=\int \Big(\delta\mathrm{D}, (-\Delta_\rho^z)\delta\mathrm{D}\Big) dx$. Construct 
 \begin{equation*}
\mathrm{I}_{a,z}(\rho):=\mathrm{I}_a(\rho)+\mathrm{I}_z(\rho)=\int \Big(\delta\mathrm{D}, (-\Delta_\rho^a-\Delta_\rho^z)\delta\mathrm{D}\Big) dx.
\end{equation*}
We next prove the following proposition.
\begin{prop}
\begin{equation*}
\frac{d}{dt}\mathrm{I}_{a,z}(\rho_t)=- 2\int \Big(\Gamma_2(\delta \mathrm{D}, \delta\mathrm{D})+ \tilde \Gamma_{2}^z(\delta \mathrm{D}, \delta\mathrm{D}))\Big)\rho_t dx, 
\end{equation*}
where 
\begin{equation}
\begin{split}\label{new gamma tilde}
\tilde\Gamma_{2}^z(\Phi, \Phi):=&\frac{1}{2}L(\Gamma_1^z(\Phi, \Phi))-\Gamma_{1}(L_z\Phi, \Phi),
\end{split}
\end{equation}
with the notation $\Delta_z=\nabla\cdot(zz^{\ts}\nabla)$ and $L_z=\nabla\cdot(zz^{\ts}\nabla)+(\nabla\log\Vol, zz^{\ts} \nabla)$.
\end{prop}
\begin{proof}
For simplicity of notation, we denote $\rho=\rho_t$. Notice the fact that 
\begin{equation*}
\frac{d}{dt}\mathrm{I}_{a,z}(\rho)=\frac{d}{dt}\mathrm{I}_a(\rho)+\frac{d}{dt}\mathrm{I}_z(\rho). 
\end{equation*}
From Proposition \ref{prop1}, we have 
\begin{equation*}
\begin{split}
\frac{d^2}{dt^2}\mathrm{I}_a(\rho)=&-2\textrm{Hess}_{g_a}\mathrm{D}(V_{\delta \mathrm{D}}, V_{\delta\mathrm{D}})\\
=&-2\int \Gamma_2(\delta\mathrm{D}, \delta\mathrm{D})\rho dx. 
\end{split}
\end{equation*}
We only need to show the following claim. 

\noindent\textbf{Claim:}
\begin{equation*}
\frac{d}{dt}\mathrm{I}_z(\rho)= -2 \int \tilde \Gamma_{2}^z(\delta \mathrm{D}, \delta\mathrm{D})\rho dx.
\end{equation*}
\begin{proof}[Proof of Claim]
The proof is similar to the ones in Proposition \ref{prop1}. We need to take care of $z$ direction.  
Notice 
\begin{equation*}
\begin{split}
\frac{d}{dt}\mathrm{I}_z(\rho)=&2 \int \delta^2\mathrm{D}\Big((-\Delta_\rho^z\delta\mathrm{D}), \partial_t\rho\Big) dx+\int (\nabla \delta\mathrm{D}, zz^{\ts}\nabla\delta\mathrm{D}) \partial_t\rho dx\\
=&2 \int \delta^2\mathrm{D} \Big((-\Delta_\rho^z\delta\mathrm{D}), \Delta_\rho^a\delta\mathrm{D}\Big) dx+\int (\nabla \delta\mathrm{D}, zz^{\ts}\nabla\delta\mathrm{D}) (\Delta_\rho^a\delta\mathrm{D}) dx\\
=&-2 \int \frac{1}{\rho} \nabla\cdot(\rho aa^{\ts}\nabla \delta\mathrm{D})\nabla\cdot(\rho zz^{\ts}\nabla\delta\mathrm{D}) dx \hspace{2cm} (I)\\
&+\int (\nabla \delta\mathrm{D}, zz^{\ts}\nabla\delta\mathrm{D}) \nabla\cdot(\rho aa^{\ts}\nabla\delta\mathrm{D}) dx \hspace{2.2cm}(II)
\end{split}
\end{equation*}
We next estimate (I), (II) separately. For (I), we notice the fact that 
\begin{equation*}
\begin{split}
\frac{1}{\rho}\nabla\cdot(\rho zz^{\ts}\nabla\delta\mathrm{D})=&(\nabla\log\rho, zz^{\ts}\nabla\delta\mathrm{D})+\nabla\cdot(zz^{\ts}\nabla\delta\mathrm{D})\\
=&(\nabla\log\frac{\rho}{\pi}, zz^{\ts}\nabla\delta\mathrm{D})+(\nabla\log\pi, zz^{\ts}\nabla\delta\mathrm{D})+\nabla\cdot(zz^{\ts}\nabla\delta\mathrm{D})\\
=&(\nabla\delta\mathrm{D}, zz^{\ts}\nabla\delta\mathrm{D})+(\nabla\log\pi, zz^{\ts}\nabla\delta\mathrm{D})+\nabla\cdot(zz^{\ts}\nabla\delta\mathrm{D}).
\end{split}
\end{equation*}
Thus 
\begin{equation*}
\begin{split}
(I)=&-2 \int \frac{1}{\rho} \nabla\cdot(\rho aa^{\ts}\nabla \delta\mathrm{D})\nabla\cdot(\rho zz^{\ts}\nabla\delta\mathrm{D}) dx\\
=&  -2\int  \nabla\cdot(\rho aa^{\ts}\nabla \delta\mathrm{D}) \Big((\nabla\delta\mathrm{D}, zz^{\ts}\nabla\delta\mathrm{D})+(\nabla\log\pi, zz^{\ts}\nabla\delta\mathrm{D})+\nabla\cdot(zz^{\ts}\nabla\delta\mathrm{D})\Big) dx\\
=&  -2\int  (\nabla\delta\mathrm{D}, zz^{\ts}\nabla\delta\mathrm{D}) L^*\rho+\nabla\cdot(\rho aa^{\ts}\nabla \delta\mathrm{D}) \Big((\nabla\log\pi, zz^{\ts}\nabla\delta\mathrm{D})+\nabla\cdot(zz^{\ts}\nabla\delta\mathrm{D})\Big) dx\\
 =&-2\int L (\nabla\delta\mathrm{D}, zz^{\ts}\nabla\delta\mathrm{D}) \rho dx \\
 &+2\int  \Big\{\nabla\Big((\nabla\log\pi, zz^{\ts}\nabla\delta\mathrm{D})+\nabla\cdot(zz^{\ts}\nabla\delta\mathrm{D})\Big), aa^{\ts}\nabla \delta\mathrm{D}\Big\}\rho dx,
\end{split}
\end{equation*}
where the last equality holds by integration by parts. 

For (II), we have 
\begin{equation*}
\begin{split}
(II)=&\int (\nabla \delta\mathrm{D}, zz^{\ts}\nabla\delta\mathrm{D}) \nabla\cdot(\rho aa^{\ts}\nabla\delta\mathrm{D}) dx\\
=&\int (\nabla \delta\mathrm{D}, zz^{\ts}\nabla\delta\mathrm{D}) L^*\rho dx\\
=& \int L (\nabla \delta\mathrm{D}, zz^{\ts}\nabla\delta\mathrm{D}) \rho dx.
\end{split}
\end{equation*}
Combining (I) and (II), we have
\begin{equation*}
\tilde\Gamma_{2}^z(\Phi, \Phi)=\frac{1}{2}L(\Gamma_1^z(\Phi, \Phi))-\Gamma_{1}(\Delta_z\Phi, \Phi)-\Gamma_{1}(\Gamma_1^z(\log \Vol, \Phi), \Phi).
\end{equation*}
Using the notation $L_z=\Delta_z+(\nabla\log\Vol, zz^{\ts}\nabla)$, we finish the proof. 
\end{proof}
\end{proof}

We next prove that $\tilde\Gamma_{2}^z$ and $\Gamma_2^{z,\Vol}$ in Definition \ref{defn:tilde gamma 2 znew} agrees each other in the weak form along the gradient flow. 
\begin{prop} 
Denote $\Phi=\delta \mathcal{D}(\rho)$, then
\begin{equation*}
\int \tilde\Gamma_{2}^z(\Phi, \Phi) \rho dx= \int \Gamma_2^{z,\Vol}(\Phi, \Phi)\rho dx.
\end{equation*}
\end{prop}
\begin{proof}
To prove the proposition, we rewrite $\tilde\Gamma_{2}^z$ as follows. 
\begin{equation*}
\begin{split}
 \tilde\Gamma_{2}^z(\Phi, \Phi)=&\frac{1}{2}L(\Gamma_1^z(\Phi, \Phi))-\Gamma_{1}(L_z\Phi, \Phi)\\
 =&\frac{1}{2}L(\Gamma_1^z(\Phi, \Phi))-\Gamma_1^z(L\Phi,\Phi)\\
 &+\Gamma_1^z(L\Phi,\Phi)-\Gamma_{1}(L_z\Phi, \Phi). 
\end{split}
\end{equation*}
Here we need to prove the following equality. 

\noindent\textbf{Claim:}
\begin{equation*}
\begin{split}
&\int \Big\{\Gamma_1^z(L\Phi,\Phi)-\Gamma_{1}(L_z\Phi, \Phi)\Big\}\rho dx\\=&\int  \rho\Big\{\frac{1}{\Vol}\nabla\cdot\Big(zz^{\ts}\Vol\big(\nabla\Phi, \nabla(aa^{\ts})\nabla\Phi\big)\Big)-\frac{1}{\Vol}\nabla\cdot\Big(aa^{\ts}\Vol\big(\nabla\Phi, \nabla(zz^{\ts})\nabla\Phi\big)\Big)\Big\} dx.
\end{split}
\end{equation*}
\begin{proof}[Proof of Claim]
For simplicity of notation, let 
\begin{equation*}
L^*\rho=\nabla\cdot(aa^{\ts}{\Vol}\nabla\frac{\rho}{\Vol})=\nabla\cdot(\rho aa^{\ts}\nabla\log\frac{\rho}{\Vol})    
\end{equation*}
and 
\begin{equation*}
L^*_{z}\rho=\nabla\cdot(zz^{\ts}{\Vol}\nabla\frac{\rho}{\Vol})=\nabla\cdot(\rho zz^{\ts}\nabla\log\frac{\rho}{\Vol}).    
\end{equation*}
The following property is also used in the proof. For any smooth test function $f$ and $\Phi=\log\frac{\rho}{\Vol}$, then
\begin{equation*}
    \int L^*_z\rho f dx=-\int \Gamma_1^z(f,\Phi)\rho dx, \quad     \int L^*\rho f dx=-\int \Gamma_{1}(f,\Phi)\rho dx.
\end{equation*}
Notice $\Phi=\log\frac{\rho}{\Vol}$, then
\begin{equation*}
\begin{split}
&\int \Gamma_1^z(L\Phi, \Phi)\rho dx\\=&\int (\nabla (\nabla\cdot(aa^{\ts}\nabla\Phi)-(A,\nabla\Phi)), zz^{\ts}\nabla\Phi)\rho dx\\
=&\int (\nabla (\nabla\cdot(aa^{\ts}\nabla\Phi)), zz^{\ts}\nabla\Phi)\rho dx
-\int (\nabla (A,\nabla\Phi), zz^{\ts}\nabla\Phi)\rho dx.\\
& \hspace{2cm}(a1) \hspace{4cm} (a2)
\end{split}
\end{equation*}
Here 
\begin{equation*}
  \begin{split}
 (a1)=&\int \Big(\nabla (\nabla\cdot(aa^{\ts}\nabla\Phi)), zz^{\ts}\nabla\Phi\Big)\rho dx  \\
=&-\int \nabla\cdot(aa^{\ts}\nabla\Phi) \nabla\cdot(\rho zz^{\ts}\nabla\Phi) dx\\
=&-\int \nabla\cdot(aa^{\ts}\nabla\log\frac{\rho}{\Vol}) \nabla\cdot(\rho zz^{\ts}\nabla\log\frac{\rho}{\Vol}) dx\\
=&-\int \nabla\cdot(\frac{1}{\rho}aa^{\ts}\Vol\nabla\frac{\rho}{\Vol}) \nabla\cdot(\rho zz^{\ts}\nabla\log\frac{\rho}{\Vol}) dx
\end{split}
\end{equation*}

\begin{equation*}
\begin{split}
=&-\int \Big\{(\nabla\frac{1}{\rho}, aa^{\ts}\Vol\nabla\frac{\rho}{\Vol})+\frac{1}{\rho}\nabla\cdot(aa^{\ts}\Vol\nabla\frac{\rho}{\Vol})\Big\} \nabla\cdot(\rho zz^{\ts}\nabla\log\frac{\rho}{\Vol}) dx\\
=&-\int (\nabla\frac{1}{\rho}, aa^{\ts}\Vol\nabla\frac{\rho}{\Vol})L_z^*\rho dx-\int \frac{1}{\rho}L^*\rho L_z^*\rho dx\\ 
=&\int \frac{1}{\rho^2}(\nabla\rho, aa^{\ts}\Vol\nabla\frac{\rho}{\Vol})L_z^*\rho dx-\int \frac{1}{\rho}L^*\rho L_z^*\rho dx\\
=&\int (\nabla\log\rho, aa^{\ts}\nabla\log\frac{\rho}{\Vol})L_z^*\rho dx-\int \frac{1}{\rho}L^*\rho L_z^*\rho dx\\ 
=&\int (\nabla\log\frac{\rho}{\Vol}, aa^{\ts}\nabla\log\frac{\rho}{\Vol})L_z^*\rho dx\\
&+\int (\nabla\log{\Vol}, aa^{\ts}\nabla\log\frac{\rho}{\Vol})L_z^*\rho dx
-\int \frac{1}{\rho}L^*\rho L_z^*\rho dx\\
=&-\int \Big(\nabla\big(\nabla\log\frac{\rho}{\Vol}, aa^{\ts}\nabla\log\frac{\rho}{\Vol}\big), zz^{\ts}\nabla\log\frac{\rho}{\Vol}\Big)\rho dx\\
&-\int \Gamma_1^z((\nabla\log{\Vol}, aa^{\ts}\nabla\log\frac{\rho}{\Vol}), \log\frac{\rho}{\Vol})\rho dx-\int \frac{1}{\rho}L^*\rho L_z^*\rho dx\\
=& -\int \Big(\big(\nabla\log\frac{\rho}{\Vol}, \nabla(aa^{\ts})\nabla\log\frac{\rho}{\Vol}\big), zz^{\ts}\nabla\frac{\rho}{\Vol}\Big)\Vol dx\\
&-\int 2\nabla^2\log\frac{\rho}{\Vol}\Big(aa^{\ts}\nabla\log\frac{\rho}{\Vol}, zz^{\ts}\nabla\log\frac{\rho}{\Vol}\Big)\rho dx\\
&-\int \Gamma_1^z((\nabla\log{\Vol}, aa^{\ts}\nabla\log\frac{\rho}{\Vol}), \log\frac{\rho}{\Vol})\rho dx-\int \frac{1}{\rho}L^*\rho L_z^*\rho dx\\
=& \int \nabla\cdot\Big(zz^{\ts}\Vol\Big(\big(\nabla\log\frac{\rho}{\Vol}, \nabla(aa^{\ts})\nabla\log\frac{\rho}{\Vol}\big)\Big)\frac{1}{\Vol}\rho dx\\
&-\int 2\nabla^2\log\frac{\rho}{\Vol}\Big(aa^{\ts}\nabla\log\frac{\rho}{\Vol}, zz^{\ts}\nabla\log\frac{\rho}{\Vol}\Big)\rho dx\\
&-\int \Gamma_1^z((\nabla\log{\Vol}, aa^{\ts}\nabla\log\frac{\rho}{\Vol}), \log\frac{\rho}{\Vol})\rho dx-\int \frac{1}{\rho}L^*\rho L_z^*\rho dx.
\end{split}
\end{equation*}
Notice the fact 
\begin{equation*}
\begin{split}
(a2)=&-\int (\nabla (A,\nabla\Phi), zz^{\ts}\nabla\Phi)\rho dx\\
=&\int \Big(\nabla(\nabla\log\Vol, aa^{\ts}\nabla\Phi), zz^{\ts}\nabla \Phi\Big)\rho dx\\
=&\int \Gamma_{1}(\Gamma_1^z(\Phi, \Phi),\Phi)\rho dx.
\end{split}
\end{equation*}
Hence 
\begin{equation*}
\begin{split}
\int \Gamma_1^z(L\Phi,\Phi)\rho dx=&(a1)+(a2)\\
=& \int \nabla\cdot\Big(zz^{\ts}\Vol\Big(\big(\nabla\log\frac{\rho}{\Vol}, \nabla(aa^{\ts})\nabla\log\frac{\rho}{\Vol}\big)\Big)\frac{1}{\Vol}\rho dx\\
&-\int 2\nabla^2\log\frac{\rho}{\Vol}\Big(aa^{\ts}\nabla\log\frac{\rho}{\Vol}, zz^{\ts}\nabla\log\frac{\rho}{\Vol}\Big)\rho dx\\
&-\int \frac{1}{\rho}L^*\rho L_z^*\rho dx.
\end{split}
\end{equation*}
Similarly, by switching $a$ and $z$, we have
\begin{equation*}
\begin{split}
\int \Gamma_1(L_z\Phi,\Phi)\rho dx
=& \int \nabla\cdot\Big(aa^{\ts}\Vol\Big(\big(\nabla\log\frac{\rho}{\Vol}, \nabla(zz^{\ts})\nabla\log\frac{\rho}{\Vol}\big)\Big)\frac{1}{\Vol}\rho dx\\
&-\int 2\nabla^2\log\frac{\rho}{\Vol}\Big(aa^{\ts}\nabla\log\frac{\rho}{\Vol}, zz^{\ts}\nabla\log\frac{\rho}{\Vol}\Big)\rho dx\\
&-\int \frac{1}{\rho}L^*\rho L_z^*\rho dx.
\end{split}
\end{equation*}
Combining the above derivation, we finish the proof. 
\end{proof}
\end{proof}
\begin{rem}\label{iterative Gamma conditon}
From the proof, we can show the following identity: Denote $\Phi=\delta\mathrm{D}$, then
\begin{equation*}
\begin{split}
&\int \Big(\Gamma_1^z(L\Phi,\Phi)-\Gamma_{1}(L_z\Phi, \Phi)\Big)\rho dx\\
=&\int \Big(\Gamma_1(\Gamma_1^z(\Phi,\Phi),\Phi)-\Gamma_{1}^z(\Gamma_1(\Phi,\Phi),\Phi)\Big)\rho dx.
\end{split}
\end{equation*}
So it is clear that if the commutative assumption $\Gamma_{1}(\Gamma_{1}^z(\Phi,\Phi),\Phi)=\Gamma_1^z(\Gamma(\Phi,\Phi),\Phi )$ holds, the above quantity equals zero. In this case, 
\begin{equation*}
\int \Gamma_2^{z,\Vol}(\Phi, \Phi)\rho dx= \int \Gamma_2^z(\Phi, \Phi)\rho dx.
\end{equation*}
This means that under the commutative assumption, the generalized Gamma z calculus agrees with the classical one \cite{BaudoinGarofalo09} in the weak sense.  
\end{rem}

With the generalized Gamma z calculus, we are ready to prove the convergence properties and functional inequalities for degenerate drift diffusion processes. 
\begin{prop}\label{prop:z log-soblov}
Suppose $\Gamma_2+\Gamma_2^{z,\Vol}\succeq \kappa (\Gamma_1+\Gamma_1^z)$ with $\kappa>0$. Denote $\rho_t$ as the solution of sub-Riemannian gradient flow \eqref{newFPE}, then
\begin{equation*}
\frac{d}{dt}\Big(I_a(\rho_t)+I_z(\rho_t)\Big)\leq -2\kappa \Big(I_a(\rho_t)+I_z(\rho_t)\Big).
\end{equation*}
In addition, the z-log-Sobolev inequalities holds:
\begin{equation*}
\int_{\hM^{n+m}}\rho\log\frac{\rho}{\Vol} dx\leq \frac{1}{2\kappa}
\mathrm{I}_{a,z}(\rho),
\end{equation*}
for any smooth density function $\rho$. 

Finally, 
    \begin{equation*}
 \int_{{\hM}^{n+m}} |\rho(t,x)-\pi(x)| dx\leq \sqrt{2  \mathrm{D}(\rho_0) }e^{-\kappa t}.
\end{equation*}
\end{prop}

\begin{proof}
Here the proof is very similar to the one in previous section. Again, consider the sub-Riemannian gradient flow in SDM. 
 $$\partial_t\rho_t=-\textrm{grad}_{W_a}\mathrm{D}(\rho_t).$$ 
We know that the log-Sobolev inequality relates to the ratio of $\frac{d}{dt}\mathrm{D}(\rho_t)$ and $\frac{d^2}{dt^2}\mathrm{D}(\rho_t)$. If we can not estimate a ratio $\kappa>0$, such that 
\begin{equation*}
\frac{d}{dt}\mathrm{I}_a(\rho_t)\leq -2\kappa \mathrm{I}_a(\rho_t),
\end{equation*}
We construct the other Lyapunov function 
\begin{equation*}
\mathrm{I}_{a,z}(\rho)=\mathrm{I}_a(\rho)+\mathrm{I}_z(\rho).    
\end{equation*}
Thus along the SDM gradient flow \eqref{newFPE}, we have
\begin{equation*}
\frac{d}{dt}\mathrm{I}_{a,z}(\rho_t)=-2\int \Big(\Gamma_2(\delta\mathrm{D}, \delta\mathrm{D}) +\Gamma_{2}^{z,\Vol}(\delta\mathrm{D},\delta\mathrm{D}) \Big)\rho_t dx.   \end{equation*}
If $\Gamma_2+\Gamma_2^{z,\Vol}\succeq \kappa(\Gamma_1+\Gamma_1^z)$, then 
\begin{equation*}\label{iq}
    \frac{d}{dt}\mathrm{I}_{a,z}(\rho_t)\leq -2\kappa \mathrm{I}_{a,z}(\rho_t). 
\end{equation*}
The convergence result follows directly from the Gronwall's equality. 

We next prove the $z$-log-Sobolev inequality. Since
\begin{equation*}
-\frac{d}{dt}\mathrm{D}(\rho_t)=\mathrm{I}_a(\rho_t)\leq \mathrm{I}_{a,z}(\rho_t),
\end{equation*}
then \eqref{iq} implies the fact that, denote $\rho_0=\rho$, then
\begin{equation*}
\begin{split}
    -\mathrm{I}_{a,z}(\rho)=&\int_0^{\infty}\frac{d}{dt}\mathrm{I}_{a,z}(\rho_t)dt\\
    \leq & -2\kappa \int_0^{\infty}\mathrm{I}_{a,z}(\rho_t)dt=-2\kappa \int_0^\infty\Big(\mathrm{I}_a(\rho_t)+\mathrm{I}_z(\rho_t)\Big)dt\\
    \leq &-2\kappa \int_0^\infty\mathrm{I}_a(\rho_t)dt\\
     =&-2\kappa  \int_0^{\infty}(-\frac{d}{dt}\mathrm{D}(\rho_t))dt\\
    =&-2\kappa\mathrm{D}(\rho). 
    \end{split}
\end{equation*}
Thus $\mathrm{I}_{a,z}(\rho)\geq 2\kappa \mathrm{D}(\rho)$. Hence we prove all results by the fact that $\mathfrak R\succeq \kappa(\Gamma_1+\Gamma_1^z)$ implies $\Gamma_2+\Gamma_2^{z,\Vol}\succeq \kappa(\Gamma_1+\Gamma_1^z)$. In a word, the generalized Gamma z calculus implies the z-log-Sobolev equality (zLSI):
$$\mathfrak R\succeq \kappa(\Gamma_1+\Gamma_1^z)  \Rightarrow \frac{d}{dt}\mathrm{I}_{a,z}(\rho_t)\leq -2\kappa \mathrm{I}_{a,z}(\rho_t)\Rightarrow \textrm{zLSI}.$$

We last prove the exponential convergence in $L^1$ distance. Notice that 
\begin{equation*}
\mathrm{D}_{\mathrm{KL}}(\rho_t\|\pi)\leq \frac{1}{2\lambda}\mathcal{I}_{a,z}(\rho_t\|\pi)\leq \frac{1}{2\lambda}e^{-2\lambda t}\mathcal{I}_{a,z}(\rho_0\|\pi).     
\end{equation*}
We apply an inequality between KL divergence and $L_1$ distance. In other words, 
\begin{equation*}
    \int_{\mathbb{R}^{n+m}} |\rho(t,x)-\pi(x)| dx\leq \sqrt{2\mathrm{D}_{\mathrm{KL}}(\rho\|\pi)}.
\end{equation*}
This finishes the proof. 
\end{proof}

\begin{rem}
It is worth mentioning with our derivation of Gamma z calculus is not a direct Hessian operator of entropy in SDM. In fact, it combines both the second order calculus in SDM and the property of $L^2$ Hessian operator of entropy. See similar relations in mean-field Bakry-{\'E}mery calculus \cite{Li2019_diffusion}. 
\end{rem}

\section{Generalized Gamma $z$ calculus}\label{sec: generalized gamma z}
In this section, we introduce the generalized Gamma $z$ calculus. For any smooth functions $f,g: \hR^{n+m}\rightarrow \hR$, the diffusion operator associated with SDE \eqref{SDE framework with drift} is denoted as
\beaa
Lf=\Delta_a f -A\nabla f+b\nabla f,
\eeaa
where we denote $A= a\otimes \nabla a$ and
\beaa
\Delta_af=\nabla \cdot (aa^{\ts}\nabla f).
\eeaa
When $b=0$, we denote the diffusion operator as
\beaa
\widetilde Lf&=&\Delta_a f -a\otimes \nabla a\nabla f.
\eeaa
We first define the Carr\'e de Champ operator $\Gamma_1$ associated with the above second order diffusion operators. It is easy to check that $\Delta_a$, $\widetilde L$, $L$ share the same $\Gamma_1$,
\bea\label{gamma 1 L}
\Gamma_{1}(f,g)=\la a^{\ts}\nabla f,a^{\ts}\nabla f \ra_{\hR^n}.
\eea
Similarly, we introduce the $\Gamma_1^z$ operator in the direction of  $z=(z_1,\cdots,z_m)$ below,
\bea \label{gamma 1 z}
\Gamma_1^z=\la z^{\ts}\nabla f , z^{\ts}\nabla f\ra_{\hR^m},
\eea
Next, we define the iterative $\Gamma_2$ and $\Gamma_2^z$ for operator $L$ ($\widetilde L$ resp.) below,
\bea\label{gamma 2}
\Gamma_{2, L}(f,f)=\frac{1}{2} L\Gamma_1^z(f,f)-\Gamma_1^z( Lf,f).
\eea
	\bea\label{gamma 2 z}
\Gamma_{2, L}^z(f,f)=\frac{1}{2} L\Gamma_1^z(f,f)-\Gamma_1^z( Lf,f).
\eea
\begin{defn}\label{def: generalized gamma 2 z with drift} 
	We define generalized Gamma $z$ for operator $L$ below,
	\bea
	\label{generalized gamma 2 z with drift}
	{\Gamma}_{2, L}^{z,\pi}(f,f)&=&\Gamma_{2, L}^z(f,f)+\div^{\pi}_z(\Gamma_{\nabla(aa^{\ts})}f,f )-\div^{\pi}_a(\Gamma_{\nabla(zz^{\ts})}f,f),
	\eea
For matrix $a\in\hR^{n\times(n+m)}$ and $z\in\hR^{m\times(n+m)}$, we denote the divergence operator as
	\beaa
\div^{\pi}_z(\Gamma_{\nabla(aa^{\ts})}f,f )&=&\frac{\nabla\cdot(zz^{\ts}\pi \Gamma_{\nabla(aa^{\ts})}(f,f)) }{\pi}, \\
\div^{\pi}_a(\Gamma_{\nabla(zz^{\ts})}f,f )&=&\frac{\nabla\cdot(aa^{\ts}\pi \Gamma_{\nabla(zz^{\ts}) }(f,f)) }{\pi},
	\eeaa
and
	\beaa
	 \Gamma_{\nabla(aa^{\ts)}}(f,f)=\la \nabla f,\nabla(aa^{\ts})\nabla f\ra,\quad \text{and}
	\quad \Gamma_{\nabla(zz^{\ts)}}(f,f)=\la \nabla f,\nabla(zz^{\ts})\nabla f\ra.
	\eeaa
Here we denote $\pi$ as the invariant distribution associated with the operator $L$.
\end{defn}
\begin{rem}
	In particular, we have the following local coordinates representation. 
\bea\label{gamma vector notation a z}
\la \nabla f,\nabla(aa^{\ts})\nabla f\ra&=&\la \nabla f,\frac{\pa }{\pa x_{\hat k}}(aa^{\ts})\nabla f\ra_{\hat k=1}^{n+m} = 2\la a^{\ts}\nabla f,\frac{\pa}{\pa x_{\hat k}}a^{\ts}\nabla f\ra_{\hat k=1}^{n+m}\nonumber \\
&=&\left(2\sum_{i=1}^n\sum_{\hat i,i'=1}^{n+m}\frac{\pa}{\pa x_{\hat k}}a^{\ts}_{i\hat i}\frac{\pa f}{\pa x_{\hat i}}a^{\ts}_{ii'}\frac{\pa f}{\pa x_{i'}}\right)_{\hat k=1}^{n+m},\nonumber \\
\la\nabla f,\nabla(zz^{\ts})\nabla f\ra&=&\left(2\sum_{j=1}^n\sum_{\hat j,j'=1}^{n+m}\frac{\pa}{\pa x_{\hat k}}z^{\ts}_{j\hat j}\frac{\pa f}{\pa x_{\hat j}}z^{\ts}_{jj'}\frac{\pa f}{\pa x_{j'}}\right)_{\hat k=1}^{n+m}.
\eea
\end{rem}
We first present the following key lemmas.
\begin{lem} \label{div z- div a}
\beaa
\div^{\pi}_z(\Gamma_{\nabla(aa^{\ts})}f,f )-\div^{\pi}_a(\Gamma_{\nabla(zz^{\ts})}f,f )&=&\mathfrak{R}^{\pi}(f,f)+2G^{\ts}X,
\eeaa
where $X$, $G$ are defined in Notation \ref{notation} and $\mathfrak R^{\pi}$ is defined in Definition \ref{def: Ricci curvature}. 
\end{lem}
\begin{lem} \label{gamma 2 estimates for Lp}
	\beaa
	&&\Gamma_{2,\widetilde L}(f,f)=X^{\ts}Q^{\ts}QX+2D^{\ts}QX+2C^{\ts}X+D^{\ts}D+\mathfrak{R}_a(\nabla f,\nabla f),
	\eeaa
where $Q,X,C,D$ are introduced in Notation \ref{notation} and $\mathfrak R_a$ is defined in Definition \ref{def: Ricci curvature}.
\end{lem}
\begin{lem}\label{thm: gamma z}
	\beaa
	\Gamma_{2,\widetilde L}^z(f,f)=X^{\ts}P^{\ts}PX+2E^{\ts}PX+2F^{\ts}X+E^{\ts}E+\mathfrak{R}_z(\nabla f,\nabla f).
	\eeaa
where $P,X,F,E$ are introduced in Notation \ref{notation} and $\mathfrak R_z$ is defined in Definition \ref{def: Ricci curvature}.
\end{lem}

We then have the following main Theorem. In order to distinguish the operator $L$ and $\widetilde L$, we rewrite Theorem \ref{thm1} as below and 
 { with some abuse of notation, we denote $\Gamma_{2}(f,f)=\Gamma_{2, L}(f,f)$ and ${\Gamma}_{2}^{z,\pi}(f,f)={\Gamma}_{2, L}^{z,\pi}(f,f)$.}
\begin{thm}\label{thm: generalized gamma 2 z with drift}[z-Bochner's formula] For smooth function $f:\hR^{n+m}\rightarrow\hR$, assume that Assumption \ref{assumption:main result} holds, then
	\beaa\Gamma_{2, L}(f,f)+{\Gamma}_{2, L}^{z,\pi}(f,f)&=&\|\mathfrak{Hess}_{a,z}f\|^2
+\mathfrak{R}(\nabla f,\nabla f),
\eeaa
where
\beaa 
\|\mathfrak{Hess}_{a,z}f\|^2&=&[X+\Lambda_1]^{\ts}Q^{\ts}Q[X+\Lambda_1]+[X+\Lambda_2]^{\ts}P^{\ts}P[X+\Lambda_2]\\
\mathfrak R(\nabla f,\nabla f)&=&-\Lambda_1^{\ts}Q^{\ts}Q\Lambda_1-\Lambda_2^{\ts}P^{\ts}P\Lambda_2+D^{\ts}D+E^{\ts}E\\
&&+\mathfrak R_{ab}(\nabla f,\nabla f)+\mathfrak R_{zb}(\nabla f,\nabla f)+\mathfrak R_{\pi}(\nabla f,\nabla f).
\eeaa 
All the terms are defined in Notation \ref{notation} and Definition \ref{def: Ricci curvature}.
\end{thm}
\begin{proof}
By Definition \ref{def: generalized gamma 2 z with drift} and formula \eqref{gamma 2}, \eqref{gamma 2 z}, we have 
\beaa
&&{\Gamma}_{2, L}(f,f)+\tilde{\Gamma}_{2, L}^{z,\pi}(f,f)\\
&=&{\Gamma}_{2, L}(f,f)+\Gamma_{2, L}^z(f,f)+\div^{\pi}_z(\Gamma_{\nabla(aa^{\ts})}f,f )-\div^{\pi}_a(\Gamma_{\nabla(zz^{\ts})}f,f ).
\eeaa
We compute the above terms explicitly in the following four steps.

\noindent \textbf{Step 1:}
\beaa
\Gamma_{2, L}(f,f)&=&\frac{1}{2}( L\Gamma_{1, L}(f,f)-2\Gamma_{1, L}( Lf,f))\\
	&=&\frac{1}{2} \Delta_a\Gamma_{1}(f,f)-\frac{1}{2}A\nabla \Gamma_{1}(f,f)+\frac{1}{2}b\nabla \Gamma_{1}(f,f)\\
	&&-\Gamma_{1}((\Delta_a -A\nabla+b\nabla)f,f)\\
	&=&\Gamma_{2,\widetilde L}(f,f)+[\frac{1}{2}b\nabla \Gamma_{1}(f,f)-\Gamma_{1}(b\nabla f,f)].
	\eeaa
The term $\Gamma_{2,\widetilde L}(f,f)$ follows from Lemma \ref{gamma 2 estimates for Lp}. We are left for the other two terms,
\beaa
\frac{1}{2}b\nabla \Gamma_{1}(f,f)&=&\frac{1}{2}\sum_{\hat k=1}^{n+m}b_{\hat k}\frac{\pa}{\pa x_{\hat k}}(\la a^{\ts}\nabla f,a^{\ts}\nabla f\ra_{\hR^n})\\
&=&\sum_{\hat k,\hat i=1}^{n+m}\sum_{i=1}^n(b_{\hat k}\frac{\pa a^{\ts}_{i\hat i}}{\pa x_{\hat k}}\frac{\pa f}{\pa x_{\hat i}}+b_{\hat k}a^{\ts}_{i\hat i}\frac{\pa^2f}{\pa x_{\hat k}\pa x_{\hat i}})(a^{\ts}\nabla f)_i,
\eeaa
and
\beaa
-\Gamma_{1}(b\nabla f,f)&=&-\la a^{\ts}\nabla(b\nabla f),a^{\ts}\nabla f\ra _{\hR^n}\\
&=&-\sum_{i=1}^n\sum_{\hat k,\hat i=1}^{n+m}(a^{\ts}_{i\hat i}\frac{\pa b_{\hat k}}{\pa x_{\hat i}}\frac{\pa f}{\pa x_{\hat k}} +a^{\ts}_{i\hat i}b_{\hat k}\frac{\pa^2f}{\pa x_{\hat i}\pa x_{\hat k}}) (a^{\ts}\nabla f)_i.
\eeaa

\noindent\textbf{Step 2:}
 \beaa
\Gamma_{2, L}^z(f,f)&=&\frac{1}{2}( L\Gamma_1^z(f,f)-2\Gamma_1^z( Lf,f))\\
	&=&\frac{1}{2} \Delta_a\Gamma_1^z(f,f)-\frac{1}{2}A\nabla \Gamma_1^z(f,f)+\frac{1}{2}b\nabla \Gamma_1^z(f,f)\\
	&&-\Gamma_1^z((\Delta_a -A\nabla+b\nabla)f,f)\\
	&=&\Gamma_{2,\widetilde L}^z(f,f)+[\frac{1}{2}b\nabla \Gamma_1^z(f,f)-\Gamma_1^z(b\nabla f,f)].
	\eeaa
 The term $\Gamma_{2,\widetilde L}^z(f,f)$ follows from Lemma \ref{thm: gamma z}. We are left to compute the last two terms,
\beaa
\frac{1}{2}b\nabla \Gamma_1^z(f,f)&=&-\frac{1}{2}\sum_{\hat k=1}^{n+m}b_{\hat k}\frac{\pa}{\pa x_{\hat k}}(\la z^{\ts}\nabla f,z^{\ts}\nabla f\ra_{\hR^m})\\
&=&\sum_{\hat k,\hat i=1}^{n+m}\sum_{i=1}^m(b_{\hat k}\frac{\pa z^{\ts}_{i\hat i}}{\pa x_{\hat k}}\frac{\pa f}{\pa x_{\hat i}}+b_{\hat k}z^{\ts}_{i\hat i}\frac{\pa^2f}{\pa x_{\hat k}\pa x_{\hat i}})(z^{\ts}\nabla f)_i,
\eeaa
and
\beaa
-\Gamma_1^z(b\nabla f,f)&=&-\la z^{\ts}\nabla(b\nabla f),z^{\ts}\nabla f\ra _{\hR^n}\\
&=&-\sum_{i=1}^m\sum_{\hat k,\hat i=1}^{n+m}(z^{\ts}_{i\hat i}\frac{\pa b_{\hat k}}{\pa x_{\hat i}}\frac{\pa f}{\pa x_{\hat k}} +z^{\ts}_{i\hat i}b_{\hat k}\frac{\pa^2f}{\pa x_{\hat i}\pa x_{\hat k}}) (z^{\ts}\nabla f)_i.
\eeaa

\noindent\textbf{Step 3:}
Following Lemma  \ref{div z- div a}, which will be proved shortly next section, we have 
\beaa
\div^{\pi}_z(\Gamma_{\nabla(aa^{\ts})}f,f )-\div^{\pi}_a(\Gamma_{\nabla(zz^{\ts})}f,f )&=&\mathfrak{R}^{\pi}(f,f)+2G^{\ts}X,
\eeaa
where $X$, $G$ are defined in Notation \ref{notation} and $\mathfrak R^{\pi}$ is defined in Definition \ref{def: Ricci curvature}. 

\noindent\textbf{Step 4:}
Combining the above terms $\Gamma_{2,\widetilde L}(f,f)$ in Lemma \ref{gamma 2 estimates for Lp}, $\Gamma_{2,\widetilde L}^z(f,f)$ in Lemma \ref{thm: gamma z}, and $\mathfrak{R}^{\pi}(f,f)+2G^{\ts}X$, we have
\beaa
&&\Gamma_{2,\widetilde L}(f,f)+\Gamma^z_{2, \widetilde L}(f,f)+\mathfrak{R}^{\pi}(f,f)+2G^{\ts}X\\
&=&{X^{\ts}P^TPX+2E^TPX+2F^TX+E^TE}\\
	&&+X^TQ^TQX+2D^TQX+2C^TX+D^TD+2G^{\ts}X\\
	&&+\mathfrak{R}_{a}(\nabla f,\nabla f) +\mathfrak{R}_z(\nabla f,\nabla f)+\mathfrak{R}_{\pi}(\nabla f,\nabla f)\\
	&=&X^T[P^{\ts}P+Q^{\ts}Q]X+2[G^{\ts}+F^{\ts}+C^{\ts}]X+2[E^{\ts}P+D^{\ts}Q]X+D^{\ts}D+E^{\ts}E\\
	&&+\mathfrak{R}_{a}(\nabla f,\nabla f) +\mathfrak{R}_z(\nabla f,\nabla f)+\mathfrak{R}_{\pi}(\nabla f,\nabla f).
		\eeaa
Assuming that Assumption \ref{assumption:main result}  is satisfied, we get 
\beaa
&&\Gamma_{2,\widetilde L}(f,f)+\Gamma^z_{2, \widetilde L}(f,f)+\mathfrak{R}^{\pi}(f,f)+2G^{\ts}X\\
&=&[X+\Lambda_1]^{\ts}Q^{\ts}Q[X+\Lambda_1]+[X+\Lambda_2]^{\ts}P^{\ts}P[X+\Lambda_2]\\
&&+\mathfrak{R}_{a}(\nabla f,\nabla f) +\mathfrak{R}_z(\nabla f,\nabla f)+\mathfrak{R}^{\pi}(\nabla f,\nabla f)-\Lambda_1^{\ts}Q^{\ts}Q\Lambda_1-\Lambda_2^{\ts}P^{\ts}P\Lambda_2+D^{\ts}D+E^{\ts}E.
\eeaa
Adding the drift terms from \textbf{Step 1} and \textbf{Step 2}, we get $\mathfrak R_{ab}$ and $\mathfrak R_{zb}$, which finishes the proof.
	\qed 
\end{proof}

\subsection{Proof of Lemma \ref{div z- div a}}
\begin{lem} 
\bea
\div^{\pi}_z(\Gamma_{\nabla(aa^{\ts})}f,f )-\div^{\pi}_a(\Gamma_{\nabla(zz^{\ts})}f,f )&=&\mathfrak{R}^{\pi}(f,f)+2G^{\ts}X,
\eea
where $X$, $G$ are defined in Notation \ref{notation} and $\mathfrak R^{\pi}$ is defined in Definition \ref{def: Ricci curvature}. 
\end{lem}

\begin{proof}
For the first term in the above lemma, we have
\beaa
&&\div^{\pi}_z(\Gamma_{\nabla(aa^{\ts})}f,f )\\
&=&\frac{\nabla\cdot(zz^{\ts}\pi \Gamma_{\nabla(aa^{\ts)}}(f,f)) }{\pi}\\
&=&\sum_{ k'=1}^{n+m} \frac{1}{\pi}\frac{\pa }{\pa x_{k'}} \left[\sum_{k=1}^mz_{k'k}\left(\pi\sum_{\hat k=1}^{n+m} z^{\ts}_{k\hat k} \left(\Gamma_{\nabla(aa^{\ts)}}(f,f))\right)_{\hat k} \right) \right]\\
&=&\sum_{ k'=1}^{n+m}\sum_{k=1}^m\left[\frac{\pa }{\pa x_{k'}} z_{k'k}\left(\sum_{\hat k=1}^{n+m} z^{\ts}_{k\hat k} \left(\Gamma_{\nabla(aa^{\ts)}}(f,f))\right)_{\hat k} \right)+z_{k'k}\frac{\pa }{\pa x_{k'}}\left(\sum_{\hat k=1}^{n+m} z^{\ts}_{k\hat k} \left(\Gamma_{\nabla(aa^{\ts)}}(f,f))\right)_{\hat k} \right) \right]  \\
&&+\sum_{ k'=1}^{n+m} \sum_{k=1}^m\frac{\pa }{\pa x_{k'}}\log\pi \left[z_{k'k}\sum_{\hat k=1}^{n+m} z^{\ts}_{k\hat k} \left(\Gamma_{\nabla(aa^{\ts)}}(f,f))\right)_{\hat k} \right]\\
&=&\sum_{ k'=1}^{n+m}\sum_{k=1}^m\left[\frac{\pa }{\pa x_{k'}} z^{\ts}_{kk'}\left(\sum_{\hat k=1}^{n+m} z^{\ts}_{k\hat k} \left(\Gamma_{\nabla(aa^{\ts)}}(f,f))\right)_{\hat k} \right)+z^{\ts}_{kk'}\frac{\pa }{\pa x_{k'}}\left(\sum_{\hat k=1}^{n+m} z^{\ts}_{k\hat k} \left(\Gamma_{\nabla(aa^{\ts)}}(f,f))\right)_{\hat k} \right) \right]  \\
&&+ \sum_{k=1}^m (z^{\ts}\nabla\log\pi)_k \left[\sum_{\hat k=1}^{n+m} z^{\ts}_{k\hat k} \left(\Gamma_{\nabla(aa^{\ts)}}(f,f))\right)_{\hat k} \right],
\eeaa
where $\Gamma_{\nabla(aa^{\ts)}}(f,f))_{\hat k}$ is defined in \eqref{gamma vector notation a z}. Plugging in \eqref{gamma vector notation a z}, we further get 
\bea\label{div a gamma z}
&&\div^{\pi}_z(\Gamma_{\nabla(aa^{\ts})}f,f )\nonumber \\
&=&\sum_{ k'=1}^{n+m}\sum_{k=1}^m\left[\frac{\pa }{\pa x_{k'}} z^{\ts}_{kk'}\left(\sum_{\hat k=1}^{n+m} z^{\ts}_{k\hat k} \left(2\sum_{i=1}^n\sum_{\hat i,i'=1}^{n+m}\frac{\pa}{\pa x_{\hat k}}a^{\ts}_{i\hat i}\frac{\pa f}{\pa x_{\hat i}}a^{\ts}_{ii'}\frac{\pa f}{\pa x_{i'}}\right) \right)\right]\nonumber\\
&&+\sum_{ k'=1}^{n+m}\sum_{k=1}^m \left[z^{\ts}_{kk'}\frac{\pa }{\pa x_{k'}}\left(\sum_{\hat k=1}^{n+m} z^{\ts}_{k\hat k} \left(2\sum_{i=1}^n\sum_{\hat i,i'=1}^{n+m}\frac{\pa}{\pa x_{\hat k}}a^{\ts}_{i\hat i}\frac{\pa f}{\pa x_{\hat i}}a^{\ts}_{ii'}\frac{\pa f}{\pa x_{i'}}\right) \right) \right] \nonumber \\
&&+ \sum_{k=1}^m (z^{\ts}\nabla\log\pi)_k \left[\sum_{\hat k=1}^{n+m} z^{\ts}_{k\hat k}\left(2\sum_{i=1}^n\sum_{\hat i,i'=1}^{n+m}\frac{\pa}{\pa x_{\hat k}}a^{\ts}_{i\hat i}\frac{\pa f}{\pa x_{\hat i}}a^{\ts}_{ii'}\frac{\pa f}{\pa x_{i'}}\right) \right]\nonumber
\eea 
\bea
&=&2\sum_{k=1}^m \sum_{i=1}^n\sum_{k',\hat k,\hat i,i'=1}^{n+m}\left[\frac{\pa }{\pa x_{k'}} z^{\ts}_{kk'} z^{\ts}_{k\hat k}\frac{\pa}{\pa x_{\hat k}}a^{\ts}_{i\hat i}\frac{\pa f}{\pa x_{\hat i}}a^{\ts}_{ii'}\frac{\pa f}{\pa x_{i'}}\right]\cdots \cS^z_1\nonumber\\
&&+2\sum_{k=1}^m \sum_{i=1}^n\sum_{k',\hat k,\hat i,i'=1}^{n+m}\left[z^{\ts}_{kk'}\frac{\pa }{\pa x_{k'}}\left( z^{\ts}_{k\hat k} \frac{\pa}{\pa x_{\hat k}}a^{\ts}_{i\hat i}\frac{\pa f}{\pa x_{\hat i}}a^{\ts}_{ii'}\frac{\pa f}{\pa x_{i'}}\right) \right]\cdots \cS^z_2\nonumber \\
&&+2\sum_{k=1}^m \sum_{i=1}^n\sum_{\hat k,\hat i,i'=1}^{n+m}(z^{\ts}\nabla\log\pi)_k \left[ z^{\ts}_{k\hat k}\frac{\pa}{\pa x_{\hat k}}a^{\ts}_{i\hat i}\frac{\pa f}{\pa x_{\hat i}}a^{\ts}_{ii'}\frac{\pa f}{\pa x_{i'}} \right]\cdots \cS^z_3 \nonumber\\
&=&\cS^z_1+\cS^z_2+\cS^z_3.
\eea
By further expanding  $\cS^z_2$, we get
\beaa
\cS^z_2&=&2\sum_{k=1}^m \sum_{i=1}^n\sum_{k',\hat k,\hat i,i'=1}^{n+m}\left[z^{\ts}_{kk'}\frac{\pa }{\pa x_{k'}}\left( z^{\ts}_{k\hat k} \frac{\pa}{\pa x_{\hat k}}a^{\ts}_{i\hat i}\frac{\pa f}{\pa x_{\hat i}}a^{\ts}_{ii'}\frac{\pa f}{\pa x_{i'}}\right) \right] \\
&=&2\sum_{k=1}^m \sum_{i=1}^n\sum_{k',\hat k,\hat i,i'=1}^{n+m}\left[z^{\ts}_{kk'}\frac{\pa }{\pa x_{k'}} z^{\ts}_{k\hat k} \frac{\pa}{\pa x_{\hat k}}a^{\ts}_{i\hat i}\frac{\pa f}{\pa x_{\hat i}}a^{\ts}_{ii'}\frac{\pa f}{\pa x_{i'}}+z^{\ts}_{kk'} z^{\ts}_{k\hat k} \frac{\pa^2}{\pa x_{k'}\pa x_{\hat k}}a^{\ts}_{i\hat i}\frac{\pa f}{\pa x_{\hat i}}a^{\ts}_{ii'}\frac{\pa f}{\pa x_{i'}} \right.\\
&&\quad\quad\quad\quad\quad\quad\quad\quad +z^{\ts}_{kk'} z^{\ts}_{k\hat k} \frac{\pa}{\pa x_{\hat k}}a^{\ts}_{i\hat i}\frac{\pa^2 f}{\pa x_{k'}\pa x_{\hat i}}a^{\ts}_{ii'}\frac{\pa f}{\pa x_{i'}}+z^{\ts}_{kk'} z^{\ts}_{k\hat k} \frac{\pa}{\pa x_{\hat k}}a^{\ts}_{i\hat i}\frac{\pa f}{\pa x_{\hat i}}a^{\ts}_{ii'}\frac{\pa^2 f}{\pa x_{k'}\pa x_{i'}}\\
&&\left.\quad\quad\quad\quad\quad\quad\quad\quad+z^{\ts}_{kk'} z^{\ts}_{k\hat k} \frac{\pa}{\pa x_{\hat k}}a^{\ts}_{i\hat i}\frac{\pa f}{\pa x_{\hat i}}\frac{\pa }{\pa x_{k'}}a^{\ts}_{ii'}\frac{\pa f}{\pa x_{i'}} \right].
\eeaa
 Similarly, we get 
\bea\label{div z gamma a}
&&\div^{\pi}_a(\Gamma_{\nabla(zz^{\ts})}f,f )\nonumber\\
&=&\sum_{ l'=1}^{n+m}\sum_{l=1}^n\left[\frac{\pa }{\pa x_{l'}} a^{\ts}_{ll'}\left(\sum_{\hat l=1}^{n+m} a^{\ts}_{l\hat l} \left(\Gamma_{\nabla(zz^{\ts)}}(f,f))\right)_{\hat l} \right)+a^{\ts}_{ll'}\frac{\pa }{\pa x_{l'}}\left(\sum_{\hat l=1}^{n+m} a^{\ts}_{l\hat l} \left(\Gamma_{\nabla(zz^{\ts)}}(f,f))\right)_{\hat l} \right) \right]  \nonumber\\
&&+ \sum_{l=1}^n(a^{\ts}\nabla\log\pi)_l \left[\sum_{\hat l=1}^{n+m}\left( a^{\ts}_{l\hat l} \Gamma_{\nabla(zz^{\ts)}}(f,f))\right)_{\hat l} \right]\nonumber\\
&=&2\sum_{j=1}^m\sum_{l=1}^n\sum_{ l',\hat l,\hat j,j'=1}^{n+m}\left[\frac{\pa }{\pa x_{l'}} a^{\ts}_{ll'} a^{\ts}_{l\hat l} \frac{\pa}{\pa x_{\hat l}}z^{\ts}_{j\hat j}\frac{\pa f}{\pa x_{\hat j}}z^{\ts}_{jj'}\frac{\pa f}{\pa x_{j'}} \right]\cdots \cS^a_1 \nonumber\\
&&+ 2\sum_{j=1}^m\sum_{l=1}^n\sum_{ l',\hat l,\hat j,j'=1}^{n+m}\left[a^{\ts}_{ll'}\frac{\pa }{\pa x_{l'}}\left(a^{\ts}_{l\hat l} \frac{\pa}{\pa x_{\hat l}}z^{\ts}_{j\hat j}\frac{\pa f}{\pa x_{\hat j}}z^{\ts}_{jj'}\frac{\pa f}{\pa x_{j'}} \right) \right] \cdots \cS^a_2\nonumber \\
&&+2\sum_{j=1}^m\sum_{l=1}^n\sum_{\hat l,\hat j,j'=1}^{n+m}(a^{\ts}\nabla\log\pi)_l \left[ a^{\ts}_{l\hat l}\frac{\pa}{\pa x_{\hat l}}z^{\ts}_{j\hat j}\frac{\pa f}{\pa x_{\hat j}}z^{\ts}_{jj'}\frac{\pa f}{\pa x_{j'}} \right]\cdots \cS^a_3\nonumber\\
&=&\cS^a_1+\cS^a_2+\cS^a_3,
\eea
where we also get  
\beaa
\cS_2^a&=&2\sum_{j=1}^m\sum_{l=1}^n\sum_{ l',\hat l,\hat j,j'=1}^{n+m}\left[a^{\ts}_{ll'}\frac{\pa }{\pa x_{l'}}\left(a^{\ts}_{l\hat l} \frac{\pa}{\pa x_{\hat l}}z^{\ts}_{j\hat j}\frac{\pa f}{\pa x_{\hat j}}z^{\ts}_{jj'}\frac{\pa f}{\pa x_{j'}} \right) \right] 
\eeaa 
\beaa 
&=&2\sum_{j=1}^m\sum_{l=1}^n\sum_{ l',\hat l,\hat j,j'=1}^{n+m}\left[ a^{\ts}_{ll'}\frac{\pa }{\pa x_{l'}}a^{\ts}_{l\hat l} \frac{\pa}{\pa x_{\hat l}}z^{\ts}_{j\hat j}\frac{\pa f}{\pa x_{\hat j}}z^{\ts}_{jj'}\frac{\pa f}{\pa x_{j'}} +a^{\ts}_{ll'}a^{\ts}_{l\hat l} \frac{\pa^2}{\pa x_{l'}\pa x_{\hat l}}z^{\ts}_{j\hat j}\frac{\pa f}{\pa x_{\hat j}}z^{\ts}_{jj'}\frac{\pa f}{\pa x_{j'}}  \right.\\
&&\quad\quad\quad\quad\quad\quad\quad\quad+a^{\ts}_{ll'}a^{\ts}_{l\hat l} \frac{\pa}{\pa x_{\hat l}}z^{\ts}_{j\hat j}\frac{\pa^2 f}{\pa x_{l'}\pa x_{\hat j}}z^{\ts}_{jj'}\frac{\pa f}{\pa x_{j'}} +a^{\ts}_{ll'}a^{\ts}_{l\hat l} \frac{\pa}{\pa x_{\hat l}}z^{\ts}_{j\hat j}\frac{\pa f}{\pa x_{\hat j}}z^{\ts}_{jj'}\frac{\pa^2 f}{\pa x_{l'}\pa x_{j'}} \\
&&\left.\quad\quad\quad\quad\quad\quad\quad\quad+a^{\ts}_{ll'}a^{\ts}_{l\hat l} \frac{\pa}{\pa x_{\hat l}}z^{\ts}_{j\hat j}\frac{\pa f}{\pa x_{\hat j}}\frac{\partial}{\pa x_{l'}}z^{\ts}_{jj'}\frac{\pa f}{\pa x_{j'}}  \right].
\eeaa
Combining all the terms above, we have 
\beaa
\div^{\pi}_z(\Gamma_{\nabla(aa^{\ts})}f,f )-\div^{\pi}_a(\Gamma_{\nabla(zz^{\ts})}f,f )
= \cS^z_1+\cS^z_2+\cS^z_3-(\cS^a_1+\cS^a_2+\cS^a_3).
\eeaa
By direct computations, we separate the above terms into two groups based on ``$\pa f\pa f$" and ``$\pa^2f\pa f$". We denote $\mathfrak{R}^{\pi}(f,f)$ as the sum of all ``$\pa f\pa f$" terms and denote $2G^{\ts}X$ as the sum of all ``$\pa^2f\pa f$" terms. Switching indices for the terms in $2G^{\ts}X$ to match $\frac{\pa^2 f}{\pa x_{\hat i}\pa x_{\hat j}}$, we get the following 
\beaa
&&2G^{\ts}X\\
&=&2\sum_{k=1}^m \sum_{i=1}^n\sum_{k',\hat k,\hat i,i'=1}^{n+m}\left[z^{\ts}_{kk'} z^{\ts}_{k\hat k} \frac{\pa}{\pa x_{\hat k}}a^{\ts}_{i\hat i}\frac{\pa^2 f}{\pa x_{k'}\pa x_{\hat i}}a^{\ts}_{ii'}\frac{\pa f}{\pa x_{i'}}+z^{\ts}_{kk'} z^{\ts}_{k\hat k} \frac{\pa}{\pa x_{\hat k}}a^{\ts}_{i\hat i}\frac{\pa f}{\pa x_{\hat i}}a^{\ts}_{ii'}\frac{\pa^2 f}{\pa x_{k'}\pa x_{i'}}\right]\nonumber\\
&&-2\sum_{j=1}^m\sum_{l=1}^n\sum_{ l',\hat l,\hat j,j'=1}^{n+m}\left[ a^{\ts}_{ll'}a^{\ts}_{l\hat l} \frac{\pa}{\pa x_{\hat l}}z^{\ts}_{j\hat j}\frac{\pa^2 f}{\pa x_{l'}\pa x_{\hat j}}z^{\ts}_{jj'}\frac{\pa f}{\pa x_{j'}} +a^{\ts}_{ll'}a^{\ts}_{l\hat l} \frac{\pa}{\pa x_{\hat l}}z^{\ts}_{j\hat j}\frac{\pa f}{\pa x_{\hat j}}z^{\ts}_{jj'}\frac{\pa^2 f}{\pa x_{l'}\pa x_{j'}} \right]\nonumber\\
&=&2\sum_{j=1}^m \sum_{i=1}^n\sum_{j',\hat j,\hat i,i'=1}^{n+m}\left[ z^{\ts}_{jj'} z^{\ts}_{j\hat j} \frac{\pa}{\pa x_{\hat j}}a^{\ts}_{i\hat i}\frac{\pa^2 f}{\pa x_{j'}\pa x_{\hat i}}a^{\ts}_{ii'}\frac{\pa f}{\pa x_{i'}}+z^{\ts}_{jj'} z^{\ts}_{j\hat j} \frac{\pa}{\pa x_{\hat j}}a^{\ts}_{i\hat i}\frac{\pa f}{\pa x_{\hat i}}a^{\ts}_{ii'}\frac{\pa^2 f}{\pa x_{j'}\pa x_{i'}}\right]\nonumber\\
&&-2\sum_{j=1}^m\sum_{i=1}^n\sum_{ i',\hat i,\hat j,j'=1}^{n+m}\left[ a^{\ts}_{ii'}a^{\ts}_{i\hat i} \frac{\pa}{\pa x_{\hat i}}z^{\ts}_{j\hat j}\frac{\pa^2 f}{\pa x_{i'}\pa x_{\hat j}}z^{\ts}_{jj'}\frac{\pa f}{\pa x_{j'}} +a^{\ts}_{ii'}a^{\ts}_{i\hat i} \frac{\pa}{\pa x_{\hat i}}z^{\ts}_{j\hat j}\frac{\pa f}{\pa x_{\hat j}}z^{\ts}_{jj'}\frac{\pa^2 f}{\pa x_{i'}\pa x_{j'}} \right]\nonumber\\
&=&2\sum_{j=1}^m \sum_{i=1}^n\sum_{j',\hat j,\hat i,i'=1}^{n+m}\left[ z^{\ts}_{j\hat j} z^{\ts}_{j j'} \frac{\pa}{\pa x_{ j'}}a^{\ts}_{i\hat i}\frac{\pa^2 f}{\pa x_{\hat j}\pa x_{\hat i}}a^{\ts}_{ii'}\frac{\pa f}{\pa x_{i'}}+z^{\ts}_{j \hat j} z^{\ts}_{j j'} \frac{\pa}{\pa x_{ j'}}a^{\ts}_{i i'}\frac{\pa f}{\pa x_{ i'}}a^{\ts}_{i\hat i}\frac{\pa^2 f}{\pa x_{\hat j}\pa x_{\hat i}}\right]\nonumber\\
&&-2\sum_{j=1}^m\sum_{i=1}^n\sum_{ i',\hat i,\hat j,j'=1}^{n+m}\left[ a^{\ts}_{i\hat i}a^{\ts}_{i i'} \frac{\pa}{\pa x_{ i'}}z^{\ts}_{j\hat j}\frac{\pa^2 f}{\pa x_{\hat i}\pa x_{\hat j}}z^{\ts}_{jj'}\frac{\pa f}{\pa x_{j'}} +a^{\ts}_{i\hat i}a^{\ts}_{ii'} \frac{\pa}{\pa x_{ i'}}z^{\ts}_{j j'}\frac{\pa f}{\pa x_{ j'}}z^{\ts}_{j\hat j}\frac{\pa^2 f}{\pa x_{\hat i}\pa x_{\hat j}} \right]\nonumber\\
&=&2\sum_{\hat i,\hat j=1}^{n+m}\frac{\pa^2 f}{\pa x_{\hat i}\pa x_{\hat j}} \left[ \sum_{i=1}^n\sum_{j=1}^m\sum_{j',\hat j,i',\hat i=1}^{n+m} \left[\left( z^{\ts}_{j\hat j} z^{\ts}_{j j'} \frac{\pa}{\pa x_{ j'}}a^{\ts}_{i\hat i}a^{\ts}_{ii'}\frac{\pa f}{\pa x_{i'}}+z^{\ts}_{j \hat j} z^{\ts}_{j j'} \frac{\pa}{\pa x_{ j'}}a^{\ts}_{i i'}\frac{\pa f}{\pa x_{ i'}}a^{\ts}_{i\hat i}\right)\right.\right.\\
&&\left.\left.\quad\quad\quad\quad\quad\quad\quad\quad\quad\quad\quad-\left(a^{\ts}_{i\hat i}a^{\ts}_{i i'} \frac{\pa}{\pa x_{ i'}}z^{\ts}_{j\hat j}z^{\ts}_{jj'}\frac{\pa f}{\pa x_{j'}} +a^{\ts}_{i\hat i}a^{\ts}_{ii'} \frac{\pa}{\pa x_{ i'}}z^{\ts}_{j j'}\frac{\pa f}{\pa x_{ j'}}z^{\ts}_{j\hat j}\right) \right]\right].
\eeaa
The first equality follows from the quantities we obtain previously, the second equality follows from switching $``k"$ to $``j"$ and $``l"$ to $``i"$, the third equality follows from switching between $``i'"$ and $``\hat i"$, $``j'"$ and $``\hat j"$. Thus the proof is completed.
\qed
\end{proof}

\subsection{Proof of Lemma \ref{gamma 2 estimates for Lp}}
From now on, we keep the following notation: $a^{\ts}\nabla f=\sum_{i=1}^n\sum_{\hat i=1}^{n+m}a^{\ts}_{i\hat i}{\frac{\partial}{\partial x_{\hat i}}}f$. Furthermore, we fix the notation for $a,a^{\ts}$ with relation $a_{\hat ii}=a^{\ts}_{i\hat i}$ for $i=1,\cdots,n$ and $\hat i=1,\cdots,n+m.$ Here we denote $a^{\ts}_{i\hat i}:=(a^{\ts})_{i\hat i}.$
Recall that, we define
\beaa 
\Gamma_{2,\widetilde L}(f,f)=\frac{1}{2}\Big( \widetilde L\Gamma_{1}(f,f)-2\Gamma_{1}(\widetilde Lf,f)\Big).
\eeaa

Next, we are ready to prove the following lemma. 

\begin{lem}
	\beaa
	&&\Gamma_{2,\widetilde L}(f,f)=X^{\ts}Q^{\ts}QX+2D^{\ts}QX+2C^{\ts}X+D^{\ts}D+\mathfrak{R}_a(\nabla f,\nabla f),
	\eeaa
where $Q,X,C,D$ are introduced in Notation \ref{notation} and $\mathfrak R_a$ is defined in Definition \ref{def: Ricci curvature}.
\end{lem}
\begin{proof}
	We plug in the operator $\widetilde L$ to our definition for $\Gamma_2$, 
	\beaa
	\Gamma_{2,\widetilde L}(f,f)&=&\frac{1}{2} \Delta_a\Gamma_{1}(f,f)-\frac{1}{2}A\nabla \Gamma_{1}(f,f)-\Gamma_{1}((\Delta_a -A\nabla)f,f)\\
	&=&\frac{1}{2}\Delta_a \Gamma_{1}(f,f)-\Gamma_{1}(\Delta_a f,f)-\frac{1}{2}A\nabla \Gamma_{1}(f,f)+\Gamma_{1}(A\nabla f,f).
	\eeaa
Now we compute the last two terms of the above equation. With $A=a\otimes\nabla a$, we get
\beaa
-\frac{1}{2}A\nabla \Gamma_{1}(f,f)&=& -\frac{1}{2}\sum_{\hat k=1}^{n+m}A_{\hat k}\nabla_{\frac{\partial}{\partial x_{\hat k}}}\la a^{\ts}\nabla f,a^{\ts}\nabla f\ra_{\hR^n}\\
&=&-\sum_{\hat k=1}^{n+m}\la A_{\hat k} (\nabla_{\frac{\partial}{\partial x_{\hat k}}} a^{\ts})\nabla f,a^{\ts}\nabla f\ra_{\hR^n}-\sum_{\hat k=1}^{n+m}\la A_{\hat k} a^{\ts}(\nabla_{\frac{\partial}{\partial x_{\hat k}}} \nabla f),a^{\ts}\nabla f\ra_{\hR^n}\\
&=&\textbf{J}_1+\textbf{J}_2,
\eeaa
and 
\beaa
\Gamma_{1}(A\nabla f,f)&=&\la a^{\ts}\nabla (\sum_{\hat k=1}^{n+m}A_{\hat k}\nabla_{\frac{\partial}{\partial x_{\hat k}}} f ) ,a^{\ts}\nabla f\ra_{\hR^n}\\
&=& \la a^{\ts}(\sum_{\hat k=1}^{n+m}A_{\hat k}\nabla \nabla_{\frac{\partial}{\partial x_{\hat k}}} f ) ,a^{\ts}\nabla f\ra_{\hR^n}+ \la a^{\ts}(\sum_{\hat k=1}^{n+m}\nabla A_{\hat k} \nabla_{\frac{\partial}{\partial x_{\hat k}}} f ) ,a^{\ts}\nabla f\ra_{\hR^n}\\
&=& \textbf{J}_3+\textbf{J}_4.
\eeaa
It is easy to see 
\beaa
\textbf{J}_2+\textbf{J}_3=0.\quad
\eeaa
We now expand $\textbf{J}_1$ and $\textbf{J}_4$ into local coordinates, 
\bea\label{J 1 term}
\textbf{J}_1=-\sum_{l=1}^{n}\left(a^{\ts}\nabla f \right)_l\left(\sum_{l',\hat k=1}^{n+m} \sum_{k=1}^n\sum_{k'=1}^{n+m}a_{\hat kk} \nabla_{\frac{\partial}{\partial x_{k'}}}a_{k'k}  \nabla_{\frac{\partial}{\partial x_{\hat k}}} a_{l l'}  \nabla_{\frac{\partial}{\partial x_{l'}}}f\right),
\eea
and
\bea\label{J 4 term}
\textbf{J}_4&=&\sum_{l=1}^n(a^{\ts}\nabla f)_l  \left(\sum_{l'=1}^{n+m} a^{\ts}_{ll'}(\sum_{\hat k=1}^{n+m}\nabla_{\frac{\partial}{\partial x_{l'}}}(\sum_{k=1}^n\sum_{k'=1}^{n+m}a_{\hat kk} \nabla_{\frac{\partial}{\partial x_{k'}}}a_{k'k})\nabla_{\frac{\partial}{\partial x_{\hat k}}} f ) \right) \nonumber \\
&=&\sum_{l=1}^n(a^{\ts}\nabla f)_l\left(\sum_{k=1}^n\sum_{l'=1}^{n+m}\sum_{\hat k,k'=1}^{n+m} a^{\ts}_{ll'}\nabla_{\frac{\partial}{\partial x_{l'}}}a_{\hat kk}  \nabla_{\frac{\partial}{\partial x_{k'}}}a_{k'k}\nabla_{\frac{\partial}{\partial x_{\hat k}}}f \right)\nonumber\\
&&+\sum_{l=1}^n(a^{\ts}\nabla f)_l\left(\sum_{k=1}^n\sum_{l'=1}^{n+m}\sum_{\hat k,k'=1}^{n+m}a^{\ts}_{ll'}a_{\hat kk} (\nabla_{\frac{\partial}{\partial x_{l'}}} \nabla_{\frac{\partial}{\partial x_{k'}}}a_{k'k})\nabla_{\frac{\partial}{\partial x_{\hat k}}}f \right).
\eea
Applying Lemma \ref{finite dim Bochner formula} which will be proved shortly below, we have 
	\beaa	
	&&\frac{1}{2}\Delta_a \Gamma_{1}(f,f)-\Gamma_{1}(\Delta_a f,f)\\
	&=&\frac{1}{2}(a^{\ts}\nabla\circ (a^{\ts}\nabla |a^{\ts}\nabla f|^2 ))-\la a^{\ts} \nabla ((a^{\ts}\nabla) \circ(a^{\ts}\nabla f)),a^{\ts}\nabla f\ra_{\hR^n} \\
	&& 
	+\sum_{l=1}^n(a^{\ts}\nabla f)_l\left(\sum_{\hat i,\hat k,l'=1}^{n+m}\sum_{k=1}^n \left(\frac{\partial}{\partial x_{\hat i}}a_{\hat ik} a^{\ts}_{k\hat k}   (\frac{\partial}{\partial x_{\hat k}} a^{\ts}_{ll'}\frac{\partial}{\partial x_{l'}}f)
	 -a^{\ts}_{ll'}  \frac{\partial}{\partial x_{\hat i}}a_{\hat ik} (\frac{\partial}{\partial x_{l'}}a^{\ts}_{k\hat k} \frac{\partial}{\partial x_{\hat k}} f) \right)\right)\nonumber\\
	 &&-\la \mathbf B_{n\times n}a^{\ts}\nabla f,a^{\ts}\nabla f\ra_{\hR^n} ,
	\eeaa
	where 
	\beaa
	\la \mathbf B_{n\times n} a^{\ts}\nabla f,a^{\ts}\nabla f\ra_{\hR^n}=\sum_{l=1}^n (a^{\ts}\nabla f)_l\left(\sum_{k=1}^n \sum_{l=1}^n\sum_{i=1}^{n+m} \sum_{k',j'=1}^{n+m} a^{\ts}_{lj'}\frac{\partial^2}{\partial x_ix_j'}a_{ik}  (a^{\ts}_{kk'}\frac{\partial}{\partial x_{k'}} f)\right).
	\eeaa
	Thus, combining with \eqref{J 1 term} and \eqref{J 4 term} , we have 
	\beaa
	\Gamma_{2,\widetilde L}(f,f)&=&\frac{1}{2}\Delta_a \Gamma_{1}(f,f)-\Gamma_{1}(\Delta_a f,f)+\textbf{J}_1+\textbf{J}_4\\
	&=&\frac{1}{2}(a^{\ts}\nabla\circ (a^{\ts}\nabla |a^{\ts}\nabla f|^2 ))-\la a^{\ts} \nabla [(a^{\ts}\nabla) \circ(a^{\ts}\nabla f)],a^{\ts}\nabla f\ra_{\hR^n} .
	\eeaa
	where the last term follows from Lemma \ref{Gamma_2_h} below. The proof is thus completed. 
 \qed
\end{proof}

\begin{lem}\label{finite dim Bochner formula}
	\bea\label{gamma 2 relation}
	&&\frac{1}{2}\Delta_a \Gamma_{1}(f,f)-\Gamma_{1}(\Delta_a f,f)\\
	&=&\frac{1}{2}(a^{\ts}\nabla\circ (a^{\ts}\nabla |a^{\ts}\nabla f|^2 ))-\la a^{\ts} \nabla ([(a^{\ts}\nabla) \circ(a^{\ts}\nabla f)]),a^{\ts}\nabla f\ra_{\hR^n}\nonumber \\
	&&-\la \textbf{B}_{n\times n}  a^{\ts}\nabla f, a^{\ts}\nabla f\ra_{\hR^n}+\textbf{B}_0.\nonumber
	\eea
Here the local representation for $\mathbf{B}_{n\times n}$ and $\textbf{B}_0$  are given as follows. For $l,k=1,\cdots,n$, we denote 
\bea\label{B and B_0}
\mathbf B_{lk}&=&\sum_{j'=1}^{n+m} a^{\ts}_{lj'}\sum_{i=1}^{n+m}\frac{\partial^2}{\partial x_i\partial x_{j'}} a_{ik}=\sum_{j'=1}^{n+m} a^{\ts}_{lj'}\sum_{i=1}^{n+m}\frac{\partial^2}{\partial x_i\partial x_{j'}} a^{\ts}_{ki},\\
\textbf{B}_0&=&\sum_{l=1}^n(a^{\ts}\nabla f)_l\left(\sum_{\hat i,\hat k,l'=1}^{n+m}\sum_{k=1}^n  \left(\frac{\partial}{\partial x_{\hat i}}a_{\hat ik} a^{\ts}_{k\hat k}   (\frac{\partial}{\partial x_{\hat k}} a^{\ts}_{ll'}\frac{\partial}{\partial x_{l'}}f) -a^{\ts}_{ll'}  \frac{\partial}{\partial x_{\hat i}}a_{\hat ik} (\frac{\partial}{\partial x_{l'}}a^{\ts}_{k\hat k} \frac{\partial}{\partial x_{\hat k}} f) \right)\right).\nonumber
\eea
We introduce the following notation $($convention$)$ that, for any function $F$ 
	\bea\label{notation circle}
	(a^{\ts}\nabla)\circ (a^{\ts}\nabla F)=\sum_{i=1}^n(a^{\ts}\nabla)_i(a^{\ts}\nabla F)_i=\sum_{i=1}^n\sum_{\hat i,i'=1}^{n+m}(a^{\ts}_{i\hat i}\frac{\pa }{\pa x_{\hat i}})(a^{\ts}_{ii'}\frac{\pa F}{\pa x_{i'}}).
	\eea
	\end{lem}


\begin{proof}[Proof of Lemma \ref{finite dim Bochner formula} ]
By our definition above, we have
\beaa
\Delta_a\Gamma_{1}(f,f)&=&\nabla \cdot (aa^{\ts} \nabla \la a^{\ts}\nabla f,a^{\ts}\nabla f \ra_{\hR^n})\\
&=&\nabla \cdot (a F)\\
&=& \sum_{\hat i=1}^{n+m} \frac{\partial}{\partial x_{\hat i}} (\sum_{k=1}^na_{\hat ik}F_k)\\
&=& \sum_{\hat i=1}^{n+m} \sum_{k=1}^n(\frac{\partial}{\partial x_{\hat i}} a_{\hat ik}F_k+a_{\hat ik} \frac{\partial}{\partial x_{\hat i}}F_k )\\
&=& \sum_{\hat i=1}^{n+m}\sum_{k=1}^n(\frac{\partial}{\partial x_{\hat i}} a_{\hat ik}F_k)+a^{\ts}\nabla\circ (a^{\ts}\nabla (a^{\ts}\nabla f)^2 ),
\eeaa
where we denote 
\beaa
F&=&a^{\ts} \nabla \la a^{\ts}\nabla f,a^{\ts}\nabla f \ra_{\hR^n}\\
&=&  a^{\ts}\nabla\sum_{l=1}^n (\sum_{\hat l=1}^{n+m} a^{\ts}_{l\hat l}\frac{\partial}{\partial x_{\hat l}}f)^2\\
&=& \left( \sum_{\hat k=1}^{n+m}a^{\ts}_{k\hat k}\frac{\partial}{\partial x_{\hat j}} \sum_{l=1}^n (\sum_{\hat l=1}^{n+m} a^{\ts}_{l\hat l}\frac{\partial}{\partial x_{\hat l}}f)^2\right)_{k=1,\cdots,n}\\
&=&(F_1,F_2,\cdots, F_n)^{\ts}.
\eeaa
 So we have 
 \bea\label{gamma 2 --1}
 \Delta_a\Gamma_{1} (f,f)&=&\sum_{\hat i=1}^{n+m}\sum_{k=1}^n\left(\frac{\partial}{\partial x_{\hat i}} a_{\hat ik}\left( \sum_{\hat k=1}^{n+m}a^{\ts}_{k\hat k}\frac{\partial}{\partial x_{\hat j}} \sum_{l=1}^n (\sum_{\hat l=1}^{n+m} a^{\ts}_{l\hat l}\frac{\partial}{\partial x_{\hat l}}f)^2\right)\right)\nonumber\\
 &&+a^{\ts}\nabla\circ  (a^{\ts}\nabla (a^{\ts}\nabla f)^2 ) \nonumber\\
 &=&\sum_{k=1}^n\sum_{\hat i=1}^{n+m}\left(\frac{\partial}{\partial x_{\hat i}} a_{\hat ik}\left( \sum_{\hat k=1}^{n+m}a^{\ts}_{k\hat k}\frac{\partial}{\partial x_{\hat k}} (a^{\ts}\nabla f)^2\right)\right)\nonumber\\
 && +(a^{\ts}\nabla)\circ (a^{\ts}\nabla (a^{\ts}\nabla f)^2 )\nonumber \\
 &=&\nabla a \circ (a^{\ts}\nabla (a^{\ts}\nabla f)^2)+(a^{\ts}\nabla)\circ (a^{\ts}\nabla (a^{\ts}\nabla f)^2 ).
 \eea
Next, we compute the following quantity. 
\beaa
\Gamma_{1}(\Delta_af,f)&=&\la a^{\ts} \nabla (\nabla\cdot (aa^{\ts}\nabla f) ),a^{\ts}\nabla f\ra_{\hR^n}, 
\eeaa
where we have
\beaa
\nabla\cdot (aa^{\ts} \nabla f)=\nabla \cdot (\sum_{k=1}^n\sum_{\hat k=1}^{n+m}a_{\hat ik}a^{\ts}_{k\hat k}\frac{\partial}{\partial x_{\hat k}}f)
=\sum_{\hat i=1}^{n+m} \frac{\partial}{\partial x_{\hat i}} (\sum_{k=1}^n\sum_{\hat k=1}^{n+m}a_{\hat ik}a^{\ts}_{k\hat k}\frac{\partial}{\partial x_{\hat k}}f) 
\eeaa 
\beaa 
&=& \sum_{\hat i,\hat k=1}^{n+m}\sum_{k=1}^n\frac{\partial}{\partial x_{\hat i}}a_{\hat ik} (a^{\ts}_{k\hat k} \frac{\partial}{\partial x_{\hat k}} f)+\sum_{\hat i,\hat k=1}^{n+m}\sum_{k=1}^n a_{\hat ik} \frac{\partial}{\partial x_{\hat i}} ( a^{\ts}_{k\hat k} \frac{\partial}{\partial x_j} f)\\ 
&=& \sum_{\hat i,\hat k=1}^{n+m}\sum_{k=1}^n\frac{\partial}{\partial x_{\hat i}}a_{\hat ik} (a^{\ts}_{k\hat k} \frac{\partial}{\partial x_{\hat k}} f)+(a^{\ts}\nabla)\circ (a^{\ts}\nabla f)\\
&=&\nabla a\circ  (a^{\ts}\nabla f)+(a^{\ts}\nabla)\circ (a^{\ts}\nabla f).
\eeaa
We continue with our computation as below, 
\bea\label{gamma 2 --2}
&&\Gamma_{1}(\Delta_a f,f)\nonumber \\
&=&\la a^{\ts} \nabla \left[ \sum_{\hat i,\hat k=1}^{n+m}\sum_{k=1}^n\frac{\partial}{\partial x_{\hat i}}a_{\hat ik} (a^{\ts}_{k\hat k} \frac{\partial}{\partial x_{\hat k}} f)+(a^{\ts}\nabla)\circ (a^{\ts}\nabla f)\right],a^{\ts}\nabla f\ra_{\hR^n}\nonumber \\
&=&\la a^{\ts}\nabla \left[  \sum_{\hat i,\hat k=1}^{n+m}\sum_{k=1}^n\frac{\partial}{\partial x_{\hat i}}a_{\hat ik} (a^{\ts}_{k\hat k} \frac{\partial}{\partial x_{\hat k}} f)\right], a^{\ts}\nabla f\ra_{\hR^n}\nonumber\\
&& +\la a^{\ts}\nabla ((a^{\ts}\nabla)\circ (a^{\ts}\nabla f)),a^{\ts}\nabla f\ra_{\hR^n}\nonumber\\
&=& \la a^{\ts} \nabla \left(\nabla a\cdot (a^{\ts}\nabla f)\right),a^{\ts}\nabla f\ra_{\hR^n} +\la a^{\ts} \nabla ((a^{\ts}\nabla) \circ(a^{\ts}\nabla f)),a^{\ts}\nabla f\ra_{\hR^n}\nonumber\\
&=&\la  \left(a^{\ts} \nabla\nabla a\cdot (a^{\ts}\nabla f)\right),a^{\ts}\nabla f\ra_{\hR^n}+\la \left(\nabla a\cdot (a^{\ts} \nabla (a^{\ts}\nabla f))\right),a^{\ts}\nabla f\ra_{\hR^n}\nonumber\\
&&+\la a^{\ts} \nabla ((a^{\ts}\nabla)\circ (a^{\ts}\nabla f)),a^{\ts}\nabla f\ra_{\hR^n}.
\eea
From the above, combining \eqref{gamma 2 --1}  and  \eqref{gamma 2 --2} we further get 
\beaa
&&\frac{1}{2}\Delta_a \Gamma_{1} (f,f)-\Gamma_{1} (\Delta_a f,f)\\
&=&\frac{1}{2}(a^{\ts}\nabla\circ (a^{\ts}\nabla |a^{\ts}\nabla f|^2 ))-\la a^{\ts} \nabla ((a^{\ts}\nabla)\circ (a^{\ts}\nabla f)),a^{\ts}\nabla f\ra_{\hR^n} \\
&&+\frac{1}{2} \sum_{\hat i=1}^{n+m}\sum_{k=1}^n\left(\frac{\partial}{\partial x_{\hat i}} a_{\hat ik}\left( \sum_{\hat k=1}^{n+m}a^{\ts}_{k\hat k}\frac{\partial}{\partial x_{\hat j}} \sum_{l=1}^n (\sum_{\hat l=1}^{n+m} a^{\ts}_{l\hat l}\frac{\partial}{\partial x_{\hat l}}f)^2\right)\right)\\
&&-\la a^{\ts}\nabla \left[  \sum_{\hat i,\hat k=1}^{n+m}\sum_{k=1}^n\frac{\partial}{\partial x_{\hat i}}a_{\hat ik} (a^{\ts}_{k\hat k} \frac{\partial}{\partial x_{\hat k}} f)\right], a^{\ts}\nabla f\ra_{\hR^n}\\
&=&\frac{1}{2}(a^{\ts}\nabla\circ (a^{\ts}\nabla |a^{\ts}\nabla f|^2 ))-\la a^{\ts} \nabla ((a^{\ts}\nabla) \circ(a^{\ts}\nabla f)),a^{\ts}\nabla f\ra_{\hR^n} \\
&&+\sum_{\hat i=1}^{n+m}\sum_{k=1}^n\left(\frac{\partial}{\partial x_{\hat i}}a_{\hat ik}\sum_{\hat k=1}^{n+m} a^{\ts}_{k\hat k}\sum_{l=1}^n (\sum_{\hat l=1}^{n+m} a^{\ts}_{l\hat l}\frac{\partial}{\partial x_{\hat l}}f) \frac{\partial}{\partial x_{\hat k}} (\sum_{\hat l=1}^{n+m} a^{\ts}_{l\hat l}\frac{\partial}{\partial x_{\hat l}}f) \right)\cdots \textbf{I}\\
&&-\sum_{l=1}^n \left((\sum_{\hat l=1}^{n+m}a^{\ts}_{l\hat l}\frac{\partial}{\partial x_{\hat l}} f) \left(\sum_{l'=1}^{n+m}a^{\ts}_{ll'}\frac{\partial}{\partial x_{l'}}\left[  \sum_{\hat i,\hat k=1}^{n+m}\sum_{k=1}^n\frac{\partial}{\partial x_{\hat i}}a_{\hat ik} (a^{\ts}_{k\hat k} \frac{\partial}{\partial x_{\hat k}} f)\right] \right) \right)\cdots \textbf{II}.
\eeaa
Recall that we denote $a^{\ts}$ to emphasize the transpose of the matrix $a$ and $a^{\ts}_{i\hat i}=a_{\hat i i}$, 
\beaa
\textbf{I}&=&\sum_{\hat i=1}^{n+m}\sum_{k=1}^n\left(\frac{\partial}{\partial x_{\hat i}}a_{\hat ik}\sum_{\hat k=1}^{n+m} a^{\ts}_{k\hat k}\sum_{l=1}^n (\sum_{\hat l=1}^{n+m} a^{\ts}_{l\hat l}\frac{\partial}{\partial x_{\hat l}}f) \frac{\partial}{\partial x_{\hat k}} (\sum_{\hat l=1}^{n+m} a^{\ts}_{l\hat l}\frac{\partial}{\partial x_{\hat l}}f) \right)\\
&=&\sum_{\hat i=1}^{n+m}\sum_{k=1}^n\left(\frac{\partial}{\partial x_{\hat i}}a_{\hat ik}\sum_{\hat k=1}^{n+m} a^{\ts}_{k\hat k}\sum_{l=1}^n (\sum_{\hat l=1}^{n+m} a^{\ts}_{l\hat l}\frac{\partial}{\partial x_{\hat l}}f)  (\sum_{\hat l=1}^{n+m}\frac{\partial}{\partial x_{\hat k}} a^{\ts}_{l\hat l}\frac{\partial}{\partial x_{\hat l}}f) \right)\\
&&+\sum_{\hat i=1}^{n+m}\sum_{k=1}^n\left(\frac{\partial}{\partial x_{\hat i}}a_{\hat ik}\sum_{\hat k=1}^{n+m} a^{\ts}_{k\hat k}\sum_{l=1}^n (\sum_{\hat l=1}^{n+m} a^{\ts}_{l\hat l}\frac{\partial}{\partial x_{\hat l}}f)  (\sum_{\hat l=1}^{n+m} a^{\ts}_{l\hat l}\frac{\partial}{\partial x_{\hat k}}\frac{\partial}{\partial x_{\hat l}}f) \right)\\
&=&\sum_{l=1}^n(\sum_{\hat l=1}^{n+m} a^{\ts}_{l\hat l}\frac{\partial}{\partial x_{\hat l}}f)\sum_{\hat i=1}^{n+m}\sum_{k=1}^n\left(\frac{\partial}{\partial x_{\hat i}}a_{\hat ik}\sum_{\hat k=1}^{n+m} a^{\ts}_{k\hat k}   (\sum_{ l'=1}^{n+m}\frac{\partial}{\partial x_{\hat k}} a^{\ts}_{ll'}\frac{\partial}{\partial x_{l'}}f) \right)\\
&&+\sum_{l=1}^n(\sum_{\hat l=1}^{n+m} a^{\ts}_{l\hat l}\frac{\partial}{\partial x_{\hat l}}f) \sum_{\hat i=1}^{n+m}\sum_{k=1}^n\left(\frac{\partial}{\partial x_{\hat i}}a_{\hat ik}\sum_{\hat k=1}^{n+m} a^{\ts}_{k\hat k}  (\sum_{\hat l=1}^{n+m} a^{\ts}_{l l'}\frac{\partial}{\partial x_{\hat k}}\frac{\partial}{\partial x_{l'}}f) \right),
\eeaa
and 
\beaa
\textbf{II}&=&\sum_{l=1}^n \left((\sum_{\hat l=1}^{n+m}a^{\ts}_{l\hat l}\frac{\partial}{\partial x_{\hat l}} f) \left(\sum_{l'=1}^{n+m}a^{\ts}_{ll'}\frac{\partial}{\partial x_{l'}}\left[  \sum_{\hat i,\hat k=1}^{n+m}\sum_{k=1}^n\frac{\partial}{\partial x_{\hat i}}a_{\hat ik} (a^{\ts}_{k\hat k} \frac{\partial}{\partial x_{\hat k}} f)\right] \right) \right)\\
&=&\sum_{l=1}^n \left((\sum_{\hat l=1}^{n+m}a^{\ts}_{l\hat l}\frac{\partial}{\partial x_{\hat l}} f) \left(\sum_{l'=1}^{n+m}a^{\ts}_{ll'}\left[  \sum_{\hat i,\hat k=1}^{n+m}\sum_{k=1}^n\frac{\partial}{\partial x_{\hat i}}a_{\hat ik} \frac{\partial}{\partial x_{l'}}(a^{\ts}_{k\hat k} \frac{\partial}{\partial x_{\hat k}} f)\right] \right) \right)\\
&&+\sum_{l=1}^n \left((\sum_{\hat l=1}^{n+m}a^{\ts}_{l\hat l}\frac{\partial}{\partial x_{\hat l}} f) \left(\sum_{l'=1}^{n+m}a^{\ts}_{ll'}\left[  \sum_{\hat i,\hat k=1}^{n+m}\sum_{k=1}^n\frac{\partial^2}{\partial x_{\hat i}x_{l'}}a_{\hat ik} (a^{\ts}_{k\hat k} \frac{\partial}{\partial x_{\hat k}} f)\right] \right) \right)\\
&=&\sum_{l=1}^n \left((\sum_{\hat l=1}^{n+m}a^{\ts}_{l\hat l}\frac{\partial}{\partial x_{\hat l}} f) \left(\sum_{l'=1}^{n+m}a^{\ts}_{ll'}\left[  \sum_{\hat i,\hat k=1}^{n+m}\sum_{k=1}^n\frac{\partial}{\partial x_{\hat i}}a_{\hat ik} (\frac{\partial}{\partial x_{l'}}a^{\ts}_{k\hat k} \frac{\partial}{\partial x_{\hat k}} f)\right] \right) \right)\\
&&+\sum_{l=1}^n \left((\sum_{\hat l=1}^{n+m}a^{\ts}_{l\hat l}\frac{\partial}{\partial x_{\hat l}} f) \left(\sum_{l'=1}^{n+m}a^{\ts}_{ll'}\left[  \sum_{\hat i,\hat k=1}^{n+m}\sum_{k=1}^n\frac{\partial}{\partial x_{\hat i}}a_{\hat ik} (a^{\ts}_{k\hat k} \frac{\partial}{\partial x_{l'}}\frac{\partial}{\partial x_{\hat k}} f)\right] \right) \right)\\
&&+\sum_{l=1}^n \left((\sum_{\hat l=1}^{n+m}a^{\ts}_{l\hat l}\frac{\partial}{\partial x_{\hat l}} f) \left(\sum_{l'=1}^{n+m}a^{\ts}_{ll'}\left[  \sum_{\hat i,\hat k=1}^{n+m}\sum_{k=1}^n\frac{\partial^2}{\partial x_{\hat i}x_{l'}}a_{\hat ik} (a^{\ts}_{k\hat k} \frac{\partial}{\partial x_{\hat k}} f)\right] \right) \right).
\eeaa
Subtracting the above two terms, we have
\beaa
&&\textbf{I}-\textbf{
II}\\
&=&-\sum_{l=1}^n \left((\sum_{\hat l=1}^{n+m}a^{\ts}_{l\hat l}\frac{\partial}{\partial x_{\hat l}} f) \left(\sum_{l'=1}^{n+m}a^{\ts}_{ll'}\left[  \sum_{\hat i,\hat k=1}^{n+m}\sum_{k=1}^n\frac{\partial^2}{\partial x_{\hat i}x_{l'}}a_{\hat ik} (a^{\ts}_{k\hat k} \frac{\partial}{\partial x_{\hat k}} f)\right] \right) \right)\\
&&+\sum_{l=1}^n(\sum_{\hat l=1}^{n+m} a^{\ts}_{l\hat l}\frac{\partial}{\partial x_{\hat l}}f)\sum_{\hat i=1}^{n+m}\sum_{k=1}^n\left(\frac{\partial}{\partial x_{\hat i}}a_{\hat ik}\sum_{\hat k=1}^{n+m} a^{\ts}_{k\hat k}   (\sum_{ l'=1}^{n+m}\frac{\partial}{\partial x_{\hat k}} a^{\ts}_{ll'}\frac{\partial}{\partial x_{l'}}f) \right)
\eeaa 
\beaa
&&-\sum_{l=1}^n \left((\sum_{\hat l=1}^{n+m}a^{\ts}_{l\hat l}\frac{\partial}{\partial x_{\hat l}} f) \left(\sum_{l'=1}^{n+m}a^{\ts}_{ll'}\left[  \sum_{\hat i,\hat k=1}^{n+m}\sum_{k=1}^n\frac{\partial}{\partial x_{\hat i}}a_{\hat ik} (\frac{\partial}{\partial x_{l'}}a^{\ts}_{k\hat k} \frac{\partial}{\partial x_{\hat k}} f)\right] \right) \right)\\ 
&=&-\sum_{l=1}^n \left((\sum_{\hat l=1}^{n+m}a^{\ts}_{l\hat l}\frac{\partial}{\partial x_{\hat l}} f) \left(\sum_{l'=1}^{n+m}a^{\ts}_{ll'}\left[  \sum_{\hat i,\hat k=1}^{n+m}\sum_{k=1}^n\frac{\partial^2}{\partial x_{\hat i}x_{l'}}a_{\hat ik} (a^{\ts}_{k\hat k} \frac{\partial}{\partial x_{\hat k}} f)\right] \right) \right)\\
&&+\sum_{l=1}^n(\sum_{\hat l=1}^{n+m} a^{\ts}_{l\hat l}\frac{\partial}{\partial x_{\hat l}}f)\left(\sum_{\hat i,\hat k,l'=1}^{n+m}\sum_{k=1}^n \left(\frac{\partial}{\partial x_{\hat i}}a_{\hat ik} a^{\ts}_{k\hat k}   (\frac{\partial}{\partial x_{\hat k}} a^{\ts}_{ll'}\frac{\partial}{\partial x_{l'}}f) -a^{\ts}_{ll'}  \frac{\partial}{\partial x_{\hat i}}a_{\hat ik} (\frac{\partial}{\partial x_{l'}}a^{\ts}_{k\hat k} \frac{\partial}{\partial x_{\hat k}} f) \right)\right)\\
&=&-\sum_{l=1}^n \left((\sum_{\hat l=1}^{n+m}a^{\ts}_{l\hat l}\frac{\partial}{\partial x_{\hat l}} f) \left(\sum_{l'=1}^{n+m}a^{\ts}_{ll'}\left[  \sum_{\hat i,\hat k=1}^{n+m}\sum_{k=1}^n\frac{\partial^2}{\partial x_{\hat i}x_{l'}}a_{\hat ik} (a^{\ts}_{k\hat k} \frac{\partial}{\partial x_{\hat k}} f)\right] \right) \right)\\
&&+\sum_{l=1}^n(a^{\ts}\nabla f)_l\left(\sum_{\hat i,\hat k,l'=1}^{n+m}\sum_{k=1}^n \left(\frac{\partial}{\partial x_{\hat i}}a_{\hat ik} a^{\ts}_{k\hat k}   (\frac{\partial}{\partial x_{\hat k}} a^{\ts}_{ll'}\frac{\partial}{\partial x_{l'}}f) -a^{\ts}_{ll'}  \frac{\partial}{\partial x_{\hat i}}a_{\hat ik} (\frac{\partial}{\partial x_{l'}}a^{\ts}_{k\hat k} \frac{\partial}{\partial x_{\hat k}} f) \right)\right).
\eeaa

Now, we eventually get the the following step
\beaa
&&\frac{1}{2}\Delta_a \Gamma_{1}(f,f)-\Gamma_{1}(\Delta_a f,f)\\
&=&\frac{1}{2}(a^{\ts}\nabla\circ (a^{\ts}\nabla |a^{\ts}\nabla f|^2 ))-\la a^{\ts} \nabla ((a^{\ts}\nabla) \circ(a^{\ts}\nabla f)),a^{\ts}\nabla f\ra_{\hR^n} \\
&&-\la\sum_{k=1}^n \sum_{l=1}^n\sum_{i=1}^{n+m} \sum_{j'=1}^{n+m} a^{\ts}_{lj'}\frac{\partial^2}{\partial x_ix_j'}a_{ik}  (a^{\ts}\nabla f)_k, a^{\ts}\nabla f\ra\\
&&+\sum_{l=1}^n(a^{\ts}\nabla f)_l\left(\sum_{\hat i,\hat k,l'=1}^{n+m}\sum_{k=1}^n \left(\frac{\partial}{\partial x_{\hat i}}a_{\hat ik} a^{\ts}_{k\hat k}   (\frac{\partial}{\partial x_{\hat k}} a^{\ts}_{ll'}\frac{\partial}{\partial x_{l'}}f) -a^{\ts}_{ll'}  \frac{\partial}{\partial x_{\hat i}}a_{\hat ik} (\frac{\partial}{\partial x_{l'}}a^{\ts}_{k\hat k} \frac{\partial}{\partial x_{\hat k}} f) \right)\right)\\
&=&\frac{1}{2}(a^{\ts}\nabla\circ (a^{\ts}\nabla |a^{\ts}\nabla f|^2 ))-\la a^{\ts} \nabla ((a^{\ts}\nabla) \circ(a^{\ts}\nabla f)),a^{\ts}\nabla f\ra_{\hR^n} -\la B_{n\times n}  a^{\ts}\nabla f, a^{\ts}\nabla f\ra_{\hR^n}\\
&&+\sum_{l=1}^n(a^{\ts}\nabla f)_l\left(\sum_{\hat i,\hat k,l'=1}^{n+m}\sum_{k=1}^n \left(\frac{\partial}{\partial x_{\hat i}}a_{\hat ik} a^{\ts}_{k\hat k}   (\frac{\partial}{\partial x_{\hat k}} a^{\ts}_{ll'}\frac{\partial}{\partial x_{l'}}f) -a^{\ts}_{ll'}  \frac{\partial}{\partial x_{\hat i}}a_{\hat ik} (\frac{\partial}{\partial x_{l'}}a^{\ts}_{k\hat k} \frac{\partial}{\partial x_{\hat k}} f) \right)\right).
\eeaa
Thus the proof is completed. \qed 
 \end{proof}
Below, we further investigate the extra term explicitly in the above Lemma \ref{finite dim Bochner formula}. 
\begin{lem}\label{Gamma_2_h}
	\bea
&&\frac{1}{2}(a^{\ts}\nabla\circ (a^{\ts}\nabla |a^{\ts}\nabla f|^2 ))-\la a^{\ts} \nabla ((a^{\ts}\nabla) \circ(a^{\ts}\nabla f)),a^{\ts}\nabla f\ra_{\hR^n} \nonumber\\
&=&X^{\ts}Q^{\ts}QX+ 2D^{\ts}QX+2C^{\ts}X+D^{\ts}D\nonumber\\
&&+\sum_{i,k=1}^n\sum_{i',\hat i,\hat k=1}^{n+m} \la a^{\ts}_{ii'} (\frac{\partial a^{\ts}_{i \hat i}}{\partial x_{i'}} \frac{\partial a^{\ts}_{k\hat k}}{\partial x_{\hat i}}\frac{\partial f}{\partial x_{\hat k}}) ,(a^{\ts}\nabla)_kf\ra_{\hR^n}  \nonumber\\
&&+\sum_{i,k=1}^n\sum_{i',\hat i,\hat k=1}^{n+m} \la a^{\ts}_{ii'}a^{\ts}_{i \hat i} (\frac{\partial }{\partial x_{i'}} \frac{\partial a^{\ts}_{k\hat k}}{\partial x_{\hat i}})(\frac{\partial f}{\partial x_{\hat k}}) ,(a^{\ts}\nabla)_kf\ra_{\hR^n} \nonumber
\eea 
\bea 
&&-\sum_{i,k=1}^n\sum_{i',\hat i,\hat k=1}^{n+m} \la (a^{\ts}_{k\hat k}\frac{\partial a^{\ts}_{ii'}}{\partial x_{\hat k}} \frac{\partial a^{\ts}_{i \hat i}}{\partial x_{i'}} \frac{\partial f}{\partial x_{\hat i}})  ,(a^{\ts}\nabla)_kf\ra_{\hR^n}  \nonumber\\
&&-\sum_{i,k=1}^n\sum_{i',\hat i,\hat k=1}^{n+m} \la a^{\ts}_{k\hat k} a^{\ts}_{ii'} (\frac{\partial }{\partial x_{\hat k}} \frac{\partial a^{\ts}_{i \hat i}}{\partial x_{i'}}) \frac{\partial f}{\partial x_{\hat i}}  ,(a^{\ts}\nabla)_kf\ra_{\hR^n}.  
\eea 
Recall that  matrix $Q$ and vectors $X$, $C$ and $D$ are defined in Notation \ref{notation}.
\end{lem}
\begin{proof}
We expand the two terms in lemma \ref{Gamma_2_h}. The first term reads as 
	\beaa
	&&\frac{1}{2}(a^{\ts}\nabla\circ  (a^{\ts}\nabla |a^{\ts}\nabla f|^2 ))\\
	&=&\frac{1}{2}\sum_{i=1}^n\sum_{k=1}^n(a^{\ts}\nabla)_i(a^{\ts}\nabla)_i|(a^{\ts}\nabla)_kf|^2\\
	&=&\sum_{i=1}^n\sum_{k=1}^n(a^{\ts}\nabla)_i\la (a^{\ts}\nabla)_i (a^{\ts}\nabla)_kf,(a^{\ts}\nabla)_kf\ra_{\hR^n}\\
	&=&\sum_{i=1}^n\sum_{k=1}^n\la (a^{\ts}\nabla)_i (a^{\ts}\nabla)_kf,(a^{\ts}\nabla)_i(a^{\ts}\nabla)_kf\ra_{\hR^n}\cdots \textbf{T}_1\\
	&&+\sum_{i=1}^n\sum_{k=1}^n\sum_{i',\hat i,\hat k=1}^{n+m} \la (a^{\ts}_{ii'}\frac{\partial}{\partial x_{i'}}) (a^{\ts}_{i \hat i}\frac{\partial}{\partial x_{\hat i}}) (a^{\ts}_{k\hat k}\frac{\partial}{\partial x_{\hat k}})f ,(a^{\ts}\nabla)_kf\ra_{\hR^n} \cdots \textbf{R}_1.
	\eeaa
And the second term reads as
	\beaa
	&&\la a^{\ts} \nabla ([(a^{\ts}\nabla)\circ (a^{\ts}\nabla f)]),a^{\ts}\nabla f\ra_{\hR^n}\\
	&=&\sum_{i,k=1}^n \la(a^{\ts}\nabla)_k[(a^{\ts}\nabla)_i(a^{\ts}\nabla)_if ],(a^{\ts}\nabla)_kf\ra \\
	&=&\sum_{i,k=1}^n\sum_{i',\hat i,\hat k=1}^{n+m} \la(a^{\ts}_{k\hat k}\frac{\partial}{\partial x_{\hat k}}) [(a^{\ts}_{ii'}\frac{\partial}{\partial x_{i'}}) (a^{\ts}_{i \hat i}\frac{\partial}{\partial x_{\hat i}}) f ],(a^{\ts}\nabla)_kf\ra  \cdots \textbf{R}_2.
	\eeaa
	Next, we expand $\textbf{R}_1$ and $\textbf{R}_2$ completely and get the following:
	\beaa
	\textbf{R}_1&=&\sum_{i,k=1}^n\sum_{i',\hat i,\hat k=1}^{n+m} \la (a^{\ts}_{ii'}\frac{\partial}{\partial x_{i'}}) (a^{\ts}_{i \hat i}\frac{\partial}{\partial x_{\hat i}}) (a^{\ts}_{k\hat k}\frac{\partial f}{\partial x_{\hat k}}) ,(a^{\ts}\nabla)_kf\ra_{\hR^n} \\
	&=&\sum_{i,k=1}^n\sum_{i',\hat i,\hat k=1}^{n+m} \la a^{\ts}_{ii'} (\frac{\partial a^{\ts}_{i \hat i}}{\partial x_{i'}} \frac{\partial a^{\ts}_{k\hat k}}{\partial x_{\hat i}}\frac{\partial f}{\partial x_{\hat k}}) ,(a^{\ts}\nabla)_kf\ra_{\hR^n} \cdots \textbf{R}_1^1\\
	&&+\sum_{i,k=1}^n\sum_{i',\hat i,\hat k=1}^{n+m} \la a^{\ts}_{ii'}a^{\ts}_{i \hat i} (\frac{\partial }{\partial x_{i'}} \frac{\partial a^{\ts}_{k\hat k}}{\partial x_{\hat i}})(\frac{\partial f}{\partial x_{\hat k}}) ,(a^{\ts}\nabla)_kf\ra_{\hR^n}  \cdots \textbf{R}_1^2\\
	&&+\sum_{i,k=1}^n\sum_{i',\hat i,\hat k=1}^{n+m} \la a^{\ts}_{ii'}a^{\ts}_{i \hat i} ( \frac{\partial a^{\ts}_{k\hat k}}{\partial x_{\hat i}})(\frac{\partial }{\partial x_{i'}}\frac{\partial f}{\partial x_{\hat k}}) ,(a^{\ts}\nabla)_kf\ra_{\hR^n}  \cdots \textbf{R}_1^3\\
	&&+\sum_{i,k=1}^n\sum_{i',\hat i,\hat k=1}^{n+m} \la (a^{\ts}_{ii'}) ((\frac{\partial}{\partial x_{i'}}a^{\ts}_{i \hat i}) a^{\ts}_{k\hat k}\frac{\partial}{\partial x_{\hat i}}\frac{\partial f}{\partial x_{\hat k}}) ,(a^{\ts}\nabla)_kf\ra_{\hR^n}  \cdots \textbf{R}_1^4
	\eeaa 
	\beaa 
	&&+\sum_{i,k=1}^n\sum_{i',\hat i,\hat k=1}^{n+m} \la a^{\ts}_{ii'}a^{\ts}_{i \hat i}(\frac{\partial}{\partial x_{i'}} a^{\ts}_{k\hat k})\frac{\partial}{\partial x_{\hat i}}\frac{\partial f}{\partial x_{\hat k}}) ,(a^{\ts}\nabla)_kf\ra_{\hR^n}  \cdots \textbf{R}_1^5\\
	&&+\sum_{i,k=1}^n\sum_{i',\hat i,\hat k=1}^{n+m} \la a^{\ts}_{ii'} a^{\ts}_{i \hat i}  a^{\ts}_{k\hat k}(\frac{\partial}{\partial x_{i'}}\frac{\partial}{\partial x_{\hat i}}\frac{\partial f}{\partial x_{\hat k}}) ,(a^{\ts}\nabla)_kf\ra_{\hR^n}  \cdots \textbf{R}_1^6.
	\eeaa
And 
	\beaa
\textbf{R}_2&=&	\sum_{i,k=1}^n\sum_{i',\hat i,\hat k=1}^{n+m} \la(a^{\ts}_{k\hat k}\frac{\partial}{\partial x_{\hat k}}) [(a^{\ts}_{ii'}\frac{\partial}{\partial x_{i'}}) (a^{\ts}_{i \hat i}\frac{\partial f}{\partial x_{\hat i}})  ],(a^{\ts}\nabla)_kf\ra \\
&=&\sum_{i,k=1}^n\sum_{i',\hat i,\hat k=1}^{n+m} \la a^{\ts}_{k\hat k}\frac{\partial a^{\ts}_{ii'}}{\partial x_{\hat k}} \frac{\partial a^{\ts}_{i \hat i}}{\partial x_{i'}} \frac{\partial f}{\partial x_{\hat i}})  ,(a^{\ts}\nabla)_kf\ra  \cdots \textbf{R}_2^1\\
&&+\sum_{i,k=1}^n\sum_{i',\hat i,\hat k=1}^{n+m} \la a^{\ts}_{k\hat k} a^{\ts}_{ii'} (\frac{\partial }{\partial x_{\hat k}} \frac{\partial a^{\ts}_{i \hat i}}{\partial x_{i'}}) \frac{\partial f}{\partial x_{\hat i}}  ,(a^{\ts}\nabla)_kf\ra \cdots \textbf{R}_2^2\\
&&+\sum_{i,k=1}^n\sum_{i',\hat i,\hat k=1}^{n+m} \la a^{\ts}_{k\hat k} a^{\ts}_{ii'} \frac{\partial a^{\ts}_{i \hat i}}{\partial x_{i'}}(\frac{\partial }{\partial x_{\hat k}}  \frac{\partial f}{\partial x_{\hat i}})  ,(a^{\ts}\nabla)_kf\ra \cdots \textbf{R}_2^3=\textbf{R}_1^4\\
&&+\sum_{i,k=1}^n\sum_{i',\hat i,\hat k=1}^{n+m} \la a^{\ts}_{k\hat k}\frac{\partial a^{\ts}_{ii'}}{\partial x_{\hat k}} a^{\ts}_{i \hat i}(\frac{\partial }{\partial x_{i'}} \frac{\partial f}{\partial x_{\hat i}})  ,(a^{\ts}\nabla)_kf\ra \cdots \textbf{R}_2^4\\
&&+\sum_{i,k=1}^n\sum_{i',\hat i,\hat k=1}^{n+m} \la a^{\ts}_{k\hat k}a^{\ts}_{ii'}\frac{\partial a^{\ts}_{i \hat i}}{\partial x_{\hat k}} (\frac{\partial }{\partial x_{i'}} \frac{\partial f}{\partial x_{\hat i}})  ,(a^{\ts}\nabla)_kf\ra \cdots \textbf{R}_2^5\\
&&+\sum_{i,k=1}^n\sum_{i',\hat i,\hat k=1}^{n+m} \la a^{\ts}_{k\hat k}a^{\ts}_{ii'}a^{\ts}_{i \hat i}(\frac{\partial }{\partial x_{\hat k}} \frac{\partial }{\partial x_{i'}} \frac{\partial f}{\partial x_{\hat i}})  ,(a^{\ts}\nabla)_kf\ra \cdots \textbf{R}_2^6=\textbf{R}_1^6.
	\eeaa
Our next step is to complete squares for all the above terms. Look at the term $\textbf{T}_1$ first. 
	\beaa
	\textbf{T}_1&=& \sum_{i,k=1}^n\left\la \sum_{\hat i,\hat k=1}^{n+m}a^{\ts}_{i\hat i}a^{\ts}_{k\hat k}\frac{\partial^2 f}{\partial x_{\hat i}\partial x_{ \hat k}}+\sum_{\hat i,\hat k=1}^{n+m}a^{\ts}_{i\hat i}\frac{\partial a^{\ts}_{k\hat k}}{\partial x_{\hat i}}\frac{\partial f}{\partial x_{ \hat k}} , \right.\\
	&&\left.\quad\quad\quad\sum_{ i', k'=1}^{n+m}a^{\ts}_{i i'}a^{\ts}_{k k'}\frac{\partial^2 f}{\partial x_{ i'}\partial x_{  k'}}+\sum_{ i', k'=1}^{n+m}a^{\ts}_{i i'}\frac{\partial a^{\ts}_{k k'}}{\partial x_{ i'}}\frac{\partial f}{\partial x_{ k'}}\right \ra\\
	&=&\sum_{i,k=1}^n\left \la \sum_{\hat i,\hat k=1}^{n+m}a^{\ts}_{i\hat i}a^{\ts}_{k\hat k}\frac{\partial^2 f}{\partial x_{\hat i}\partial x_{ \hat k}},\sum_{ i', k'=1}^{n+m}a^{\ts}_{i i'}a^{\ts}_{k k'}\frac{\partial^2 f}{\partial x_{ i'}\partial x_{  k'}} \right\ra \cdots \textbf{T}_{1a}\\
	&&+\sum_{i,k=1}^n\left\la \sum_{\hat i,\hat k=1}^{n+m}a^{\ts}_{i\hat i}a^{\ts}_{k\hat k}\frac{\partial^2 f}{\partial x_{\hat i}\partial x_{ \hat k}},\sum_{ i', k'=1}^{n+m}a^{\ts}_{i i'}\frac{\partial a^{\ts}_{k k'}}{\partial x_{ i'}}\frac{\partial f}{\partial x_{ k'}} \right\ra\cdots \textbf{T}_{1b}\\
	&&+\sum_{i,k=1}^n\left\la \sum_{\hat i,\hat k=1}^{n+m}a^{\ts}_{i\hat i}\frac{\partial a^{\ts}_{k\hat k}}{\partial x_{\hat i}}\frac{\partial f}{\partial x_{ \hat k}} ,\sum_{ i', k'=1}^{n+m}a^{\ts}_{i i'}a^{\ts}_{k k'}\frac{\partial^2 f}{\partial x_{ i'}\partial x_{  k'}} \right\ra\cdots \textbf{T}_{1c}\\
	&&+\sum_{i,k=1}^n\left\la \sum_{\hat i,\hat k=1}^{n+m}a^{\ts}_{i\hat i}\frac{\partial a^{\ts}_{k\hat k}}{\partial x_{\hat i}}\frac{\partial f}{\partial x_{ \hat k}} ,\sum_{ i', k'=1}^{n+m}a^{\ts}_{i i'}\frac{\partial a^{\ts}_{k k'}}{\partial x_{ i'}}\frac{\partial f}{\partial x_{ k'}} \right\ra\cdots \textbf{T}_{1d}.\\
	\eeaa
The terms $\textbf{T}_{1b}=\textbf{T}_{1c}$, $\textbf{R}_{1}^3=\textbf{R}_1^5$ and $\textbf{R}_2^5=\textbf{R}_2^4$ play the role of crossing terms inside the complete squares. In particular, for convenience, we change the index inside the sum of $\textbf{R}_1^3$ and $\textbf{R}_2^5$, switching $i',\hat i$ for $\textbf{R}_1^3$ and switching $i',\hat k$ for $\textbf{R}_2^5$. Then we get the following. 
	\beaa
	2\textbf{R}_1^3&=& 2\sum_{i,k=1}^n\sum_{i',\hat i,\hat k=1}^{n+m} \left\la a^{\ts}_{i\hat i}a^{\ts}_{i i'} ( \frac{\partial a^{\ts}_{k\hat k}}{\partial x_{i'}})(\frac{\partial }{\partial x_{\hat i}}\frac{\partial f}{\partial x_{\hat k}}) ,(a^{\ts}\nabla)_kf\right\ra \\
	&=&2\sum_{i,k=1}^n\sum_{\hat i,\hat k=1}^{n+m} \sum_{i',l=1}^{n+m} \left( a^{\ts}_{i\hat i}a^{\ts}_{i i'} ( \frac{\partial a^{\ts}_{k\hat k}}{\partial x_{i'}})(\frac{\partial }{\partial x_{\hat i}}\frac{\partial f}{\partial x_{\hat k}})a^{\ts}_{kl} \frac{\pa f}{\pa x_l} \right) \\
	-2 \textbf{R}_2^5&=&-2\sum_{i,k=1}^n\sum_{i',\hat i,\hat k=1}^{n+m}\left  \la a^{\ts}_{k i'}a^{\ts}_{i\hat k}\frac{\partial a^{\ts}_{i \hat i}}{\partial x_{i'}} (\frac{\partial }{\partial x_{\hat k}} \frac{\partial f}{\partial x_{\hat i}})  ,(a^{\ts}\nabla)_kf\right\ra \\
	&=&-2\sum_{i,k=1}^n \sum_{\hat i,\hat k=1}^{n+m}\sum_{i',l=1}^{n+m} \left( a^{\ts}_{k i'}a^{\ts}_{i\hat k}\frac{\partial a^{\ts}_{i \hat i}}{\partial x_{i'}} (\frac{\partial }{\partial x_{\hat k}} \frac{\partial f}{\partial x_{\hat i}})a^{\ts}_{kl} \frac{\pa f}{\pa x_l}\right).
	\eeaa

We denote 
\bea\label{vector z}
\sum_{\hat i,\hat k=1}^{n+m}a^{\ts}_{i\hat i}a^{\ts}_{k\hat k}\frac{\pa^2 f}{\pa x_{\hat i} \pa x_{\hat k}}&=&\gamma_{ik}.
\eea
The above equality \eqref{vector z} can be represented in the following matrix form
\beaa
Q_{n^2\times(n+m)^2}X_{(n+m)^2\times 1}=(\gamma_{11},\cdots,\gamma_{ik},\cdots,\gamma_{nn})^{\ts}_{n^2\times 1},
\eeaa
where $Q$ and $X$ are defined in \eqref{matrix Q} and \eqref{vector X}.
Now, we can represent term $\textbf{T}_{1a}$ as $
\sum_{i,k=1}^n \gamma_{ik}^2= \gamma^{\ts}\gamma=(QX)^{\ts}QX=X^{\ts}Q^{\ts}QX$.
Next we want to represent $\textbf{R}_1^3$ and $\textbf{R}_2^5$ in the following form in terms of vector $X$, 
\beaa
&&2\textbf{R}_1^3-2\textbf{R}_2^5\\
&=&2\sum_{i,k=1}^n\sum_{i',\hat i,\hat k=1}^{n+m}\left( \left\la a^{\ts}_{i\hat i}a^{\ts}_{i i'} ( \frac{\partial a^{\ts}_{k\hat k}}{\partial x_{i'}})(\frac{\partial }{\partial x_{\hat i}}\frac{\partial f}{\partial x_{\hat k}}) ,(a^{\ts}\nabla)_kf\right\ra_{\hR^n}
\right.\\
&&\left.\quad\quad\quad\quad\quad\quad-2\sum_{i,k=1}^n\sum_{i',\hat i,\hat k=1}^{n+m} \left\la a^{\ts}_{k i'}a^{\ts}_{i\hat k}\frac{\partial a^{\ts}_{i \hat i}}{\partial x_{i'}} (\frac{\partial }{\partial x_{\hat k}} \frac{\partial f}{\partial x_{\hat i}})  ,(a^{\ts}\nabla)_kf\right\ra \right)\\
&=&2\sum_{\hat i,\hat k=1}^{n+m}\left[\sum_{i,k=1}^n\sum_{i'=1}^{n+m}\left(\la a^{\ts}_{i\hat i}a^{\ts}_{i i'} ( \frac{\partial a^{\ts}_{k\hat k}}{\partial x_{i'}}) ,(a^{\ts}\nabla)_kf\ra - \la a^{\ts}_{k i'}a^{\ts}_{i\hat k}\frac{\partial a^{\ts}_{i \hat i}}{\partial x_{i'}}   ,(a^{\ts}\nabla)_kf\ra\right)\right](\frac{\partial }{\partial x_{\hat i}}\frac{\partial f}{\partial x_{\hat k}})\\
&=&2C^{\ts}X,
\eeaa
where $C$ is defined in \eqref{vector C}.
Similarly, we can represent $\textbf{T}_{1b}=\textbf{T}_{1c}$ by $X$,
\beaa
\textbf{T}_{1b}=\textbf{T}_{1c}&=&\sum_{i,k=1}^n\la \sum_{\hat i,\hat k=1}^{n+m}a^{\ts}_{i\hat i}\frac{\partial a^{\ts}_{k\hat k}}{\partial x_{\hat i}}\frac{\partial f}{\partial x_{ \hat k}} ,\sum_{ i', k'=1}^{n+m}a^{\ts}_{i i'}a^{\ts}_{k k'}\frac{\partial^2 f}{\partial x_{ i'}\partial x_{  k'}} \ra\\
&=&D^{\ts}QX,
\eeaa
where $D$ is defined in \eqref{notation}.  Summering over the above terms, we have the following quadratic form: 
\bea\label{complete squares}
\textbf{T}_1+2\textbf{R}_1^3-2\textbf{R}_2^5
=X^{\ts}Q^{\ts}QX+ 2D^{\ts}QX+2C^{\ts}X+D^{\ts}D.
\eea
Taking into account the fact that $\textbf{R}_1^6-\textbf{R}_2^6=0$ and $\textbf{R}_1^4-\textbf{R}_2^3=0$, we have 
\beaa
\textbf{T}_1+\textbf{R}_1-\textbf{R}_2
=\textbf{T}_1+2\textbf{R}_1^3-2\textbf{R}_2^5+\textbf{R}_1^1+\textbf{R}_1^2-\textbf{R}_2^1-\textbf{R}_2^2,
\eeaa	
which completes the proof.  \qed 
\end{proof}

\subsection{Proof of Lemma \ref{thm: gamma z}}
\begin{lem}
	\bea
	\Gamma_{2,\widetilde L}^z(f,f)=X^{\ts}P^{\ts}PX+2E^{\ts}PX+2F^{\ts}X+E^{\ts}E+\mathfrak{R}_z(\nabla f,\nabla f).
	\eea
where $\mathfrak R_z$ is defined in Definition \ref{def: Ricci curvature}.
\end{lem}
\begin{proof} The proof follows directly from Lemma \ref{gamma z 1} and Lemma \ref{gamma z 2}. \qed
\end{proof}

\begin{lem}\label{gamma z 1}
	\beaa
&&\frac{1}{2}\widetilde L\Gamma_1^z(f,f)-\Gamma_1^z(\widetilde Lf,f)\\
&=&\frac{1}{2}(a^{\ts}\nabla\circ (a^{\ts}\nabla |z^{\ts}\nabla f|^2 ))-\la z^{\ts} \nabla ((a^{\ts}\nabla) \circ(a^{\ts}\nabla f)),z^{\ts}\nabla f\ra_{\hR^m}.
\eeaa
\end{lem}

\begin{proof}\noindent
\textbf{Step 1:} We first define $\Gamma_1^z=\la z^{\ts}\nabla f , z^{\ts}\nabla f\ra_{\hR^m}$, we have 
\beaa
\widetilde L\Gamma_1^z(f,f)=\D_p	\Gamma_1^z(f,f)-A\nabla\Gamma_1^z(f,f),\quad 
\Gamma_1^z(\widetilde L f,f)=\Gamma_1^z(\D_pf,f)-\Gamma_1^z(A\nabla f,f).
\eeaa
By our definition above, we directly get
\beaa
\Delta_a\Gamma_1^z(f,f)&=&\nabla \cdot (aa^{\ts} \nabla \la z^{\ts}\nabla f,z^{\ts}\nabla f \ra_{\hR^m})=\nabla \cdot (a F^z)\\
&=& \sum_{\hat i=1}^{n+m} \frac{\partial}{\partial x_{\hat i}} (\sum_{k=1}^na_{\hat ik}F^z_k)\\
&=& \sum_{\hat i=1}^{n+m} \sum_{k=1}^n(\frac{\partial}{\partial x_{\hat i}} a_{\hat ik}F^z_k+a_{\hat ik} \frac{\partial}{\partial x_{\hat i}}F^z_k )\\
&=& \sum_{\hat i=1}^{n+m}\sum_{k=1}^n(\frac{\partial}{\partial x_{\hat i}} a_{\hat ik}F^z_k)+a^{\ts}\nabla\circ (a^{\ts}\nabla (z^{\ts}\nabla f)^2 ),
\eeaa
where we denote 
\beaa
F^z&=&a^{\ts} \nabla \la z^{\ts}\nabla f,z^{\ts}\nabla f \ra_{\hR^m}
=  a^{\ts}\nabla\sum_{l=1}^m (\sum_{\hat l=1}^{n+m} z^{\ts}_{l\hat l}\frac{\partial}{\partial x_{\hat l}}f)^2
\\
&=& \left( \sum_{\hat k=1}^{n+m}a^{\ts}_{k\hat k}\frac{\partial}{\partial x_{\hat k}} \sum_{l=1}^m (\sum_{\hat l=1}^{n+m} z^{\ts}_{l\hat l}\frac{\partial}{\partial x_{\hat l}}f)^2\right)_{k=1,\cdots,n}
=(F^z_1,F^z_2,\cdots, F^z_n)^{\ts}.
\eeaa
 we have
 \bea\label{gamma 2 a first term}
 &&\Delta_a\Gamma_1^z (f,f)\nonumber \\
 &=&\sum_{\hat i=1}^{n+m}\sum_{k=1}^n\left(\frac{\partial}{\partial x_{\hat i}} a_{\hat ik}\left( \sum_{\hat k=1}^{n+m}a^{\ts}_{k\hat k}\frac{\partial}{\partial x_{\hat k}} \sum_{l=1}^m (\sum_{\hat l=1}^{n+m} z^{\ts}_{l\hat l}\frac{\partial}{\partial x_{\hat l}}f)^2\right)\right)+a^{\ts}\nabla\circ  (a^{\ts}\nabla (z^{\ts}\nabla f)^2 )) \nonumber\\
 &=&\sum_{k=1}^n\sum_{\hat i=1}^{n+m}\left(\frac{\partial}{\partial x_{\hat i}} a_{\hat ik}\left( \sum_{\hat k=1}^{n+m}a^{\ts}_{k\hat k}\frac{\partial}{\partial x_{\hat k}} (z^{\ts}\nabla f)^2\right)\right) +(a^{\ts}\nabla)\circ (a^{\ts}\nabla (z^{\ts}\nabla f)^2 )\nonumber \\
 &=&\nabla a \circ (a^{\ts}\nabla (z^{\ts}\nabla f)^2)+(a^{\ts}\nabla)\circ (a^{\ts}\nabla (z^{\ts}\nabla f)^2 ).
 \eea
Next, we compute the following quantity. 
\beaa
\Gamma_1^z(\Delta_af,f)&=&\la z^{\ts} \nabla (\nabla\cdot (aa^{\ts}\nabla f) ),z^{\ts}\nabla f\ra_{\hR^m}.
\eeaa
From Lemma \ref{finite dim Bochner formula}, we have
\beaa
\nabla\cdot (aa^{\ts} \nabla f)=\nabla a\circ  (a^{\ts}\nabla f)+(a^{\ts}\nabla)\circ (a^{\ts}\nabla f).
\eeaa
We continue with our computation as below, 
\bea\label{gamma 2 a second term}
&&\Gamma_1^z(\Delta_a f,f)\nonumber \\
&=&\la z^{\ts} \nabla \left[ \sum_{\hat i,\hat k=1}^{n+m}\sum_{k=1}^n\frac{\partial}{\partial x_{\hat i}}a_{\hat ik} (a^{\ts}_{k\hat k} \frac{\partial}{\partial x_{\hat k}} f)+(a^{\ts}\nabla)\circ (a^{\ts}\nabla f)\right],z^{\ts}\nabla f\ra_{\hR^m}\nonumber \\
&=&\la z^{\ts}\nabla \left[  \sum_{\hat i,\hat k=1}^{n+m}\sum_{k=1}^n\frac{\partial}{\partial x_{\hat i}}a_{\hat ik} (a^{\ts}_{k\hat k} \frac{\partial}{\partial x_{\hat k}} f)\right], z^{\ts}\nabla f\ra_{\hR^m}+\la z^{\ts}\nabla ((a^{\ts}\nabla)\circ (a^{\ts}\nabla f)),z^{\ts}\nabla f\ra_{\hR^m}\nonumber\\
&=& \la z^{\ts} \nabla \left(\nabla a\cdot (a^{\ts}\nabla f)\right),z^{\ts}\nabla f\ra_{\hR^m}+\la z^{\ts} \nabla ((a^{\ts}\nabla) \circ(a^{\ts}\nabla f)),z^{\ts}\nabla f\ra_{\hR^m}\nonumber\\
&=&\la  \left(z^{\ts} \nabla\nabla a\cdot (a^{\ts}\nabla f)\right),z^{\ts}\nabla f\ra_{\hR^m}+\la \left(\nabla a\cdot (z^{\ts} \nabla (a^{\ts}\nabla f))\right),z^{\ts}\nabla f\ra_{\hR^m}\nonumber\\
&&+\la z^{\ts} \nabla ((a^{\ts}\nabla)\circ (a^{\ts}\nabla f)),z^{\ts}\nabla f\ra_{\hR^m}.
\eea
From the above, combining \eqref{gamma 2 a first term}  and  \eqref{gamma 2 a second term} we further get 
\beaa
&&\frac{1}{2}\Delta_a \Gamma_1^z (f,f)-\Gamma_1^z (\Delta_a f,f)\\
&=&\frac{1}{2}(a^{\ts}\nabla\circ (a^{\ts}\nabla |z^{\ts}\nabla f|^2 ))-\la z^{\ts} \nabla ((a^{\ts}\nabla)\circ (a^{\ts}\nabla f)),z^{\ts}\nabla f\ra_{\hR^m} \\
&&+\frac{1}{2} \sum_{\hat i=1}^{n+m}\sum_{k=1}^n\left(\frac{\partial}{\partial x_{\hat i}} a_{\hat ik}\left( \sum_{\hat k=1}^{n+m}a^{\ts}_{k\hat k}\frac{\partial}{\partial x_{\hat j}} \sum_{l=1}^n (\sum_{\hat l=1}^{n+m} z^{\ts}_{l\hat l}\frac{\partial}{\partial x_{\hat l}}f)^2\right)\right)
\eeaa 
\beaa 
&&-\la z^{\ts}\nabla \left[  \sum_{\hat i,\hat k=1}^{n+m}\sum_{k=1}^n\frac{\partial}{\partial x_{\hat i}}a_{\hat ik} (a^{\ts}_{k\hat k} \frac{\partial}{\partial x_{\hat k}} f)\right], z^{\ts}\nabla f\ra_{\hR^m}\\ 
&=&\frac{1}{2}(a^{\ts}\nabla\circ (a^{\ts}\nabla |z^{\ts}\nabla f|^2 ))-\la z^{\ts} \nabla ((a^{\ts}\nabla) \circ(a^{\ts}\nabla f)),z^{\ts}\nabla f\ra_{\hR^m} \\
&&+\sum_{\hat i=1}^{n+m}\sum_{k=1}^n\left(\frac{\partial}{\partial x_{\hat i}}a_{\hat ik}\sum_{\hat k=1}^{n+m} a^{\ts}_{k\hat k}\sum_{l=1}^m (\sum_{\hat l=1}^{n+m} z^{\ts}_{l\hat l}\frac{\partial}{\partial x_{\hat l}}f) \frac{\partial}{\partial x_{\hat k}} (\sum_{\hat l=1}^{n+m} z^{\ts}_{l\hat l}\frac{\partial}{\partial x_{\hat l}}f) \right)\cdots \textbf{I}\\
&&-\sum_{l=1}^m \left((\sum_{\hat l=1}^{n+m}z^{\ts}_{l\hat l}\frac{\partial}{\partial x_{\hat l}} f) \left(\sum_{l'=1}^{n+m}z^{\ts}_{ll'}\frac{\partial}{\partial x_{l'}}\left[  \sum_{\hat i,\hat k=1}^{n+m}\sum_{k=1}^n\frac{\partial}{\partial x_{\hat i}}a_{\hat ik} (a^{\ts}_{k\hat k} \frac{\partial}{\partial x_{\hat k}} f)\right] \right) \right)\cdots \textbf{II}.
\eeaa
Recall here, we denote $a^{\ts}$ to emphasize the transpose of the matrix $a$ and $a^{\ts}_{i\hat i}=a_{\hat i i}$, 
\beaa
\textbf{I}&=&\sum_{\hat i=1}^{n+m}\sum_{k=1}^n\left(\frac{\partial}{\partial x_{\hat i}}a_{\hat ik}\sum_{\hat k=1}^{n+m} a^{\ts}_{k\hat k}\sum_{l=1}^m(\sum_{\hat l=1}^{n+m} z^{\ts}_{l\hat l}\frac{\partial}{\partial x_{\hat l}}f) \frac{\partial}{\partial x_{\hat k}} (\sum_{\hat l=1}^{n+m} z^{\ts}_{l\hat l}\frac{\partial}{\partial x_{\hat l}}f) \right)\\
&=&\sum_{\hat i=1}^{n+m}\sum_{k=1}^n\left(\frac{\partial}{\partial x_{\hat i}}a_{\hat ik}\sum_{\hat k=1}^{n+m} a^{\ts}_{k\hat k}\sum_{l=1}^m (\sum_{\hat l=1}^{n+m} z^{\ts}_{l\hat l}\frac{\partial}{\partial x_{\hat l}}f)  (\sum_{\hat l=1}^{n+m}\frac{\partial}{\partial x_{\hat k}} z^{\ts}_{l\hat l}\frac{\partial}{\partial x_{\hat l}}f) \right)\\
&&+\sum_{\hat i=1}^{n+m}\sum_{k=1}^n\left(\frac{\partial}{\partial x_{\hat i}}a_{\hat ik}\sum_{\hat k=1}^{n+m} a^{\ts}_{k\hat k}\sum_{l=1}^m (\sum_{\hat l=1}^{n+m} z^{\ts}_{l\hat l}\frac{\partial}{\partial x_{\hat l}}f)  (\sum_{\hat l=1}^{n+m} z^{\ts}_{l\hat l}\frac{\partial}{\partial x_{\hat k}}\frac{\partial}{\partial x_{\hat l}}f) \right)\\
&=&\sum_{l=1}^m(\sum_{\hat l=1}^{n+m} z^{\ts}_{l\hat l}\frac{\partial}{\partial x_{\hat l}}f)\sum_{\hat i=1}^{n+m}\sum_{k=1}^n\left(\frac{\partial}{\partial x_{\hat i}}a_{\hat ik}\sum_{\hat k=1}^{n+m} a^{\ts}_{k\hat k}   (\sum_{ l'=1}^{n+m}\frac{\partial}{\partial x_{\hat k}} z^{\ts}_{ll'}\frac{\partial}{\partial x_{l'}}f) \right)\\
&&+\sum_{l=1}^m(\sum_{\hat l=1}^{n+m} z^{\ts}_{l\hat l}\frac{\partial}{\partial x_{\hat l}}f) \sum_{\hat i=1}^{n+m}\sum_{k=1}^n\left(\frac{\partial}{\partial x_{\hat i}}a_{\hat ik}\sum_{\hat k=1}^{n+m} a^{\ts}_{k\hat k}  (\sum_{\hat l=1}^{n+m} z^{\ts}_{l l'}\frac{\partial}{\partial x_{\hat k}}\frac{\partial}{\partial x_{l'}}f) \right);\\ 
\textbf{II}&=&\sum_{l=1}^m \left((\sum_{\hat l=1}^{n+m}z^{\ts}_{l\hat l}\frac{\partial}{\partial x_{\hat l}} f) \left(\sum_{l'=1}^{n+m}z^{\ts}_{ll'}\frac{\partial}{\partial x_{l'}}\left[  \sum_{\hat i,\hat k=1}^{n+m}\sum_{k=1}^n\frac{\partial}{\partial x_{\hat i}}a_{\hat ik} (a^{\ts}_{k\hat k} \frac{\partial}{\partial x_{\hat k}} f)\right] \right) \right)\\
&=&\sum_{l=1}^m \left((\sum_{\hat l=1}^{n+m}z^{\ts}_{l\hat l}\frac{\partial}{\partial x_{\hat l}} f) \left(\sum_{l'=1}^{n+m}z^{\ts}_{ll'}\left[  \sum_{\hat i,\hat k=1}^{n+m}\sum_{k=1}^n\frac{\partial}{\partial x_{\hat i}}a_{\hat ik} \frac{\partial}{\partial x_{l'}}(a^{\ts}_{k\hat k} \frac{\partial}{\partial x_{\hat k}} f)\right] \right) \right)\\
&&+\sum_{l=1}^m \left((\sum_{\hat l=1}^{n+m}z^{\ts}_{l\hat l}\frac{\partial}{\partial x_{\hat l}} f) \left(\sum_{l'=1}^{n+m}z^{\ts}_{ll'}\left[  \sum_{\hat i,\hat k=1}^{n+m}\sum_{k=1}^n\frac{\partial^2}{\partial x_{\hat i}x_{l'}}a_{\hat ik} (a^{\ts}_{k\hat k} \frac{\partial}{\partial x_{\hat k}} f)\right] \right) \right)\\ 
&=&\sum_{l=1}^m \left((\sum_{\hat l=1}^{n+m}z^{\ts}_{l\hat l}\frac{\partial}{\partial x_{\hat l}} f) \left(\sum_{l'=1}^{n+m}z^{\ts}_{ll'}\left[  \sum_{\hat i,\hat k=1}^{n+m}\sum_{k=1}^n\frac{\partial}{\partial x_{\hat i}}a_{\hat ik} (\frac{\partial}{\partial x_{l'}}a^{\ts}_{k\hat k} \frac{\partial}{\partial x_{\hat k}} f)\right] \right) \right)\\
&&+\sum_{l=1}^m \left((\sum_{\hat l=1}^{n+m}z^{\ts}_{l\hat l}\frac{\partial}{\partial x_{\hat l}} f) \left(\sum_{l'=1}^{n+m}z^{\ts}_{ll'}\left[  \sum_{\hat i,\hat k=1}^{n+m}\sum_{k=1}^n\frac{\partial}{\partial x_{\hat i}}a_{\hat ik} (a^{\ts}_{k\hat k} \frac{\partial}{\partial x_{l'}}\frac{\partial}{\partial x_{\hat k}} f)\right] \right) \right)\\
&&+\sum_{l=1}^n \left((\sum_{\hat l=1}^{n+m}z^{\ts}_{l\hat l}\frac{\partial}{\partial x_{\hat l}} f) \left(\sum_{l'=1}^{n+m}z^{\ts}_{ll'}\left[  \sum_{\hat i,\hat k=1}^{n+m}\sum_{k=1}^n\frac{\partial^2}{\partial x_{\hat i}x_{l'}}a_{\hat ik} (a^{\ts}_{k\hat k} \frac{\partial}{\partial x_{\hat k}} f)\right] \right) \right).
\eeaa
Subtracting the above two terms, we obtain the following
\beaa
\textbf{I}-\textbf{II}&=&-\sum_{l=1}^m \left((\sum_{\hat l=1}^{n+m}z^{\ts}_{l\hat l}\frac{\partial}{\partial x_{\hat l}} f) \left(\sum_{l'=1}^{n+m}z^{\ts}_{ll'}\left[  \sum_{\hat i,\hat k=1}^{n+m}\sum_{k=1}^n\frac{\partial}{\partial x_{\hat i}}a_{\hat ik} (\frac{\partial}{\partial x_{l'}}a^{\ts}_{k\hat k} \frac{\partial}{\partial x_{\hat k}} f)\right] \right) \right)\\
&&-\sum_{l=1}^m \left((\sum_{\hat l=1}^{n+m}z^{\ts}_{l\hat l}\frac{\partial}{\partial x_{\hat l}} f) \left(\sum_{l'=1}^{n+m}z^{\ts}_{ll'}\left[  \sum_{\hat i,\hat k=1}^{n+m}\sum_{k=1}^n\frac{\partial^2}{\partial x_{\hat i}x_{l'}}a_{\hat ik} (a^{\ts}_{k\hat k} \frac{\partial}{\partial x_{\hat k}} f)\right] \right) \right)\\
&&+\sum_{l=1}^m(\sum_{\hat l=1}^{n+m} z^{\ts}_{l\hat l}\frac{\partial}{\partial x_{\hat l}}f)\sum_{\hat i=1}^{n+m}\sum_{k=1}^n\left(\frac{\partial}{\partial x_{\hat i}}a_{\hat ik}\sum_{\hat k=1}^{n+m} a^{\ts}_{k\hat k}   (\sum_{ l'=1}^{n+m}\frac{\partial}{\partial x_{\hat k}} z^{\ts}_{ll'}\frac{\partial}{\partial x_{l'}}f) \right).
\eeaa

Now, we eventually end up with the following formula, 
\beaa
&&\frac{1}{2}\Delta_a \Gamma_1^z(f,f)-\Gamma_1^z(\Delta_a f,f)\\
&=&\frac{1}{2}(a^{\ts}\nabla\circ (a^{\ts}\nabla |z^{\ts}\nabla f|^2 ))-\la z^{\ts} \nabla ((a^{\ts}\nabla) \circ(a^{\ts}\nabla f)),z^{\ts}\nabla f\ra_{\hR^m} \\
&&-\sum_{l=1}^m \left((\sum_{\hat l=1}^{n+m}z^{\ts}_{l\hat l}\frac{\partial}{\partial x_{\hat l}} f) \left(\sum_{l'=1}^{n+m}z^{\ts}_{ll'}\left[  \sum_{\hat i,\hat k=1}^{n+m}\sum_{k=1}^n\frac{\partial}{\partial x_{\hat i}}a_{\hat ik} (\frac{\partial}{\partial x_{l'}}a^{\ts}_{k\hat k} \frac{\partial}{\partial x_{\hat k}} f)\right] \right) \right)\\
&&-\sum_{l=1}^m \left((\sum_{\hat l=1}^{n+m}z^{\ts}_{l\hat l}\frac{\partial}{\partial x_{\hat l}} f) \left(\sum_{l'=1}^{n+m}z^{\ts}_{ll'}\left[  \sum_{\hat i,\hat k=1}^{n+m}\sum_{k=1}^n\frac{\partial^2}{\partial x_{\hat i}x_{l'}}a_{\hat ik} (a^{\ts}_{k\hat k} \frac{\partial}{\partial x_{\hat k}} f)\right] \right) \right)\\
&&+\sum_{l=1}^m(\sum_{\hat l=1}^{n+m} z^{\ts}_{l\hat l}\frac{\partial}{\partial x_{\hat l}}f)\sum_{\hat i=1}^{n+m}\sum_{k=1}^n\left(\frac{\partial}{\partial x_{\hat i}}a_{\hat ik}\sum_{\hat k=1}^{n+m} a^{\ts}_{k\hat k}   (\sum_{ l'=1}^{n+m}\frac{\partial}{\partial x_{\hat k}} z^{\ts}_{ll'}\frac{\partial}{\partial x_{l'}}f) \right).
\eeaa

\noindent
\textbf{Step 2:} Computation of
$-\frac{1}{2}A\nabla\Gamma_1^z(f,f)+\Gamma_1^z(A\nabla f,f)$.
Now we compute the last two terms of the above equation, with $A=a\otimes\nabla a$,
\beaa
-\frac{1}{2}A\nabla \Gamma_1^z(f,f)&=& -\frac{1}{2}\sum_{\hat k=1}^{n+m}A_{\hat k}\nabla_{\frac{\partial}{\partial x_{\hat k}}}\la z^{\ts}\nabla f,z^{\ts}\nabla f\ra_{\hR^m}\\
&=&-\sum_{\hat k=1}^{n+m}\la A_{\hat k} (\nabla_{\frac{\partial}{\partial x_{\hat k}}} z^{\ts})\nabla f,z^{\ts}\nabla f\ra_{\hR^m}-\sum_{\hat k=1}^{n+m}\la A_{\hat k} z^{\ts}(\nabla_{\frac{\partial}{\partial x_{\hat k}}} \nabla f),z^{\ts}\nabla f\ra_{\hR^m}\\
&=&\tilde{\textbf{J}}_1+\tilde{\textbf{J}}_2,\eeaa 
\beaa 
\Gamma_1^z(A\nabla f,f)&=&\la z^{\ts}\nabla (\sum_{\hat k=1}^{n+m}A_{\hat k}\nabla_{\frac{\partial}{\partial x_{\hat k}}} f ) ,z^{\ts}\nabla f\ra_{\hR^m}\\
&=& \la z^{\ts}(\sum_{\hat k=1}^{n+m}A_{\hat k}\nabla \nabla_{\frac{\partial}{\partial x_{\hat k}}} f ) ,z^{\ts}\nabla f\ra_{\hR^m}+ \la z^{\ts}(\sum_{\hat k=1}^{n+m}\nabla A_{\hat k} \nabla_{\frac{\partial}{\partial x_{\hat k}}} f ) ,z^{\ts}\nabla f\ra_{\hR^m}\\
&=& \tilde{\textbf{J}}_3+\tilde{\textbf{J}}_4.
\eeaa
It is easy to see $\tilde{\textbf{J}}_2+\tilde{\textbf{J}}_3=0.$
We now expand $\tilde{\textbf{J}}_1$ and $\tilde{\textbf{J}}_4$ into local coordinates, 
\bea\label{tilde J 1 term}
\tilde{\textbf{J}}_1&=&-\sum_{l=1}^{m}\left(z^{\ts}\nabla f \right)_l\left(\sum_{l',\hat k=1}^{n+m} \sum_{k=1}^n\sum_{k'=1}^{n+m}a_{\hat kk} \nabla_{\frac{\partial}{\partial x_{k'}}}a_{k'k}  \nabla_{\frac{\partial}{\partial x_{\hat k}}} z_{l l'}  \nabla_{\frac{\partial}{\partial x_{l'}}}f\right),
\eea
\bea 
\tilde{\textbf{J}}_4&=&\sum_{l=1}^m(z^{\ts}\nabla f)_l  \left(\sum_{l'=1}^{n+m} z^{\ts}_{ll'}(\sum_{\hat k=1}^{n+m}\nabla_{\frac{\partial}{\partial x_{l'}}}(\sum_{k=1}^n\sum_{k'=1}^{n+m}a_{\hat kk} \nabla_{\frac{\partial}{\partial x_{k'}}}a_{k'k})\nabla_{\frac{\partial}{\partial x_{\hat k}}} f ) \right) \nonumber \\
&=&\sum_{l=1}^m(z^{\ts}\nabla f)_l\left(\sum_{k=1}^n\sum_{l'=1}^{n+m}\sum_{\hat k,k'=1}^{n+m} z^{\ts}_{ll'}\nabla_{\frac{\partial}{\partial x_{l'}}}a_{\hat kk}  \nabla_{\frac{\partial}{\partial x_{k'}}}a_{k'k}\nabla_{\frac{\partial}{\partial x_{\hat k}}}f \right)\nonumber\\
&&+\sum_{l=1}^m(z^{\ts}\nabla f)_l\left(\sum_{k=1}^n\sum_{l'=1}^{n+m}\sum_{\hat k,k'=1}^{n+m}z^{\ts}_{ll'}a_{\hat kk} (\nabla_{\frac{\partial}{\partial x_{l'}}} \nabla_{\frac{\partial}{\partial x_{k'}}}a_{k'k})\nabla_{\frac{\partial}{\partial x_{\hat k}}}f \right).
\eea
Combine the above two steps, we thus get 
\beaa
\frac{1}{2}\widetilde L\Gamma_1^z(f,f)-\Gamma_1^z(\widetilde Lf,f)=\frac{1}{2}(a^{\ts}\nabla\circ (a^{\ts}\nabla |z^{\ts}\nabla f|^2 ))-\la z^{\ts} \nabla ((a^{\ts}\nabla) \circ(a^{\ts}\nabla f)),z^{\ts}\nabla f\ra_{\hR^m}.
\eeaa \qed 
\end{proof}
\begin{lem}\label{gamma z 2}
	\beaa
	&&\frac{1}{2}(a^{\ts}\nabla\circ (a^{\ts}\nabla |z^{\ts}\nabla f|^2 ))-\la z^{\ts} \nabla ((a^{\ts}\nabla) \circ(a^{\ts}\nabla f)),z^{\ts}\nabla f\ra_{\hR^m}\\
	&=&X^{\ts}P^{\ts}PX+2E^{\ts}PX+2F^{\ts}X+E^{\ts}E+\mathfrak R_z(\nabla f,\nabla f).
	\eeaa
\end{lem}
\begin{proof}
We expand the two terms in lemma \ref{gamma z 2} .
	\beaa
	&&\frac{1}{2}(a^{\ts}\nabla\circ  (a^{\ts}\nabla |z^{\ts}\nabla f|^2 ))\\
	&=&\frac{1}{2}\sum_{i=1}^n\sum_{k=1}^m(a^{\ts}\nabla)_i(a^{\ts}\nabla)_i|(z^{\ts}\nabla)_kf|^2\\
	&=&\sum_{i=1}^n\sum_{k=1}^m(a^{\ts}\nabla)_i\la (a^{\ts}\nabla)_i (z^{\ts}\nabla)_kf,(z^{\ts}\nabla)_kf\ra_{\hR^m}\\
	&=&\sum_{i=1}^n\sum_{k=1}^m\la (a^{\ts}\nabla)_i (z^{\ts}\nabla)_kf,(a^{\ts}\nabla)_i(z^{\ts}\nabla)_kf\ra_{\hR^m}\cdots \widetilde{\textbf{T}}_1\\
	&&+\sum_{i=1}^n\sum_{k=1}^m \la (a^{\ts}_{ii'}\frac{\partial}{\partial x_{i'}}) (a^{\ts}_{i \hat i}\frac{\partial}{\partial x_{\hat i}}) (z^{\ts}_{k\hat k}\frac{\partial}{\partial x_{\hat k}})f ,(z^{\ts}\nabla)_kf\ra_{\hR^m} \cdots \widetilde{\textbf{R}}_1.
	\eeaa
	\beaa
	&&\la z^{\ts} \nabla ([(a^{\ts}\nabla)\circ (a^{\ts}\nabla f)]),z^{\ts}\nabla f\ra_{\hR^m}\\
	&=&\sum_{i=1}^n \sum_{k=1}^m\la(z^{\ts}\nabla)_k[(a^{\ts}\nabla)_i(a^{\ts}\nabla)_if ],(z^{\ts}\nabla)_kf\ra_{\hR^m} \\
	&=&\sum_{i=1}^n\sum_{k=1}^m \la(z^{\ts}_{k\hat k}\frac{\partial}{\partial x_{\hat k}}) [(a^{\ts}_{ii'}\frac{\partial}{\partial x_{i'}}) (a^{\ts}_{i \hat i}\frac{\partial}{\partial x_{\hat i}}) f ],(z^{\ts}\nabla)_kf\ra_{\hR^m}  \cdots \widetilde{\textbf{R}}_2.
	\eeaa
	Next, we expand $\widetilde{\textbf{R}}_1$ and $\widetilde{\textbf{R}}_2$ completely and get the following,
	\beaa
	\widetilde{\textbf{R}}_1&=&\sum_{i=1}^n\sum_{k=1}^m\sum_{i',\hat i,\hat k=1}^{n+m} \la (a^{\ts}_{ii'}\frac{\partial}{\partial x_{i'}}) (a^{\ts}_{i \hat i}\frac{\partial}{\partial x_{\hat i}}) (z^{\ts}_{k\hat k}\frac{\partial f}{\partial x_{\hat k}}) ,(z^{\ts}\nabla)_kf\ra_{\hR^m} \\
	&=&\sum_{i=1}^n\sum_{k=1}^m\sum_{i',\hat i,\hat k=1}^{n+m} \la a^{\ts}_{ii'} (\frac{\partial a^{\ts}_{i \hat i}}{\partial x_{i'}} \frac{\partial z^{\ts}_{k\hat k}}{\partial x_{\hat i}}\frac{\partial f}{\partial x_{\hat k}}) ,(z^{\ts}\nabla)_kf\ra_{\hR^m} \cdots \widetilde{\textbf{R}}_1^1
	\eeaa 
	\beaa 
	&&+\sum_{i=1}^n\sum_{k=1}^m\sum_{i',\hat i,\hat k=1}^{n+m} \la a^{\ts}_{ii'}a^{\ts}_{i \hat i} (\frac{\partial }{\partial x_{i'}} \frac{\partial z^{\ts}_{k\hat k}}{\partial x_{\hat i}})(\frac{\partial f}{\partial x_{\hat k}}) ,(z^{\ts}\nabla)_kf\ra_{\hR^m}  \cdots \widetilde{\textbf{R}}_1^2\\
	&&+\sum_{i=1}^n\sum_{k=1}^m\sum_{i',\hat i,\hat k=1}^{n+m} \la a^{\ts}_{ii'}a^{\ts}_{i \hat i} ( \frac{\partial z^{\ts}_{k\hat k}}{\partial x_{\hat i}})(\frac{\partial }{\partial x_{i'}}\frac{\partial f}{\partial x_{\hat k}}) ,(z^{\ts}\nabla)_kf\ra_{\hR^m}  \cdots \widetilde{\textbf{R}}_1^3\\
	&&+\sum_{i=1}^n\sum_{k=1}^m\sum_{i',\hat i,\hat k=1}^{n+m} \la (a^{\ts}_{ii'}) ((\frac{\partial}{\partial x_{i'}}a^{\ts}_{i \hat i}) z^{\ts}_{k\hat k}\frac{\partial}{\partial x_{\hat i}}\frac{\partial f}{\partial x_{\hat k}}) ,(z^{\ts}\nabla)_kf\ra_{\hR^m}  \cdots \widetilde{\textbf{R}}_1^4\\
	&&+\sum_{i=1}^n\sum_{k=1}^m\sum_{i',\hat i,\hat k=1}^{n+m} \la a^{\ts}_{ii'}a^{\ts}_{i \hat i}(\frac{\partial}{\partial x_{i'}} z^{\ts}_{k\hat k})\frac{\partial}{\partial x_{\hat i}}\frac{\partial f}{\partial x_{\hat k}}) ,(z^{\ts}\nabla)_kf\ra_{\hR^m}  \cdots \widetilde{\textbf{R}}_1^5\\
	&&+\sum_{i=1}^n\sum_{k=1}^m\sum_{i',\hat i,\hat k=1}^{n+m} \la a^{\ts}_{ii'} a^{\ts}_{i \hat i}  z^{\ts}_{k\hat k}(\frac{\partial}{\partial x_{i'}}\frac{\partial}{\partial x_{\hat i}}\frac{\partial f}{\partial x_{\hat k}}) ,(z^{\ts}\nabla)_kf\ra_{\hR^m}  \cdots \widetilde{\textbf{R}}_1^6\\
	\eeaa
	\beaa
\widetilde{\textbf{R}}_2&=&	\sum_{i=1}^n\sum_{k=1}^m\sum_{i',\hat i,\hat k=1}^{n+m} \la(z^{\ts}_{k\hat k}\frac{\partial}{\partial x_{\hat k}}) [(a^{\ts}_{ii'}\frac{\partial}{\partial x_{i'}}) (a^{\ts}_{i \hat i}\frac{\partial f}{\partial x_{\hat i}})  ],(z^{\ts}\nabla)_kf\ra_{\hR^m} \\
&=&\sum_{i=1}^n\sum_{k=1}^m\sum_{i',\hat i,\hat k=1}^{n+m} \la z^{\ts}_{k\hat k}\frac{\partial a^{\ts}_{ii'}}{\partial x_{\hat k}} \frac{\partial a^{\ts}_{i \hat i}}{\partial x_{i'}} \frac{\partial f}{\partial x_{\hat i}})  ,(z^{\ts}\nabla)_kf\ra_{\hR^m}  \cdots \widetilde{\textbf{R}}_2^1\\
&&+\sum_{i=1}^n\sum_{k=1}^m\sum_{i',\hat i,\hat k=1}^{n+m} \la z^{\ts}_{k\hat k} a^{\ts}_{ii'} (\frac{\partial }{\partial x_{\hat k}} \frac{\partial a^{\ts}_{i \hat i}}{\partial x_{i'}}) \frac{\partial f}{\partial x_{\hat i}}  ,(z^{\ts}\nabla)_kf\ra_{\hR^m} \cdots \widetilde{\textbf{R}}_2^2\\
&&+\sum_{i=1}^n\sum_{k=1}^m\sum_{i',\hat i,\hat k=1}^{n+m} \la z^{\ts}_{k\hat k} a^{\ts}_{ii'} \frac{\partial a^{\ts}_{i \hat i}}{\partial x_{i'}}(\frac{\partial }{\partial x_{\hat k}}  \frac{\partial f}{\partial x_{\hat i}})  ,(z^{\ts}\nabla)_kf\ra_{\hR^m} \cdots \widetilde{\textbf{R}}_2^3=\widetilde{\textbf{R}}_1^4\\
&&+\sum_{i=1}^n\sum_{k=1}^m\sum_{i',\hat i,\hat k=1}^{n+m} \la z^{\ts}_{k\hat k}\frac{\partial a^{\ts}_{ii'}}{\partial x_{\hat k}} a^{\ts}_{i \hat i}(\frac{\partial }{\partial x_{i'}} \frac{\partial f}{\partial x_{\hat i}})  ,(z^{\ts}\nabla)_kf\ra_{\hR^m} \cdots \widetilde{\textbf{R}}_2^4\\ 
&&+\sum_{i=1}^n\sum_{k=1}^m\sum_{i',\hat i,\hat k=1}^{n+m} \la z^{\ts}_{k\hat k}a^{\ts}_{ii'}\frac{\partial a^{\ts}_{i \hat i}}{\partial x_{\hat k}} (\frac{\partial }{\partial x_{i'}} \frac{\partial f}{\partial x_{\hat i}})  ,(z^{\ts}\nabla)_kf\ra_{\hR^m} \cdots \widetilde{\textbf{R}}_2^5\\ 
&&+\sum_{i=1}^n\sum_{k=1}^m\sum_{i',\hat i,\hat k=1}^{n+m} \la z^{\ts}_{k\hat k}a^{\ts}_{ii'}a^{\ts}_{i \hat i}(\frac{\partial }{\partial x_{\hat k}} \frac{\partial }{\partial x_{i'}} \frac{\partial f}{\partial x_{\hat i}})  ,(z^{\ts}\nabla)_kf\ra_{\hR^m} \cdots \widetilde{\textbf{R}}_2^6=\widetilde{\textbf{R}}_1^6
	\eeaa
Our next step is to complete squares for all the above terms. We look at term $\widetilde{\textbf{T}}_1$ first. 
	\beaa
	\widetilde{\textbf{T}}_1&=&\sum_{i=1}^n\sum_{k=1}^m\left\la \sum_{\hat i,\hat k=1}^{n+m}a^{\ts}_{i\hat i}z^{\ts}_{k\hat k}\frac{\partial^2 f}{\partial x_{\hat i}\partial x_{ \hat k}}+\sum_{\hat i,\hat k=1}^{n+m}a^{\ts}_{i\hat i}\frac{\partial z^{\ts}_{k\hat k}}{\partial x_{\hat i}}\frac{\partial f}{\partial x_{ \hat k}},\right.\\
	&&\left.\quad\quad\quad\quad\sum_{ i', k'=1}^{n+m}a^{\ts}_{i i'}z^{\ts}_{k k'}\frac{\partial^2 f}{\partial x_{ i'}\partial x_{  k'}}+\sum_{ i', k'=1}^{n+m}a^{\ts}_{i i'}\frac{\partial z^{\ts}_{k k'}}{\partial x_{ i'}}\frac{\partial f}{\partial x_{ k'}} \right\ra\\
	&=&\sum_{i=1}^n\sum_{k=1}^m\sum_{k=1}^m\left\la \sum_{\hat i,\hat k=1}^{n+m}a^{\ts}_{i\hat i}z^{\ts}_{k\hat k}\frac{\partial^2 f}{\partial x_{\hat i}\partial x_{ \hat k}},\sum_{ i', k'=1}^{n+m}a^{\ts}_{i i'}z^{\ts}_{k k'}\frac{\partial^2 f}{\partial x_{ i'}\partial x_{  k'}} \right\ra \cdots \widetilde{\textbf{T}}_{1a}
		\eeaa 
	\beaa
	&&+\sum_{i=1}^n\sum_{k=1}^m\left\la \sum_{\hat i,\hat k=1}^{n+m}a^{\ts}_{i\hat i}z^{\ts}_{k\hat k}\frac{\partial^2 f}{\partial x_{\hat i}\partial x_{ \hat k}},\sum_{ i', k'=1}^{n+m}a^{\ts}_{i i'}\frac{\partial z^{\ts}_{k k'}}{\partial x_{ i'}}\frac{\partial f}{\partial x_{ k'}} \right\ra\cdots \widetilde{\textbf{T}}_{1b}\\ 
	&&+\sum_{i=1}^n\sum_{k=1}^m\left\la \sum_{\hat i,\hat k=1}^{n+m}a^{\ts}_{i\hat i}\frac{\partial z^{\ts}_{k\hat k}}{\partial x_{\hat i}}\frac{\partial f}{\partial x_{ \hat k}} ,\sum_{ i', k'=1}^{n+m}a^{\ts}_{i i'}z^{\ts}_{k k'}\frac{\partial^2 f}{\partial x_{ i'}\partial x_{  k'}}\right\ra\cdots \widetilde{\textbf{T}}_{1c}\\
	&&+\sum_{i=1}^n\sum_{k=1}^m\left\la \sum_{\hat i,\hat k=1}^{n+m}a^{\ts}_{i\hat i}\frac{\partial z^{\ts}_{k\hat k}}{\partial x_{\hat i}}\frac{\partial f}{\partial x_{ \hat k}} ,\sum_{ i', k'=1}^{n+m}a^{\ts}_{i i'}\frac{\partial z^{\ts}_{k k'}}{\partial x_{ i'}}\frac{\partial f}{\partial x_{ k'}} \right\ra\cdots \widetilde{\textbf{T}}_{1d}.
	\eeaa

The terms $\widetilde{\textbf{T}}_{1b}=\widetilde{\textbf{T}}_{1c}$, $\widetilde{\textbf{R}}_{1}^3=\widetilde{\textbf{R}}_1^5$ and $\widetilde{\textbf{R}}_2^5=\widetilde{\textbf{R}}_2^4$ plays the role of crossing terms inside the complete squares. In particular, for convenience, we change the index inside the sum of $\widetilde{\textbf{R}}_1^3$ and $\widetilde{\textbf{R}}_2^5$, switch $i',\hat i$ for $\widetilde{\textbf{R}}_1^3$ and switch $i',\hat k$ for $\widetilde{\textbf{R}}_2^5$, then we get following. 
	\beaa
	2\widetilde{\textbf{R}}_1^3&=& 2\sum_{i=1}^n\sum_{k=1}^m\sum_{i',\hat i,\hat k=1}^{n+m} \la a^{\ts}_{i\hat i}a^{\ts}_{i i'} ( \frac{\partial z^{\ts}_{k\hat k}}{\partial x_{i'}})(\frac{\partial }{\partial x_{\hat i}}\frac{\partial f}{\partial x_{\hat k}}) ,(z^{\ts}\nabla)_kf\ra_{\hR^n} \\
	&=&2\sum_{i=1}^n\sum_{k=1}^m\sum_{\hat i,\hat k=1}^{n+m} \sum_{i',l=1}^{n+m} \left( a^{\ts}_{i\hat i}a^{\ts}_{i i'} ( \frac{\partial z^{\ts}_{k\hat k}}{\partial x_{i'}})(\frac{\partial }{\partial x_{\hat i}}\frac{\partial f}{\partial x_{\hat k}})z^{\ts}_{kl} \frac{\pa f}{\pa x_l} \right) \\
	-2 \widetilde{\textbf{R}}_2^5&=&-2\sum_{i=1}^n\sum_{k=1}^m\sum_{i',\hat i,\hat k=1}^{n+m} \la z^{\ts}_{k i'}a^{\ts}_{i\hat k}\frac{\partial a^{\ts}_{i \hat i}}{\partial x_{i'}} (\frac{\partial }{\partial x_{\hat k}} \frac{\partial f}{\partial x_{\hat i}})  ,(z^{\ts}\nabla)_kf\ra \\
	&=&-2\sum_{i=1}^n\sum_{k=1}^m \sum_{\hat i,\hat k=1}^{n+m}\sum_{i',l=1}^{n+m} \left( z^{\ts}_{k i'}a^{\ts}_{i\hat k}\frac{\partial a^{\ts}_{i \hat i}}{\partial x_{i'}} (\frac{\partial }{\partial x_{\hat k}} \frac{\partial f}{\partial x_{\hat i}})z^{\ts}_{kl} \frac{\pa f}{\pa x_l}\right)
	\eeaa

We denote 
\bea\label{vector z 2}
\sum_{\hat i,\hat k=1}^{n+m}a^{\ts}_{i\hat i}z^{\ts}_{k\hat k}\frac{\pa^2 f}{\pa x_{\hat i} \pa x_{\hat k}}&=&\omega_{ik}.
\eea
The above equality \eqref{vector z 2} can be represented in the following matrix form
\beaa
P_{(n*m)\times(n+m)^2}X_{(n+m)^2\times 1}=(\omega_{11},\cdots,\omega_{ik},\cdots,\omega_{nm})^{\ts}_{(n*m)\times 1}
\eeaa
where $P$ and $X$ are defined in \eqref{matrix P} and \eqref{vector X}.
Now, we can represent term $\widetilde{\textbf{T}}_{1a}$ as $
\sum_{i=1}^n\sum_{k=1}^m \omega_{ik}^2= \omega^{\ts}\omega=(PX)^{\ts}PX=X^{\ts}P^{\ts}PX$.
Next we want to represent $\widetilde{\textbf{R}}_1^3$ and $\widetilde{\textbf{R}}_2^5$ in the following form in terms of vector $X$, 
\beaa
&&2\widetilde{\textbf{R}}_1^3-2\widetilde{\textbf{R}}_2^5\\
&=&2\sum_{i,k=1}^n\sum_{i',\hat i,\hat k=1}^{n+m} \la a^{\ts}_{i\hat i}a^{\ts}_{i i'} ( \frac{\partial z^{\ts}_{k\hat k}}{\partial x_{i'}})(\frac{\partial }{\partial x_{\hat i}}\frac{\partial f}{\partial x_{\hat k}}) ,(z^{\ts}\nabla)_kf\ra_{\hR^n}\\
&&-2\sum_{i,k=1}^n\sum_{i',\hat i,\hat k=1}^{n+m} \la z^{\ts}_{k i'}a^{\ts}_{i\hat k}\frac{\partial a^{\ts}_{i \hat i}}{\partial x_{i'}} (\frac{\partial }{\partial x_{\hat k}} \frac{\partial f}{\partial x_{\hat i}})  ,(z^{\ts}\nabla)_kf\ra 
\eeaa 
\beaa 
&=&2\sum_{\hat i,\hat k=1}^{n+m}\left[\sum_{i,k=1}^n\sum_{i'=1}^{n+m}\left(\la a^{\ts}_{i\hat i}a^{\ts}_{i i'} ( \frac{\partial z^{\ts}_{k\hat k}}{\partial x_{i'}}) ,(z^{\ts}\nabla)_kf\ra - \la z^{\ts}_{k i'}a^{\ts}_{i\hat k}\frac{\partial a^{\ts}_{i \hat i}}{\partial x_{i'}}   ,(z^{\ts}\nabla)_kf\ra\right)\right](\frac{\partial }{\partial x_{\hat i}}\frac{\partial f}{\partial x_{\hat k}})\\
&=&2F^{\ts}X,
\eeaa
where $F$ is defined in \eqref{vector F}.
Similarly, we can represent $\widetilde{\textbf{T}}_{1b}=\widetilde{\textbf{T}}_{1c}$ by $X$,
\beaa
\widetilde{\textbf{T}}_{1b}=\widetilde{\textbf{T}}_{1c}&=&\sum_{i,k=1}^n\la \sum_{\hat i,\hat k=1}^{n+m}a^{\ts}_{i\hat i}\frac{\partial z^{\ts}_{k\hat k}}{\partial x_{\hat i}}\frac{\partial f}{\partial x_{ \hat k}} ,\sum_{ i', k'=1}^{n+m}a^{\ts}_{i i'}z^{\ts}_{k k'}\frac{\partial^2 f}{\partial x_{ i'}\partial x_{  k'}} \ra=E^{\ts}PX
\eeaa
where $E$ is defined in \eqref{vector E}. We thus have the following form,  
\beaa
\widetilde{\textbf{T}}_1+2\widetilde{\textbf{R}}_1^3-2\widetilde{\textbf{R}}_2^5=X^{\ts}P^{\ts}PX+2E^{\ts}PX+2F^{\ts}X+E^{\ts}E
\eeaa
Taking into account the fact that $\textbf{R}_1^6-\textbf{R}_2^6=0$ and $\textbf{R}_1^4-\textbf{R}_2^3=0$, we have 
\beaa
\widetilde{\textbf{T}}_1+\widetilde{\textbf{R}}_1-\widetilde{\textbf{R}}_2=\widetilde{\textbf{T}}_1+2\widetilde{\textbf{R}}_1^3-2\widetilde{\textbf{R}}_2^5+\widetilde{\textbf{R}}_1^1+\widetilde{\textbf{R}}_1^2-\widetilde{\textbf{R}}_2^1-\widetilde{\textbf{R}}_2^2,
\eeaa	
which completes the proof.  \qed
\end{proof}


\section{Further discussions on other inequalities}\label{sec:func inq}
In this section, we apply the generalized Gamma calculus to study the entropic inequality for the semi-group $P_t$ associated with the drift-diffusion process. With a little abuse of notations, we denote the generator of the semi-group $P_t$ as $\frac{1}{2}L$ instead of $L$ and we denote $X_t$ as the corresponding diffusion process.
\begin{defn}\label{def: diffusion semi group}
We define the semigroup $P_t=e^{\frac{1}{2}tL}$, where $L$ is invariant w.r.t the invariant measure $d\mu=\pi(x)dx$. We denote $P_tf(x)=\hE(f(X_t))$, and 
\beaa
\hE(f(X_t))&=&\int_{\hM^{n+m}} f(y)p(t,x,y)d\mu(y)
=\int_{\hM^{n+m}}f(y)\rho(t,x,y)dy,
\eeaa
where the infinitesimal generator of this process $X_t$ is $\frac{1}{2}L$ and we denote $\rho(t,\cdot,\cdot)$ as the product of the transition kernel $p(t,\cdot,\cdot)$ and the volume measure $\Vol$.
\end{defn}
\begin{rem}
Following the standard treatment as in \cite{BaudoinGarofalo09}[Section 5], whenever we consider differentiating operation on $P_tf$, we shall always consider $P_tf_{\varepsilon}$ first with $f_{\varepsilon}=f+\varepsilon$, for $\forall \varepsilon>0$. Then we take the limit as $\varepsilon\rightarrow 0$. Throughout this section, we will directly use $P_tf$ instead of $P_tf_{\varepsilon}$ for convenience. 
\end{rem}
\begin{rem}
In the standard sub-Riemannian setting, the semi-groups is in general defined with respect to the invariant measure $d\mu(y)$. In this paper, we formulate the semi-group and the transition kernel with respect to the Lebesgue measure $dy$.
\end{rem}{}
Following the framework in \cite{BaudoinGarofalo09}, we also need the following assumption which is necessary to rigorously justify computations on functionals of the heat semigroup.
\begin{assum}\label{stochastic complete}
    The semigroup $P_t$ is stochastically complete that is, for $t\ge 0$, $P_t \textbf{1}=\textbf{1}$ and for any $T>0$ and $f\in\mathcal C^{\infty}(\hM^{n+m})$ with compact support, we assume that
    \bea
    \sup_{t\in [0,T]}\|\Gamma(P_tf)\|_{\infty}+\|\Gamma_1^z(P_tf)\|_{\infty}<+\infty.
    \eea
\end{assum}{}
We believe that the above Assumption \ref{stochastic complete} should follow from the the assumption $\mathfrak R\ge\kappa ( \Gamma_1+\Gamma_1^z)$ if we assume the appropriate lower bound $\kappa$. We leave this for further studies.  Related gradient estimates are presented in order below.  For the infinitesimal generator $\frac{1}{2}L$ associated with linear semi-group $P_t$, we have the following property.
\begin{prop}For all smooth function $f$, we have 
\begin{itemize}
    \item $P_0=Id$; 
    \item For all functions $f\in \cC_b(\hR^{n+m})$, the map $t\mapsto P_tf$ is continuous from $\hR^{+}$ to $L^2(d\mu)$;
    \item For all $s,t\ge 0,$ one has $P_t\circ P_s=P_{t+s}$;
    \item $\forall x\in \hR^{n+m}$, $\forall t\ge 0,$ $\frac{\partial}{\partial_t}P_tf(x)=\frac{1}{2}L(P_tf)(x)=\frac{1}{2}P_t(Lf)(x)$.
\end{itemize}
\end{prop}

Next, we present the entropic inequality under the Assumption \ref{assumption:main result}. We follow closely the framework introduced in \cite{BaudoinGarofalo09} and define the following two functionals, 
\beaa
\phi_a(x,t)&=P_{T-t}f\Gamma_{1}(\log P_{T-t}f)(x),\quad \text{and}\quad \phi_z(x,t)&=P_{T-t}f\Gamma_1^z(\log P_{T-t}f)(x).
\eeaa
\begin{lem}
	We have the following relation
	\bea\label{phi z} \frac{1}{2}L\phi_a+\frac{\pa}{\pa t}\phi_a&=&(P_{T-t}f)(x)\Gamma_2(\log P_{T-t}f,\log P_{T-t}f)(x),\\
	\label{phi a}
	 \frac{1}{2}L\phi_z+\frac{\pa}{\pa t}\phi_z&=&(P_{T-t}f)(x)\Gamma_2^z(\log P_{T-t}f,\log P_{T-t}f)(x)\nonumber \\
	 &&+(P_{T-t}f)(x)\Gamma_{1}(\log P_{T-t}f,\Gamma_1^z(P_{T-t}f,P_{T-t}f))(x)\nonumber \\
	 &&-(P_{T-t}f)(x)\Gamma_1^z(\log P_{T-t}f,\Gamma_{1}(P_{T-t}f,P_{T-t}f))(x). 
\eea
\end{lem}
\begin{proof}
Denote $g(t,x)=P_{T-t}f(x)=\int \rho(t,x,\tilde x)f(\tilde x)d\tilde x$, we have the following relation
\beaa 
L(\log g)=-\frac{\Gamma_{1}(g,g)}{(g)^2}-2\frac{\pa_t g}{g}. 
\eeaa
By direct computation, one obtains
\beaa
\pa_t\phi_a&=&\pa_t g\Gamma_{1}(\log g,\log g)+2g\la a^{\ts}\nabla \log g, a^{\ts}\nabla(\frac{\pa_t g}{g})\ra_{\hR^n}\\
&=&-\frac{1}{2}Lg\Gamma_{1}(\log g,\log g)-g\Gamma_{1}(\log g,L\log g)-g\Gamma_{1}(\log g, \Gamma_{1}(\log g,\log g)),\\
\frac{1}{2}L\phi_a&=&\frac{1}{2}Lg\Gamma_{1}(\log g,\log g)+\frac{1}{2}gL\Gamma_{1}(\log g,\log g)+\Gamma_{1}(g,\Gamma_{1}(\log g,\log g)),
\eeaa
where we have $\Gamma_{1}(g,\Gamma_{1}(\log g,\log g))=g\Gamma_{1}(\log g,\Gamma_{1}(\log g,\log g))$, thus \eqref{phi a} is proved. Similarly, we obtain the following for $\phi_z$
\beaa
\pa_t\phi_z&=&\pa_t g\Gamma_1^z(\log g,\log g)+2g\la z^{\ts}\nabla \log g, z^{\ts}\nabla(\frac{\pa_t g}{g})\ra_{\hR^m}\\
&=&-\frac{1}{2}Lg\Gamma_1^z(\log g,\log g)-g\Gamma_1^z(\log g,L\log g)-g\Gamma_1^z(\log g, \Gamma_{1}(\log g,\log g)),\\
\frac{1}{2}L\phi_z&=&\frac{1}{2}Lg\Gamma_1^z(\log g,\log g)+\frac{1}{2}gL\Gamma_1^z(\log g,\log g)+\Gamma_{1}(g,\Gamma_1^z(\log g,\log g)).
\eeaa
The proof then follows. \qed 
\end{proof}
Now, we are ready to present the following important lemma which prepares us to prove the new entropy inequality without the assumption: $$\Gamma_{1}(\log P_{T-t}f,\Gamma_1^z(P_{T-t}f,P_{T-t}f))(x)=\Gamma_1^z(\log P_{T-t}f,\Gamma_{1}(P_{T-t}f,P_{T-t}f))(x).$$
\begin{lem}\label{non-commutative lemma} For any $0<s<T$, we denote $\rho(s,x,y)=p(s,x,y)\Vol(y)$ as the transition kernel of diffusion process $X^x_s$ starting at $x$ defined in Definition \ref{def: diffusion semi group}, the following equality is satisfied
	\beaa
	&&\mathbb E[g\Gamma_{1}(\log g,\Gamma_1^z(\log g,\log g))-g\Gamma_1^z(\log g,\Gamma_{1}(\log g,\log g)) ]\\
	&=&\int \frac{\nabla\cdot (\rho(s,x,y)zz^{\ts}\Gamma_{\nabla(aa^{\ts})}(\log g(s,y),\log g(s,y)))}{\rho(s,x,y)} g(s,y)\rho(s,x,y)dy\\
	&&-\int \frac{\nabla\cdot (\rho(s,x,y) aa^{\ts}\Gamma_{\nabla(zz^{\ts})}(\log g(s,y),\log g(s,y)))}{\rho(s,x,y)} g(s,y)\rho(s,x,y)dy.
	\eeaa
Here we denote	$g(s,y)=P_{T-s}f(y)=\int \rho(s,y,\tilde y)f(\tilde y)d\tilde y$ and 
\beaa
&&\mathbb E[g\Gamma_{1}(\log g,\Gamma_1^z(\log g,\log g))]\\
&=&\mathbb E[g(s,X_s)\Gamma_{1}(\log g(s,X^x_s),\Gamma_1^z(\log g(s,X^x_s),\log g(s,X^x_s)))]\\
&=&\mathbb \int g(s,y)\Gamma_{1}(\log g(s,y),\Gamma_1^z(\log g(s,y),\log g(s,y)))\rho(s,x,y)dy.
\eeaa

\end{lem}
\begin{proof}
We first expand in the following integral form. 
\beaa
&&\mathbb E[g\Gamma_{1}(\log g,\Gamma_1^z(\log g,\log g))-g\Gamma_1^z(\log g,\Gamma_{1}(\log g,\log g)) ]\\
&=&\int g(s,y)\Gamma_{1}(\log g(s,y),\Gamma_1^z(\log g(s,y),\log g(s,y)))\rho(s,x,y) dy\\
&&-\int g(s,y)\Gamma_1^z(\log g(s,y),\Gamma_{1}(\log g(s,y),\log g(s,y)))\rho(s,x,y) dy.
\eeaa
We skip $x,y,s$ for simplicity.  Take $\log g=h$, 

\noindent\textbf{Claim 1}:
\beaa
&&\int \Gamma_{1}(h,\Gamma_1^z(h,h))\rho gdy-\int \Gamma_1^z(h,\Gamma_{1}(h,h))\rho gdy\\
&=&\int \Gamma_1^z(h,\Delta_a h )\rho gdy-\int \Gamma_1^z(h,\frac{\Delta_a g}{g})\rho gdy 
-\int \Gamma_{1}(h,\Delta_z h )\rho gdy+\int \Gamma_{1}(h,\frac{\Delta_z g}{g})\rho gdy.
\eeaa
Recall that we denote $\Delta_a=\nabla\cdot(aa^{\ts}\nabla)$ and $\Delta_z=\nabla\cdot(zz^{\ts}\nabla)$. Use the following identity 
\beaa
\Delta_a h=\frac{\Delta_a g}{g}-\frac{\Gamma_{1}(g,g)}{g^2},\quad \text{and}\quad 
\Delta_z h=\frac{\Delta_z g}{g}-\frac{\Gamma_1^z(g,g)}{g^2}.
\eeaa
We then get 
\beaa
\int \Gamma_1^z(h,\Delta_a h) \rho g dy&=&\int \Gamma_1^z\left(h,\frac{\Delta_a g}{g}-\frac{\Gamma_{1}(g,g)}{g^2}\right) \rho g dy\\
&=&-\int \Gamma_1^z(h,\Gamma_{1}(h,h) )\rho gdy+\int \Gamma_1^z(h,\frac{\Delta_a g}{g})\rho gdy.
\eeaa
Similarly, the other equality is satisfied. 

\noindent \textbf{Claim 2:}
\beaa
&&\int \Gamma_1^z(h,\Delta_a h )\rho gdy-\int \Gamma_1^z(h,\frac{\Delta_a g}{g})\rho gdy-\int \Gamma_{1}(h,\Delta_z h )\rho gdy+\int \Gamma_{1}(h,\frac{\Delta_z g}{g})\rho gdy\\
&=&\int \frac{\nabla\cdot (\rho zz^{\ts}\Gamma_{\nabla(aa^{\ts})}(h,h))}{\rho} g\rho dy-\int \frac{\nabla\cdot (\rho aa^{\ts}\Gamma_{\nabla(zz^{\ts})}(h,h))}{\rho} g\rho dy.
\eeaa
First observe that 
\beaa
\int \Gamma_1^z(h,\frac{\Delta_a g}{g})\rho gdy&=&\int \la zz^{\ts}\nabla h,\nabla ( \frac{\Delta_a g}{g})\ra \rho gdy
=-\int \nabla\cdot (\rho zz^{\ts}\nabla g)\frac{\Delta_a g}{g}dy\\
&=&-\int \frac{\rho}{g}\Delta_a g\Delta_z g dy-\int \la \nabla \rho,zz^{\ts}\nabla g\ra \frac{\Delta_a g}{g}dy.
\eeaa
Similarly, one gets 
\beaa
\int \Gamma_{1}(h,\frac{\Delta_z g}{g})\rho gdy
&=&-\int \frac{\rho }{g}\Delta_a g\Delta_z g dy-\int \la \nabla \rho ,aa^{\ts}\nabla g\ra \frac{\Delta_z g}{g}dy.
\eeaa
For the next term, one obtains 
\beaa
&&\int \Gamma_1^z(h,\Delta_a h) \rho g dy\\
&=&\int \la \nabla(\nabla\cdot(aa^{\ts}\nabla h)),zz^{\ts}\nabla h\ra \rho g dy
=-\int[\nabla\cdot(aa^{\ts}\nabla h)][\nabla\cdot (\rho gzz^{\ts}\nabla h)]dy\\
&=&-\int [\nabla\cdot (aa^{\ts} \frac{1}{g}\nabla g) ]\nabla\cdot (\rho zz^{\ts}\nabla g)dy
=-\int \left(\la \nabla \frac{1}{g},aa^{\ts}\nabla g\ra +\frac{1}{g}\Delta_a g \right)\nabla\cdot (\rho zz^{\ts}\nabla g)dy\\
&=&\int \la  \frac{1}{g^2}\nabla g,aa^{\ts}\nabla g\ra(\nabla\cdot (\rho zz^{\ts}\nabla g))dy -\int\frac{1}{g}\Delta_a g( \nabla\cdot (\rho zz^{\ts}\nabla g))dy\\
&=&-2\int \nabla^2h(aa^{\ts}\nabla h,zz^{\ts}\nabla h)\rho gdy-\int \la  \la \nabla h, \nabla(aa^{\ts})\nabla h\ra , zz^{\ts}\nabla h\ra  \rho gdy\\
&&-\int\frac{1}{g}\Delta_a g \la\nabla \rho ,zz^{\ts}\nabla g\ra dy-\int\frac{\rho }{g}\Delta_a g \Delta_z gdy,
\eeaa
where the last equality follows from integration by parts for the first term and direct expansion of the divergence for the second term. Similarly, we obtain 
\beaa
\int \Gamma_{1}(h,\Delta_z h) \rho g dy&=&-2\int \nabla^2h(zz^{\ts}\nabla h,aa^{\ts}\nabla h)\rho gdy-\int \la  \la \nabla h, \nabla(zz^{\ts})\nabla h\ra , aa^{\ts}\nabla h\ra  \rho gdy\\
&&-\int\frac{1}{g}\Delta_z g \la\nabla \rho ,aa^{\ts}\nabla g\ra dy-\int\frac{\rho }{g}\Delta_a g \Delta_z gdy.
\eeaa
Observing that by integration by parts, we get 
\beaa
&&-\int \la  \la \nabla h, \nabla(aa^{\ts})\nabla h\ra , zz^{\ts}\nabla h\ra  \rho gdy+\int \la  \la \nabla h, \nabla(zz^{\ts})\nabla h\ra , aa^{\ts}\nabla h\ra  \rho gdy\\
&=&\int \frac{\nabla\cdot (\rho zz^{\ts}\Gamma_{\nabla(aa^{\ts})}(h,h))}{\rho } g\rho dy-\int \frac{\nabla\cdot (\rho aa^{\ts}\Gamma_{\nabla(zz^{\ts})}(h,h))}{\rho } g\rho dy.
\eeaa
Combining the above formulas, the proof is completed. 
	\qed
\end{proof}
With the above Lemma in hand, we are ready to prove the following entropic inequality. We first define the following engery form
\beaa
\Phi_a(x,t)&=P_t\left(P_{T-t}f\Gamma_{1}(\log P_{T-t}f)\right)(x),\quad 
\Phi_z(x,t)&=P_t\left(P_{T-t}f\Gamma_1^z(\log P_{T-t}f)\right)(x).
\eeaa
Recall that, we define 
\beaa
\phi_a(x,t)&=P_{T-t}f\Gamma_{1}(\log P_{T-t}f)(x),\quad \text{and}\quad \phi_z(x,t)&=P_{T-t}f\Gamma_1^z(\log P_{T-t}f)(x).
\eeaa
\begin{thm}\label{thm:entropic inequality}
Denote $\phi=\phi_a+\phi_z$, if the following condition is satisfied 
\begin{equation*}
\mathfrak{R}\succeq \kappa ( \Gamma_1+\Gamma_1^z),
\end{equation*}
we then conclude 
	\bea
	P_T(\phi(\cdot,T))(x)&\ge &\phi(x,0)+ \int_0^{\ts}\kappa_s(\Phi_a(x,s)+\Phi_z(x,s))ds,
	\eea
	where $\kappa_s$ depends on the estimate of transition kernel $\nabla \log \rho(s,\cdot,\cdot)$ associated with semi-group $P_s($see Definition \ref{def: diffusion semi group}$)$.
\end{thm}

\begin{rem}
	Based on Theorem \ref{thm: generalized gamma 2 z with drift}, we can also prove the above theorem for operator $\widetilde L$ with drift term involved. Since the proof is similar, we skip the proof here.  
\end{rem}
\begin{proof}
Take $\phi=\phi_a+\phi_z$. Let $(X_t^x)_{t\ge 0}$ be the diffusion Markov process with semigroup $P_t$.(Similar proofs are referred to \cite{BaudoinGarofalo09}[Proposition 4.5]) Let smooth function $u:\hR^{n+m}\rightarrow \hR$ be such that for every $T>0$, $\sup_{t\in [0,T]}\|u(t,\cdot)\|_{\infty}<\infty$ and $\sup_{t\in[0,T]}\|\frac{1}{2}Lu(t,\cdot)+\pa_t u(t,\cdot)\|_{\infty}<\infty$. We have for every $t>0$,
\beaa
u(t,X_t^x)&=&u(0,x)+\int_0^{\ts}(\frac{1}{2}Lu+\pa_s u)(s,X_s^x)ds+M_t,
\eeaa
where $(M_t)_{t\ge 0}$ is a local martngale. Let $T_n,n\in \hN$ be an increasing sequence of stopping times such that almost surely $T_n\rightarrow \infty$ and $(M_{t\wedge T_n})_{t\ge 0}$ is a martingale. We get
\beaa
\mathbb E[u(t\wedge T_n,X_{t\wedge T_n}^{x})]&=&u(0,x)+\mathbb E[\int_0^{t\wedge T_n}(\frac{1}{2}Lu+\pa_s u)(s,X_s^x)ds].
\eeaa
By using the dominated convergence theorem, we get
\beaa
\mathbb E[u(t,X_{t}^{x})]&=&u(0,x)+\mathbb E[\int_0^{t}(\frac{1}{2}Lu+\pa_s u)(s,X_s^x)ds].
\eeaa
Applying the above equality to $\phi(t,X_t^x)$, we obtain  
\beaa
\mathbb E[\phi(t,X_t^x)]&=&\phi(0,x)+\mathbb E[\int_0^{\ts}(\frac{1}{2}L\phi+\pa_s \phi)(s,X_s^x)ds]\\
&=&\phi(0,x)+\int_0^{\ts} \mathbb E[ (\frac{1}{2}L\phi+\pa_s \phi)(s,X_s^x)]ds.
\eeaa
We now look at the term $ \mathbb E[ (\frac{1}{2}L\phi+\pa_s \phi)(s,X_s^x)]$ with $g(s,x)=(P_{T-s}f)(x)=\mathbb E[f(X_t^x)]=\int \rho(x,y,s)f(y)dy$, 
\beaa
 \mathbb E[ (\frac{1}{2}L\phi+\pa_s \phi)(s,X_s^x)]&=&\mathbb E[g\Gamma_2(\log g,\log g)+g\Gamma^z_{2}(\log g,\log g) ]\\
 &&+\mathbb E[g\Gamma_{1}(\log g,\Gamma_1^z(\log g,\log g))-g\Gamma_1^z(\log g,\Gamma_{1}(\log g,\log g)) ].
\eeaa
By using the above Lemma \ref{non-commutative lemma}, let $h=\log g$ we get 
\beaa
&&\mathbb E[ (\frac{1}{2}L\phi+\pa_s \phi)(s,X_s^x)]\\
&=&\int g\rho \left( \Gamma_2(h,h)+\Gamma_2^z(h,h)  + \frac{\nabla\cdot (\rho zz^{\ts}\Gamma_{\nabla(aa^{\ts})}(h,h))}{\rho } -\frac{\nabla\cdot (\rho aa^{\ts}\Gamma_{\nabla(zz^{\ts})}(h,h))}{\rho } \right)dy\\
&=&\int g\rho \left( \Gamma_2(h,h)+\tilde\Gamma_2^{z,\rho }(h,h)\right)dy.
\eeaa
Applying Theorem \ref{thm: generalized gamma 2 z with drift} here with $\pi=\rho (s,\cdot,\cdot)$ as the transition kernel function, we get a time dependent version of Theorem \ref{thm: generalized gamma 2 z with drift}. Assuming that the following bound 
is satisfied where the bound $\kappa_s$ depends on kernel $\rho (s,\cdot,\cdot)$,
\beaa
\mathfrak{R}(\nabla f,\nabla f)\ge \kappa_s(\Gamma_{1}(f,f)+\Gamma_1^z(f,f) ).
\eeaa 
We then conclude with the following bound
\beaa
\mathbb E[ (\frac{1}{2}L\phi+\pa_s \phi)(s,X_s^x)]&\ge &\int \rho(s,x,y)g\kappa_s(\Gamma_{1}(h,h)(y)+\Gamma_1^z(h,h)(y))dy\\
&=&\int p(s,x,y)g\kappa_s(\Gamma_{1}(h,h)(y)+\Gamma_1^z(h,h)(y))\Vol(y)dy \\
&\ge &P_s(\kappa_sg(\Gamma_{1}(\log g,\log g)+\Gamma_1^z(\log g,\log g) )).
\eeaa
Plugging into the time integral $\int_0^{\ts} \mathbb E[ (\frac{1}{2}L\phi+\pa_s \phi)(s,X_s^x)]ds$, the proof follows.
	\qed
\end{proof}
\begin{rem}
 We prove the entropic inequality Theorem \ref{thm:entropic inequality} in this section without the the assumption: $\Gamma_1(f,\Gamma_1^z(f,f))=\Gamma_1^z(f,\Gamma_1(f,f))$. A similar entropic inequality under the assumption $\Gamma_1(f,\Gamma_1^z(f,f))=\Gamma_1^z(f,\Gamma_1(f,f))$ is first proved in \cite{BaudoinGarofalo09}[Proposition 4.5] and [Theorem 5.2]. With this new inequality Theorem \ref{thm:entropic inequality} in hand, similar gradient estimates and other inequalities from \cite{BaudoinGarofalo09} follow. We leave them for future studies. Proposition 4.5 in \cite{BaudoinGarofalo09} is based on a point-wise estimate given the commutative assumption of $\Gamma_1$ and $\Gamma_1^z$. We remove the commutative assumption and our estimate is in a weak form, which is presented in the above Lemma \ref{non-commutative lemma}.
\end{rem}{}

\appendix
\section{Appendix: Degenerate SDEs and Sub-Riemannian manifold}\label{section:Ricci}
In appendix, we briefly illustrate the formulation of degenerate diffusion process and sub-Riemannian geometry. 

For a smooth connected $n+m$ dimensional Riemannian manifold $\mathbb M^{n+m}$, we denote $T\mathbb M^{n+m}$ as the tangent bundle of $\mathbb M^{n+m}$ and denote $\tau$ as a sub-bundle of $T\mathbb M^{n+m}$. The sub-Riemannian structure associated with the sub-bundle $\tau$ on $\mathbb M^{n+m}$ is denoted as $(\tau,g_{\tau})$, where $g_{\tau}(\cdot,\cdot)$ is the metric associated with the sub-bundle $\tau$. In particular, if we take distribution $\tau$ to be the horizontal sub-bundle, denoted as $\mathcal H$, of the tangent bundle $T\mathbb M^{n+m}$ (see \cite{BaudoinGarofalo09, baudoin2014book} for more details), then we denote the sub-Riemannian structure as $(\mathbb M^{n+m},\cH,g_{\cH})$. In this paper, we will not distinguish distribution $\tau$ and  $\mathcal H$ and call it horizontal sub-bundle. We will assume that the horizontal distribution $\cH$ is bracket generating (with any steps). And the distribution $\cH$ has dimension $n$. 

For a vector field $b\in\hR^{n+m}$ and a general matrix $a\in \hR^{(n+m)\times n}$, we denote $a=(a_1,a_2,\cdots,a_n)$ with each $a_i, i=1\cdots,n$, as a $n+m$-dimensional column vector. For any Stratonovich SDE, 
\bea\label{SDE framework}
dX_t=b(X_t)dt+\sqrt{2}\sum_{i=1}^n  a_i(X_t)\circ dB_t^i,
\eea
 where $(B_t^1,B_t^2,\cdots, B_t^n)$ is a $n$-dimensional Brownian motion in $\hR^n$ and $a_i$ has local coordinates  $a_i(x)=\sum_{\hat i=1}^{n+m}a_{\hat ii}(x)\frac{\partial}{\partial x_{\hat i}}$.
We consider \eqref{SDE framework} as the SDE associated with a given sub-Riemannian structure, which is defined through the Lie algebra spanned by the driving vector fields of the SDE $\{a_1,a_2,\cdots,a_n\}$. In general, we assume that $\mathcal H:=\{a_1,a_2,\cdots,a_n\}$ is of rank $n$ and satisfies the bracket generating condition (or H\"ormander condition). To be precise, for any $x\in \mathbb M^{n+m}$, the Lie brackets of $\{a_1(x),a_2(x)$, $\cdots,a_n(x)\}$ spans the whole tangent space at $x$ with dimension $n+m$. We define the manifold $\mathbb M^{n+m}$ as the subspace of $\hR^{n+m}$, where the diffusion process $X_t$ lives on. This spaces is described as the triple $(\mathbb M^{n+m},\cH,g_{\cH})$, and we denote  $\cH$ as the $n$-dimensional horizontal distribution of the tangent bundle $T\mathbb M^{n+m}$ generated by the vector fields $\{(a_1(x),a_2(x),\cdots,a_n(x)\}$. 
In this paper, we consider the case where the generator of the diffusion process \eqref{SDE framework} coincides with the horizontal Laplacian operator (or sub-Laplacian operator) associated with the sub-Riemannian structure $(\mathbb M^{n+m},\cH,g_{\cH})$. Furthermore, we assume that there exists symmetric and invariant  volume measure associated with the horizontal Laplacian operator. The Stratonovich SDE \eqref{SDE framework} without drift ($dt$) term could be treated as a special case, where the horizontal Laplacian can be presented as the sum of squares of the horizontal vector fields in $\cH$. In particular, we consider the precise metric defined through the diffusion matrix $a$, which could be seen as an analogue for non-degenerate SDEs on Riemannian manifolds. The problem is that the rank of $aa^{\ts}$ is $n$, thus the $(n+m)\times(n+m)$ matrix $aa^{\ts}$ is  degenerate and cannot serve as a metric.
We thus introduce the following metric, which is to formulate this sub-Riemannian structure in Euclidean space. 
 \begin{defn}
 	Consider an orthonormal basis $c=\{c_{n+1}(x),\cdots, c_{n+m}(x)\}$ in $\mathbb{R}^{n+m}$, such that $a_i^{\ts}c_j=0$ for any $1\leq i\leq n$, $n+1\leq j\leq m+n$. We define a metric $g=(aa^\ts+cc^\ts)^{-1}=(aa^{\ts})^{\dd}+cc^{\ts}$, and metric on the horizontal sub-bundle $g_\tau=(aa^{\ts})^{\dd}$, the pseudo-inverse of  matrix $aa^{\ts}$, on manifold $\mathbb M^{n+m}$.
 \end{defn}
 The above definition is based on the following lemma.
\begin{lem}
	The metric is $g=(aa^\ts+cc^\ts)^{-1}=(aa^{\ts})^{\dd}+cc^{\ts}$.
\end{lem}
\begin{proof}
	For rank $n$ matrix $aa^{\ts}$, we denote its eigenvalue decomposition and the corresponding pseudo-inverse $(aa^{\ts})^{\dd}$ as 
	\beaa
	aa^{\ts}=\sum_{i=1}^n\lambda_i V_iV_i^{\ts}, \quad (aa^{\ts})^{\dd}=\sum_{i=1}^n\frac{1}{\lambda_i} V_iV_i^{\ts}.
	\eeaa
Thus we have $aa^{\ts}+cc^{\ts}=\sum_{i=1}^n\lambda_i V_iV_i^{\ts}+\sum_{j=n+1}^{n+m}c_jc_j^{\ts}$. Furthermore, we have 
\beaa
aa^{\ts}+cc^{\ts}&=&(V_1,\cdots,V_n,c_{n+1},\cdots,c_{n+m})\begin{pmatrix}
	\Lambda_n &\\
	&\mathbf{I}_m
\end{pmatrix}(V_1,\cdots,V_n,c_{n+1},\cdots,c_{n+m})^{\ts},\\
(aa^{\ts}+cc^{\ts})^{-1}&=&(V_1,\cdots,V_n,c_{n+1},\cdots,c_{n+m})\begin{pmatrix}
	\Lambda_n^{-1} &\\
	&\mathbf{I}_m
\end{pmatrix}(V_1,\cdots,V_n,c_{n+1},\cdots,c_{n+m})^{\ts},
\eeaa
where we denote $\Lambda_n=\mathbf{diag}(\lambda_1,\cdots,\lambda_n)$ as the diagonal matrix for eigenvalues $\lambda_i's$ and $\mathbf{I}_m$ as the $m$-dimensional identity matrix. Thus the proof follows directly with
\beaa
(aa^{\ts}+cc^{\ts})^{-1}=\sum_{i=1}^n\frac{1}{\lambda_i}V_iV_i^{\ts}+\sum_{j=n+1}^{n+m}c_jc^{\ts}_j=(aa^{\ts})^{\dd}+cc^{\ts}.   
\eeaa
 \qed 
\end{proof}
With the new metric introduced above, we have the following lemma.
  \begin{lem}\label{lemma: ON frame}
 	The vectors $\{a_1,\cdots, a_n\}$ are orthonormal basis under the metric $g=(aa^{\ts})^{\dd}+cc^{\ts}$.
 \end{lem}
 \begin{proof}
 	We just need to prove for $a=(a_1,\cdots,a_n)$ with each $(a_i)_{(n+m)\times 1}$, we have
 	\beaa
a^{\ts}g a=a^{\ts}(aa^{\ts})^{\dd}a=\textbf{Id}_{n\times n}. 	
 	\eeaa
	Notice $a_i^{\ts}c_j=0$, then we only need to prove $a^{\ts}(aa^{\ts})^{\dd}a=\textbf{Id}_{n\times n}$. Let us denote $a^{\ts}(aa^{\ts})^{\dd}a=B$, then we have 
 	\beaa
 	aa^{\ts}(aa^{\ts})^{\dd}aa^{\ts}&=&aBa^{\ts}\\
 	aa^{\ts}&=&aBa^{\ts}\\
 	a^{\ts}aa^{\ts}a&=&a^{\ts}aBa^{\ts}a\\
 	(a^{\ts}a)^{-1}a^{\ts}aa^{\ts}a(a^{\ts}a)^{-1}&=&B\\
 	\textbf{Id}_{n\times n} &=&B,
 	\eeaa
 	where the second equality follows from the property of pseudo-inverse matrix and the last step follows from the fact that $a^{\ts}a$ is a non-degenerate $n\times n$ matrix, hence invertible. The proof then follows directly. \qed
 \end{proof}
  
  We are now ready to introduce the following definition.
 
 \begin{defn}\label{def: sub-Rie str}
  Define $(\mathbb M^{n+m},\tau,g_{\tau})$ as the sub-Riemannian structure associated with the degenerate SDE \eqref{SDE framework},
 where $g_\tau=(aa^{\ts})^{\dd}$ denotes the horizontal metric, i.e. metric $g$ restricts on the horizontal bundle $\tau$. And we denote ${\nabla^{R}}$ as the Levi-Civita connection on $\mathbb M^{n+m}$ associated with our metric $g=(aa^{\ts})^{\dd}+cc^{\ts}$, and let $P^{\tau}\nabla^R$ as the projection of the connection on the horizontal distribution $\tau$.  In particular, in our framework, we have $P^{\tau}\nabla^R f=aa^{\ts}\nabla f$, for any function $f:\mathbb M^{n+m}\rightarrow \hR$. Where $\nabla$ is the Euclidean gradient in $\hR^{n+m}.$
 \end{defn}
\begin{rem}
In Lemma \ref{lemma: ON frame}, we show that $\{a_1,a_2,\cdots,a_n\}$ are orthonormal basis for horizontal distribution $\tau$  under our metric $g$. In particular, we have 
\beaa
aa^T\nabla f=(a_1,\cdots,a_n)\begin{pmatrix}
	a_1\\
	\cdots \\
	a_n
\end{pmatrix}f=\sum_{i=1}^n(a_if)a_i\in\tau,
\eeaa
which gives the local representation of $P^{\tau}\nabla^Rf.$
\end{rem}
 
To demonstrate the definition clearly, we give the following example. On the Heisenberg group $\mathbb H^1$ , we know that $X=\frac{\partial}{\partial x_1}-\frac{1}{2}x_2\frac{\partial}{\partial x_3},\quad Y=\frac{\partial}{\partial x_2}+\frac{1}{2}x_1\frac{\partial}{\partial x_3},\quad Z=\frac{\partial}{\partial x_3}
$ forms an orthonormal basis for the tangent bundle of $\mathbb H^1$. In particular, $X$ and $Y$ generate the horizontal distribution $\tau$. If we start with the following SDE 
\bea\label{Heisenberg SDE}
dW_t=X\circ  dB_t^1+Y\circ dB_t^2,
\eea
then we know $W_t=(B_t^1,B_t^2,\frac{1}{2} \int_0^tB_s^1dB_s^2-B_s^2dB_s^1)$, which 
 is the horizontal Brownian motion on the Heisenberg group $\mathbb H^1$. The generator of the horizontal Brownian motion and the sub-Laplaican operator are the same which is given by $\Delta_{\cH}=X^2+Y^2$, and the volume measure associated with $\Delta_{\cH}$ is the Lebesgue measure on the Heisenberg group with volume element equals 1.  Then $W_t$ is a diffusion process in $\hR^3$. In terms of our general sub-Riemannian structure introduced above, we can define
 \beaa
 a=\begin{pmatrix}
 	1&0\\
 	0&1\\
 	-\frac{x_2}{2}&\frac{x_1}{2}
 \end{pmatrix}=(a_1,a_2)=(X,Y),
\qquad 
c=\frac{1}{\sqrt{\frac{x_1^2}{4}+\frac{x_2^2}{4}+1}}\begin{pmatrix}
\frac{x_2}{2}
\\
-\frac{x_1}{2}\\
1
\end{pmatrix},
 \eeaa
 and 
 \beaa
  g_{\mathbb H^1,\tau}=(aa^{T})^{\dd}=\begin{pmatrix}
 	1&0& -\frac{x_2}{2}\\
 	0&1& \frac{x_1}{2}\\
 	-\frac{x_2}{2}&\frac{x_1}{2}& \frac{x_1^2+x_2^2}{4}
 \end{pmatrix}^{\dd},\quad    g_{\mathbb H^1}=(aa^{T})^{\dd}+cc^{\ts}.
 \eeaa 
 In particular, the horizontal gradient is given by 
 \beaa
 aa^{\ts}\nabla f=\begin{pmatrix}
 	Xf\\
 	Yf\\
 	-\frac{x_2}{2}Xf+\frac{x_1}{2}Yf
 \end{pmatrix}=Xf\begin{pmatrix}
 	1\\
 	0\\
 	-\frac{x_2}{2}
 \end{pmatrix}+Yf\begin{pmatrix}
 	0\\
 	1\\
 	\frac{x_1}{2}
 \end{pmatrix}=(Xf)X+(Yf)Y.
 \eeaa
 Thus the sub-Riemannian structure associated with Stratonovitch SDE \eqref{Heisenberg SDE} is just $(\mathbb H^1,\tau,    g_{\mathbb H^1,\tau})$, where $g_{\mathbb H^1,\tau}$ is the restriction of metric $g_{\mathbb H^1}$ on the horizontal sub-bundle $\tau$. Different from the standard construction of Brownian motion on a given Riemannian (sub-Riemannian) manifold by Ells-Elworthy-Malliavin \cite{elworthy1982, Malliavin}, we can directly define our diffusion on the manifold $\mathbb M^{n+m}$ by \eqref{SDE framework} without doing projection from the orthonormal frame bundles. It is because that the new metric $g=(aa^{\ts})^{\dd}+cc^{\ts}$, and $\{a_1,a_2,\cdots,a_n\}$ are globally defined orthonormal basis of the (horizontal) sub-bundle on the tangent bundle  $T\mathbb M^{n+m}$. Essentially, we first define \eqref{SDE framework} in $\hR^{n+m}$, and then introduce the associated sub-Riemannian structure.
\begin{rem}
	Comparing to the definition of horizontal Brownian motion introduced in \cite{baudoin2019integration},   the sub-Riemannian structure comes first  with a totally geodesic Riemannian foliation structure, and then the SDE \eqref{Heisenberg SDE} is defined on the given totally geodesic Riemannian foliation. In the current setting, we directly define the degenerate diffusion process by a first given matrix $a$, then we define the sub-Riemannian structure by introducing the new metric $(aa^T)^{\dd}+cc^{\ts}$. 
\end{rem}

\subsection{Proof of Gradient flow assumption}\label{section4}
In this subsection, we demonstrate that equation \eqref{FPE} is in fact a Fokker-Planck equation of SDE \eqref{SDE framework}.  
\begin{lem}\label{lem:canonical drift}
    Consider the drift diffusion process 
    \bea\label{label1}
    dX_t=b(X_t)dt+\sqrt{2} a(X_t)\circ dB_t,
    \eea 
Suppose that $b$, $a$, $\pi$ satisfy
\beaa 
a\otimes \nabla a-b=-aa^{\ts}\nabla \log \Vol.
\eeaa
Then the Fokker-Planck equation of $X_t$ satisfies
\begin{equation*}
    \pa_t \rho(t,x)=\nabla\cdot\Big(\rho(t,x)\Big(a(x)a(x)^{\ts}\Big)\nabla\log\frac{\rho(t,x)}{\pi(x)}\Big). 
\end{equation*} 
\end{lem}{}
\begin{proof}
{ Recall that we denote $\{a_1,\cdots, a_n\}$ as the column vectors of matrix a.}
For Stratonovich SDE \eqref{label1}, we can write 
 \beaa
    dX_t=b(X_t)dt+\sqrt{2} \sum_{i=1}^na_i(X_t)\circ dB_t^i.
    \eeaa 
{ According to \cite{baudoinflow}[Appendix 7], the corresponding It\^o SDE is }
\beaa
dX_t=\sqrt{2}\sum_{i=1}^n a_i dB_t^i+(\sum_{i=1}^n\nabla_{a_i}a_i+b)dt.
\eeaa
Thus the Fokker-Plank equation (Kolmogorov forward equation) satisfies
{\beaa
\partial_t\rho(t,x) &=&\sum_{i'=1}^{n+m}\sum_{j'=1}^{n+m} \frac{\partial^2}{\partial x_{i'}\partial x_{j'}}((aa^{\ts})_{i'j'} \rho)-\nabla\cdot(( \sum_{i=1}^n\nabla_{a_i}a_i +b) \rho)\\
&=&\nabla \cdot (aa^{\ts}\nabla\rho)+\nabla\cdot \left(\rho (\sum_{j'=1}^{n+m}\frac{\partial}{\partial x_{j'}}(aa^{\ts})_{i'j'})_{i'=1}^{n+m}-\rho\sum_{i=1}^n\nabla_{a_i}a_i -b\rho\right)\\
&=&\nabla \cdot (aa^{\ts}\nabla\rho)+\nabla\cdot (\rho (a\otimes \nabla a-b)).
\eeaa}
Namely, we have 
\bea \label{F-P equation 1}
\partial_t\rho(t,x) =\nabla \cdot (aa^{\ts}\nabla\rho)+\nabla\cdot (\rho (a\otimes \nabla a-b)).
\eea
Plugging in the relation $
a\otimes \nabla a-b=-aa^{\ts}\nabla \log \Vol$,
we have
\bea \label{F-P equation 1}
\begin{split}
\partial_t\rho(t,x) =&\nabla \cdot (aa^{\ts}\nabla\rho)-\nabla\cdot (\rho aa^{\ts}\nabla\log\pi)\\
=&\nabla \cdot (\rho aa^{\ts}\nabla\log\rho)-\nabla\cdot (\rho aa^{\ts}\nabla\log\pi)\\
=&\nabla \cdot (\rho aa^{\ts}\nabla\log\frac{\rho}{\pi}).
\end{split}
\eea
Here we use the fact that $$\rho\nabla\log\rho=\nabla\rho.$$ 
This finishes the proof.
\end{proof}
\begin{example}
 The Lie group $\textbf{SU}(2)$ is a compact connected Lie group, diffeomorphic to the 3-sphere $\mathbb S^3$. Following the construction of the left-invariant vector fields in \cite[Section 6.2]{GL2016},
we change the coordinates in terms of coordinate system $(\theta,\phi,\psi)$. We get new left-invariant vector fields on $\textbf{SU}(2)$, with
\beaa
X&=&\cos \psi \frac{\partial}{\partial \theta} +\frac{\sin \psi}{\sin \theta}\frac{\partial}{\partial \phi}-\cos \theta \frac{\sin \psi}{\sin \theta}\frac{\partial}{\partial \psi},\\
Y&=& -\sin \psi \frac{\partial}{\partial \theta}+\frac{\cos \psi}{\sin \theta} \frac{\partial}{\partial \phi}-\cos\theta \frac{\cos \psi}{\sin \theta}\frac{\partial}{\partial \psi},\\
Z&=& \frac{\partial}{\partial\psi}.
\eeaa
 Thus we have $a=(a_1,a_2)=(X,Y)$ in the new coordinate system. We define the metric $g=(aa^{\ts})^{\dd}$. Here $X,Y$ are orthonormal basis for the horizontal bundle generated by $X,Y$ under metric $(aa^{\ts})^{\dd}$. According to \cite[Lemma 6.4]{GL2016}, the invariant measure on $\textbf{SU}(2)$ has the form of $\mu=\sin(\theta)d\theta \wedge d\phi \wedge d\psi$. It is easy to check that the above Lemma is satisfied for $b=0$, $\pi=\sin(\theta)$, and
\beaa
aa^{\ts}\nabla \log \Vol=-a\otimes\nabla a=\begin{pmatrix}
	\frac{\cos \theta}{\sin \theta}\\
	0\\
	0
\end{pmatrix},
\eeaa
where
\beaa
a=\begin{pmatrix}
	\cos \psi &  -\sin \psi\\
	\frac{\sin \psi}{\sin \theta}& \frac{\cos \psi}{\sin \theta} \\
	-\cos \theta \frac{\sin \psi}{\sin \theta}& -\cos \theta \frac{\cos \psi}{\sin \theta}
\end{pmatrix}.
\eeaa
\end{example}

\bibliographystyle{alpha}

\begin{thebibliography}{10}

\bibitem{AC}
Arnold, A.; Carlen, E. 
\newblock{A generalized Bakry--{\'E}mery condition for non-symmetric diffusions.}
\newblock In Proceedings of the EQUADIFF 99---International Conference on Differential Equations,
Berlin, Germany, 1--7 August 1999; pp. 732--734.

\bibitem{agrachev2012}
Agrachev, A.; Barilari, D.; Boscain, U.
\newblock On the \textsc{H}ausdorff volume in sub-Riemannian geometry.
\newblock {\em Calc. Var. Partial Differ. Equations} \textbf{2012}, \emph{43}, 355--388.

\bibitem{agrachev2009optimal}
Agrachev, A.; Lee, P.
\newblock Optimal transportation under nonholonomic constraints.
\newblock {\em Trans. Am. Math. Soc.} \textbf{2009}, \emph{361}, 6019--6047.
 


\bibitem{ARNOLD20186843}
Arnold, A.; Einav, A.; W{\"o}hrer, T.
\newblock On the rates of decay to equilibrium in degenerate and defective
 \textsc{F}okker-\textsc{P}lanck equations.
\newblock {\em J. Differ. Equations} \textbf{2018}, \emph{264}, 6843--6872.

\bibitem{AE}
Arnold, A.; Erb, J.
\newblock{Sharp entropy decay for hypocoercive and non-symmetric Fokker--Planck equations with linear drift.} \emph{arXiv} \textbf{2014}, arXiv:1409.5425.


\bibitem{bakryemery1985}
Bakry, D.; {\'E}mery, M.
\newblock Diffusions hypercontractives.
\newblock In {\em S{\'e}minaire de Probabilit{\'e}s XIX 1983/84}; Springer: Berlin/Heidelberg, Germany, 1985; pp. 177--206.


\bibitem{barilari2013formula}
Barilari, D.; Rizzi, L.
\newblock A formula for Popp’s volume in sub-Riemannian geometry.
\newblock {\em Anal. Geom. Metr. Spaces} \textbf{2013}, \emph{1}, 42--57.


\bibitem{barlownualart}
Barlow, M.; Nualart, D.
\newblock Lectures on Probability Theory and Statistics.
\newblock{\em Ecole d'Ete de Probabilites de Saint-Flour XXV}; Springer: Berlin/Heidelberg, Germany, 1995.


\bibitem{baudoinflow}
Baudoin, F.
\newblock \emph{An Introduction to the Geometry of Stochastic Flows}; 
\newblock World Scientific: Singapore, 2004.



\bibitem{baudoin2014book}
Baudoin, F.
\newblock Sub-Laplacians and hypoelliptic operators on totally geodesic
 \textsc{R}iemannian foliations. \emph{arXiv} \textbf{2014}, arXiv:1410.3268.

\bibitem{baudoin2016wasserstein}
Baudoin, F.
\newblock Wasserstein contraction properties for hypoelliptic diffusions. \emph{arXiv} \textbf{2016}, arXiv:1602.04177.

\bibitem{Baudoin2017}
F.~Baudoin.
\newblock Bakry--{\'E}mery meet Villani.
\newblock {\em J. Funct. Anal.} \textbf{2017}, \emph{273}, 2275--2291.


\bibitem{baudoin2012log}
Baudoin, F.; Bonnefont, M.
\newblock Log-Sobolev inequalities for subelliptic operators satisfying a
 generalized curvature dimension inequality.
\newblock {\em J. Funct. Anal.} \textbf{2012}, \emph{262}, 2646--2676.

\bibitem{BBG}
Baudoin, F.; Bonnefont, M.; Garofalo, N.
\newblock A sub-{{Riemannian}} curvature-dimension inequality, volume doubling
 property and the {{Poincar\'e}} inequality.
\newblock {\em Math. Ann.} \textbf{2014}, \emph{358}, 833--860.


\bibitem{baudoin2015log}
Baudoin, F.; Feng, Q.
\newblock Log-Sobolev inequalities on the horizontal path space of a totally
 geodesic foliation. \emph{arXiv} \textbf{2015}, arXiv:1503.08180.


\bibitem{baudoin2019integration}
Baudoin, F.; Feng, Q.; Gordina, M.
\newblock Integration by parts and quasi-invariance for the horizontal
 \textsc{W}iener measure on foliated compact manifolds.
\newblock {\em J. Funct. Anal.} \textbf{2019}, \emph{277}, 1362--1422.


\bibitem{BaudoinGarofalo09}
Baudoin, F.; Garofalo, N.
\newblock Curvature-dimension inequalities and \textsc{R}icci lower bounds for
 sub-Riemannian manifolds with transverse symmetries.
\newblock {\em J. EMS} \textbf{2017}, \emph{19}, 151--219.


\bibitem{baudoin2015subelliptic}
Baudoin, F.; Cecil, M.
\newblock The subelliptic heat kernel on the three-dimensional solvable Lie groups.
\newblock {\em Forum Math.} \textbf{2015}, \emph{27}, 2051--2086.

\bibitem{baudoin2019gamma}
Baudoin, F.; Gordina, M.; Herzog, D.P.
\newblock Gamma calculus beyond \textsc{V}illani and explicit convergence
 estimates for Langevin dynamics with singular potentials.
\newblock {\em Arch. Ration. Mech. Anal.} \textbf{2021}, \emph{241}, 765--804.

\bibitem{BGK}
Baudoin, F.; Grong, E.; Kuwada, K.; Thalmaier, A.
\newblock Sub-{{Laplacian}} comparison theorems on totally geodesic Riemannian foliations.
\newblock {\em Calc. Var.} \textbf{2019}, \emph{58}, 130.

\bibitem{baudoinwang2012}
Baudoin, F.; Wang, J.
\newblock Curvature dimension inequalities and subelliptic heat kernel gradient
 bounds on contact manifolds.
\newblock {\em Potential Anal.} \textbf{2014}, \emph{40}, 163--193.


\bibitem{BNOT}
Baudoin, F.; Nualart, E.; Ouyang, C.; Tindel, S.
\newblock On probability laws of solutions to differential systems driven by a fractional Brownian motion.
\newblock{\em Ann. Probab.} \textbf{2016}, \emph{44}, 2554--2590.

\bibitem{benarous}
Arous, B.; L{\'e}andre, R.
\newblock D{\'e}croissance exponentielle du noyau de la chaleur sur la diagonale (II).
 \newblock{\em Probab. Theory Relat. Fields} \textbf{1991}, \emph{90}, 377--402.



\bibitem{Bismut}
Bismut, J.M.
\newblock{Martingales, the Malliavin calculus and hypoellipticity under general H{\"o}rmander's conditions}.
\newblock In {\emph{Zeitschrift f{\"u}r Wahrscheinlichkeitstheorie und Verwandte Gebiete}}; Springer: Berlin/Heidelberg, Germany, 1981; Volume 56, pp. 469--505.


\bibitem{eldredge2018}
Eldredge, N.; Gordina, M.; Saloff-Coste, L.
\newblock Left-invariant geometries on SU(2) are uniformly doubling.
\newblock {\em Geom. Funct. Anal.} \textbf{2018}, \emph{28}, 1321--1367.

\bibitem{elworthy1982}
Elworthy, K.D.
\newblock {\em Stochastic Differential Equations on Manifolds}; Cambridge University Press: Cambridge, UK, 1982; Volume 70.


\bibitem{Feng}
Feng, Q.
\newblock Harnack inequalities on totally geodesic foliations with transverse
 {{Ricci}} flow. \emph{arXiv} \textbf{2017}, arXiv:1712.02275.




\bibitem{FengLi2021}
Feng, Q.; Li, W.
\newblock Hypoelliptic entropy dissipation for stochastic differential equations. \emph{arXiv} \textbf{2021}, arXiv:2102.00544.


\bibitem{figalli2010mass}
Figalli, A.; Rifford, L.
\newblock Mass transportation on sub-Riemannian manifolds.
\newblock {\em Geom. Funct. Anal.} \textbf{2010}, \emph{20}, 124--159.

\bibitem{GL2016}
Gordina, M.; Laetsch, T.
\newblock Sub-Laplacians on sub-Riemannian manifolds.
\newblock {\em Potential Anal.} \textbf{2016}, \emph{44}, 811--837.

\bibitem{GL2017}
Gordina, M.; Laetsch, T.
\newblock A convergence to Brownian motion on sub-Riemannian manifolds.
\newblock {\em Trans. Am. Math. Soc.} \textbf{2017}, \emph{369}, 6263--6278.

\bibitem{Grong_2015_1}
Grong, E.; Thalmaier, A.
\newblock Curvature-dimension inequalities on sub-Riemannian manifolds obtained
 from Riemannian foliations: Part I.
\newblock {\em Math. Z.} \textbf{2015}, \emph{282}, 99--130.

\bibitem{Grong_2015_2}
Grong, E.; Thalmaier, A.
\newblock Curvature-dimension inequalities on sub-Riemannian manifolds obtained
 from Riemannian foliations: Part II.
\newblock {\em Math. Z.} \textbf{2015}, \emph{282}, 131--164.




\bibitem{Hormander}
H\"ormander, L.
\newblock{Hypoelliptic second-order differential equations.}
\newblock \emph{Acta Math.} \textbf{1967}, \emph{119}, 147--171.


\bibitem{inglis2009logarithmic}
Inglis, J.; Papageorgiou, I.
\newblock Logarithmic Sobolev inequalities for infinite-dimensional H{\"o}rmander type generators on the Heisenberg group.
\newblock {\em Potential Anal.} \textbf{2009}, \emph{31}, 79--102.




\bibitem{Juilleet}
Juillet, N.
\newblock Diffusion by optimal transport in {{Heisenberg}} groups.
\newblock {\em Calc. Var. Partial Differ. Equations} \textbf{2014}, \emph{50}, 693--721.

\bibitem{jungel2016entropy}
J{\"u}ngel, A.
\newblock {\em Entropy Methods for Diffusive Partial Differential Equations}; 
\newblock Springer: Berlin/Heidelberg, Germany, 2016.

\bibitem{KL}
Khesin, B.; Lee, P.
\newblock A nonholonomic {{Moser}} theorem and optimal transport.
\newblock {\em J. Symplectic Geom.} \textbf{2009}, \emph{7}, 381--414.

\bibitem{Karatzas}
Karatzas, I.; Shreve, S.E.
\newblock Brownian Motion and Stochastic Calculus, 2nd ed.;
\newblock{\em Graduate Texts in Mathematics}; Springer: New York, NY, USA, 1991; Volume 113. 



\bibitem{Lafferty}
Lafferty, J.D.
\newblock The {{Density Manifold}} and {{Configuration Space Quantization}}.
\newblock {\em Trans. Am. Math. Soc.} \textbf{1988}, \emph{305}, 699--741.


\bibitem{LiG}
Li, W.
\newblock Transport information geometry: Riemannian calculus on probability simplex. 
\newblock {\em Inf. Geom.} \textbf{2022}, \emph{5}, 161–207.

\bibitem{Li2019_diffusion}
Li, W.
\newblock Diffusion {{Hypercontractivity}} via {{Generalized Density
 Manifold}}. \emph{arXiv} \textbf{2019}, arXiv:1907.12546.


\bibitem{Lott_Villani}
Lott, J.; Villani, C.
\newblock Ricci {{Curvature}} for {{Metric}}-{{Measure Spaces}} via {{Optimal
 Transport}}.
\newblock {\em Ann. Math.} \textbf{2009}, \emph{169}, 903--991.

\bibitem{Malliavin}
Malliavin, P.
\newblock Stochastic Analysis. In {\em Grundlehren der
 Mathematischen Wissenschaften [Fundamental Principles of Mathematical
 Sciences]};
\newblock Springer: Berlin/Heidelberg, Germany, 1997; Volume 313.

\bibitem{MV}
Markowich, P.A.; Villani, C.
\newblock{On The Trend To Equilibrium For The Fokker--Planck Equation: An Interplay Between Physics And Functional Analysis. Physics and Functional Analysis}. \emph{Mat. Contemp.} \textbf{1999}, \emph{19}, 1--29.




\bibitem{otto2001}
Otto, F.
\newblock The geometry of dissipative evolution equations the porous medium
 equation.
\newblock {\em Commun. Partial Differ. Equations} \textbf{2001}, \emph{26}, 101--174.

\bibitem{OV}
Otto, F.; Villani, C.
\newblock Generalization of an {{Inequality}} by {{Talagrand}} and {{Links}}
 with the {{Logarithmic Sobolev Inequality}}.
\newblock {\em J. Funct. Anal.} \textbf{2000}, \emph{173}, 361--400.


\bibitem{Stroock}
Stroock, D.W.
\newblock Partial differential equations for probabilists.
\newblock{\em Cambridge Studies in Advanced Mathematics}; Cambridge University Press: Cambridge, UK, 2008; Volume 112.




\bibitem{woit2017quantum}
Woit, P.
\newblock \emph{Quantum Theory, Groups and Representations};
\newblock Springer: Berlin/Heidelberg, Germany, 2017.



\bibitem{Strum}
Sturm, K.-T.
\newblock On the {{Geometry}} of {{Metric Measure Spaces}}.
\newblock {\em Acta Math.} \textbf{2006}, \emph{196}, 65--131.


\bibitem{wang1997logarithmic}
Wang, F.-Y.
\newblock Logarithmic Sobolev inequalities on noncompact Riemannian manifolds.
\newblock {\em Probab. Theory Relat. Fields} \textbf{1997}, \emph{109}, 417--424.

\end{thebibliography}



\end{document}